\PassOptionsToPackage{unicode}{hyperref}
\PassOptionsToPackage{hyphens}{url}
\PassOptionsToPackage{dvipsnames,svgnames,x11names}{xcolor}
\documentclass[12pt]{article}
\usepackage{amsmath,amssymb, amsthm, bm}
\usepackage{mathrsfs}
\usepackage{hyperref}
\usepackage{algorithm}
\usepackage{algorithmic}
\usepackage{enumitem}
\usepackage{multirow}
\usepackage{tabularx}
\usepackage{subcaption}
\usepackage{caption}
\usepackage{comment}
\usepackage{iftex}
\usepackage{siunitx}
\usepackage{xr-hyper}              
\ifPDFTeX
  \usepackage[T1]{fontenc}
  \usepackage[utf8]{inputenc}
  \usepackage{textcomp} 
\else 
  \usepackage{unicode-math}
  \defaultfontfeatures{Scale=MatchLowercase}
  \defaultfontfeatures[\rmfamily]{Ligatures=TeX,Scale=1}
\fi
\usepackage{lmodern}
\ifPDFTeX\else  
\fi
\IfFileExists{upquote.sty}{\usepackage{upquote}}{}
\IfFileExists{microtype.sty}{
  \usepackage[]{microtype}
  \UseMicrotypeSet[protrusion]{basicmath} 
}{}
\makeatletter
\@ifundefined{KOMAClassName}{
  \IfFileExists{parskip.sty}{%
    \usepackage{parskip}
  }{
    \setlength{\parindent}{0pt}
    \setlength{\parskip}{6pt plus 2pt minus 1pt}}
}{
  \KOMAoptions{parskip=half}}
\makeatother
\usepackage{xcolor}
\makeatletter
\ifx\paragraph\undefined\else
  \let\oldparagraph\paragraph
  \renewcommand{\paragraph}{
    \@ifstar
      \xxxParagraphStar
      \xxxParagraphNoStar
  }
  \newcommand{\xxxParagraphStar}[1]{\oldparagraph*{#1}\mbox{}}
  \newcommand{\xxxParagraphNoStar}[1]{\oldparagraph{#1}\mbox{}}
\fi
\ifx\subparagraph\undefined\else
  \let\oldsubparagraph\subparagraph
  \renewcommand{\subparagraph}{
    \@ifstar
      \xxxSubParagraphStar
      \xxxSubParagraphNoStar
  }
  \newcommand{\xxxSubParagraphStar}[1]{\oldsubparagraph*{#1}\mbox{}}
  \newcommand{\xxxSubParagraphNoStar}[1]{\oldsubparagraph{#1}\mbox{}}
\fi
\makeatother

\usepackage{longtable,booktabs,array}
\usepackage{mathtools}
\usepackage{calc} 
\usepackage{etoolbox}
\makeatletter
\patchcmd\longtable{\par}{\if@noskipsec\mbox{}\fi\par}{}{}
\makeatother
\IfFileExists{footnotehyper.sty}{\usepackage{footnotehyper}}{\usepackage{footnote}}
\makesavenoteenv{longtable}
\usepackage{graphicx}
\makeatletter
\def\maxwidth{\ifdim\Gin@nat@width>\linewidth\linewidth\else\Gin@nat@width\fi}
\def\maxheight{\ifdim\Gin@nat@height>\textheight\textheight\else\Gin@nat@height\fi}
\makeatother
\setkeys{Gin}{width=\maxwidth,height=\maxheight,keepaspectratio}
\makeatletter
\def\fps@figure{htbp}
\makeatother

\makeatletter
\@ifpackageloaded{caption}{}{\usepackage{caption}}
\AtBeginDocument{%
\ifdefined\contentsname
  \renewcommand*\contentsname{Table of contents}
\else
  \newcommand\contentsname{Table of contents}
\fi
\ifdefined\listfigurename
  \renewcommand*\listfigurename{List of Figures}
\else
  \newcommand\listfigurename{List of Figures}
\fi
\ifdefined\listtablename
  \renewcommand*\listtablename{List of Tables}
\else
  \newcommand\listtablename{List of Tables}
\fi
\ifdefined\figurename
  \renewcommand*\figurename{Figure}
\else
  \newcommand\figurename{Figure}
\fi
\ifdefined\tablename
  \renewcommand*\tablename{Table}
\else
  \newcommand\tablename{Table}
\fi
}
\@ifpackageloaded{float}{}{\usepackage{float}}
\floatstyle{ruled}
\@ifundefined{c@chapter}{\newfloat{codelisting}{h}{lop}}{\newfloat{codelisting}{h}{lop}[chapter]}
\floatname{codelisting}{Listing}

\makeatother
\makeatletter
\@ifpackageloaded{caption}{}{\usepackage{caption}}
\@ifpackageloaded{subcaption}{}{\usepackage{subcaption}}
\makeatother

\newtheorem{thm}{Theorem}
\newtheorem{lem}{Lemma}
\newtheorem{prop}{Proposition}
\newtheorem{cor}{Corollary}
\newtheorem{defn}{Definition}
\newtheorem{rem}{Remark}

\newcommand{\de}{\delta}

\newcommand{\ff}{\infty}
\newcommand{\ra}{\rightarrow}

\newcommand{\ep}{\epsilon}

\def\theequation{\thesubsection.\arabic{equation}}

\newcommand{\QG}{Q_G}

\let\Q\Rational
\def\Real{\mathrm{I\!R}}

\def\theequation{\thesection.\arabic{equation}}

\ifLuaTeX
  \usepackage{selnolig}  
\fi
\usepackage[]{natbib}
\usepackage{bookmark}

\IfFileExists{xurl.sty}{\usepackage{xurl}}{} 
\hypersetup{
  pdftitle={Title},
  pdfauthor={Author 1; Author 2},
  pdfkeywords={3 to 6 keywords, that do not appear in the title},
  colorlinks=true,
  linkcolor={blue},
  filecolor={Maroon},
  citecolor={Blue},
  urlcolor={Blue},
  pdfcreator={LaTeX via pandoc}}

\newcommand{\anon}{1}

\begin{document}

\def\spacingset#1{\renewcommand{\baselinestretch}%
{#1}\small\normalsize} \spacingset{1}
\newcommand{\Q}{\mathcal{Q}}
\newcommand{\QGc}{\mathcal{Q}_{G}}
\newcommand{\QGn}{\mathcal{Q}_{G,n}}
\newcommand{\QGone}{\mathcal{Q}_{G,1n}}
\newcommand{\QGtwo}{\mathcal{Q}_{G,2n}}
\newcommand{\QEM}{\mathcal{Q}_{\mathrm{EM}}}

\newcommand{\DG}{D_{G}}
\newcommand{\DGn}{D_{G,n}}
\newcommand{\Done}{D_{1n}}
\newcommand{\Dtwo}{D_{2n}}

\newcommand{\Mop}{M}
\newcommand{\Mnop}{M_{n}}
\newcommand{\Mone}{M_{1n}}
\newcommand{\Mtwo}{M_{2n}}

\newcommand{\mix}{f}
\newcommand{\trueg}{g}
\newcommand{\gn}{g_n}
\newcommand{\wk}{w_k}
\newcommand{\hcomp}{h}
\newcommand{\thet}{\theta}
\newcommand{\thetp}{\theta'}

\providecommand{\Env}{\mathsf{Env}}

\newcommand{\RAF}{A}
\newcommand{\resid}{\delta}


\if1\anon
{
  \title{\bf Divergence-Minimization for Latent-Structure Models: 
  Monotone Operators, Contraction Guarantees, and Robust Inference}
  \author{Lei Li
    \hspace{.2cm}\\
    Department of Statistics, George Mason University\\
    and \\
    Anand N. Vidyashankar \\
    Department of Statistics, George Mason University}
  \maketitle
} \fi

\if0\anon
{
  \bigskip
  \bigskip
  \bigskip
  \begin{center}
    {\LARGE\bf Title}
\end{center}
  \medskip
} \fi

\bigskip
\begin{abstract}
We develop a divergence-minimization (DM) framework for robust and efficient inference in latent-mixture models. By optimizing a residual-adjusted divergence, the DM approach recovers EM as a special case and yields robust alternatives
through different divergence choices. We establish that the sample objective decreases monotonically along the iterates, leading the DM sequence to stationary
points under standard conditions, and that at the population level the operator
exhibits local contractivity near the minimizer. Additionally, we verify
consistency and $\sqrt{n}$-asymptotic normality of minimum-divergence estimators
and of finitely many DM iterations, showing that under correct specification their
limiting covariance matches the Fisher information. Robustness is analyzed via the residual-adjustment function, yielding bounded influence functions and a strictly positive breakdown bound for bounded-RAF divergences, and we contrast
this with the non-robust behaviour of KL/EM. Next, we address the challenge of
determining the number of mixture components by proposing a penalized divergence
criterion combined with repeated sample splitting, which delivers consistent
order selection and valid post-selection inference. Empirically, DM
instantiations based on Hellinger and negative exponential divergences deliver
accurate inference and remain stable under contamination in mixture and image-segmentation tasks. The results clarify connections to MM and proximal-point methods and offer practical defaults, making DM a drop-in alternative to EM for
robust latent-structure inference.
\end{abstract}



\noindent%
\vfill

\newpage
\spacingset{1.8} 
\section{Introduction}

Models with latent structure, such as finite mixture models (FMM), are
ubiquitous in clustering, image segmentation, and density modeling across
fields ranging from astronomy to medical research. A standard approach in
many of these settings is the classical expectation--maximization (EM)
algorithm and its variants (e.g., stochastic EM) for model fitting and
inference. The properties of EM, which are likelihood-based, have been
well studied when the model is correctly specified. However, when the
model is misspecified or the data contain aberrant observations,
likelihood-based methods can perform poorly, and minimum divergence
estimators, also referred to as minimum disparity estimators (see
\cite{Lind94,BSV97}), have been proposed as robust and efficient
alternatives to likelihood-based inference. For models with latent
structure, however, a unified, divergence-agnostic treatment with
operator-level guarantees remains underdeveloped.

The primary objective of this paper is to develop a general
divergence-minimization (DM) algorithm for latent-structure models and
(i) to establish monotone descent and local contractivity properties,
(ii) large-sample guarantees (consistency and $\sqrt{n}$-asymptotic
normality), and (iii) a divergence-based mixture complexity estimator via
multiple sample splitting. As a consequence of our operator-theoretic
view, existing methods such as EM and proximal-point updates appear as
special cases of the DM algorithm. A brief literature review of
divergence-based methods for latent structure models appears below; an
extended review is deferred to Appendix~A.

\subsection{Background and literature review}

Divergence-based methods are widely used due to their ability to unify
classical and robust inference. From \cite{Ber77} through \cite{Lind94}
(among others), it is well established that divergence-based estimators
are robust to misspecification and first-order efficient under correct
specification. In finite mixtures, \cite{Woo95} analyzed a two-component
normal mixture and used the minimum Hellinger distance (MHD) to estimate
the mixing proportion. \cite{Cut96} estimated all parameters in normal
mixtures via MHD, referred to as the HMIX algorithm. For discrete
mixtures, \cite{karlis98} studied Poisson mixtures using HELMIX, a
\emph{Poisson-specific} variant whose update relies on recurrence
relations for the Poisson pmf and
on the special algebraic form of the of the Hellinger distance; it
does not readily generalize either to other divergences or to other count families. In particular, HELMIX does
not extend to the Poisson--Gamma (Negative Binomial) or
Poisson--lognormal mixtures that we study here, for which no
Hellinger-type EM algorithm appears to be available in the literature.
Our divergence–minimization (DM) algorithm provides EM-type procedures
for these models in a unified way, yielding estimators that are both
robust and first-order efficient under the correctly specified model
(their limiting covariance equals the Fisher information). Beyond
Hellinger distance, the DM framework accommodates negative exponential
and related divergences, which are known to improve stability in the
presence of outliers and inliers \citep{Lind94,BSV97}.

Our finite–step convergence and large–sample results are essentially
divergence–agnostic within the DM class: under mild curvature and
regularity conditions, the DM iterates are $\sqrt{n}$–consistent and
first–order efficient under the correctly specified model, with limiting
covariance equal to the Fisher information, for any disparity $G$ in
this class. Robustness, however, is driven not by the DM scheme itself
but by the residual adjustment factor $A$ associated with $G$: when $A$
is bounded away from $-1$ on the relevant range, the resulting DM
estimators have a strictly positive breakdown point and bounded
gross–error sensitivity to outliers, while uniform boundedness of $A$
yields additional stability under inlier contamination, in line with the
behaviour of negative exponential and related divergences. In contrast,
classical EM, including commonly used weight–flooring schemes, does not
yield robustness: its influence function is unbounded and its breakdown
point is essentially zero.

The remainder of the paper is organized as follows.
Section~\ref{sec:DMLS} introduces the DM functional and algorithm in a
principled way via variational elimination of the latent weights, and
clarifies its relationships to EM, MM, and proximal–point methods.
Section~\ref{sec:convergence} studies the resulting algorithm,
establishing monotone descent, local contraction for the population and
sample DM maps, and path properties of the iterates.
Section~\ref{sec:asymptotics} investigates the asymptotic behaviour of
the sample–level iterates, showing $\sqrt{n}$–consistency, first–order
efficiency under the correctly specified model, finite–step CLT and
(Godambe–)Wilks properties, and robustness in the form of bounded
influence functions and strictly positive breakdown points.
Finally, Section~\ref{sec:modelselection} addresses unknown mixture
complexity: we introduce a generalized divergence information criterion
(GDIC), study its robustness, and combine it with a simple
sample–splitting scheme (using a majority-vote aggregation across splits)
to obtain valid post–selection inference, where the selection split is
used to compute GDIC while the estimation split inherits the finite–step
CLT and (Godambe–)Wilks guarantees developed earlier. This yields, to
our knowledge, the first end–to–end divergence–based procedure that both
estimates $K$ and delivers robust inference in overdispersed latent
count mixtures such as Poisson–Gamma and Poisson–lognormal models.
Section~\ref{sec:simulation} presents empirical studies and
Section~\ref{sec:case} case studies; detailed background and historical
notes appear in Appendix~A. Additional technical details are presented
in Appendices~B through~N.
\section{DM Algorithm for Models with Latent Structure}{\label{sec:DMLS}}
We begin this section with a brief description of models with latent structure and minimum divergence estimation. 

{\bf{Finite mixtures and latent structure.}}
The random variable $Y$ with density given by
\[
f(y;\bm{\theta})=\sum_{k=1}^K \pi_k\,h_k(y; \bm{\phi}_k),\qquad
\bm{\theta}=(\pi,\bm{\phi})\in\Delta_{K-1}\times\Phi,\ \ \pi_k>0,\ \sum_k\pi_k=1,
\]
is referred to as the finite mixture model (FMM). When 
\begin{align*}
h_k(y; \bm{\phi}_k) = \int_{\Real} s_k(y\mid\lambda) r_k(\lambda\mid\bm{\phi}_k) d \rho(\lambda),
\end{align*}
where $s_k(\cdot\mid\lambda)$ is a density for each $\lambda$ and $r(\cdot\mid\bm{\phi}_k)$ is the density on $\Lambda$ with parameter $\bm{\phi}_k = (\phi_{k1}, \dots, \phi_{kd})$, we refer to it as Hierarchical Finite Mixture Model (HFMM).  In applications, we assume a common conditional kernel, that is, $s_k \equiv s$, 
and allow only the mixing law to vary across components; hence, we write
\[
h_k(y;\boldsymbol\phi_k)\equiv h(y;\boldsymbol\phi_k)=\int_{\Lambda} s(y| \lambda)\, r(\lambda|\boldsymbol\phi_k)\,d \rho(\lambda),
\]
where $\Lambda \subset \Real$. Turning to the model, let $Z \in \mathcal{Z} \coloneqq \{1, \ldots, K\}$ denote the random variable representing the unobserved class of $Y$ and set for $z=k$, $p(y,z; \bm{\theta})=\pi_k h(y;\bm{\phi}_k)$. Then, the marginal density of $Y$ is given by $f(y;\bm{\theta})=\int_{\mathcal{Z}}p(y,z;\bm{\theta}) d\nu(z)$, where $\nu$ is the counting measure on $\left(\mathcal{Z}, 2^{\mathcal{Z}}\right)$.
$\bm{\pi}$ represents the vector of class-inclusion probabilities and is also referred to as mixing weights.  To remove label ambiguity, we order the components by increasing mixing weights and break ties deterministically:
$
\pi_1 \le \pi_2 \le \cdots \le \pi_K,\quad\text{and if }\pi_j=\pi_{j+1}\text{ we use a fixed scalar summary }m(\phi)\text{ (e.g., a location) with }m(\phi_j)\le m(\phi_{j+1}).
$ All results are permutation‑invariant and hold for any fixed labeling rule; this convention only selects a canonical representative of the label-equivalence class.

Several commonly used models are particular cases of the HFMM. First, setting  $s(y|\lambda)=\mathrm{Poisson}(y| \lambda)=e^{-\lambda}\lambda^y/y!$ and 
$r_k( \lambda; \bm{\phi}_k) =r(\lambda;\bm{\phi}_k)=\mathrm{Gamma}(\lambda| \alpha_k,\beta_k)=\{\beta_k^{\alpha_k}/\Gamma(\alpha_k)\}\,\lambda^{\alpha_k-1}e^{-\beta_k\lambda}$, 
with $\bm{\phi}_k=(\alpha_k,\beta_k)$, $\alpha_k>0$, $\beta_k>0$, one obtains a finite mixture of Poisson-Gamma models.
The resulting marginal is the same as a mixture of negative binomial models; that is,
\[
h(y;\alpha_k,\beta_k)
=\frac{\Gamma(y+\alpha_k)}{\Gamma(\alpha_k)\,y!}\,\Big(\frac{\beta_k}{\beta_k+1}\Big)^{\alpha_k}\Big(\frac{1}{\beta_k+1}\Big)^y.
\]
Continuing with the Poisson distribution, setting $s(y\mid\lambda)$ as above and 
$r(\lambda;\bm{\phi}_k)=\mathrm{Lognormal}(\lambda\mid \mu_k,\sigma_k^2)=\{\lambda\,\sqrt{2\pi}\sigma_k\}^{-1}
\exp\!\big[-(\log\lambda-\mu_k)^2/(2\sigma_k^2)\big]$, with $\bm{\phi}_k=(\mu_k,\sigma_k^2)$, $\sigma_k^2>0$, we obtain 
\[
h(y;\mu_k,\sigma_k^2)=\int_0^\infty \mathrm{Poisson}(y\mid \lambda)\;
\mathrm{Lognormal}(\lambda\mid \mu_k,\sigma_k^2)\,d\lambda.
\]
The above expression is evaluated numerically as there is no closed-form expression. We will investigate these two HFMMs, along with the standard $K$-component Poisson mixture, in the simulation section. In the rest of the manuscript, we use the following ratio convention:
\noindent\textit{Ratio convention.} For any ratio \(a/b\) used below
(e.g., \(r(y)=g(y)/f(y;\bm{\theta}')\), \(t(z\mid y)=w(z\mid y;\bm{\theta})/w(z\mid y;\bm{\theta}')\)),
we interpret it only where the denominator is positive; on null sets, we define the ratio to be \(1\).

{\bf{Residuals and divergences.}}
Let $g$ denote the target density and let the \emph{Pearson residual} be
\[
\delta(y;\bm{\theta})=\frac{g(y)}{f(y;\bm{\theta})}-1\quad(\text{so }g=f\,(1+\delta)).
\]
Given a thrice differentiable convex generator $G: [-1, \ff) \ra [0, \ff)$  satisfying $G(0)=G^{\prime}(0)=0$ and $G^{\prime\prime}(0)=1$, its \emph{residual‑adjustment function} (RAF) is given by $A(\delta)\coloneqq  (1+\delta)G'(\delta)-G(\delta)$. From the properties of $G$, $A(\delta)$ is an increasing function on $[-1, \infty)$ and carries the relevant information about the trade-off between efficiency and robustness. The residual‑adjusted divergence
\[
D_G(g,f_{\bm{\theta}})=\int G\!\big(\delta(y;{\bm{\theta}})\big)\,f(y;{\bm{\theta}})\,dy.
\]
The population minimum‑divergence estimator (MDE) solves $\bm{\theta}^\star\in\arg\min_{\bm{\theta}} D_G(g,f_{\bm{\theta}})$; in practice $g$
is replaced by an empirical version $g_n$ yielding $\widehat{\bm{\theta}}_n^{(G)}\in\arg\min_{\bm{\theta}} D_G(g_n,f_{\bm{\theta}})$.  Different generators $G$ induce different RAFs and hence different weighting of residuals in the estimating equations. Some of the canonical choices are (i) squared Hellinger distance:
$G_{\mathrm{HD}}(t)=2(\sqrt{1+t}-1)^2,\quad
A_{\mathrm{HD}}(\delta)
=2[(1+\delta)^{1/2}-1]$ and (ii) the
Negative exponential divergence (NED): $G_{\mathrm{NED}}(t)=[\exp(-t)-1 +t],\quad A_{\mathrm{NED}}= 2-2 e^{-\delta}-\delta e^{-\delta}$.
A Variant of NED (vNED), with g and $f_{\bm{\theta}}$ flipped, equivalently replacing $\delta= g/f-1$ by $\delta^*=f/g-1$, yields a softened left‑tail for $\delta<0$ to avoid over‑penalizing ``holes'' and helps improve segmentation and count-mixture fits. The details are in the appendix. 

\noindent \textbf{Remark (Limitations of KL for robustness).} For the Kullback--Leibler generator $G_{\mathrm{KL}}(t) = (1+t)\log(1+t)-t$, the
corresponding residual adjustment factor is $A_{\mathrm{KL}}(\delta) = \delta$,
which is unbounded in the residual $\delta = g/f_{\bm{\theta}} - 1$. As a consequence,
large density ratios $g/f_{\bm{\theta}}$ can exert arbitrarily large leverage on the
estimating equations: the gradient $\nabla_{\bm{\theta}} D_G(g,f_{\bm{\theta}})$ is not
uniformly bounded over contaminated distributions, and individual data points
can dominate the fit. Common EM heuristics such as \emph{weight--flooring} (enforcing
$\pi_k \ge \pi_{\mathrm{floor}}$ and renormalizing) only constrain the mixing
weights and do not cap these per–observation contributions: a single extreme
observation can still drive the fitted component parameters and, in
model--selection settings, spuriously favor an extra component. In particular,
the influence function remains unbounded and the breakdown point is essentially
zero. In Sections~\ref{sec:robustness}--\ref{sec:breakdown} we make this
precise, showing that our general convergence and finite--step CLT results still
hold for KL, but the uniform contraction and breakdown lower bounds available
for bounded--RAF divergences (such as NED/vNED) do not extend to KL without
additional local density--ratio assumptions.

We now turn to a heuristic description of the DM method in terms of divergences, borrowing terminology from the EM literature.  As above, let $Y$ be an observable real-valued random variable and $Z$ is a latent random variable representing the class membership of $Y$. The pair $(Y, Z)$ is referred to as the complete data. Let $\mathcal{G}$ denote the densities on $\Real$, and $g(\cdot)$ be the true density of $Y$ and $f(\cdot; \bm{\theta}) \in \mathcal{F}_{\bm{\Theta}}$ denote the postulated density. The postulated joint density of $(Y, Z)$ is denoted by $p(\cdot, \cdot; \bm{\theta})$ and the marginal postulated density is $f(y; \bm{\theta}) = \int_{\mathcal{Z}} p(y, z; \bm{\theta}) d\nu(z)$.
The marginal density is referred to as the \emph{incomplete data} model. Let $w(z|y;\bm{\theta}) \coloneqq \tfrac{p(y,z;\bm{\theta})}{f(y;\bm{\theta})} =\mathbb{P}_{\bm{\theta}}(Z=z|Y=y)$ be the conditional density of $Z$ given $Y$, it follows that $p(y, z;\bm{\theta}) = f(y;\bm{\theta}) w(z|y;\bm{\theta}).$ This is referred to as the \emph{responsibility}. Finally, we frequently use the notations $w(z|y, \bm{\theta})$ and $w_z(y;\bm{\theta})$ interchangeably. 

The cross-entropy loss for the joint distribution of $(Y, Z)$, referred to as the \emph{EM $Q-$functional}, is $\bm{E}\left[-\log p(Y, Z; \bm{\theta}\right]$,  and can be expressed as
$\mathbf{E}_Y\left[ \mathbf E_{Z\mid Y;\,\bm{\theta}}\!\big[-\log p(Y,Z;\bm{\theta}')\big]\right]$. The conditional expectation, conditioned on $Y=y$, is referred to as the EM surrogate. To obtain a robust variant of the EM surrogate, we replace the cross-entropy loss with a residual adjusted divergence on complete data. Specifically, we introduce the complete-data residual (parameterized by an auxiliary conditional $\tilde q$)
\[
\delta(y,z; \bm{\theta}',\tilde q)=\frac{g(y)\,\tilde q(z\mid y)}{p(y,z;\bm{\theta}')}-1,
\quad
\QG(\bm{\theta}',\tilde q\mid\bm{\theta})
:=\iint G\!\big(\delta(y,z;\bm{\theta}',\tilde q)\big)\,p(y,z;\bm{\theta}')\,d\nu(z)\,dy.
\]
We begin with a simple lemma describing the variational elimination of the auxiliary conditional from the complete-data residual.

\begin{lem}[Variational elimination of the auxiliary conditional]
\label{lem:innerq}
For any $\bm{\theta}'$, 
\[
\min_{\tilde q}\ \widetilde\QG(\bm{\theta}',\tilde q\mid \bm{\theta})
\ =\ \DG(g, f_{\bm{\theta}'})\quad\text{attained at}\quad
\tilde q(\cdot\mid y)=w(\cdot\mid y;\bm{\theta}'):=\frac{p(y,\cdot;\bm{\theta}')}{f(y;\bm{\theta}')}.
\]
\end{lem}

\begin{proof}
Let $\bm{\theta}'$ be fixed and set $w(\cdot\mid y;\bm{\theta}')$ as above. Then,  by Jensen's inequality, for each $y$,
\[
\int G\!\Big(\frac{g(y)}{f(y;\bm{\theta}')}\cdot \frac{\tilde q(z\mid y)}{w(z\mid y;\bm{\theta}')}-1\Big)\,w(z\mid y;\bm{\theta}')\,d\nu(z)
\ \ge\ 
G\!\Big(\frac{g(y)}{f(y;\bm{\theta}')}-1\Big),
\]
and $\int \frac{\tilde q}{w}\,w\,d\nu=\int \tilde q\,d\nu=1$. Next, integrating over $y$, we obtain
$\min_{\tilde q}\widetilde\QG(\bm{\theta}',\tilde q\mid\bm{\theta})=\DG(g,f_{\bm{\theta}'})$, attained at $\tilde q=w(\cdot\mid y;\bm{\theta}')$.
\end{proof}

\noindent We notice that fixing the auxiliary conditional at the current responsibilities yields a valid majorizer. Hence, we set
$$
Q_{G}(\bm{\theta}'\mid\bm{\theta}):=Q_{G}\big(\bm{\theta}',\,w(\cdot\mid\cdot;\bm{\theta})\,\bigm|\bm{\theta}\big)\ \ge\ \DG(g,f_{\bm{\theta}'})\quad
\text{with equality at }\ \bm{\theta}'=\bm{\theta}.
$$
Thus, if $\bm{\theta}_{m+1}\in\arg\min_{\bm{\theta}'} Q_G(\bm{\theta}'\mid\bm{\theta}_m)$ then
$D_G(g,f_{\bm{\theta}_{m+1}})\le D_G(g,f_{\bm{\theta}_m}).$ In other words, $Q_G(\bm{\theta}'\mid\bm{\theta})$ inherits the properties of EM surrogate. Next, setting
$t(z\mid y)=\tfrac{w(z\mid y;\bm{\theta})}{w(z\mid y;\bm{\theta}')}$ and recalling $\delta(y; \bm{\theta}')+1=\tfrac{g(y)}{f(y;\bm{\theta}')}$, observe that
\begin{align}{\label{def:Q_G}}
Q_G(\bm{\theta}'\mid\bm{\theta})
=\int_{\mathcal Y} f(y;\bm{\theta}')\Bigg\{\int_{\mathcal Z}
G\!\big((1+\de(y;\bm{\theta}'))\,t(z\mid y)-1\big)\,w(z\mid y;\bm{\theta}')\,d\nu(z)\Bigg\}dy.
\end{align}
By subtracting $G(-1+\tfrac{g(y)}{f(y;\bm{\theta}')})$ from the inner integral, it follows that the difference $Q_G(\bm{\theta}'\mid \bm{\theta})-D_G(g, \bm{\theta}')$ represents average latent-conditional gap. Applying Jensen's inequality w.r.t. $w(\cdot\mid y; \bm{\theta}')$ (noting $\int t(\cdot\mid y)\,w(\cdot\mid y;\bm{\theta}')\,d\nu=1$) it follows that $H_G(\bm{\theta}'\mid\bm{\theta})\ \ge\ 0,~\text{and}~
H_G(\bm{\theta}\mid\bm{\theta})=0$. This yields the separation-majorization identity, namely, $Q_G(\bm{\theta}'\mid\bm{\theta})=\Psi_G(\bm{\theta}')+H_G(\bm{\theta}'\mid\bm{\theta})\ \ \ge\ \ \Psi_G(\bm{\theta}')=D_G(g_n,f_{\bm{\theta}'})$ with equality iff $w(\cdot\mid y;\bm{\theta})=w(\cdot\mid y;\bm{\theta}')$ for $f(\cdot;\bm{\theta}')$-a.e.\ $y$.

\noindent \textbf{The DM Algorithm:}
The expression (\ref{def:Q_G}) can alternatively be  expressed as follows:
\begin{align}{\label{eq:iter}}
\nonumber Q_G(\bm{\theta}'\mid \bm{\theta})&=\bm{E}_Y\left[ \bm{E}_{Z|Y, \bm{\theta}'}\left[ G\left( -1 + \frac{g(Y)w(Z|Y;\bm{\theta})}{f(Y;\bm{\theta}') w(Z|Y;\bm{\theta}') } \right)\right]  \right]\\
&= \bm{E}_Y\left[ \bm{E}_{Z|Y, \bm{\theta}'}\left[ G\left( -1 + \frac{g(Y)w(Z|Y;\bm{\theta})}{p(Y,Z;\bm{\theta}')} \right)\right]  \right].
\end{align}
This suggests a natural approach to obtain the MDE by iterating the RHS of (\ref{eq:iter}). We will show, in the following lemma, that if $\bm{\theta}'$ is MDE, then $Q^{(G)}(\bm{\theta}'|\bm{\theta}') = D^{(G)}(\bm{\theta}') \leq Q^{(G)}(\bm{\theta}' |\bm{\theta})$ for $\bm{\theta} \in \bm{\Theta}$. Using Supplement Lemma~\ref{S-Inequality00} (variational elimination), choosing 
$\tilde q(\cdot\mid y)=w(\cdot\mid y;\bm \theta)$ yields the following majorization.

\begin{lem}\label{lem:MF}
For any $\bm{\theta},\bm{\theta}'\in \bm{\Theta}$, $D_{G}(\bm{\theta}') \leq Q_{G}(\bm{\theta}'|\bm{\theta})$ for all $\bm{\theta} \in \bm{\Theta}$ and the equality holds if and only if $\bm{\theta}' = \bm{\theta}$. Hence, for any  $\bm{\theta}$, the map $Q_{G}(\cdot | \bm{\theta})$ majorizes $D^{(G)}(\cdot)$.
\end{lem}

Based on Lemma \ref{lem:MF}, it follows that $Q_{G}(\bm{\theta}' |\bm{\theta})$ is a majorizing function for the MM algorithm (see \cite{Hun00a} ).  Also, note that the divergence in $Q_{G}(\bm{\theta}'|\bm{\theta})$ includes additional information regarding the unobserved conditional density $w(z|y;\bm{\theta})$, which enables the calculation of the posterior distribution, as in the EM algorithm. This is based on knowing the true density $g(\cdot)$ and hence $D_{G}(\bm{\theta})$ and $Q_{G}(\bm{\theta}'|\bm{\theta})$ will be referred to as the \emph{population level} DM objective functions. 

When $g(\cdot)$ is unknown it can be replaced by $g_n(\cdot)$ (estimated from observed data), leading to the \emph{sample-level} objective functions. One choice of $g_n(\cdot)$ is the kernel density estimator
\begin{align}\label{KDE-1}
g_n(y) = \frac{1}{n c_n}\sum_{i=1}^n \mathcal{K}\left( \frac{y-Y_i}{c_n} \right),
\end{align}
where $\mathcal{K}(\cdot)$ is a probability density and $c_n$ is the bandwidth. Other choices of $g_n(\cdot)$ include wavelet density estimator (see \cite{Dou90}) and local polynomial regression estimator (see \cite{Fan96}). We now introduce notations to describe the sample-level objective function and surrogate. Let \(\delta_n(y;\bm{\theta})=g_n(y)/f(y;\bm{\theta})-1\) denote the sample Pearson's residual and
\[
D_G(g_n,f_{\bm{\theta}})=\int G(\delta_n(y; \bm{\theta}))\,f(y;\bm{\theta})\,dy ~~ \text{and} ~~
A(\delta_n)=(1+\delta_n)G'(\delta_n)-G(\delta_n).
\]
denote the sample-level objective function and the RAF. Given responsibilities \(w_k(y;\bm{\theta})\), we set
\[
\delta_{k,n}(y;\bm{\theta},\bm{\theta}')=\frac{g_n(y)\,w_k(y;\bm{\theta})}{\pi_k' h(y;\phi_k')}-1,~~ \text{and} ~~
Q_G(\bm{\theta}'\mid\bm{\theta};g_n)=\sum_{k=1}^K \int G\!\big(\delta_{k,n}(y;\bm{\theta},\bm{\theta}')\big)\,\pi_k' h(y;\phi_k')\,dy.
\]
Thus, the DM algorithm can be broadly divided into two steps. Below, we fix $\mu \in \{g_n, g\}$. Then the steps of the DM algorithm are:
\begin{enumerate}
    \item \textbf{D-step}. Determine $Q_{G}\left(\bm{\theta}'| \bm{\theta};\mu \right)$ .
\item \textbf{M-step}. Choose $\bm{\theta}_{m+1} \in \bm{\Theta}$ to minimize $Q_{G}\left(\bm{\theta}' |\bm{\theta}_m; \mu \right)$ over $\bm{\theta}' \in \bm{\Theta}$. 
\end{enumerate}
As a \emph{standing convention} for the rest of the paper, for every fixed  \(n\), we denote by \(D_{G,n}(\bm{\theta}):=D_G(g_n,f_{\bm{\theta}})\) and \(Q_{G,n}(\bm{\theta}'\mid\bm{\theta}):=Q_G(\bm{\theta}'\mid\bm{\theta};g_n)\). Also, for a fixed $G$, we suppress the dependence on $G$ and simply refer to it as $D_n(\bm{\theta})$ and $Q_n(\bm{\theta}'\mid\bm{\theta})$. 

\noindent \textbf{The generalized DM Algorithm:}
It is possible that in the M-step, the minimizer $\bm{\theta}_{m+1}$ is not unique. Let $\bm{\Theta}_m$ denote the set of all minimizers at step $m$; that is, $\bm{\theta}_{m+1} \in  \bm{\Theta}_{m+1}$.
Sometimes it may be difficult to perform M-step numerically; in this case, we can define a generalized DM algorithm (referred to as the G-DM algorithm) as follows: Let $M: \bm{\theta}_m \to \bm{\Theta}_{m+1}$ be a point to set map: then the G-DM algorithm is an iterative scheme such that 
\begin{align*}
Q^{(G)}(\bm{\theta}'|\bm{\theta}_m) \leq Q^{(G)}(\bm{\theta}_m |\bm{\theta}_m) \quad \text{for all} \quad \bm{\theta}' \in M(\bm{\theta}_m).
\end{align*}
We notice here that for any G-DM sequence $\{\bm{\theta}_m \}$, $D_{G}(\bm{\theta}_{m+1}) \leq D_{G}(\bm{\theta}_m)$ and DM algorithm is a special case of G-DM algorithm. We emphasize here that by choosing different divergences, we obtain many existing algorithms, including the EM, HMIX, and HELMIX algorithms. These examples are given in the Supplementary analysis \ref{S-Special_cases}.

The key ingredient for the proposed population DM algorithm is the conditional Pearson-type ratio  between $g(y)w(z|y;\bm{\theta})$ and $f(y;\bm{\theta}')w(z|y;\bm{\theta}')$,  which we denote by
\begin{align*}
\tau(y, z; g, \bm{\theta}, \bm{\theta}') = \frac{g(y)w(z|y;\bm{\theta})}{f(y;\bm{\theta}')w(z|y;\bm{\theta}')}.
\end{align*}
When $z=k$, the denominator of $\tau(y, k; g, \bm{\theta}, \bm{\theta}')$ reduces to $\pi_k'h(y;\bm{\phi}'_k).$ We observe here that, unlike the Pearson residual, $\tau$ does not subtract 1; it plays the role of a conditional analogue instead. When there is no scope for confusion, for notational simplicity, we suppress $g$, $\bm{\theta}$ and $\bm{\theta}'$ in $\tau_k(y,g, \bm{\theta}, \bm{\theta}')$ and denote it as $\tau_k(y)$ and when $g$ is replaced by $g_n$ we refer to it as $\tau_{k,n}(y)$. 



\subsection{The DM-MIX Algorithm for FMM}\label{Chap5:S:DM-Alg}
A general description of the DM algorithm for HFMM (i.e., the DM-MIX algorithm) is given below. As noted above, $\tau_k(y) \coloneqq \tau_k(y, g; \bm{\theta}, \bm{\theta}') = \frac{g(y)w_k(y; \bm{\theta})}{\pi_k' h(y;\bm{\theta}')}$. Also, set
\[Q^{(k)}(\bm{\theta}'|\bm{\theta}) = \int_\Real G(\tau_k(y)) \pi_k' h(y; \bm{\phi}_k')dy.\]
When $g$ is replaced by $g_n$, we denote the above quantities by $\tau_{k,n}(\cdot, \cdot;\cdot)$ and $Q^{(k)}(\cdot|\cdot)$ by $Q_{n}^{(k)}(\cdot|\cdot)$. Finally, define $B(u)=G(u)-u G'(u)$. We now turn to the algorithm. For a given starting point $\bm{\theta}^{(0)}$ and a density estimate $g_n(y)$,
\begin{algorithm}[H]
    \scriptsize
	\caption{The DM-Mix Algorithm}
	\label{SP:1}
	Set m = 0.
	\begin{algorithmic}
		\REPEAT
		\STATE 1. Compute
		\begin{align*}
		Q_n\left(\bm{\theta}^{(m)}|\bm{\theta}^{(m)}\right) = \sum_{k=1}^{K} \int_{\mathcal{Y}} G \left(-1+ \frac{g_n(y)w_{k}(y;\bm{\theta}^{(m)})}{\pi_k^{(m)}h(y;\bm{\phi}_k^{(m)}) } \right) \pi_k^{(m)} h(y | \bm{\phi}_k^{(m)} ) dy, \quad \text{where} \quad w_{k}(y;\bm{\theta}) = \frac{\pi_k h(y;\bm{\phi}_k)}{\sum_{l=1}^K \pi_l h(y;\bm{\phi}_l)}.
		\end{align*}
		\STATE 2. Update $\bm{\phi}_k^{(m+1)}$:
		\begin{align*}
		\bm{\phi}_k^{(m+1)} ~ &= ~ \underset{\bm{\phi}_k' \in \bm{\Theta}}{\text{argmin}} ~  \int_{\mathcal{Y}} G\left(-1 + \frac{g_n(y)w_{k}(y;\bm{\theta}^{(m)})}{\pi_k^{(m)}h(y;\bm{\phi}_k')}\right)\pi_k^{(m)}h\left(y;\bm{\phi}_k'\right) dy.
		\end{align*}		
        \STATE 3. Update the mixing probabilities using the Lagrangian multiplier step:
        \begin{align*}
        \pi_k^{(m+1)}
        = \frac{\Phi_k\!\left(\bm{\pi}^{(m)}, \bm{\phi}^{(m+1)}\right)}
        {\sum_{\ell=1}^K \Phi_\ell\!\left(\bm{\pi}^{(m)}, \bm{\phi}^{(m+1)}\right)},
        \qquad k=1,\dots,K,
        \end{align*}
        where
        \[
        \Phi_k\!\left(\bm{\pi}, \bm{\phi}\right)
        := -\pi_k \,\frac{\partial Q_{n}^{(k)}(\pi_k,\bm{\phi}_k \mid \bm{\theta})}{\partial \pi_k},
        \]
        \[
        \frac{\partial Q_{n}^{(k)}}{\partial \pi_k}(\pi_k,\bm\phi_k\mid\bm\theta)
        = \int_{\mathcal Y} h(y;\bm\phi_k)\,\Big\{B(\tau_{k,n}(y))\Big\}\,dy,
        \]
        \[
        \Phi_k(\bm\pi,\bm\phi)
        := \pi_k \int_{\mathcal Y} h(y;\bm\phi_k)\,\Big\{B(\tau_{k,n}(y))\Big\}\,dy.
        \]
		\STATE 4. Update $w_{k}(y; \bm{\theta}^{(m+1)})$, 
        and compute $Q_n\left(\bm{\theta}^{(m+1)}|\bm{\theta}^{(m+1)}\right)$ and the difference $\epsilon_{m+1} = | Q_n\left(\bm{\theta}^{(m+1)}|\bm{\theta}^{(m+1)}\right) - Q_n\left(\bm{\theta}^{(m)}|\bm{\theta}^{(m)}\right) |.$	
		\STATE 5. Set m = m+1.
	\UNTIL{$\epsilon_{m}< \text{threshold}$.}
	\end{algorithmic}
\end{algorithm}
The EM, HMIX, and HELMIX algorithms are special cases of this framework for particular choices of $G(\cdot)$ and probability distributions (see Supplementary analysis \ref{S-Special_cases}), wherein we also illustrate that by reformulating the DM objective function algorithm can be interpreted as an MM algorithm, a proximal point algorithm, and a coordinate descent algorithm. 

The DM algorithm applied to HFMM has a special structure for updating the mixing probability $\bm{\pi}$. To see this, we partition the parameter vector $\bm{\theta}' = (\bm{\pi}', \bm{\phi}')$ into two sub-vectors $\bm{\pi}' = (\pi'_1, \cdots, \pi'_K)$ and $\bm{\phi}' = (\bm{\phi}_1', \cdots, \bm{\phi}_K')$, and we can update $\bm{\phi}'$ and $\bm{\pi}'$ iteratively. Specifically, starting with updating parameters $\bm{\phi}_k'$, we update $\pi_k'$ using the Lagrangian multiplier and then iterate the steps. This process leads to the estimate of $\pi_k$ as
\begin{align}\label{Update-Pi11}
\pi_k'
=\frac{\displaystyle \pi_k'\,\frac{\partial Q_n^{(k)}(\bm{\theta}'\mid\bm{\theta})}{\partial \pi_k'}}
{\displaystyle \sum_{\ell=1}^K \pi_\ell'\,\frac{\partial Q_n^{(\ell)}(\bm{\theta}'\mid\bm{\theta})}{\partial \pi_\ell'}}.
\end{align}
Details of the derivation are provided in the Supplementary analysis \ref{S-supp:update-pi}. The update of the mixing probability in (\ref{Update-Pi11}) naturally arises from leveraging the majorizing divergence function $Q(\cdot|\cdot)$.

\section{Convergence Guarantees of the DM Algorithm}\label{sec:convergence}

\noindent Figure~\ref{fig:dm-illustration} summarizes the DM map and the descent property of $D_G$.
\begin{figure}[H]
\centering
\includegraphics[width=0.4\textwidth]{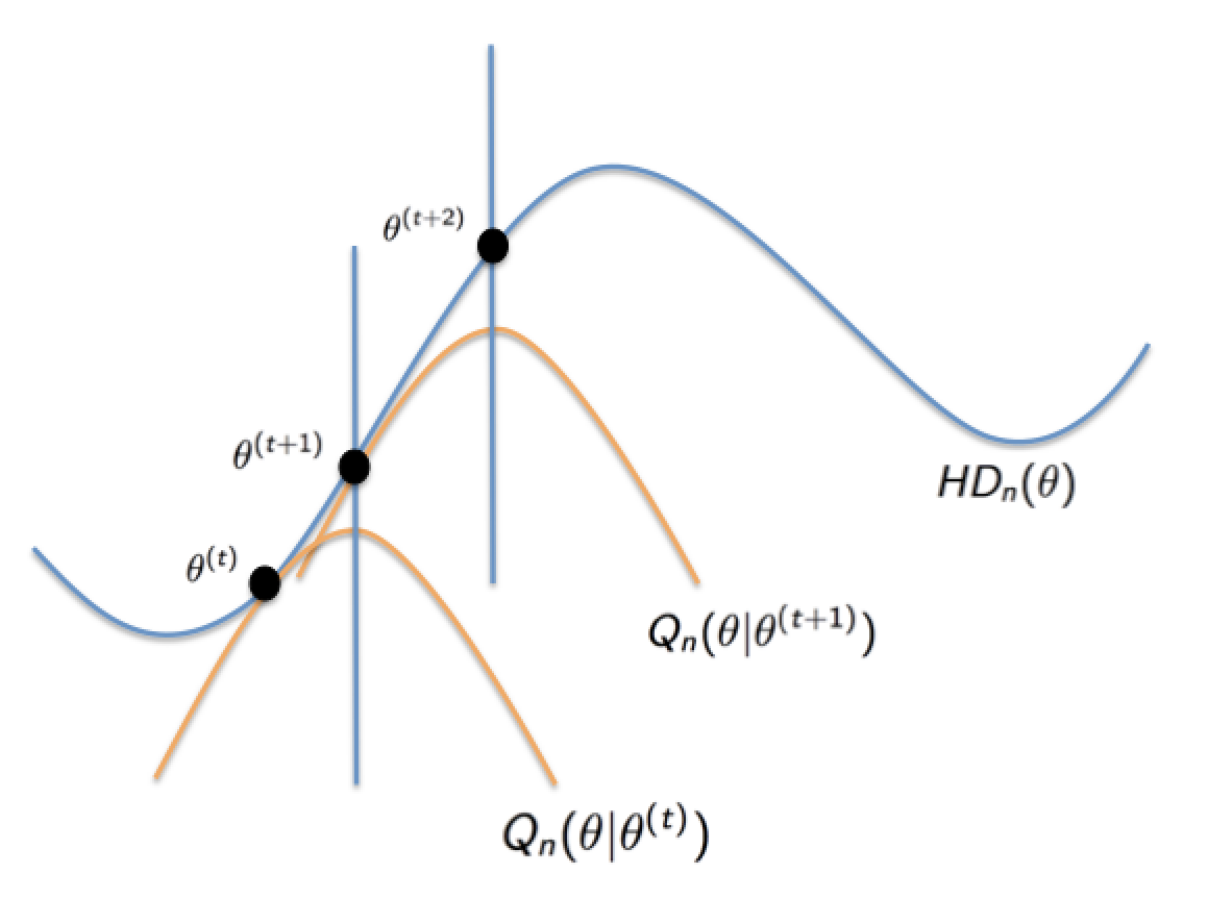}
	\caption{The DM Algorithm Illustration}
    \label{fig:dm-illustration}
\end{figure}

Let $M_n(\bm{\theta}): \bm \Theta \rightrightarrows \bm \Theta$ and $M(\bm{\theta}):\bm \Theta \rightrightarrows \bm \Theta$ be two set-valued maps defined as follows:
\begin{align*}
M_n(\bm{\theta})\coloneqq \arg\min_{\bm{\theta}'\in\bm{\Theta}} Q_n(\bm{\theta}'\mid\bm{\theta})\quad \text{and} \quad M(\bm{\theta})\coloneqq \arg\min_{\bm{\theta}'\in\bm{\Theta}} Q(\bm{\theta}'\mid\bm{\theta}).
\end{align*}
\begin{defn}[DM sequences]\label{DM-seq}
A sequence $\{\bm{\theta}_{m}\}$ is a \emph{population} DM sequence if $\bm{\theta}_{m+1}\in M(\bm{\theta}_m)$; 
a sequence $\{\bm{\theta}_{m,n}\}$ is a \emph{sample} DM sequence if $\bm{\theta}_{m+1,n}\in M_n(\bm{\theta}_{m,n})$.
\end{defn}
In this section, we establish the monotone descent property and stationarity of the limit points for the sample-level iterates. Using these, we establish local contraction of the population map $M(\cdot)$ together with the finite-sample perturbation bounds. 

\noindent \textbf{DM surrogate decomposition (recall).} With $G$ fixed, let $D_n(\bm{\theta}):=D_G(g_n,f_{\bm{\theta}})$ and
$Q_n(\bm{\theta}'\!\mid\bm{\theta}):=Q_G(\bm{\theta}'\!\mid{\bm{\theta}};g_n)$.
Define the remainder
\[
H_G(\bm{\theta}'\!\mid\bm{\theta})\ :=\ Q_n(\bm{\theta}'\!\mid\bm{\theta})\;-\;D_n(\bm{\theta}).
\]
Then $Q_n(\bm{\theta}'\!\mid\bm{\theta})=D_n(\bm{\theta}')+H_G(\bm{\theta}'\!\mid\bm{\theta})$ with
$H_G(\bm{\theta}'\!\mid\bm{\theta})\ge 0$ and $H_G(\bm{\theta}\!\mid\bm{\theta})=0$.
(See Supplementary material Section D for a derivation.)

\begin{lem}[Monotone descent for DM updates]\label{lem:dm-descent}
Let $\bm{\theta}^+\in M_n(\bm{\theta})$ be a DM update (set–valued selection). Then
\[
D_n(\bm{\theta}^+)\ \le\ Q_n(\bm{\theta}^+\!\mid\bm{\theta})\ \le\ Q_n(\bm{\theta}\!\mid\bm{\theta})\ =\ D_n(\bm{\theta}),
\]
hence $D_n$ is nonincreasing along DM iterates.
\end{lem}

\begin{proof}
By the decomposition above, $Q_n(\bm \theta^+\!\mid \bm \theta)\ge D_n(\bm \theta^+)$ and
$Q_n(\bm \theta\!\mid \bm \theta)=D_n(\bm \theta)$. Since $\bm \theta^+\in\arg\min_{\bm \theta'}Q_n(\bm \theta'\!\mid \bm \theta)$,
$Q_n(\bm \theta^+\!\mid \bm \theta)\le Q_n(\bm \theta\!\mid \bm\theta)$.
\end{proof}

\noindent\textit{Stationarity} will be established below via our self‐consistency and first‐order conditions.

\subsection{Global Properties of the DM Sequences}\label{Properties-DM}
Our first proposition concerns the convergence properties of the sample-level DM sequence. 

\begin{prop}\label{Prop:DM-Conv}
Assume (B1)-(B3) and (D1) holds and let $\{\bm{\theta}_{m,n}\}$ be any sample-level DM sequence. Then
$D_n(\bm{\theta}_{m+1,n}) \le D_n(\bm{\theta}_{m,n})$, with equality if and only if $\bm{\theta}_{m,n}$ is a fixed point of $M_n$. Moreover, every limit point $\bar{\bm{\theta}}$ of $\{\bm{\theta}_{m,n}\}$ is stationary for $D_n$, i.e. $\nabla D_n(\bar{\bm{\theta}})=0$.

\begin{proof}[Proof (majorization--minimization)]
By construction $D_n(\bm{\theta})\le Q_n(\bm{\theta}\mid\bm{\theta}')$ for all $\bm{\theta},\bm{\theta}'$, with equality at $\bm{\theta}=\bm{\theta}'$. Hence, noticing that $\theta_{m+1,n}$ is a minimizer,it follows that $D_n(\bm{\theta}_{m+1,n}) \le Q_n(\bm{\theta}_{m+1,n}\mid\bm{\theta}_{m,n}) \le Q_n(\bm{\theta}_{m,n}\mid\bm{\theta}_{m,n}) = D_n(\bm{\theta}_{m,n}).$ Next, if equality holds, then ${\bm{\theta}}_{m,n}$ minimizes $Q_n(\cdot \mid {\bm{\theta}}_{m,n})$ and is thus a fixed point. 
Now let $\bar{\bm{\theta}}$ be any limit point of $\{{\bm{\theta}}_{m,n}\}$. 
By the closed–graph property (D1), any such limit point satisfies $\bar{\bm{\theta}} \in M_n(\bar{\bm{\theta}})$, 
so $\bar{\bm{\theta}}$ minimizes $Q_n(\cdot \mid \bar{\bm{\theta}})$. 
On the other hand, for every ${\bm{\theta}}$, using the tangency identity $Q_n({\bm{\theta}} \mid {\bm{\theta}}) = D_n({\bm{\theta}})$, and differentiating this with respect to the first argument at ${\bm{\theta}} = \bar{\bm{\theta}}$ gives
\[
\nabla_{{\bm{\theta}}'} Q_n({\bm{\theta}}' \mid \bar{\bm{\theta}})\big|_{{\bm{\theta}}'=\bar{\bm{\theta}}}
  = \nabla D_n(\bar{\bm{\theta}}).
\]
Since $\bar{\bm{\theta}}$ minimizes $Q_n(\cdot \mid \bar{\bm{\theta}})$, the left–hand side is zero, and hence
$\nabla D_n(\bar{\bm{\theta}}) = 0$.
\end{proof}
\end{prop}
\begin{rem}[When “stationary” coincides with “fixed point”]\label{rem:stat-vs-fp}
If, in addition, for each $\bm{\theta}$ the function $\bm{\theta}'\mapsto Q_n(\bm{\theta}'\mid\bm{\theta})$ is \emph{strongly convex} near $\bm{\theta}$ and \emph{tangent} to $D_n$ at $\bm{\theta}$ in the sense that
$\nabla_{\bm{\theta}'} Q_n(\bm{\theta}'\mid\bm{\theta})|_{\bm{\theta}'=\bm{\theta}}=\nabla D_n(\bm{\theta})$,
then $\bar{\bm{\theta}}$ is stationary for $D_n$ if and only if it minimizes $Q_n(\cdot\mid\bar{\bm{\theta}})$. Under the uniqueness of that minimizer, this is equivalent to $\bar{\bm{\theta}}$ being a fixed point of the DM update map.
\end{rem}
The next proposition concerns the self-consistency of the DM updates, namely, any minimizer of the observed-data objective also minimizes the complete-data surrogate at the same parameter value. That is, if 
$\bm{\theta}_n^*$ minimizes $D_n(\bm{\theta}')$ over $\bm{\theta}' \in \bm{\Theta}$, then $\bm{\theta}_n^*$ also minimizes $Q_n(\cdot |\bm{\theta}_n^*)$. 
\begin{prop}{\label{self-consist}}
Let $\bm{\theta}_n^\star \in \arg\min_{\bm{\theta}' \in \Theta} D_n(\bm{\theta}')$ and assume that $Q_n(\cdot \mid \bm{\theta})$
attains a minimum for each fixed $\bm{\theta}$, then $\bm{\theta}_n^\star \in \arg\min_{\bm{\theta}' \in \Theta} Q_n(\bm{\theta}' \mid \bm{\theta}_n^\star)$.
\begin{proof}
For any $\bm{\theta}'$, $Q_n(\bm{\theta}' \mid \bm{\theta}_n^\star)\,\ge\, D_n(\bm{\theta}')\,\ge\, D_n(\bm{\theta}_n^\star)
= Q_n(\bm{\theta}_n^\star \mid \bm{\theta}_n^\star)$, so $Q_n(\cdot \mid \bm{\theta}_n^\star)$ is minimized at $\bm{\theta}_n^\star$. 
\end{proof}
An immediate consequence of the above proposition is the following corollary.
\begin{cor}[Fixed point under uniqueness]\label{cor:fixed-point-sample}
If, in addition, $Q_n(\cdot \mid \bm{\theta}_n^\star)$ has a unique minimizer, then the update map satisfies
$M_n(\bm{\theta}_n^\star)=\bm{\theta}_n^\star$.
\end{cor}
\end{prop}

We now turn to the cyclical behavior of the iterates of the  algorithm.

\noindent \textbf{Cyclical behavior}

For fixed $n$, a finite set of distinct points $\{\bm{\theta}^\ast_{1,n},\dots,\bm{\theta}^\ast_{t,n}\}\subset\Theta$ is a \emph{cycle of length $t\ge2$} for $M_n$ if
$M_n(\bm{\theta}^\ast_{i,n})=\bm{\theta}^\ast_{i+1,n}$ for $i=1,\dots,t-1$ and $M_n(\bm{\theta}^\ast_{t,n})=\bm{\theta}^\ast_{1,n}$.
We use assumptions (B1)-(B4) and the isolated-stationary-points condition (B4)  throughout this subsection.

\begin{lem}[Finiteness at a level]\label{lem:finiteness-level}
Assume (B1)-(B4). Fix any value $D_n^\star\in\Real$. Then the set of stationary points of $D_n$ with objective value $D_n^\star$ is at most finite.
\end{lem}

\begin{proof}[Proof sketch]
If infinitely many stationary points shared the same value, compactness (B2) would yield a convergent subsequence to another stationary point, contradicting isolation (B4). More detailed proof is in the appendix.
\end{proof}

\begin{prop}[Limit–set structure: convergence or finite cycle]\label{prop:cycle-structure}
Let $\{\bm{\theta}_{m,n}\}$ be any sample-level DM sequence with $\bm{\theta}_{m+1,n}\in M_n(\bm{\theta}_{m,n})$, and assume (B1)-(B4), and \textup{(D1)}$^{\prime}$. Then $D_n(\bm{\theta}_{m,n})\downarrow D_n^\star$, and every limit point of $\{\bm{\theta}_{m,n}\}$ is a stationary point of $D_n$ with value $D_n^\star$.
Moreover, the update $M_n$ permutes the (finite) set of limit points; hence either $\bm{\theta}_{m,n}\to\bm{\theta}^\ast_n$ (a stationary point with $M_n(\bm{\theta}^\ast_n)=\bm{\theta}^\ast_n$), or the limit points form a finite cycle.
\end{prop}

\begin{proof}[Proof sketch]
Monotone majorization gives $D_n(\bm{\theta}_{m+1,n})\le D_n(\bm{\theta}_{m,n})$ with a finite limit $D_n^\star$; any limit point is stationary by Prop.~\ref{Prop:DM-Conv}. By Lemma~\ref{lem:finiteness-level}, the limit set is finite. Continuity (D1) implies $M_n$ maps the set to itself, so $M_n$ acts as a permutation; this yields either a singleton (convergence) or a finite cycle.
\end{proof}

\begin{prop}[No cycles under uniqueness]\label{prop:no-cycles}
Assume (B1)-(B4), and \textup{(D2)} (for every $\bm{\theta}$, $Q_n(\cdot\mid\bm{\theta})$ has a unique minimizer). Then any sample-level DM sequence $\{\bm{\theta}_{m,n}\}$ converges to a stationary point $\bm{\theta}^\ast_n$ with $M_n(\bm{\theta}^\ast_n)=\bm{\theta}^\ast_n$. Moreover, if $\bm{\theta}_{m,n}\neq\bm{\theta}^\ast_n$ for all $m$, then $D_n(\bm{\theta}_{m+1,n})<D_n(\bm{\theta}_{m,n})$ and $D_n(\bm{\theta}_{m,n})\searrow D_n(\bm{\theta}^\ast_n)$.
\end{prop}

\begin{proof}[Proof sketch]
Uniqueness excludes nontrivial cycles and enforces single-valuedness/continuity of $M_n$ near stationary points; strict inequality follows from majorization unless at a minimizer of $Q_n(\cdot\mid \bm{\theta}_{m,n})$, which coincides with stationarity.
\end{proof}
\subsection{Contraction Properties}
We analyze the DM operator at both the population and sample levels. Throughout this subsection, assume the true model holds, $g(y)=f(y;\bm{\theta}^\star)$, and define $q(\cdot):=Q(\cdot\mid\bm{\theta}^\star)$. We will use the strong convexity of $q$ near $\bm{\theta}^\star$ and a first–order stability (FOS) condition that controls the sensitivity of the gradient map $\nabla Q(\cdot\mid\bm{\theta})$ with respect to its second argument.

\noindent \textbf{Geometric Convergence of the DM Sequences}

In this section, we focus on the guarantees for the population-level DM algorithm. 
\text{For }$q:\bm{\Theta}\to\mathbb{R},\ q \text{ is }\lambda\text{–strongly convex on }B_2(r;{\bm{\theta}}^\ast)\text{ if}
\quad
q({\bm{\theta}}_1)-q({\bm{\theta}}_2)-\langle \nabla q({\bm{\theta}}_2),\,{\bm{\theta}}_1-{\bm{\theta}}_2\rangle
\ge \tfrac{\lambda}{2}\,\|{\bm{\theta}}_1-{\bm{\theta}}_2\|_2^2.$
By self-consistency, $\bm{\theta}^\star=M(\bm{\theta}^\star)$, and the first-order optimality conditions are, for all $\bm{\theta}$, $\langle \nabla_{\bm{\theta'}} Q(\bm{\theta}^*\mid \bm{\theta}^*),\,\bm{\theta}-\bm{\theta}* \rangle \ge 0
\quad \text{and} \quad \langle \nabla_{\bm{\theta}'} Q(M(\bm{\theta})\mid\bm{\theta}),\,\bm{\theta}-M(\bm{\theta})\rangle \ge 0$.

\begin{defn}
(First-order Stability (FOS)) 
We say that $\{Q(\cdot\mid\bm{\theta}):\bm{\theta}\in\bm{\Theta} \}$ satisfies \textup{FOS($\gamma$)} on $B_2(r';\bm{\theta}^\star)$ if
\begin{align}\label{Defn:FOS}
||\nabla Q(M(\bm{\theta})|\bm{\theta}^* ) - \nabla Q(M(\bm{\theta})|\bm{\theta}) ||_2 \leq \gamma||\bm{\theta} -\bm{\theta}^* ||_2 \quad \text{for all} \quad \bm{\theta} \in \mathbb{B}_2(r';\bm{\theta}^*).
\end{align}    
\end{defn}

\begin{thm}\label{THM:Population:contraction}
For some radius $r'>0$ and pair $(\gamma, \lambda)$ such that $0<\gamma < \lambda$, suppose that the function $Q(\cdot|\bm{\theta}^*)$ is $\lambda$-strongly convex  and that the FOS($\gamma$) condition (\ref{Defn:FOS}) holds on the ball $\mathbb{B}_2(r';\bm{\theta}^*)$. Then the population DM operator $M$ is contractive over $\mathbb{B}_2(r';\bm{\theta}^*)$; in particular, the following inequality holds: 
\begin{align*}
||M(\bm{\theta}) - \bm{\theta}^* ||_2 \leq \frac{\gamma}{\lambda} ||\bm{\theta}-\bm{\theta}^* ||_2 \quad \text{for all} \quad \bm{\theta} \in \mathbb{B}_2(r';\bm{\theta}^*).
\end{align*}
As an immediate consequence, under the conditions of the theorem, for any initial point $\bm{\theta}_0 \in \mathbb{B}_2(r';\bm{\theta}^*)$, the population DM sequence $\{\bm{\theta}_m \}$ exhibits geometric convergence; that is,
\begin{align*}
|| \bm{\theta}_m - \bm{\theta}^*||_2 \leq \left( \frac{\gamma}{\lambda} \right)^m ||\bm{\theta}^0 - \bm{\theta}^* ||_2 \quad \text{for all} \quad m = 1, 2, \cdots .
\end{align*}
\end{thm}
The FOS($\gamma$) condition is standard in contraction analyses of EM-type
algorithms (cf. \cite{Siv17}) and, in our mixture setting, follows from Lipschitz regularity of the responsibilities and component scores; see the
Component Lipschitz regularity assumption and Corollary~\ref{cor:Kexplicit}
for an explicit bound on $\gamma_K$. We emphasize that FOS($\gamma$) is only
required locally in a neighbourhood of $\theta^\star$, and is not needed for the global monotone descent or stationarity results in Section~\ref{Properties-DM}; it is used solely to obtain geometric rates for the population map and its noisy sample analogue. We now turn to theoretical results on the sample-level DM algorithm. To this end, let
\begin{align}\label{pop-sam-lower}
\mathcal M_{\mathrm{unif}}(n,\rho)
:=\sup_{\bm{\theta}\in B_2(r;\bm{\theta}^*)}\ \inf_{\eta\in M_n(\bm{\theta})}
\big\|\eta - M(\bm{\theta})\big\|_2,
\qquad \mathbb{P}\!\left[\mathcal M_{\mathrm{unif}}(n,\rho)\le \varepsilon\right]\ge 1-\rho.
\end{align}
Our next result concerns the rate of convergence of the sample-level DM sequence to the population level minimizer of $D(\bm{\theta})$.
\begin{thm}[Noisy contraction]\label{thm:noisy}
Assume that the conditions of Theorem~\ref{THM:Population:contraction} with contraction factor $\kappa\in(0,1)$ on $B_2(r;\bm{\theta}^\ast)$ hold.
If $M_{\mathrm{unif}}(n,\rho)\le (1-\kappa)\,r$, then
\[
\|\bm \theta_{m,n}-\bm \theta^\ast\|_2 \le \kappa^m\|\bm \theta_{0,n}-\bm \theta^\ast\|_2 + \frac{M_{\mathrm{unif}}(n,\rho)}{1-\kappa}
\quad\text{with prob. }\ge 1-\rho.
\]
\end{thm}
We next provide an explicit form of the bound in Theorem \ref{thm:noisy}. We need one additional assumption and a related notation.

\textbf{Assumption (Component Lipschitz regularity).}
There exists $L_{\mathrm{comp}}>0$ such that for all $\theta,\theta'\in B_2(r';{\bm{\theta}}^\ast)$,
\[
\sup_{y}\max_{k}\Big\{|w_k(y;{\bm{\theta}})-w_k(y;{\bm{\theta}}')|
+ \|\nabla_{\phi}\log h(y;\phi_k({\bm{\theta}}))-\nabla_{\phi}\log h(y;\phi_k({\bm{\theta}}'))\|_2\Big\}
\;\le\; L_{\mathrm{comp}}\;\|{\bm{\theta}}-{\bm{\theta}}'\|_2 .
\]
Also, set $J (\bm{\eta}, \bm{\theta}) \coloneqq  D_2\big(\nabla_1 Q\big)(\bm{\eta}\mid \bm{\theta})$ where $\nabla_1 Q(\bm{\eta}|\bm{\theta})$ is the gradient of $Q(\cdot|\bm{\theta})$ wrt the first argument and $D_2$ is the Fr\'echet derivative wrt the second argument. Set 
\[
C_{\mathrm{fos}}
\coloneqq
\sup_{\bm{\theta}\in B_2(r';\bm{\theta}^*)}\ \sup_{\bm{\eta}\in M(\bm{\theta})}\|J(\bm{\eta}\mid \bm{\theta})\,\|_{\mathrm{op}}=   \sup_{\bm{\theta}\in B_2(r';\bm{\theta}^*)}\ \sup_{\bm{\eta}\in M(\bm{\theta})}\sqrt{\lambda_{max}(J^TJ)},
\]
where $\lambda_{max}(J^TJ)$ is the maximal eigen-value of the matrix $J^TJ$. Also, set $\gamma_K= \inf\{\gamma: \text{FOS($\gamma$) holds for the $K-$ component family on}~ \mathbb{B}_2(r';\bm{\theta}^*)\}$. Let $\pi_{min}$ denote the minimal mixing weight on the neighborhood. That is, $\pi_{min}=\inf_{\bm{\theta} \in \mathbb{B}(r';\bm{\theta}*)}\min_{1 \le k \le K} \pi_k(\bm{\theta}).$ Let $p(K)$ denote the model complexity; that is $p(K)=(K-1) (\text{mixing weights})+ K d_{\bm{\phi}} (\text{component parameters)}.$

\begin{cor}[Explicit $K$-scaling]\label{cor:Kexplicit}
If in addition $\gamma_K\le (C_{\mathrm{fos}}/\pi_{\min})K L_{\mathrm{comp}}$ so that $\kappa_K=\gamma_K/\lambda<1$, and
\[
M_{\mathrm{unif}}(n,\rho) \le C_{\mathrm{op}}A_{\max}\sqrt{\frac{p(K)+\log(1/\rho)}{n}},
\]
then, with probability at least $(1-\rho)$,
\[
\|{\bm{\theta}}_{m,n}-{\bm{\theta}}^\ast\|_2 \le \kappa_K^{\,m}\|{\bm{\theta}}_{0,n}-{\bm{\theta}}^\ast\|_2
+ \frac{C_{\mathrm{op}}A_{\max}}{1-\kappa_K}\sqrt{\frac{p(K)+\log(1/\rho)}{n}}.
\]
\end{cor}

\begin{rem}[How $d$ enters $M_{\mathrm{unif}}(n,r)$ when $M_n$ and $M$ are set-valued]
\label{rem:Munif-d-vs-pK}
Notice that $\Psi(\bm\theta;g)=\nabla_{\bm\theta} D_G(g,f_{\bm\theta})
= -\int A\!\Big(\frac{g(y)}{f_{\bm\theta}(y)}-1\Big) s_{\bm\theta}(y)\,f_{\bm\theta}(y)\,dy$, 
$s_{\bm\theta}=\nabla_{\bm\theta}\log f_{\bm\theta}$.
For any signed perturbation $h$ with $\int h=0$,
\[
\partial_g \Psi(\bm\theta;g)[h]
= -\!\int A'\!\Big(\frac{g(y)}{f_{\bm\theta}(y)}-1\Big)\, s_{\bm\theta}(y)\, h(y)\,dy.
\]
Thus, $\Psi(\bm\theta;g_{n})-\Psi(\bm\theta;g)
= \partial_g \Psi(\bm\theta;g)[\,g_n-g\,] \;+\; r_n(\bm\theta)$,and 
\quad 
$\sup_{\bm\theta\in\mathbb{B}_2(r;\bm\theta^\star)}\!\|r_n(\bm\theta)\|_2
= o_p\!\big(\|g_n-g\|_{\mathcal H}\big)$, where the remainder bound follows from local Lipschitz continuity of $A'$. Under the calibration $A'(0)=1$ and at the model ($g=f_{\bm\theta^\star}$), this simplifies to
$\partial_g \Psi(\bm\theta^\star;g)[h]=-\int s_{\bm\theta^\star}(y) h(y)\,dy$.
Hence, the leading plug-in effect depends only on the \emph{score class} $\{s_{\bm\theta}:\bm\theta\in\mathbb{ B}_2(r;\bm\theta^\star)\}$, and is bounded by its dual-norm envelope:
\[
\sup_{\bm\theta\in\mathbb{ B}_2(r;\bm\theta^\star)}\!
\|\Psi(\bm\theta;g_n)-\Psi(\bm\theta;g)\|_2
\;\le\; A'_{\max}\,\Env(K)\,\|g_n-g\|_{\mathcal H}\;+\;o_p\!\big(\|g_n-g\|_{\mathcal H}\big).
\]
Hence, for an
envelope/entropy constant $\Env(K)$ of $\{s_\xi:\ \xi\in M(\theta),\ \theta\in\mathbb B_2(r;\theta^\star)\}$ and
the seminorm
\[
\|h\|_{\mathcal H}\;:=\;\sup_{\theta\in\mathbb B_2(r;\theta^\star)}\ \sup_{\xi\in M(\theta)}\ \int \|s_\xi(y)\|_2\,|h(y)|\,dy,
\]
one obtains the high-probability bound
\begin{equation}\label{eq:Munif-decomp-set}
M_{\mathrm{unif}}(n,r)\ \lesssim\ C_\mathrm{fos}\,A'_{\max}\,\Env(K)\,\|g_n-g\|_{\mathcal H}\qquad\text{with prob.\ }\ge 1-\rho,
\end{equation}
where $C_\mathrm{fos}$ is the local FOS modulus and $A'_{\max}:=\sup_{\delta\ge -1}|A'(\delta)|$ is the RAF envelope
(both as used in Corollary~2). Strong convexity then converts score-level perturbations to argmin-level
deviations, yielding the noisy-contraction bound with $M_{\mathrm{unif}}(n,r)$ in place of $\|M_n-M\|$.  All dependence on the data dimension $d$ enters \emph{only} through the plug-in rate $\|g_n-g\|_{\mathcal H}$. For the discrete (finite or countable support) $\|g_n-g\|_{\mathcal H}=O_p(n^{-1/2})$; thus
$M_{\mathrm{unif}}(n,r)\lesssim C_\mathrm{fos}A'_{\max}\Env(K)\,n^{-1/2}$ and hence there is no dependence on $d$. In the continuous $d$-variate (kernel plug-in) case, for $\beta$-Hölder $g$ and a product kernel,
$\|g_n-g\|_{\mathcal H}=O_p\!\big(h^{\beta}+\sqrt{1/(n h^{d})}\big)$, optimized at
$n^{-\beta/(2\beta+d)}$—the standard nonparametric rate. Consequently, the \emph{model-side} constants—$C_\mathrm{fos}$, $A_{\max}$, local curvature/strong convexity, and any mixture-specific envelopes such as $\Env(K)\lesssim C_1\sqrt{K}$ or $C_2K/\pi_{\min}$—govern the explicit
$p(K)$-dependence, while $d$ affects only $\|g_n-g\|_{\mathcal H}$ via known density-estimation rates.
\end{rem}





\section{Asymptotic Results}\label{sec:asymptotics}

Throughout Sections~4.1–4.4 we take the number of mixture components $K$ as fixed and known.
Section~4.5 treats the case of unknown $K$ via split–select–estimate and dimension matching.
For each sample size $n$, let $m_n\in\mathbb{N}$ denote the number of DM iterations we run on $D_n$.
Accordingly, $\bm{\theta}_{m_n,n}$ denotes the iterate after $m_n$ updates at sample size $n$.

\subsection{Properties of Truncated Iterates}

\begin{defn}
A population (resp. sample) level DM algorithm sequence $\{\bm{\theta}_m\}$ (resp. $\{\bm{\theta}_{m, n} \}$) is called a population (resp. sample) level optimal sequence if $\bm{\theta}^* =
\lim_{m\to \infty} \bm{\theta}_m \in \underset{\bm{\theta}' \in \bm{\Theta}}{\text{argmin}} ~D(\bm{\theta}')$ 
(resp. $\bm{\theta}_n^* = \lim_{m\to \infty} \bm{\theta}_{m, n} \in \underset{\bm{\theta}' \in \bm{\Theta}}{\text{argmin}}~ D_n(\bm{\theta}')$). 
\end{defn}

The following theorem is concerned with the consistency and asymptotic normality of the finitely iterated sample-level optimal DM sequence $\{\bm{\theta}_{m,n} \}$. 
\begin{thm}\label{Thm:1}
Assume the number of components $K$ is fixed and the model is correctly specified $g=f_{\bm{\theta}^\star}$.
\smallskip
\begin{enumerate}
\item \noindent\textbf{ Consistency.}
Under \textbf{(C0)–(C3)}, any sample optimal DM sequence $\{\bm{\theta}_{m,n}\}$ satisfies
$\lim_{n\to\infty}\lim_{m\to\infty}\bm{\theta}_{m,n}=\bm{\theta}^\star$ almost surely. If $m_n\to\infty$ then $\bm{\theta}_{m_n,n}\stackrel{p}{\to}\bm{\theta}^\star$.

\smallskip
\item \noindent\textbf{ $\sqrt n$–normality (truncated iterates).}
Assume Theorem~\ref{THM:Population:contraction} holds: $M$ is $\kappa$–contractive on $B_2(r;\theta^\star)$ with $\kappa\in(0,1)$.
If $\theta_{0,n}\in B_2(r;\theta^\star)$, 
\[
m_n\ \ge\ \Big\lceil \frac{(\tfrac12+\delta)\log n + c_0}{|\log\kappa|}\Big\rceil\qquad (c_0>0, \delta >0 \text{ fixed}),
\]
then under the conditions  \textbf{(F1)}, \textbf{(C1)}–\textbf{(C2)}, \textbf{(K1)}–\textbf{(K2)}, \textbf{(M1)}–\textbf{(M8)}  
\[
\sqrt n\,(\bm{\theta}_{m_n,n}-\theta^\star)\ \Rightarrow\ \mathcal N\!\big(0,\ I(\bm{\theta}^\star)^{-1}\big).
\]
\emph{In particular, for some $\delta >0$} \(m_n = \lceil ( (\tfrac12+\delta) \log n)/|\log \kappa| \rceil + O(1)\), so only \(O(\log n)\) iterations are required.
\end{enumerate}
\end{thm}

\begin{cor}[Finite-step Godambe CLT]\label{cor:finite-godambe}
Let $\hat{\bm{\theta}}_n$ solve $\Psi_n(\bm{\theta}):=\nabla_{\bm{\theta}} D_G(g_n,f_{\bm{\theta}})=0$ and let $\bm{\theta}_{m+1,n}\in M_n(\bm{\theta}_{m,n})$ with $\bm{\theta}_{0,n}\in B_2(r;\bm{\theta}^\star)$.
Assume \textup{(G1)–(G4)} at ${\bm{\theta}}^\dagger:=\arg\min_{\bm{\theta}} D_G(g,f_{\bm{\theta}})$, and the contraction in Theorem~\ref{THM:Population:contraction}.
If $\sqrt n\,\kappa^{m_n}\to0$, then
\[
\sqrt n\,(\bm{\theta}_{m_n,n}-\bm{\theta}^\dagger)\ \Rightarrow\ \mathcal N\!\big(0,\ H^{-1} V H^{-1}\big),
\]
with $H:=\nabla_{\bm{\theta}}\Psi(\bm \theta^\dagger;g)$ and $V:=Var_g\!\big[A'(\tfrac{g}{f_{\bm{\theta}^\dagger}}-1)s_{\bm{\theta}^\dagger}(Y)\big]$.
In particular, under correct model specification and $A'(0)=1$, the covariance reduces to $I(\bm{\theta}^\star)^{-1}$.
\end{cor}

\begin{cor}[Finite-step (Godambe-)Wilks]\label{cor:finite-wilks}
Under the assumptions of Corollary~\ref{cor:finite-godambe} and $\sqrt n\,\kappa^{m_n}\to0$,
\[
2n\Big\{D_G(g_n,f_{\bm{\theta}^\dagger})-D_G(g_n,f_{\bm{\theta}_{m_n,n}})\Big\}
\ \Rightarrow\ \sum_{j=1}^{p(K_0)}\lambda_j\,\chi^2_{1,j},
\]
where $\{\lambda_j\}$ are eigenvalues of $J:=H^{-1/2} V H^{-1/2}$. 
If $g=f_{\bm{\theta}^\star}$ and $A'(0)=1$ then $H=V=I(\theta^\star)$ and the limit is $\chi^2_{p(K_0)}$.
\end{cor}

\subsection{Robustness}{\label{sec:robustness}}
For $\epsilon\in[0,1)$ and a contamination density $\eta_n$, define the $\epsilon$-contaminated model
\[
f_{\epsilon,n}(y;\bm{\theta}) \;:=\; (1-\epsilon)\,f(y;\bm{\theta}) \,+\, \epsilon\,\eta_n(y),
\qquad \bm{\theta}\in\Theta.
\]
Our first focus isn on contraction and noisy contraction under small contamination. For $\epsilon\in[0,\epsilon_0]$, set $g_\epsilon:=(1-\epsilon)g+\epsilon\,\eta$ and 
$\bm\theta^\dagger_\epsilon:=\arg\min_{\bm\theta} D_G(g_\epsilon,f_{\bm\theta})$.
Define
\[
Q_\epsilon(\bm\theta'\!\mid\bm\theta):=Q_G(\bm\theta'\!\mid\bm\theta;\,g_\epsilon),\quad
M_\epsilon(\bm\theta):=\arg\min_{\bm\theta'} Q_\epsilon(\bm\theta'\!\mid\bm\theta),
\]
and analogously $Q_{\epsilon,n}$ and $M_{\epsilon,n}$ with $g_{\epsilon,n}$.
Assume the fixed-order smoothness/curvature and FOS conditions of Theorem~\ref{THM:Population:contraction}
hold uniformly for $\epsilon\in[0,\epsilon_0]$ on a common ball $B_2(r;\bm\theta^\star)$.

\begin{cor}[Population contraction under contamination]\label{cor:pop-contraction-eps}
There exists $\bar\kappa\in(0,1)$ and $r>0$ such that, for every $\epsilon\in[0,\epsilon_0]$,
\[
\bigl\|M_\epsilon(\bm\theta)-\bm\theta^\dagger_\epsilon\bigr\|_2
\ \le\ \bar\kappa\,\|\bm\theta-\bm\theta^\dagger_\epsilon\|_2,\qquad
\forall\,\bm\theta\in B_2(r;\bm\theta^\dagger_\epsilon).
\]
\end{cor}
The proof of the corollary is in Appendix L. We now turn to the sample based contraction whose proof is relegated to the Appendix L.

\begin{cor}[Noisy contraction and opt-to-stat under contamination]\label{cor:noisy-contraction-eps}
Let
\[
\mathcal M_{\mathrm{unif}}(n,r;\epsilon)
:=\sup_{\bm\theta\in B_2(r;\bm\theta^\star)}\ \sup_{\bm\xi\in M_\epsilon(\bm\theta)}\
\inf_{\bm\eta\in M_{\epsilon,n}(\bm\theta)} \|\bm\eta-\bm\xi\|_2.
\]
Then any selection $\bm\theta_{t+1,\epsilon,n}\in M_{\epsilon,n}(\bm\theta_{t,\epsilon,n})$ with
$\bm\theta_{t,\epsilon,n}\in B_2(r;\bm\theta^\dagger_\epsilon)$ obeys
\[
\|\bm\theta_{t+1,\epsilon,n}-\bm\theta^\dagger_\epsilon\|_2
\ \le\ \bar\kappa\,\|\bm\theta_{t,\epsilon,n}-\bm\theta^\dagger_\epsilon\|_2
\ +\ \mathcal M_{\mathrm{unif}}(n,r;\epsilon).
\]
If, in addition, $\|g_{\epsilon,n}-g_\epsilon\|_{\mathcal H}=o_p(n^{-1/2})$ and $A'(0)=1$ with
$A'_{\max}:=\sup_{\delta\ge-1}|A'(\delta)|<\infty$, then $\mathcal M_{\mathrm{unif}}(n,r;\epsilon)=o_p(n^{-1/2})$ and any ($\delta >0$)
\[
m_n \ \ge\ \Big\lceil \frac{(\tfrac12+\delta)\log n + c_0}{|\log\bar\kappa|} \Big\rceil
\]
yields the opt-to-stat bound $\sqrt{n}\,\|\bm\theta^{(m_n)}_{\epsilon,n}-\widehat{\bm\theta}_{\epsilon,n}\|_2\to^p 0$
\end{cor}

\noindent \textbf{Remark (Robustness under contamination: bounded–RAF vs. KL).}
The sample–level term in Corollary~\ref{cor:noisy-contraction-eps} is controlled by the uniform operator–deviation bound
\[
\mathcal M_{\mathrm{unif}}(n,r;\varepsilon)\ \lesssim\
\frac{C_{\mathrm{fos}}\,A'_{\max}\,\mathrm{Env}(K)}{\lambda_\varepsilon}\,
\|g_{\varepsilon,n}-g_\varepsilon\|_{\mathcal H},
\]
on $B_2(r;\bm\theta^\dagger_\varepsilon)$ (see Supplement, Corollary S.R.1).  
Hence \textbf{bounded–RAF} generators (NED, vNED), for which $A'_{\max}<\infty$, yield
$\mathcal M_{\mathrm{unif}}(n,r;\varepsilon)=o_p(n^{-1/2})$ under the usual plug-in rate, and any
$m_n=O(\log n)$ gives the opt-to-stat conclusion.  
For \textbf{KL}, the same conclusion requires a local density/score floor; without it one can have
$\mathcal M_{\mathrm{unif}}(n,r;\varepsilon)\not=o_p(n^{-1/2})$ and the opt-to-stat step may fail (see Supplement, Proposotion~S.R.2).

We now turn to evaluate the influence function. Let $T(\cdot)$ denote the DM population functional and write $\bm{\theta}^\star_\epsilon(n):=T\!\big(f_{\epsilon,n}(\cdot;\bm{\theta})\big)$, with $\bm{\theta}^\star=\bm{\theta}^\star_0$.
\begin{thm}\label{Robust}
Let $\{\bm{\theta}_{\epsilon,m,n}\}$ be a sample-level DM optimal sequence at contamination level $\epsilon$, and set
$\bm{\theta}^\star_{\epsilon,n}:=\lim_{m\to\infty}\bm{\theta}_{\epsilon,m,n}$ whenever the limit exists.
Assume \textbf{(C2)} and \textbf{(O2)} hold uniformly in $n$ and $\epsilon\in[0,\epsilon_0)$.
\begin{enumerate}
\item For each fixed $\epsilon\in[0,\epsilon_0)$, suppose $\bm{\theta}^\star_{\epsilon,n}$ is unique for all $n\ge 1$.
Then $\{\bm{\theta}^\star_{\epsilon,n}\}_{n\ge1}$ is a \emph{bounded} sequence and
\[
\lim_{n\to\infty}\bm{\theta}^\star_{\epsilon,n} \;=\; \bm{\theta}^\star_\epsilon.
\]
\item If \textbf{(M1)}–\textbf{(M2)} hold, then $T$ is Gâteaux differentiable at $f(\cdot;\bm{\theta}^\star)$ along the mixture direction $\eta_n$, and
\[
\lim_{\epsilon\downarrow 0}\frac{\bm{\theta}^\star_{\epsilon,n}-\bm{\theta}^\dagger_{\epsilon}}{\epsilon}
\;=\; \big[I(\bm{\theta}^\star)\big]^{-1} \int_{\Real} \eta_n(y)\,u(y;\bm{\theta}^\star)\,dy,
\]
where $u(y;\bm{\theta})\coloneqq \nabla_{\bm{\theta}} \log f(y;\bm{\theta})$ and $I(\bm{\theta}^\star)$ is the Fisher information.
\end{enumerate}
\end{thm}

\begin{rem}[Influence function and gross-error sensitivity]
For point-mass contamination $\eta_n=\delta_y$, (ii) gives the influence function
$\mathrm{IF}(y;T,f)=I(\bm{\theta}^\star)^{-1}u(y;\bm{\theta}^\star)$,
so the gross-error sensitivity is $\sup_y\|I(\bm{\theta}^\star)^{-1}u(y;\bm{\theta}^\star)\|$.
(First-order efficiency also follows since $I(\bm{\theta}^\star)$ matches the MLE information.)
\end{rem}

\subsection{Breakdown Point Analysis}\label{sec:breakdown}
Roughly speaking, the breakdown point of a functional is the smallest proportion of contaminated value(s) that can produce arbitrary estimates.We consider the contamination path
\[
g_{\epsilon,n}(y) \;=\; (1-\epsilon)\,g(y)\;+\;\epsilon\,\eta_n(y),\qquad \epsilon\in[0,1].
\]
Following \cite{Simp87}, the (asymptotic) \emph{breakdown point} of a functional $T$ is
\[
\epsilon^\star \;:=\; \inf\Big\{\epsilon\in[0,1]:\ \exists\,\{\eta_n\}\ \text{with}\ \|T(g_{\epsilon,n})-T(g)\|\to\infty\ \text{as }n\to\infty\Big\}.
\]
\begin{defn}[Contaminated surrogate and update]\label{def:Qeps}
Define, for $\bm{\theta}',\bm{\theta}\in\Theta$,
\[
Q_{\epsilon,n}(\bm{\theta}'\mid \bm{\theta})
\;=\;
\mathbf{E}_Y\!\left[
\mathbf{E}_{Z\mid Y}\!\left\{
G\!\left(
-1 + \frac{g_{\epsilon,n}(Y)\,w(Z\mid Y;\bm{\theta})}{f(Y;\bm{\theta}')\,w(Z\mid Y;\bm{\theta}')}
\right)\right\}\right].
\]
A sample-level DM sequence at contamination level $\epsilon$ satisfies $\bm \theta_{\epsilon,m+1,n}\in
M_\epsilon(\bm{\theta}_{\epsilon,m,n})$ with $M_\epsilon(\bm{\theta}):=\arg\min_{\bm{\theta}'} Q_{\epsilon,n}(\bm{\theta}'\mid \bm \theta)$.
If $\lim_{n\to\infty} \bm \theta_{\epsilon,m,n}=\bm{\theta}^\star_{\epsilon,m}$ with $\bm{\theta}^\star_{\epsilon,m}=T(g_{\epsilon,m})$,
we call $\{\bm{\theta}^\star_{\epsilon,m}\}_{m\ge 0}$ a \emph{population} DM optimal sequence at contamination level $\epsilon$.
\end{defn}

\textbf{Notation for contamination and finite steps.}
For $\epsilon\in[0,\epsilon_0]$, let $\bm{\theta}^{\dagger}_{\epsilon}:=\arg\min_{\bm{\theta}} D_G(g_{\epsilon},f_{\bm{\theta}})$,
$\hat{\bm{\theta}}_{\epsilon,n}:=\arg\min_{\bm{\theta}} D_G(g_{\epsilon,n},f_{\bm{\theta}})$, and
$\bm{\theta}^{(m)}_{\epsilon,n}\in M_{\epsilon,n}^{(m)}(\bm{\theta}^{(0)}_{\epsilon,n})$ be the $m$-step DM iterate.
We assume the contraction in Theorem~\ref{THM:Population:contraction} holds uniformly on $\mathbb {B}_2(r;\bm{\theta}^\star)$
with factor $\kappa\in(0,1)$, and choose $m_n$ so that $\sqrt{n}\kappa^{m_n}\to0$ (e.g., 
$m_n\ge\lceil(\tfrac12\log n + c_0)/|\log\kappa|\rceil$).

\begin{thm}[Finite-step Godambe CLT under contamination]\label{thm:finite-godambe-eps}
Assume (G1)–(G4) hold uniformly for $\epsilon\in[0,\epsilon_0]$ at $\bm{\theta}^\dagger_\epsilon$ and
the bandwidth/plug-in rate satisfies $\|g_{\epsilon,n}-g_\epsilon\|_{\mathcal H}=o_p(n^{-1/2})$.
Then, with $m_n$ as above,
\[
\sqrt{n}\,\big(\bm{\theta}^{(m_n)}_{\epsilon,n}-\bm{\theta}^\dagger_\epsilon\big)
\ \Rightarrow\ \mathcal N\!\big(0,\ H_\epsilon^{-1} V_\epsilon H_\epsilon^{-1}\big),
\]
where $H_\epsilon:=\nabla_\theta\Psi(\theta^\dagger_\epsilon;g_\epsilon)$ and
$V_\epsilon:=Var_{g_\epsilon}\!\Big[A'\!\Big(\tfrac{g_\epsilon(Y)}{f_{\bm{\theta}^\dagger_\epsilon}(Y)}-1\Big)
s_{\bm \theta^\dagger_\epsilon}(Y)\Big]$.
\emph{In particular}, at the model ($\epsilon=0$) and with $A'(0)=1$, the covariance reduces to $I(\bm{\theta}^\star)^{-1}$.
\end{thm}
\begin{thm}Assume the conditions of Theorem~\ref{thm:finite-godambe-eps} and $\sqrt n\,\kappa^{m_n}\to0$.
Then the same limit holds with $\bm{\theta}^{(m)}_{\epsilon,n}$ replaced by $\bm{\theta}^{(m_n)}_{\epsilon,n}$.
\end{thm}

\begin{thm}[Uniform breakdown lower bound for bounded RAF]\label{thm:breakdown}
Let $g_\epsilon=(1-\epsilon)g+\epsilon q$, and let $\hat{\bm{\theta}}_\epsilon\in\arg\min_{\bm{\theta}} D_G(g_\epsilon,f_{\bm{\theta}})$.
Assume: (i) $D_G(g,\cdot)$ is $\lambda$-strongly convex on $B({\bm{\theta}}^\star,r)$ and attains its minimum at ${\bm{\theta}}^\star$; 
(ii) for some $S_K<\infty$ and all ${\bm{\theta}}\in B({\bm{\theta}}^\star,r)$, 
$\|\nabla_{\bm{\theta}} D_G(q,f_{\bm{\theta}})\|\le S_K\,A_{\max}$ with $A_{\max}:=\sup_{\delta\ge-1}|\RAF(\delta)|<\infty$ (e.g., NED/vNED); 
(iii) $B({\bm{\theta}}^\star,r)$ contains no other local minimum. Then, for any
$\epsilon<\epsilon^\dagger:=\lambda r/(S_K A_{\max})$,
we have $\hat{\bm{\theta}}_\epsilon\in B({\bm{\theta}}^\star,r)$ and 
$\|\hat{\bm{\theta}}_\epsilon-{\bm{\theta}}^\star\|\le (\epsilon S_K A_{\max})/\lambda$.
\end{thm}

\begin{rem}
In regular mixtures one may take $S_K\lesssim C_1\sqrt{p(K)}$ or $S_K\lesssim C_2 K/\pi_{\min}$, yielding
$\epsilon^\dagger\gtrsim (\lambda r)/(A_{\max} C_1\sqrt{p(K)})$ (resp.\ $\lambda r\,\pi_{\min}/(A_{\max} C_2 K)$).
For HD (unbounded RAF), the same bound holds under a bounded density‑ratio condition on $q$ in the neighborhood.
\end{rem}

\noindent\textbf{Remark(Breakdown under contamination: bounded–RAF vs unbounded–RAF).}
The lower bound in Theorem~\ref{thm:breakdown} hinges on the uniform envelope
$\| \nabla_{\bm\theta} D_G(q,f_{\bm\theta})\|\le S_K\,A_{\max}$ with 
$A_{\max}:=\sup_{\delta\ge -1}|A(\delta)|<\infty$ (bounded RAF).
Thus generators with bounded RAF (e.g. NED, vNED) enjoy a \emph{uniform} breakdown lower bound 
$\epsilon^\dagger \asymp \lambda r /(S_K A_{\max})$ on $B_2(r;\bm\theta^\star)$.
For unbounded RAFs (e.g. Hellinger, KL), $A_{\max}=+\infty$ and the uniform bound in Theorem~\ref{thm:breakdown} need not hold.
If a local density-ratio/score floor is imposed on the contamination within the neighborhood
(e.g. $q(y)\le (1+\Gamma)\,f(y;\bm\theta)$ for all $y$ and $\bm\theta\in B_2(r;\bm\theta^\star)$), 
then the same proof yields a local lower bound with $A_{\max}$ replaced by 
$A_\Gamma:=\sup_{\delta\in[-1,\Gamma]}|A(\delta)|<\infty$; see Supplement, Corollary in S.BD. 1 (for HD, $A_\Gamma=2(\sqrt{1+\Gamma}-1)$; for KL, $A_\Gamma=\max\{1,\Gamma-1\}$). 
Conversely, without any local floor one can construct contaminations for unbounded RAFs that violate any uniform
lower bound; see Supplement, Proposition in S.BD.2.

\section{\texorpdfstring{Extension of Asymptotic Analysis when $K$ is unknown}{alpha-spending Designs}}\label{sec:modelselection}
We treat the case where the number of components $K$ is unknown. Split the sample into two independent parts $D_{1n}$ and $D_{2n}$ of sizes $n_1$ and $n_2$ with $n_1/n \to \tau \in (0,1)$ and $n_2/n \to 1-\tau$. We use $D_{1n}$ to select $\hat{K}_n$ (e.g., via a divergence-based information criterion), and $D_{2n}$ to estimate the parameters conditional on $\hat {K}_n$. Let $K_0$ denote the true number of components and let $\bm{\theta}^\star$ be the true parameter with $K_0$ components. 

{\textbf{Dimension Matching:}}
Let $\widehat K_n$ be the estimator of the order $K_0$ using the selection split $\mathcal D_{1n}$ of size $n_1$. On the estimation split $\mathcal D_{2n}$ of size $n_2$, let
\[
\widehat{\bm\theta}_{n_2}\ \in\ \arg\min_{\bm\theta\in\Theta_{\widehat K_n}} D_{2n}(\bm\theta)
\quad\text{with}\quad
D_{2n}(\bm\theta)\coloneqq D\!\big(g_{n_2}, f_{\bm\theta}\big).
\]
Define the \emph{dimension‑matched estimator} $\overline{\bm\theta}_{n_2}$ and the
\emph{dimension‑matched truth} $\bm\theta^\star(\widehat K_n)$ by
\[
\overline{\bm\theta}_{n_2}=
\begin{cases}
(\widehat{\bm\theta}_{n_2},\mathbf 0), & \widehat K_n<K_0,\\[2pt]
\widehat{\bm\theta}_{n_2}, & \widehat K_n\ge K_0,
\end{cases}
\qquad
\bm\theta^\star(\widehat K_n)=
\begin{cases}
\bm\theta^\star, & \widehat K_n\le K_0,\\[2pt]
(\bm\theta^\star,\mathbf 0), & \widehat K_n>K_0,
\end{cases}
\]
so both live in $\mathbb R^{\,p\,\max\{\widehat K_n,K_0\}}$.  See \cite{kav18} for background on
dimension‑matching in mixtures.

{\textbf{Model Selection.}} To determine the number of unknown mixture components, we use the divergence-based mixture complexity estimator. Let $\Delta_K:=\{\pi\in[0,1]^K:\ \sum_{k=1}^K \pi_k=1\}$ and 
\[
\bm{\Theta}_K
:= \Big\{\, (\pi,\phi_1,\ldots,\phi_K):
\ \pi\in\Delta_K,\ \phi_k\in\Phi\ (k=1,\ldots,K) \,\Big\}.
\]

On the selection split $D_{1n}$ of size $n_1$, define the generalized divergence information criterion
\[
\mathrm{GDIC}_{n_1}(K)
\;:=\;
\widehat{\mathcal{R}}_{n_1}(K)
\;+\;
\frac{b_{n_1}}{n_1}\,p(K),
\qquad
\widehat{\mathcal{R}}_{n_1}(K)
:= \inf_{\bm{\theta} \in \bm{\Theta}_K} D_{1n}(\bm{\theta}),
\]
where $p(K)$ is the (identified) parameter dimension of $\bm{\Theta}_K$, and $b_{n_1}$ is a penalty weight with $b_{n_1}\to\infty$ and $\frac{b_{n_1}}{n_1}\to0$. We then choose
$\hat{K}_n \in \arg\min_{1\le K \le K_{\max}} \mathrm{GDIC}_{n_1}(K)$.

\emph{Default choice.} To match the usual BIC when $D_{1n}$ is the average negative log-likelihood, take
$b_{n_1}=\tfrac12 \log n_1$, so the penalty is $(p(K)\log n_1)/ (2n_1)$.

\begin{thm}[Consistency of the GDIC selector]\label{thm:gdic-consistency}
The following hold:
\begin{enumerate}
\item \textbf{Uniform LLN:} For each fixed $K\le K_{\max}$,
$\sup_{\bm{\theta}\in\bm{\Theta}_K}\,|D_{1n}(\bm{\theta})-D(\bm{\theta})| \xrightarrow{p} 0$ as $n_1\to\infty$.
\item \textbf{Identifiability gap:} Let $K_0$ be the true component number and $D_K^\star:=\inf_{\bm{\theta}\in \bm{\Theta}_K} D(\bm{\theta})$.
Then $D_K^\star>D_{K_0}^\star$ for all $K<K_0$.
\item \textbf{Local regularity at $K_0$:} $D$ is twice continuously differentiable at its (unique) minimizer
$\bm{\theta}^\star\in \bm{\Theta}_{K_0}$ with positive definite Hessian $H(\bm{\theta}^\star)$, and the sample minimizer
$\widehat{\bm{\theta}}_{K_0,1n}\in\arg\min_{\bm{\Theta}_{K_0}} D_{1n}$ satisfies
$\widehat{\bm{\theta}}_{K_0,1n}$ converges in probability to $\bm{\theta}^\star$ and
$D_{1n}(\widehat{\bm{\theta}}_{K_0,1n})-D(\bm{\theta}^\star)=O_p(n_1^{-1})$.
\item \textbf{Overfit control:} For each $K>K_0$, there exists a (possibly boundary) population minimizer
$\bm{\theta}_K^\dagger\in\bm{\Theta}_K$ with $D(\bm{\theta}_K^\dagger)=D(\bm{\theta}^\star)$ such that the sample minimum obeys
$D_{1n}(\widehat{\bm{\theta}}_{K,1n})-D(\bm{\theta}^\star)=O_p(n_1^{-1})$.

\item If $b_{n_1}\to\infty$ and $b_{n_1}/n_1\to0$, then
$
\mathbb{P}\!\big(\widehat K_n=K_0\big) \;\longrightarrow\; 1.
$
\end{enumerate}
 
\end{thm}

We now state the asymptotic theorem when $K$ is unknown.

\begin{thm}[Plug-in and post-selection CLTs on the estimation split]\label{thm:plugin-post}
Let the sample be split into independent parts $\mathcal D_{1n}$ (size $n_1$) and $\mathcal D_{2n}$ (size $n_2$),
with $n_1/n\to\tau\in(0,1)$ and $n_2/n\to 1-\tau$. On $\mathcal D_{1n}$ select $\widehat K_n$ (e.g., by GDIC),
and on $\mathcal D_{2n}$ compute the minimum-divergence estimator
\[
\widehat{\bm\theta}_{n_2}\in\arg\min_{\bm\theta\in\bm{\Theta}_K} D\!\big(g_{n_2},f_{\bm\theta}\big)
\quad\text{(for the order $K$ indicated below).}
\]
Assume \emph{correct model specification} at the true order $K_0$ (i.e., $g=f_{\bm\theta^\star}$ for some
$\bm\theta^\star\in\bm{\Theta}_{K_0}$), generator calibration $A'(0)=1$, and the plug-in rate
$\|g_{n_2}-g\|_{\mathcal H}=o_p(n_2^{-1/2})$ (e.g., empirical pmf or discrete-kernel with $h\!\to\!0$ and $n_2h\!\to\!\infty$).
Also suppose the fixed-order regularity conditions hold at $K_0$ (e.g., (F1), (C1)--(C2), (K1)--(K2), (M1)--(M8)).

\smallskip
\noindent\textbf{(a) Fixed order.}
For any fixed $K$ (in particular $K=K_0$),
\[
\sqrt{n_2}\,\big(\widehat{\bm\theta}_{n_2}-\bm\theta^\star\big)\ \Rightarrow\ 
\mathcal N\!\big(0,\ I(\bm\theta^\star)^{-1}\big).
\]

\noindent\textbf{(b) Post-selection (unconditional).}
If, in addition, the selector is consistent, $\mathbb{P}(\widehat K_n=K_0)\to 1$ (e.g., under the conditions of the GDIC
consistency theorem), define the dimension-matched estimator $\overline{\bm\theta}_{n_2}$ and truth
$\bm\theta^\star(\widehat K_n)$ by
\[
\overline{\bm\theta}_{n_2}=
\begin{cases}
(\widehat{\bm\theta}_{n_2},\mathbf 0), & \widehat K_n<K_0,\\[2pt]
\widehat{\bm\theta}_{n_2}, & \widehat K_n\ge K_0,
\end{cases}
\qquad
\bm\theta^\star(\widehat K_n)=
\begin{cases}
\bm\theta^\star, & \widehat K_n\le K_0,\\[2pt]
(\bm\theta^\star,\mathbf 0), & \widehat K_n>K_0,
\end{cases}
\]
so both live in $\mathbb R^{\,p\,\max\{\widehat K_n,K_0\}}$. Then
\[
\sqrt{n_2}\,\big\{\overline{\bm\theta}_{n_2}-\bm\theta^\star(\widehat K_n)\big\}
\ \Rightarrow\ \mathcal N\!\big(0,\ I(\bm\theta^\star)^{-1}\big).
\]
\emph{Remark.} If the DM algorithm is truncated on $\mathcal D_{2n}$, it suffices to take 
$m_{n_2}=O(\log n_2)$ iterations (Theorem~3(ii)) so that the optimization error is $o_p(n_2^{-1/2})$.
\end{thm}

\noindent \textbf{Robust selection under bounded–RAF disparities.}
For $\epsilon\in[0,\epsilon_0)$ let $g_\epsilon:=(1-\epsilon)g+\epsilon\,\eta$ and write
$D_\epsilon(\bm \theta):=D_G(g_\epsilon,f_{\bm\theta})$ and $D_{\epsilon,1n}({\bm\theta}):=D_G(g_{\epsilon,1n},f_{\bm\theta})$
on the selection split $\mathcal D_{1n}$.
For $K<K_0$ set the (contaminated) population gap
\[
\Delta_\epsilon(K)\ :=\ \inf_{{\bm\theta}\in\bm\Theta_K} D_\epsilon({\bm\theta})\;-\;\inf_{{\bm\theta}\in\bm\Theta_{K_0}} D_\epsilon({\bm\theta}),
\qquad \Delta_0(K)>0 \ \text{by Theorem~\ref{thm:gdic-consistency}(2)}.
\]

We next turn to investigate the  robustness features of the model selector described in Theorem ~\ref{thm:gdic-consistency}.

\begin{thm}[GDIC: finite-sample over/under-estimation bounds for bounded RAFs]\label{thm:gdic-robust-bounds}
Assume \textup{(F1)}, \textup{(C0)–(C3)}, \textup{(K1)–(K2)}, \textup{(M1)–(M8)} on a fixed neighborhood of $\bm\theta^\star$,
and that the RAF derivative is bounded and nonincreasing on $[0,\infty)$:
$A'(0)=1$, $A'_{\max}:=\sup_{\delta\ge-1}|A'(\delta)|<\infty$. Suppose the selection-split plug-in satisfies
$\|g_{\epsilon,1n}-g_\epsilon\|_{\mathcal H}=o_p(n_1^{-1/2})$. Let $b_{n_1}\to\infty$ with $b_{n_1}/n_1\to 0$ (e.g. $b_{n_1}=\tfrac12\log n_1$).

\smallskip
\noindent\textbf{(Underestimation, $K<K_0$).}
There exist constants $c_1,c_2>0$ (depending only on the local envelopes and curvature) such that for every fixed $K<K_0$ and all sufficiently small $\epsilon\le \epsilon^\flat$,
\[
\mathbb{P}\!\big(\widehat K_n=K\big)
\ \le\ \exp\!\Big(-c_1\,n_1\,\Delta_\epsilon(K)^2\Big)\ +\ c_2\,\mathbb{P}\!\Big(\|g_{\epsilon,1n}-g_\epsilon\|_{\mathcal H}>\tfrac12\Delta_\epsilon(K)\Big),
\]
where $\Delta_\epsilon(K)\ge \Delta_0(K)-C\,\epsilon$ with $C\lesssim A'_{\max}\,\mathrm{Env}(K)$. In particular,
\(
\mathbb{P}(\widehat K_n<K_0)\le \sum_{K<K_0}\exp(-c_1 n_1(\Delta_0(K)-C\epsilon)^2)+o(1).
\)

\smallskip
\noindent\textbf{(Overestimation, $K>K_0$).}
Let $\nu_K:=p(K)-p(K_0)$ and assume the (Godambe–)Wilks expansion for the selection-split DM contrast:
\[
2n_1\Big\{D_{1n}(\widehat{\bm\theta}_{K_0,1n})-D_{1n}(\widehat{\bm\theta}_{K,1n})\Big\}
\ \rightsquigarrow\ \chi^2_{\nu_K}\qquad(K>K_0).
\]
Then, with $b_{n_1}=\tfrac12\log n_1$ (BIC-type GDIC),
\[
\mathbb{P}\!\big(\widehat K_n\ge K_0+1\big)
\ \le\ \sum_{K=K_0+1}^{K_{\max}}
\mathbb{P}\!\Big(\chi^2_{\nu_K}\ \ge\ \nu_K\,\log n_1 + o(1)\Big)
\ =\ O\!\Big(\sum_{K=K_0+1}^{K_{\max}} n_1^{-\nu_K/2}\Big).
\]
Hence GDIC with BIC-penalty is (over-fit) consistent even under small contamination for the bounded-RAF class.
\end{thm}
The proof of the Theorem is relegated to the appendix.

\noindent \textbf{Remark (Finite-sample robustness of GDIC vs BIC/AIC under contamination).}
On the selection split, each observation contributes $O(n_1^{-1})$ to $D_{1n}(\cdot)$.
For divergence generators with \emph{bounded RAF} $A$ (e.g. NED, vNED), the per-point decrement in $D_{1n}$ achievable by inserting an additional component is uniformly bounded by $A_{\max}/n_1$. With a BIC-type penalty $b_{n_1}=\tfrac12\log n_1$, a single “extraneous” point cannot overturn the penalty once $n_1$ is large, so GDIC does not overestimate $K$ on the basis of a single contaminated datum. In contrast, for KL (likelihood disparity), the per-point contribution to $D_{1n}$ is unbounded in the tail; absent a local density/score floor, one extreme contaminated point can reduce the KL-based contrast by more than the BIC penalty and spuriously trigger an extra component. The next two results make this precise. (Full proofs are in the Supplement.)

\begin{prop}[GDIC overfit control with bounded RAFs]\label{prop:gdic-bounded-overfit}
Assume the setting of Theorem~\ref{thm:gdic-consistency} on the selection split, and that the RAF is bounded:
$A_{\max}:=\sup_{\delta\ge -1}|A(\delta)|<\infty$. Fix $K>K_0$ and let $\Delta p:=p(K)-p(K_0)\ge 1$.
Consider $b_{n_1}\ge c\log n_1$ (e.g. $b_{n_1}=\tfrac12\log n_1$). Then for every data set of size $n_1$, set $SS_K \coloneqq \{\text{selection–split points that a $K$–only component captures}\}$
\[
\mathrm{GDIC}_{n_1}(K)-\mathrm{GDIC}_{n_1}(K_0)
\;\ge\; \frac{\Delta p\,b_{n_1}}{n_1}\;-\;\frac{A_{\max}}{n_1}\,\#SS_K\;-\;R_{n_1},
\]
with $R_{n_1}=O_p(n_1^{-1/2})$ from sampling noise. In particular, for any threshold $m_\star$
\[
\#\{SS_K\}\ \le\ m_\star
\quad\Longrightarrow\quad
\mathrm{GDIC}_{n_1}(K)-\mathrm{GDIC}_{n_1}(K_0)\ \ge\ \frac{\Delta p\,b_{n_1}-A_{\max} m_\star}{n_1} - R_{n_1}.
\]
Hence a \emph{single} extraneous point ($m_\star=1$) cannot induce overestimation for all sufficiently large $n_1$,
because $\Delta p\,b_{n_1}/n_1 \gg A_{\max}/n_1$. More generally, at least
\[
m_{\min}(n_1,K)\ :=\ \big\lceil (\Delta p)\,b_{n_1}/A_{\max}\big\rceil
\]
contaminated points on the selection split are necessary (up to $o_p(1)$) to force $\widehat K_n\ge K$.
\end{prop}
The proof is in the Supplement, Proposition in S-GDIC, B1.

\begin{cor}[Tail bound under $\varepsilon$-contamination; bounded RAFs]\label{cor:gdic-overfit-tail}
Under the assumptions of Proposition~5, if the selection split has i.i.d.\ $\varepsilon$-contamination with rate $\varepsilon\in(0,1)$, then for any fixed $K>K_0$ and all large $n_1$,
\[
\mathbb{P}\!\big(\widehat K_n\ge K\big)\ \le\ \mathbb{P}\!\big(\mathrm{Bin}(n_1,\varepsilon)\ge m_{\min}(n_1,K)\big)\;+\;o(1)
\ \le\ \exp\!\Big(-\,n_1\,\mathrm{kl}\!\big(\tfrac{m_{\min}(n_1,K)}{n_1}\,\|\,\varepsilon\big)\Big)\;+\;o(1),
\]
where $\mathrm{kl}(x\|y)$ is the Bernoulli KL divergence. In particular, with $b_{n_1}=\tfrac12\log n_1$,
$m_{\min}(n_1,K)\asymp (\Delta p)\log n_1/A_{\max}$ and $\mathbb{P}(\widehat K_n\ge K)=O(n_1^{-c})$ for some $c>0$.
\end{cor}
The Proof of the result standard.

\begin{thm}[Sample-level breakdown lower bounds for GDIC with bounded RAF]\label{thm:gdic-breakdown-sample}
Work on the selection split $\mathcal D_{1n}$ of size $n_1$.
Assume \textup{(F1)}, \textup{(C0)–(C3)}, \textup{(K1)–(K2)}, \textup{(M1)–(M8)} on a neighborhood of $\bm\theta^\star$.
Suppose the divergence generator has bounded RAF $A$ (i.e. $A_{\max}:=\sup_{\delta\ge -1}|A(\delta)|<\infty$), and the plug-in satisfies
$\|g_{1n}-g\|_{\mathcal H}=o_p(n_1^{-1/2})$. Let $b_{n_1}\to\infty$ with $b_{n_1}/n_1\to 0$ and write $\nu_K:=p(K)-p(K_0)$.

\begin{enumerate}
\item \textbf{(Underestimation)} For each fixed $K<K_0$ let $\,\Delta_0(K):=\inf_{\bm\theta\in\bm\Theta_K}\!D(\bm\theta)-D(\bm\theta^\star)>0$ (Theorem~8(2)). Then for any $0<\varepsilon<\varepsilon_{\rm under}(K):=\Delta_0(K)/(4A_{\max})$,
\[
\mathbb{P}\!\big(\widehat K_n=K\big)\ \le\ 
\exp\!\Big(-c_1\,n_1\,[\,\Delta_0(K)-2A_{\max}\varepsilon\,]^2\Big)\ +\ o(1),
\]
for some $c_1>0$ depending only on the local envelopes; hence $\mathbb{P}(\widehat K_n<K_0)\to 0$ for any fixed $\varepsilon<\min_{K<K_0}\varepsilon_{\rm under}(K)$.

\item \textbf{(Overestimation)} For each $K>K_0$ define 
$ m_{\min}(K,n_1):=\big\lceil \nu_K\,b_{n_1}/A_{\max}\big\rceil$.
Let $X_{n_1}$ be the number of contaminated points falling in $\mathcal D_{1n}$ (e.g. $X_{n_1}\sim{\rm Bin}(n_1,\varepsilon)$ under i.i.d. $\varepsilon$–contamination).
Then for every $K>K_0$ and all large $n_1$,
\[
\mathbb{P}\!\big(\widehat K_n\ge K\big)\ \le\ \mathbb{P}\!\big(X_{n_1}\ge m_{\min}(K,n_1)\big)\;+\;o(1).
\]
In particular, with BIC-type penalty $b_{n_1}=\tfrac12\log n_1$ and any fixed $\varepsilon<\varepsilon_{\rm over}(K):=\frac{\nu_K}{2A_{\max}}\frac{\log n_1}{n_1}$,
\[
\mathbb{P}\!\big(\widehat K_n\ge K\big)\ \le\ \exp\!\big(-\,c_2\,n_1\big)\;+\;o(1),
\]
for some $c_2=c_2(\varepsilon,\nu_K,A_{\max})>0$. In particular, a \emph{single} extraneous point cannot force overestimation ($\widehat K_n\ge K_0{+}1$) once $n_1$ is large.
\end{enumerate}
\end{thm}

\noindent \textbf{Remark (Unbounded RAFs).}
For generators with unbounded RAF (e.g. KL, HD), $A_{\max}=+\infty$ and the uniform bounds in Theorem~\ref{thm:gdic-breakdown-sample} do not apply. If a local density/score floor holds on the selection split (e.g. $q(y)\le (1+\Gamma)\,f(y;\bm\theta)$ and $\inf_{\bm\theta\in B_2(r;\bm\theta^\star)}f(y;\bm\theta)\ge c_*>0$ on the relevant neighborhood), the same conclusions hold with $A_{\max}$ replaced by $A_\Gamma:=\sup_{\delta\in[-1,\Gamma]}|A(\delta)|<\infty$. Without such a floor, a single extreme point can overturn a BIC penalty for KL, and no sample-level breakdown lower bound holds (see Supplement, Prop. S.GDIC-KL).

\noindent\textbf{Remark (Positive breakdown and per–point influence for GDIC--DM).}
On the selection split, each observation contributes $O(n_1^{-1})$ to the GDIC contrast
$D_{1n}(\cdot)$. For divergence generators with bounded RAF $A$ (e.g. NED, vNED), the
per–point decrement achievable by inserting an extra component is uniformly bounded by
$A_{\max}/n_1$, whereas the BIC–type penalty contributes $(\Delta p)\,b_{n_1}/n_1$ with
$\Delta p = p(K) - p(K_0) \ge 1$. Thus, a single contaminated point cannot overturn the
penalty once $n_1$ is large; more generally, at least a nonvanishing fraction of contaminated points is required to force overestimation, as quantified in Theorem~\ref{thm:gdic-breakdown-sample}. In particular, GDIC--DM with bounded RAF enjoys a strictly positive sample–level breakdown bound, in sharp contrast to KL–based EM and BIC, whose unbounded RAF allows a single extreme point to destroy any uniform breakdown lower bound.

Our next result is the stable post-selection limit distribution of the minimum divergence estimator.

\begin{thm}[Stable post-selection CLT]\label{thm:stable-post}
Let $\mathcal D_{1n}$ and $\mathcal D_{2n}$ be independent splits with $|\mathcal D_{2n}|=n_2$.
Let $\widehat K_n=\widehat K_n(\mathcal D_{1n})$ be the order selected on $\mathcal D_{1n}$ (e.g., GDIC),
and let $\overline{\bm\theta}_{n_2}=\overline{\bm\theta}_{n_2}(\mathcal D_{2n};\widehat K_n)$ denote the
dimension-matched MDE on $\mathcal D_{2n}$ at order $\widehat K_n$.
Assume that model is correctly specified at $K_0$; that is, $g=f_{\bm\theta^\star}$ for some $\bm\theta^\star\in\bm{\Theta}_{K_0}$ and $A'(0)=1$. Assume that $\|g_{n_2}-g\|_{\mathcal H}=o_p(n_2^{-1/2})$ and that the conditions used for the fixed-order plug-in CLT hold at $K_0$; that is, conditions (F1), (C1)–(C2), (K1)–(K2), (M1)–(M8) hold. Let $\widehat{K}_n$ be a consistent estimator of $K_0$. Then conditionally on $\mathcal{D}_{1n}$,
\[
\sqrt{n_2}\,\big\{\overline{\bm\theta}_{n_2}-\bm\theta^\star(K_0)\big\}
\ \Rightarrow\ \mathcal N\!\big(0,\,I(\bm\theta^\star)^{-1}\big)\quad\text{in probability},
\]
and hence the same convergence holds unconditionally.
\end{thm}
\begin{rem}[Repeated splits / majority vote]
In practice, we may repeat the split $C$ times and aggregate $\widehat K_n$ by majority vote; if each split’s selector is consistent, the aggregated selector is also consistent, and Theorem~\ref{thm:gdic-consistency} continues to hold.
\end{rem}

\section{Numerical Experiments}\label{sec:simulation}

We evaluate DM instantiations with Hellinger (HD), and vNED divergences on synthetic finite mixtures (Poisson, Poisson--Gamma/negative binomial, Poisson--lognormal), focusing on contamination and model selection. We use 5000 Monte Carlo repetitions. Full experimental details, additional models (Poisson, Poisson-lognormal), and runtime comparisons are provided in the Supplementary materials \ref{S-add-simulations}. We first demonstrate robustness with known $K$ and then evaluate the full split-select-estimate pipeline with unknown $K$. In the following tables, 'Ave' represents the average value of the estimates and 'StD' denotes the standard deviation." All methods are initialized using k-means clustering.



\subsection{Robustness Simulation}

\noindent \textbf{PG Mixture}

We simulate a two-component PG mixture ($K=2$) and true parameters
$(\pi_1,\alpha_1,\beta_1,\alpha_2,\beta_2)=(0.3,10,1,1,2)$ (details in supplementary materials \ref{S-add-simulations}). We inject point-mass contamination by replacing an $\ep-$ fraction of values with the value 50. Figure \ref{fig:pg_outlier} plots the average estimates versus $\epsilon$.


\begin{figure}[H]
    \centering
    \includegraphics[width=0.8\textwidth]{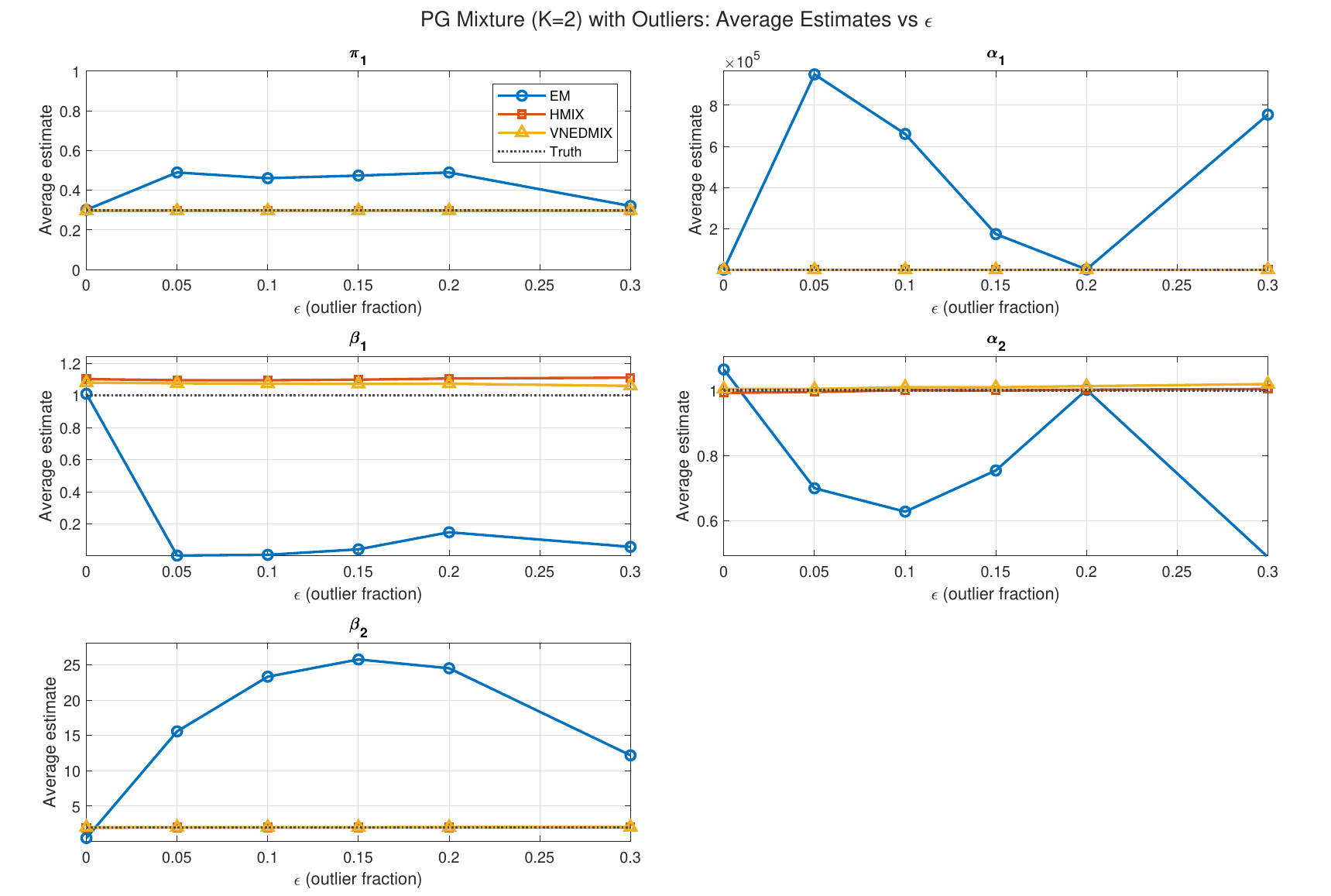}
    \caption{\small Average parameter estimates versus contamination level $\epsilon$ in a two–component PG mixture with known $K=2$.}
    \label{fig:pg_outlier}
\end{figure}
As expected of likelihood-based methods, EM is sensitive to large injected counts and its averages drift with $\epsilon$.
HMIX and vNEDMIX, which replace the likelihood by divergences against the empirical distribution, down-weight tail mismatches and remain closer to truth over a wider range of $\epsilon$. 
For readability, $y$-axes are trimmed using robust quantiles so occasional non-convergent runs at high contamination do not distort the display.


\subsection{The number of mixtures \texorpdfstring{$K$}{K} is unknown}\label{chap5:subS:K-unknown}



We now estimate both the number of components and the parameters via split–select–estimate with repeated splits.
Let $\mathcal{D}_{1n}$ and $\mathcal{D}_{2n}$ be an even random split of the sample of sizes $n_1$ and $n_2$. We adopt the following four-step process:
\begin{enumerate}[leftmargin=1.5em,itemsep=0.25em]
\item \emph{Select $K$ on $\mathcal{D}_{1n}$:} compute $\widehat K_n$ by GDIC with penalty $b_{n_1}=\tfrac12\log n_1$.
\item \emph{Estimate on $\mathcal{D}_{2n}$:} fit the MDE at order $\widehat K_n$ to obtain $\bm{\widehat\theta}$.
\item \emph{Repeat $C$ times:} take a majority vote over $\widehat K_n^{(1)},\dots,\widehat K_n^{(C)}$; 
average parameters across runs that select the voted order.
\item \emph{Repeat for $R$ Monte Carlo runs} and report Ave/StD/MSE (we use $R=5{,}000$ in the tables).
\end{enumerate}
All results in this subsection use the empirical kernel for $g_n(\cdot)$.


\begin{table}[H]
    \scriptsize
    \centering
    \captionsetup{font=small}
    \caption{Parameter Estimation with Unknown $K$ with Data Splitting Method}
    \label{PG_m1_66}
    \begin{tabular}{ccccccc}
        \toprule
        PG model &     & $\hat{\pi}_1$ & $\hat{\alpha}_1$ & $\hat{\beta}_1$ & $\hat{\alpha}_2$ & $\hat{\beta}_2$ \\ \hline
        \multirow{2}{*}{EM ($100.0\%$)} 
            & Ave & 0.297 & 10.62 & 1.054 & 0.937 & 1.826 \\ \cline{2-7}
            & StD & 0.004 & 0.675  & 0.063 & 0.078 & 0.192 \\ \cline{2-7}
            \hline
        \multirow{2}{*}{HMIX ($98.2\%$)}  
            & Ave & 0.298 & 10.61 & 1.060 & 0.988 & 1.962 \\ \cline{2-7}
            & StD & 0.004 & 0.692  & 0.065 & 0.083 & 0.205 \\ \cline{2-7}
            \hline
        \multirow{2}{*}{VNEDMIX ($86.2\%$)} 
            & Ave & 0.298 & 10.48 & 1.048 & 0.996 & 1.986 \\ \cline{2-7}
            & StD & 0.004 & 0.698  & 0.066 & 0.084 & 0.209 \\ 
             \bottomrule
    \end{tabular}
\end{table}


Table~\ref{PG_m1_66} summarizes a two-component PG mixture with truth $(\pi_1,\alpha_1,\beta_1,\alpha_2,\beta_2)=(0.3,10,1,1,2)$ using $C=5$ repeats in the split–select–estimate pipeline.
For each method (EM, HMIX, VNEDMIX) we first identify $K$ and then report Ave/StD computed \emph{only} on datasets where that method correctly identified $K=2$.
All three methods are essentially unbiased for component means. For example, the implied means $\alpha_1/\beta_1$ are $\approx 10.08$ (EM), $10.01$ (HMIX), and $10.00$ (VNEDMIX) versus the truth $10$, and for component~2 they are approximately $0.51$, $0.50$, and $0.50$ versus the truth $0.5$. Dispersion parameters are close to the truth with small variability.
When $K$ is correctly identified, EM, HMIX, and VNEDMIX produce comparably accurate estimates in this PG setting.

\section{Data Analysis: Image Segmentation}\label{sec:case}

In this section, we apply our method to the image segmentation problem. 
Our methods will be tested on the Lena image
(see original image \ref{S-ori:image_b}) which contains $770\times 776$ pixels. The gray-scale intensity values of the images range from 0 to 255 (all are integer-valued). We treat each pixel as one data point and fit a three-component Poisson mixture model for both images and apply EM algorithm, HMIX algorithm, and VNEDMIX algorithm.
The empirical density estimate is used as the non-parametric density estimate. The gray-scale intensity values are labeled as 1, 100, and 200 to represent three classes. 


\textbf{Robust Image Recovery}

We study the robustness property of the proposed method. 
Specifically, we generate contaminated pixel data following the Poisson distribution with mean $250$ with probability 0.3; if the generated value exceeds 255, then set it as 255 (since the maximum value for color is 255 under this format). We fit a three-component Poisson mixture model using EM algorithm, HMIX algorithm, and VNEDMIX algorithm. From Figure \ref{fig:lena-orig-hmix},
we observe that the EM algorithm captures a lot of ``noise'' added to the image (the face has a lot of shadow area). 
In comparison, the HMIX algorithm and VNEDMIX algorithm are able to ignore the ``outliers'' well (the EM and VNEDMIX algorithm recovered figure are shown in supplementary material \ref{S-em_a:poiss} and \ref{S-vNEDmix_a:poiss} respectively); in particular, the face contours are captured better. Furthermore, the point estimates based on the EM algorithm are highly affected while the other two algorithms give similar results to those from uncorrupted image. 

One may argue that once we add $30\%$ outliers, these outliers may be treated as another component. However, even if we fit a four-component Poisson mixture model and apply EM algorithm, it still produces as much shadow similar to using three-component Poisson mixture (image not displayed here). In contrast, HMIX and VNEDMIX will potentially eliminate this additional component and produce more ``reasonable'' image. Additional image analysis including original image recovery, model selection tables can be found in supplementary analysis \ref{S-supp:add-image-ana}.

\begin{figure}[H]
  \centering
  \captionsetup{font=footnotesize}
  \newcommand{\panelw}{0.47\linewidth} 

  \begin{subfigure}{\panelw}
    \centering
    \includegraphics[width=\linewidth,keepaspectratio]{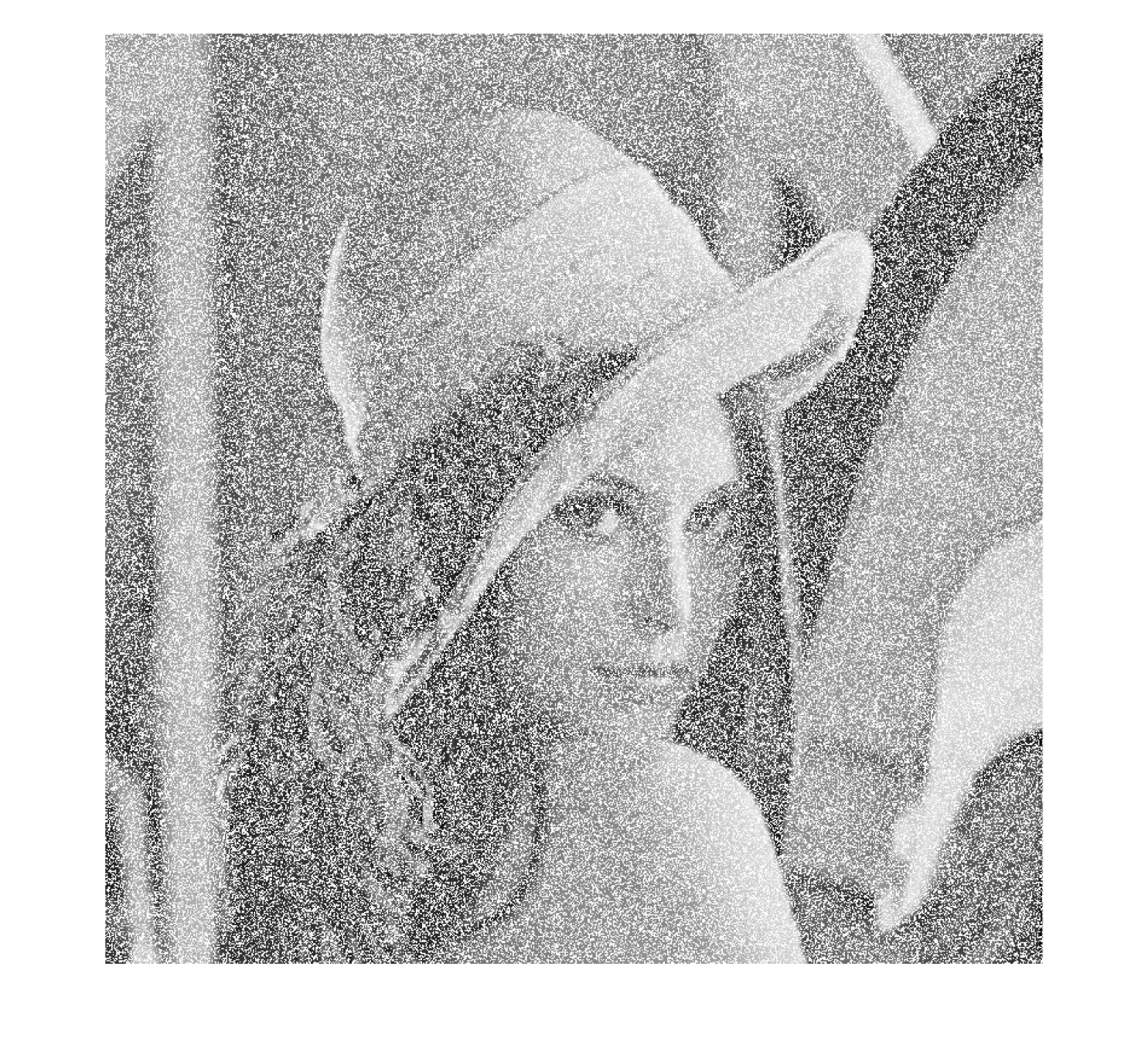}
    \caption{Corrupted image for Lena}
    \label{corrupted_c}
  \end{subfigure}
  \hfill
  \begin{subfigure}{\panelw}
    \centering
    \includegraphics[width=\linewidth,keepaspectratio]{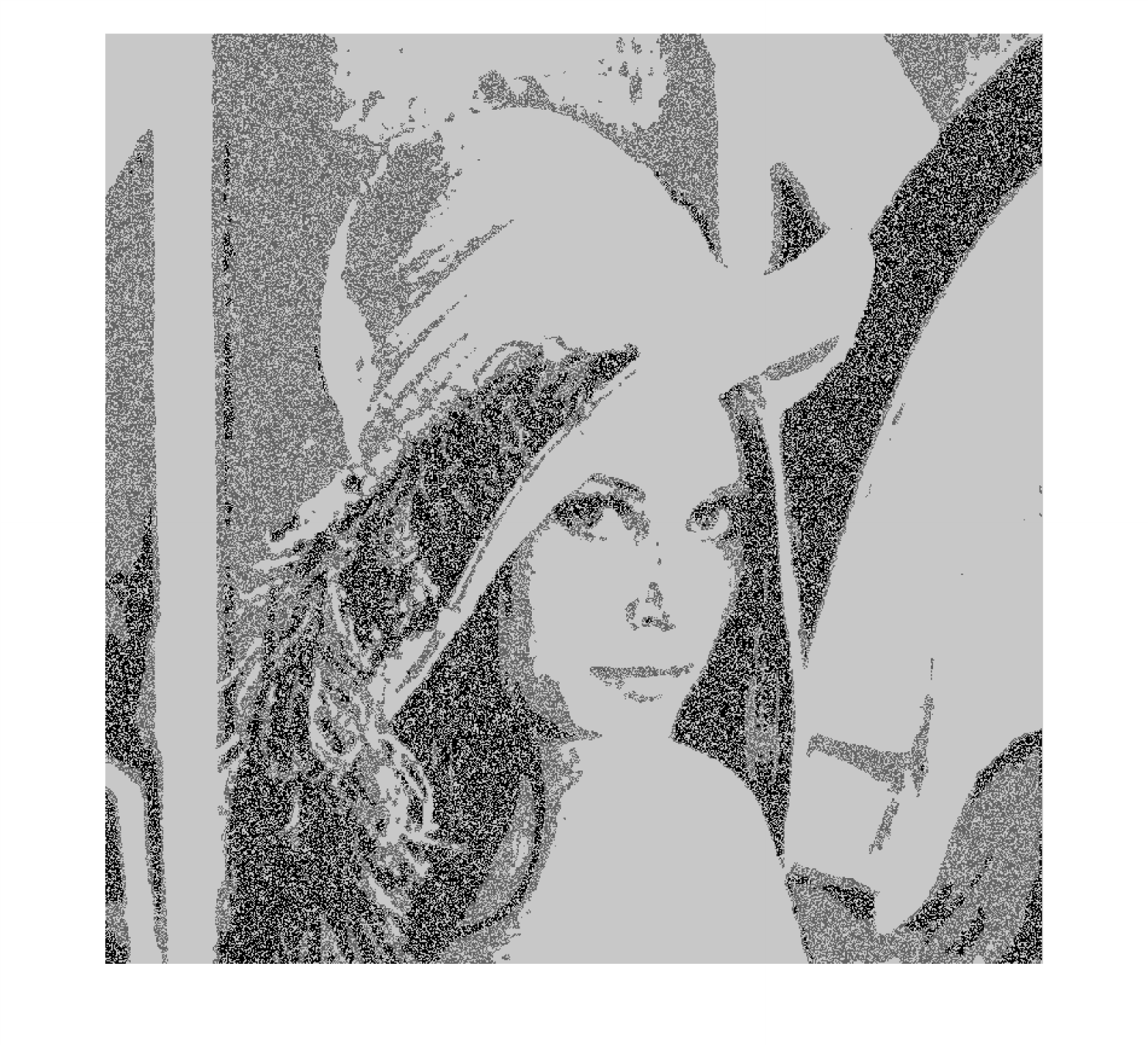}
    \caption{Recovered image from HMIX algorithm}
    \label{fig:hmix_recon_lena}
  \end{subfigure}

  \caption{Lena image reconstruction after adding $30\%$ outliers.}
  \label{fig:lena-orig-hmix}
\end{figure}


\addtolength{\textheight}{-.2in}%

\section{Concluding Remarks}\label{Chap5:S:Conclusion}

We developed a divergence–minimization (DM) framework for models with latent structure that unifies classical algorithms (EM, HMIX, HELMIX) and yields robust, likelihood–free updates through a single operator. At the algorithmic level, we established a majorization–separation identity that links a complete–data surrogate to the observed–data divergence, and used it to prove monotone descent at the sample level together with global stationarity of limit points. At the population level, we provided a local contraction result under strong convexity and first–order stability, and derived a noisy-contraction bound for the sample operator.

For model selection, we proposed a divergence-based criterion (GDIC) with repeated sample-splitting and showed that it consistently recovers the mixture order under mild conditions; the post-selection MDE enjoys the usual $\sqrt{n}$-asymptotic normality. On the empirical side, we studied discrete-kernel estimators of $g_n$ in finite mixtures of counts, quantified bandwidth effects, and found that simple triangular kernels are competitive at small sample sizes. Finally, synthetic experiments and image-segmentation case studies demonstrate that HD/NED/vNED instantiations of DM are competitive with EM in well-specified regimes and deliver markedly greater stability under contamination and misspecification.

{\bf{Future work.}}
Natural extensions of our work include (i) scalable DM for high-dimensional latent structures (sparse or low-rank components), (ii) semiparametric DM with estimated nuisance $g_n$ under weaker smoothness and discrete-support constraints, (iii) selection consistency beyond fixed $K_{\max}$ (growing-model regimes), and (iv) tighter nonasymptotic guarantees for repeated sample-splitting and majority-vote selection.

\section{Disclosure statement}\label{disclosure-statement}
The authors have no conflict of interest.

\newpage





\bigskip

\begin{center}

{\large\bf SUPPLEMENTARY MATERIAL}\label{supplementary-material}

\end{center}


\section*{A: Background and Literature Review}

Finite mixture models (FMM), a class of models for data with latent structure, have gained increasing attention for their flexibility in capturing multiple modes and unobserved heterogeneity. They are applied across diverse fields including astronomy, social sciences, biology, engineering, and medicine, and more recently have been adopted in neural network and deep learning research. Since the pioneering work of \cite{pea94} for mixtures of univariate normal distributions, FMM have been extended to a general class of distributions. Turning to fitting the models, starting with the work of \cite{rao48} who used Fisher's method of scoring, the field has evolved into the routine use of EM algorithm and its variants as in the seminal paper of \cite{Dem77} (See also \cite{gan78}, \cite{gan79a}, \cite{gan79}, \cite{gan80}, \cite{oneil78}, and \cite{ait80} for more details). 

As is well known, the EM algorithm introduces a latent label for each observation, representing group membership, and treats it as missing data. This facilitates the formulation of ``complete data'' (missing and observed data) log-likelihood function, under the assumption of independence. The Expectation step (E-step) computes the conditional expectation of the complete-data log-likelihood given the current parameter estimates and observed data.
The Maximization step (M-step) then updates the parameters by maximizing this expected log-likelihood. These steps are iterated until convergence. Regarding its convergence properties,
\cite{Dem77} established that the log-likelihood function is non-decreasing after each iteration and \cite{Wu83} provided regularity conditions under which the limit points of EM algorithm are stationary points of the likelihood function. Further convergence properties are classical; see \cite{Wu83} and \cite{Vai05}. Acceleration and variants are surveyed in \citet{lou82,Mcl95a, Paul04}.


Despite the popularity of EM algorithm, it is well known that it has several limitations. First, the EM algorithm may converge slowly, and the situation is worse when the ``incomplete information'' dominates the likelihood. Various methods and modifications have been tried to improve the speed of convergence, such as  Aitken's acceleration method (see \cite{Mcl95a}), Louis' method (see \cite{lou82} ), Conjugate Gradient method (see \cite{jam93}), and EM Gradient algorithm (see \cite{lange95a}). 
Second, it is not stable as it may often converge to the local optima. 
This problem is worse for contaminated data (see \cite{Cut96}). To address the issue, \cite{Hu17} proposed a robust EM-type algorithm for log-concave mixtures regression models, where they use the trimmed least squares technique (see \cite{Rou85}) to achieve the robustness properties. However, the performance is not satisfactory when the percentage of outliers increase. Additionally, the use of trimmed least squares lead to loss of efficiency at the model.
We address these challenges using the divergence method. 

Divergence-based methods (see \cite{Lind94} and \cite{Basu94}) are appealing for parametric inference when the model is misspecified or when the data are contaminated. These methods have the property that they are first-order efficient when the model is correctly specified and are robust under model misspecification and the presence of outliers.
Starting with the work of \cite{Ber77}, who proposed minimum Hellinger distance (MHD) estimation for independent identically distributed (i.i.d.) data, extensions and variations of the theme have been studied, for example, in \cite{sta81}, \cite{Dono88}, \cite{esl91}, \cite{basuha94},\cite{basusra97}, \cite{ACAV06}, \cite{Simp87}, \cite{Simp89}, and \cite{LiAV19}.
Turning to mixture models, \cite{Woo95} considered two- component normal mixture model and proposed to use MHD method to estimate the mixing proportion. \cite{Cut96} estimated all the unknown parameters using MHD method for normal mixture by the use of the so called HMIX algorithm. 
Note that the genesis of the HMIX algorithm and its connection to the EM algorithm are unknown. For discrete data, \cite{karlis98} considered the MHD estimation for Poisson mixtures and used the so called HELMIX algorithm, a variant of the HMIX algorithm. The HELMIX algorithm is specialized  for Poisson models as it utilizes a recurrence relation for Poisson probabilities that do not generalize to other distributions. More recent references regarding the robust EM algorithm and divergences 
are referred \cite{nielsen2016}, \cite{qin2013}, \cite{sam2019}, \cite{hu2019EMseg}, \cite{lucke2019k}, \cite{zhao2020}. Minimum-divergence (disparity) methods remain first-order efficient at the model yet damp the effect of large residuals via the RAF, which motivates our DM operator in Section~2 of the main text.

\newpage 

\section*{B: Plot of Residual Adjustment Function (RAF)}

\begin{figure}[H] 
    \centering
    \includegraphics[width=0.65\linewidth]{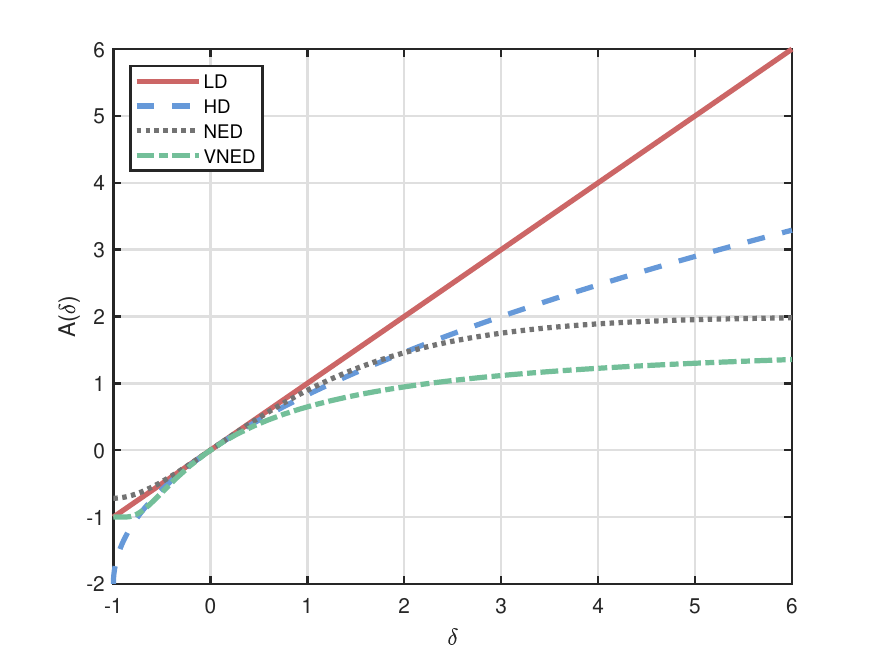}
    \caption{Plot of Residual Adjustment Function $A(\delta)$ for LD, HD, NED, and vNED} 
    \label{RAF} 
\end{figure}

\newpage 

\section*{C: Derivation of updating mixing probability in (\ref{Update-Pi11})}\label{supp:update-pi}

\begin{proof}[Proof of the update $\pi$]
Recall
\[
Q(\bm{\theta}'\mid\bm{\theta})
=\sum_{k=1}^K \int_{\mathcal{Y}}
\pi_k'\,h(y;\bm{\phi}_k')\,
G\!\left(\tau_k(y)\right)\,dy,
\qquad
\tau_k(y):=\frac{g(y)\,w_k(y;\bm{\theta})}{\pi_k'\,h(y;\bm{\phi}_k')},
\]
where \(w_k(y;\bm{\theta})=\dfrac{\pi_k h(y;\bm{\phi}_k)}{\sum_{\ell=1}^K \pi_\ell h(y;\bm{\phi}_\ell)}\)
and the mixing weights satisfy \(\pi_k'>0\) and \(\sum_{k=1}^K \pi_k'=1\).
For brevity write \(h_k'(y):=h(y;\bm{\phi}_k')\) and \(\tau_k=\tau_k(y)\).

We minimize \(Q(\bm{\theta}'\mid\bm{\theta})\) over \(\bm{\pi}'\) on the simplex via the Lagrangian
\[
\mathcal{L}(\bm{\pi}',\lambda)=Q(\bm{\theta}'\mid\bm{\theta})+\lambda\!\left(\sum_{k=1}^K \pi_k' - 1\right).
\]

\textbf{Step 1: Gradient with respect to \(\pi_k'\).}
For the integrand \(\pi_k' h_k'(y) G(\tau_k)\) with \(\tau_k=(g\,w_k)/(\pi_k' h_k')\), we have
\[
\frac{\partial}{\partial \pi_k'}\big[\pi_k' h_k' G(\tau_k)\big]
= h_k' G(\tau_k)+\pi_k' h_k' G'(\tau_k)\,\frac{\partial \tau_k}{\partial \pi_k'}
= h_k'\Big\{G(\tau_k)-\tau_k\,G'(\tau_k)\Big\},
\]
because \(\dfrac{\partial \tau_k}{\partial \pi_k'}=-\dfrac{\tau_k}{\pi_k'}\).
Therefore
\[
\frac{\partial \mathcal{L}}{\partial \pi_k'}
= \int_{\mathcal{Y}} h_k'(y)\Big\{G(\tau_k(y))-\tau_k(y)\,G'(\tau_k(y))\Big\}\,dy + \lambda.
\]
Define
\[
\Xi_k(\bm{\pi}',\bm{\phi}') \;:=\; \int_{\mathcal{Y}} h_k'(y)\Big\{G(\tau_k(y))-\tau_k(y)\,G'(\tau_k(y))\Big\}\,dy.
\]
\noindent\textbf{Abbreviation.}\quad $B(u):=G(u)-uG'(u)\,.$
\quad\text{Then}\quad
$\frac{\partial Q^{(k)}(\theta'\mid\theta)}{\partial \pi'_k}
=\int h'_k(y)\,B(\tau_k(y))\,dy\,.$

Stationarity (KKT) gives, for all \(k\),
\begin{equation}\label{eq:Xi-const}
\Xi_k(\bm{\pi}',\bm{\phi}') \;=\; -\,\lambda.
\end{equation}

\textbf{Step 2: Product identity and compact reparametrization.}
Note that \(\pi_k' h_k'(y)\,\tau_k(y)\equiv g(y)\,w_k(y;\bm{\theta})\).
Multiplying \eqref{eq:Xi-const} by \(\pi_k'\) and using this identity yields
\[
\pi_k'\,\Xi_k(\bm{\pi}',\bm{\phi}')
=\int_{\mathcal{Y}} \Big\{\pi_k' h_k'(y)\,G(\tau_k(y))
- g(y)\,w_k(y;\bm{\theta})\,G'(\tau_k(y))\Big\}dy.
\]
Introduce
\[
\Phi_k(\bm{\pi}',\bm{\phi}')
\;:=\; \int_{\mathcal{Y}} \Big\{\pi_k' h_k'(y)\,G(\tau_k(y))
- g(y)\,w_k(y;\bm{\theta})\,G'(\tau_k(y))\Big\}dy,
\]
so that
\begin{equation}\label{eq:Psi-piXi}
\Phi_k(\bm{\pi}',\bm{\phi}') \;=\; \pi_k'\,\Xi_k(\bm{\pi}',\bm{\phi}') \;=\; -\,\lambda\,\pi_k'.
\end{equation}
Moreover, since \(Q_k(\bm{\theta}'\mid\bm{\theta})=\int_{\mathcal{Y}}\pi_k' h_k'(y) G(\tau_k(y))\,dy\),
we have
\[
\frac{\partial Q_k(\bm{\theta}'\mid\bm{\theta})}{\partial \pi_k'} = \Xi_k(\bm{\pi}',\bm{\phi}'),
\qquad\text{hence}\qquad
\Phi_k(\bm{\pi}',\bm{\phi}')=\pi_k'\,\frac{\partial Q_k(\bm{\theta}'\mid\bm{\theta})}{\partial \pi_k'}.
\]

\textbf{Step 3: Normalization and the update.}
Summing \eqref{eq:Psi-piXi} over \(k\) and using \(\sum_k \pi_k'=1\) gives
\[
\sum_{k=1}^K \Phi_k(\bm{\pi}',\bm{\phi}') = -\,\lambda.
\]
Therefore, for each \(k\),
\[
\pi_k'
=\frac{\Phi_k(\bm{\pi}',\bm{\phi}')}{\sum_{\ell=1}^K \Phi_\ell(\bm{\pi}',\bm{\phi}')}
=\frac{\displaystyle \pi_k'\,\frac{\partial Q_k(\bm{\theta}'\mid\bm{\theta})}{\partial \pi_k'}}
{\displaystyle \sum_{\ell=1}^K \pi_\ell'\,\frac{\partial Q_\ell(\bm{\theta}'\mid\bm{\theta})}{\partial \pi_\ell'}}.
\]
By construction, \(\sum_k \pi_k'=1\). The integrability and finiteness of \(\Phi_k\) follow from
the stated regularity conditions.
\end{proof}

\newpage

\section*{D: Relation of DM algorithm with other popular algorithms}

\textbf{1. Relation with the Proximal Point Algorithm}

\cite{Mar70}, \cite{Roc76a} and \cite{Roc76b} proposed an iterative algorithm which is referred to as the proximal point algorithm and can be described as
\begin{align*}
\bm{\theta}_{m+1} = \underset{\bm{\theta}' \in \bm{\Theta}}{\text{argmax}}\left\{\Psi(\bm{\theta}') - \frac{\bm{\beta}_m}{2}||\bm{\theta}' - \bm{\theta}_m  ||_2^2  \right\},
\end{align*}
where $\Psi:\bm{\Theta} \to \Real$ and
the quadratic penalty is relaxed using a positive sequence $\{\bm{\beta}_m \}$.
\cite{Paul04} introduced the entropy-like proximal point algorithm as
\begin{align*}
\bm{\theta}_{m+1} = \underset{\bm{\theta}' \in \bm{\Theta}}{\text{argmax}}\left\{\Psi(\bm{\theta}') - H(\bm{\theta}'|\bm{\theta}_m) \right\},
\end{align*}
where $H: \bm{\Theta}\times \bm{\Theta}  \to \Real_+$ satisfies $H(\bm{\theta}'|\bm{\theta}') = 0$ for all $\bm{\theta}' \in \bm{\Theta}$.
It is known that the EM algorithm can be considered as an example of proximal point algorithm (see \cite{Stph00}), and other extensions of proximal point algorithm are referred to \cite{Paul04}, \cite{Stph08}, \cite{Cun10}, and \cite{Moh16}.

The DM algorithm can also be treated as proximal point method by letting $\Psi(\bm{\theta}') = -D(\bm{\theta}')$ and $H(\bm{\theta}'|\bm{\theta}) = -D(\bm{\theta}') + Q(\bm{\theta}'|\bm{\theta})$. 

We now provide another relation via some examples between DM algorithm and the proximal point algorithm. In particular, we can view $\Psi(\bm{\theta}')$ as the divergence information from the observed data and treat $H(\bm{\theta}'|\bm{\theta}_m)$ as the divergence information for the latent variable. 
For example, for EM objective function, note that,
\begin{align}
Q_{\text{\tiny{KL}}}(\bm{\theta}'|\bm{\theta}) &= \nonumber - \int_{\mathcal{Y}} \int_{\mathcal{Z}} g(y) w(z|y;\bm{\theta}) \log \left(\frac{p(y, z;\bm{\theta}') }{g(y) w(z|y;\bm{\theta}) } \right) dz dy \\
&= \nonumber - \int_{\mathcal{Y}}  \log \left( \frac{f(y;\bm{\theta}')}{g(y)}  \right) g(y) dy - \int_{\mathcal{Y}} \left( \int_{\mathcal{Z}} w(z|y;\bm{\theta}) \log \left( \frac{w(Z|Y;\bm{\theta}')}{w(z|y;\bm{\theta})}  \right) dz \right) g(y) dy. 
\end{align}
Let 
\begin{align}
\Psi_{\text{\tiny{EM}}}(\bm{\theta}') &= -\int_{\mathcal{Y}}  \log \left( \frac{f(y;\bm{\theta}')}{g(y)}  \right) g(y) dy \quad \label{EM-Proximal-1} \text{and} \\
H_{\text{\tiny{EM}}}(\bm{\theta}'|\bm{\theta}) &= \int_{\mathcal{Y}} \left( \int_{\mathcal{Z}} w(z|y;\bm{\theta}) \log \left( \frac{w(Z|Y;\bm{\theta}')}{w(z|y;\bm{\theta})}  \right) dz \right) g(y) dy, \label{EM-Proximal-2}
\end{align}
then it becomes the proximal point algorithm. More importantly, the benefit for splitting $Q_{\text{\tiny{KL}}}(\bm{\theta}'|\bm{\theta})$ into $\Psi_{\text{\tiny{EM}}}(\bm{\theta}')$ and $H_{\text{\tiny{EM}}}(\bm{\theta}'|\bm{\theta})$ is that $Q_{\text{\tiny{KL}}}(\bm{\theta}'|\bm{\theta})$ is considered as the Kullback-Leibler divergence for complete data, $\Psi_{\text{\tiny{EM}}}(\bm{\theta}')$ is exactly the Kullback-Leibler divergence for observed data, and $H_{\text{\tiny{EM}}}(\bm{\theta}'|\bm{\theta})$ can be viewed as the Kullback-Leibler divergence for the latent variable given observed data (and parameters in previous step). 
As another example, taking $\Psi_{\text{\tiny{EM}}}(\bm{\theta}')$ and $H_{\text{\tiny{EM}}}(\bm{\theta}'|\bm{\theta})$ as (\ref{EM-Proximal-1}) and (\ref{EM-Proximal-2}) respectively, we have that 
\begin{align}
\Psi_{\text{\tiny{EM}}}(\bm{\theta}') &\approx 2 \int_{\mathcal{Y}} \left( g^{\frac{1}{2}}(y) - f^{\frac{1}{2}}(y;\bm{\theta}')  \right)^2 dy \equiv 2 \Psi_{\text{\tiny{HD}}}(\bm{\theta}'), \quad \text{and} \label{HD-Proximal-1}  \\
H_{\text{\tiny{EM}}}(\bm{\theta}'|\bm{\theta}) &\approx  -2 \int_{\mathcal{Y}} \left(\int_{\mathcal{Z}} \left( c^{\frac{1}{2}}(z|y;\bm{\theta}') - c^{\frac{1}{2}} (z|y;\bm{\theta})  \right)^2 dz \right) g(y) dy \nonumber \\
&\equiv -2 H_{\text{\tiny{HD}}}(\bm{\theta}'|\bm{\theta}). \label{HD-Proximal-2}
\end{align}
In addition, from (\ref{HD-Proximal-1}) and (\ref{HD-Proximal-2}) and using substitution principle, we have 
\begin{align*}
Q_{\text{\tiny{HD}}}(\bm{\theta}'|\bm{\theta}) = \Psi_{\text{\tiny{HD}}}(\bm{\theta}') - H_{\text{\tiny{HD}}}(\bm{\theta}'|\bm{\theta}).
\end{align*}
Similar as Kullback-Leibler divergence, the Hellinger distance divergence for complete data can also be divided into the ``observed data divergence'' part and the ``latent data divergence'' part, i.e., $Q_{\text{\tiny{HD}}}(\bm{\theta}'|\bm{\theta})$ is the Hellinger distance divergence for complete data, $\Psi_{\text{\tiny{HD}}}(\bm{\theta}')$ is the Hellinger distance divergence for observed data, and $H_{\text{\tiny{HD}}}(\bm{\theta}'|\bm{\theta})$ is the Hellinger distance divergence for latent variable given observed data (and current parameter).
More generally, it is believed that for any divergence, $Q(\bm{\theta}'|\bm{\theta})$ the divergence objective function for complete data, $\Psi(\bm{\theta}')$ is the divergence objective function for observed data, and $H(\bm{\theta}'|\bm{\theta})$ is the divergence objective function for latent variable given observed data (and current parameter). The detailed discussion is considered elsewhere.

\noindent \textbf{2. Relation with MM algorithm}

\cite{Hun00a} proposed the general MM algorithm to construct optimization algorithms. Specifically, an MM algorithm creates a surrogate function that minorizes or majorizes the objective function and when the surrogate function is optimized, the objective function is forced to decrease or increase correspondingly. MM algorithms are widely used in a broad application areas, for instance, EM algorithm, robust regression (see \cite{Hub81}) quantile regression (see \cite{Hun00}), survival analysis (see \cite{Hun02}), paired and multiple comparisons (specifically on generalized Bradley-Terry models, see \cite{HunD04}), etc. 
For a more detailed review of MM algorithm, see \cite{Hun04}.

The MM algorithm proceeds as follows. Let $\bm{\theta}_m$ represent a fixed value of the parameter $\bm{\theta}$, and let $\psi(\bm{\theta}|\bm{\theta}_m)$ denote a real-valued function of $\bm{\theta}$ depending on $\bm{\theta}_m$. The function $\psi(\bm{\theta}|\bm{\theta}_m)$ is said to majorize a real-valued function $t(x)$ at the point $\bm{\theta}_m$ provided 
\begin{align}\label{MM:1} 
\psi(\bm{\theta}|\bm{\theta}_m) \geq t(\bm{\theta}) \quad \text{for all} ~ \bm{\theta}, ~ \text{and}~  \psi(\bm{\theta}_m|\bm{\theta}_m) = t(\bm{\theta}_m).
\end{align}
Ordinarily, $\bm{\theta}_m$ denotes the current iteration in a search surface of $t(\bm{\theta})$. In a majorize-minimize MM algorithm, we minimize the majorizing function $\psi(\bm{\theta}|\bm{\theta}_m)$ rather than $t(\bm{\theta})$. If $\bm{\theta}_{m+1}$ represents the minimizer of $\psi(\bm{\theta}|\bm{\theta}_m)$, then MM procedure forces $t(\bm{\theta})$ downhill.
To see this, note that
\begin{small}
\begin{align*}
t(\bm{\theta}_{m+1}) =\psi(\bm{\theta}_{m+1}|\bm{\theta}_m) + t(\bm{\theta}_{m+1}) - \psi(\bm{\theta}_{m+1}|\bm{\theta}_m) \leq  \psi(\bm{\theta}_{m}|\bm{\theta}_m) + t(\bm{\theta}_{m}) - \psi(\bm{\theta}_{m}|\bm{\theta}_m) =t(\bm{\theta}_m).
\end{align*}
\end{small}
It turns out that DM algorithm also belongs to the class of MM algorithms by letting $\psi(\bm{\theta}|\bm{\theta}_m) \equiv Q(\bm{\theta}|\bm{\theta}_m) $, $t(\bm{\theta}) \equiv D(\bm{\theta})$.

\noindent \textbf{Proximal Point Algorithm and MM algorithm}

Note that the proximal point algorithm also belongs to the class of MM algorithms. Specifically, let $\psi(\bm{\theta}|\bm{\theta}_m) = -(\Psi(\bm{\theta}) - H(\bm{\theta}|\bm{\theta}_m))$, $t(\bm{\theta}) = -\Psi(\bm{\theta})$. By definition of $\Psi(\cdot)$ and $H(\cdot|\cdot)$, it follows that (\ref{MM:1}) is satisfied. Furthermore, if $\bm{\theta}_{m+1}$ represents the minimizer of $\psi(\bm{\theta}|\bm{\theta}_m)$, then $t(\bm{\theta}_{m+1}) \leq t(\bm{\theta}_m)$.

\noindent \textbf{3. Relation of DM algorithm with coordinate descent.}
Following \cite{nea98}, EM can be viewed as a coordinate ascent
algorithm. Similarly, DM can be viewed as a coordinate descent algorithm.
Define
\[
\mathscr D(\tilde q,\theta') = \mathbf{E}_Y\!\left[\,
\mathbf{E}_{Z\sim \tilde q(\cdot|Y)}\,
G\!\left(-1+\frac{g(Y)\,\tilde q(Z|Y)}{f(Y;\bm \theta')\,w(Z|Y;\bm \theta')}\right)
\right],
\]
where $\tilde q(\cdot|y)$ is any conditional density on $Z$ given $Y=y$.  
By Lemma~\ref{Inequality00}, $\mathscr D(\tilde q,\bm \theta')\ge D(\bm \theta')$ with equality iff
$\tilde q(\cdot|y)=w(\cdot|y;\bm \theta')$. Thus the DM algorithm alternates:
\begin{enumerate}
  \item \emph{D-step.} For fixed $\bm \theta$, minimize $\mathscr D(\tilde q,\bm \theta)$
  over $\tilde q$, yielding $\tilde q(z|y)=w(z|y;\bm\theta)$.
  \item \emph{M-step.} For this choice of $\tilde q$, minimize $\mathscr D(\tilde q, \bm \theta')$
  over $\bm \theta'$, which recovers the DM update.
\end{enumerate}
Hence DM belongs to the class of two-block coordinate descent algorithms.

\begin{lem}\label{Inequality00}
For any conditional density $\tilde{q}(\cdot|y)$ and all $\bm{\theta}' \in \bm{\Theta}$,
\[
\mathscr{D}(\tilde{q}, \bm{\theta}') \;\geq\; D(\bm{\theta}').
\]
Equality holds if and only if $\tilde{q}(z|y) = w(z|y;\bm{\theta}')$ for almost every $y \in \mathcal{Y}$.
\end{lem}

\noindent \textbf{The generalized DM Algorithm:}
It is possible that in the M-step, the minimizer $\bm{\theta}_{m+1}$ is not unique. Let $\bm{\Theta}_m$ denote the set of all minimizers at step $m$; that is, $\bm{\theta}_{m+1} \in  \bm{\Theta}_{m+1}$.
Sometimes it may be difficult to perform M-step numerically; in this case, we can define a generalized DM algorithm (referred to as the G-DM algorithm) as follows: Let $M: \bm{\theta}_m \to \bm{\Theta}_{m+1}$ be a point to set map: then the G-DM algorithm is an iterative scheme such that 
\begin{align*}
Q^{(G)}(\bm{\theta}'|\bm{\theta}_m) \leq Q^{(G)}(\bm{\theta}_m |\bm{\theta}_m) \quad \text{for all} \quad \bm{\theta}' \in M(\bm{\theta}_m).
\end{align*}
We notice here that for any G-DM sequence $\{\bm{\theta}_m \}$, $D^{(G)}(\bm{\theta}_{m+1}) \leq D^{(G)}(\bm{\theta}_m)$ and DM algorithm is a special case of G-DM algorithm. We emphasize here that by choosing different divergences, we obtain many existing algorithms, including the EM, HMIX, and HELMIX algorithms.

\newpage

\section*{\texorpdfstring{E: Special Cases of DM Algorithms for Various Choices of $G(\cdot)$}{Special Cases of DM}}\label{Special_cases}

\noindent \textbf{Special case 1: EM Algorithm}

Let $G(\delta) = (\delta + 1) \log(\delta + 1)$, which corresponds to the Kullback-Leibler divergence, we get the objective function obtained from the E-step in the EM algorithm. Specifically, since MLE can be obtained by minimizing the Kullback-Leibler divergence, it follows that
\begin{align}\label{EM:11}
Q_{\text{\tiny{KL}}}(\bm{\theta}'| \bm{\theta}) &= \nonumber  \int_{\mathcal{Y}} \int_{\mathcal{Z}} \left( \frac{g(y) w(z|y;\bm{\theta}) }{f(y;\bm{\theta}')w(Z|Y;\bm{\theta}') } \right) \log \left(\frac{g(y) w(z|y;\bm{\theta}) }{f(y;\bm{\theta}')w(Z|Y;\bm{\theta}') } \right)f(y;\bm{\theta}')w(Z|Y;\bm{\theta}') dz dy \\
&= \nonumber \int_{\mathcal{Y}} \int_{\mathcal{Z}} g(y) w(z|y;\bm{\theta}) \log \left(\frac{g(y) w(z|y;\bm{\theta}) }{f(y;\bm{\theta}')w(Z|Y;\bm{\theta}') } \right) dz dy  \\
&= \nonumber \int_{\mathcal{Y}} \int_{\mathcal{Z}} g(y) w(z|y;\bm{\theta}) \log \left(g(y) w(z|y;\bm{\theta})\right) dz dy  \\
&\quad \nonumber- \int_{\mathcal{Y}} \int_{\mathcal{Z}} g(y) w(z|y;\bm{\theta}) \log \left(f(y;\bm{\theta}')w(Z|Y;\bm{\theta}') \right) dz dy \\
&\equiv  \int_{\mathcal{Y}} \int_{\mathcal{Z}} g(y) w(z|y;\bm{\theta}) \log \left(g(y) w(z|y;\bm{\theta})\right) dz dy - \tilde{Q}(\bm{\theta}'|\bm{\theta}) ,
\end{align}
where 
\begin{align*}
\tilde{Q}(\bm{\theta}'|\bm{\theta}) = \int_{\mathcal{Y}} \left( \int_{\mathcal{Z}}  w(z|y;\bm{\theta}) \log \left(f(y;\bm{\theta}')w(Z|Y;\bm{\theta}') \right) dz \right) g(y) dy.
\end{align*}
The first term in (\ref{EM:11}) does not involve $\bm{\theta}'$ and hence can be omitted in terms of optimization. The second term $\tilde{Q}(\bm{\theta}'|\bm{\theta})$, in the current literature, is referred to as the population level EM objective function, see \cite{Siv17}, \cite{Dwi18}. 
If we estimate $g(\cdot)$ through empirical measure, then $\tilde{Q}(\bm{\theta}'|\bm{\theta})$ will reduce to the sample level EM objective function, and is given by 
\begin{align*}
\tilde{Q}_n(\bm{\theta}'|\bm{\theta}) = \frac{1}{n}\sum_{i=1}^n \int_{\mathcal{Z}} c(z|y_i;\bm{\theta}) \log \left(f(y_i;\bm{\theta}')c(z|y_i;\bm{\theta}') \right) dz .
\end{align*}

\noindent \textbf{Special case 2: HMIX Algorithm}

Next consider another example. By taking $G(\delta) = 2[(\delta+1)^{1/2} - 1 ]^2$, which corresponds to the Hellinger distance divergence, we get the population level HMIX objective function 
\begin{align*}
Q_{\text{\tiny{HD}}}(\bm{\theta}'|\bm{\theta}) &= \nonumber 2\int_{\mathcal{Y}} \int_{\mathcal{Z}} \left( \left( \frac{g(y) w(z|y;\bm{\theta}) }{f(y;\bm{\theta}')w(Z|Y;\bm{\theta}')}  \right)^{\frac{1}{2}} - 1 \right)^2 f(y;\bm{\theta}')w(Z|Y;\bm{\theta}') dz dy \\
&= \nonumber 4 - 4 \int_{\mathcal{Y}} \int_{\mathcal{Z}} \left[ g(y) w(z|y;\bm{\theta}) p(y, z;\bm{\theta}') \right]^{\frac{1}{2}} dz dy,\\
&\equiv \nonumber 4 - 4 \mathcal{A}(\bm{\theta}'|\bm{\theta}).
\end{align*} 
On the other hand, the sample level HMIX objective function is
\begin{align*}
    \mathcal{A}_n(\bm{\theta}'| \bm{\theta}) = \int_{\mathcal{Y}} \int_{\mathcal{Z}} \left[ g_n(y) w(z|y;\bm{\theta}) f(y;\bm{\theta}')w(Z|Y;\bm{\theta}') \right]^{\frac{1}{2}} dz dy.
\end{align*}

\noindent \textbf{Special case 3: Algorithm from the Negative Exponential Divergence}

Next consider another example. By taking $G(\delta) = [e^{-\delta} - 1 +\delta]$, which corresponds to the negative exponential divergence, we get the population level objective function based on Negative exponential divergence. Specifically, it is given by 
\begin{align}
Q_{\text{\tiny{NED}}}(\bm{\theta}'|\bm{\theta})
&= \nonumber \exp(1) \int_{\mathcal{Y}} \int_{\mathcal{Z}} \exp\left(-\frac{g(y)w(z|y;\bm{\theta})}{f(y;\bm{\theta}')w(Z|Y;\bm{\theta}') } \right)f(y;\bm{\theta}')w(Z|Y;\bm{\theta}') dz dy - 1\\
&\equiv \nonumber \exp(1) \text{NED}(\bm{\theta}'|\bm{\theta}) - 1. 
\end{align}
So $Q_{\text{NED}}(\bm{\theta}'|\bm{\theta})$ is the population level objective function generated by Negative exponential divergence. The sample level $Q_{\text{NED}_n}(\bm{\theta}'|\bm{\theta})$ is defined analogously, where $g(\cdot)$ is replaced by $g_n(\cdot)$.

\noindent \textbf{Special case 4: Algorithm from the Variant Negative Exponential Divergence}

We introduce a divergence similar to Negative exponential divergence called vNED. Specifically, $G(\cdot)$ is given by 
\begin{align*}
G_{\text{\tiny{vNED}}}(\delta) = \exp\left(-\frac{1}{1+\delta} +1\right)(1+\delta)-(2\delta+1).
\end{align*}
Then the corresponding population DM objective function $Q(\cdot|\cdot)$ is as follows:
\begin{align*}
Q_{\text{\tiny{vNED}}}(\bm{\theta}'|\bm{\theta})
&= \exp(1) \int_{\mathcal{Y}} \int_{\mathcal{Z}} \exp\left( -\frac{f(y;\bm{\theta}')w(Z|Y;\bm{\theta}')}{g(y)w(z|y;\bm{\theta})} \right) g(y)w(z|y;\bm{\theta}) dz dy-1  \\
&\equiv \exp(1) \text{vNED}(\bm{\theta}'|\bm{\theta})-1.
\end{align*}



\noindent \textbf{Special case 5: Algorithm from Blended Weighted Hellinger Distance Divergence}

\cite{Basu94} proposed the blended weighted Hellinger distance (BWHD) divergence as
\begin{align*}
G_{\text{\tiny{BWHD}}}(\delta) = \frac{1}{2} \frac{\delta^2}{[ \tau (\delta+1)^{\frac{1}{2}} + 1-\tau ]^2}, \quad \text{where} \quad \tau \in [0, 1].
\end{align*}
Then the BWHD objective function from DM algorithm is given as 
\begin{align*}
Q_{\text{\tiny{BWHD}}}(\bm{\theta}'|\bm{\theta}) = \int_{\mathcal{Y}} \int_{\mathcal{Z}} \left( \frac{g(y)w(z|y;\bm{\theta}) - f(y;\bm{\theta}')w(Z|Y;\bm{\theta}') }{\tau(g(y)w(z|y;\bm{\theta}))^{\frac{1}{2}} + (1-\tau) (f(y;\bm{\theta}')w(Z|Y;\bm{\theta}'))^{\frac{1}{2}}  } \right)^2 dz dy.
\end{align*}
Apart from the examples described as above, one can get other DM algorithm objective functions following similar ideas (e.g., blended weighted Negative exponential divergence, blended weight chi-square divergence, etc), hence we omit here.

\noindent \textbf{Special case 6: Algorithm from the Cressie-Read Family}

Next we consider an important subfamily of general divergence, the Cressie-Read (CR) family (see \cite{cri84}, \cite{read88}). For CR family, $G(\cdot)$ is given by
\begin{align*}
G_{\text{\tiny{CR}}}(\delta) = \frac{(\delta+1)^{\alpha+1}-1}{\alpha(\alpha+1)}, \quad \text{where}\quad  \alpha \in \Real.
\end{align*}
When $\alpha = -1$, it corresponds to the Kullback-Leibler divergence; when $\alpha = -\frac{1}{2}$, it corresponds to Hellinger distance divergence; when $\alpha = 1,$ it corresponds to Pearson's $\chi^2$ and when $\alpha = -2$, it corresponds to Neyman's $\chi^2$. Note that when $\alpha = -1$, the divergence is defined by continuity.
The DM algorithm objective function for CR family is 
\begin{align*}
Q_{\text{\tiny{CR}}}(\bm{\theta}'|\bm{\theta}) = \frac{1}{\alpha(1+\alpha)} \int_{\mathcal{Y}} \int_{\mathcal{Z}} \left[\left( \frac{g(y)w(z|y;\bm{\theta}) }{f(y;\bm{\theta}')w(Z|Y;\bm{\theta}')} \right)^{\alpha+1}-1  \right] f(y;\bm{\theta}')w(Z|Y;\bm{\theta}') dz dy.
\end{align*}

\noindent \textbf{Special case 7: Algorithm from Power Divergence Family}

Another important subfamily of general divergence is the power divergence (PD) family (see \cite{basu13}).
For PD family, $G(\cdot)$ is given by 
\begin{align*}
G_{\text{\tiny{PD}}}(\delta) = \frac{1}{\alpha(1+\alpha)} \left[ \left(\delta+1 \right)^{1+\alpha} - \left(\delta+1 \right)  \right] - \frac{\delta}{1+\alpha} \quad \text{where} \quad \alpha \in \Real.
\end{align*}
The DM algorithm objective function for PD family is 
\begin{align*}
Q_{\text{\tiny{PD}}}(\bm{\theta}'|\bm{\theta}) &= \int_{\mathcal{Y}} \int_{\mathcal{Z}} \Bigg\{\frac{1}{\alpha(1+\alpha)}\left[ \left(\frac{g(y)w(z|y;\bm{\theta}) }{f(y;\bm{\theta}')w(Z|Y;\bm{\theta}')} \right)^{1+\alpha} - \left(\frac{g(y)w(z|y;\bm{\theta}) }{f(y;\bm{\theta}')w(Z|Y;\bm{\theta}')} \right)  \right]  \\
&\quad +\frac{1}{1+\alpha} \left[1- \frac{g(y)w(z|y;\bm{\theta}) }{f(y;\bm{\theta}')w(Z|Y;\bm{\theta}')} \right] \Bigg\} f(y;\bm{\theta}')w(Z|Y;\bm{\theta}') dzdy.
\end{align*}

\begin{table}[H]
\centering
\scriptsize
\renewcommand{\arraystretch}{1.2}
\setlength{\tabcolsep}{4pt}
\caption{Representative special cases of DM algorithm for different $G(\cdot)$.}
\label{tab:DM-special-cases}
\begin{tabular}{p{2.8cm} p{4.2cm} p{7.5cm}}
\hline
\textbf{Algorithm} & \textbf{Generator $G(\delta)$} & \textbf{Population DM objective (up to constants)} \\
\hline
KL (EM) &
$(\delta+1)\log(\delta+1)$ &
$\iint g(y)w(z|y;\theta)\,
\log\!\frac{g(y)w(z|y;\theta)}{f(y;\theta')w(z|y;\theta')}\,dz\,dy$ \\[6pt]

Hellinger (HMIX) &
$2\big((\delta+1)^{1/2}-1\big)^2$ &
$4-4\iint \big[g(y)w(z|y;\theta)\,p(y,z;\theta')\big]^{1/2}\,dz\,dy$ \\[6pt]

NED &
$e^{-\delta}-1+\delta$ &
$e^{1}\!\iint \exp\!\left(-\tfrac{g(y)w(z|y;\theta)}{f(y;\theta')w(z|y;\theta')}\right)
f(y;\theta')w(z|y;\theta')\,dz\,dy-1$ \\[6pt]

vNED &
$\exp\!\left(-\tfrac{1}{1+\delta}+1\right)(1+\delta)-(2\delta+1)$ &
$e^{1}\!\iint \exp\!\left(-\tfrac{f(y;\theta')w(z|y;\theta')}{g(y)w(z|y;\theta)}\right)
g(y)w(z|y;\theta)\,dz\,dy-1$ \\[6pt]

BWHD &
$\tfrac12 \frac{\delta^2}{[\tau(\delta+1)^{1/2}+1-\tau]^2}$ &
$\iint \left(\frac{g(y)w(z|y;\theta)-f(y;\theta')w(z|y;\theta')}
{\tau\sqrt{g(y)w(z|y;\theta)}+(1-\tau)\sqrt{f(y;\theta')w(z|y;\theta')}}\right)^{2}\,dz\,dy$ \\[6pt]

CR family &
$\tfrac{(\delta+1)^{\alpha+1}-1}{\alpha(\alpha+1)}$ &
$\tfrac{1}{\alpha(1+\alpha)}\iint 
\Big[\big(\tfrac{g(y)w}{f(y;\theta')w'}\big)^{\alpha+1}-1\Big] f(y;\theta')w'\,dz\,dy$ \\[6pt]

PD family &
$\frac{1}{\alpha(1+\alpha)}\!\Big[(\delta+1)^{1+\alpha}-(\delta+1)\Big]-\frac{\delta}{1+\alpha}$ &
$\iint\!\Bigg\{\frac{\big(\tfrac{g(y)w}{f(y;\theta')w'}\big)^{1+\alpha}-\tfrac{g(y)w}{f(y;\theta')w'}}{\alpha(1+\alpha)}
+\frac{1}{1+\alpha}\Big(1-\tfrac{g(y)w}{f(y;\theta')w'}\Big)\Bigg\}
f(y;\theta')w'\,dz\,dy$ \\
\hline
\end{tabular}
\end{table}

\newpage 

\section*{\texorpdfstring{F: DM Algorithm for Various Choices of $G(\cdot)$ with Application to FMM}{DM Algorithm with Application to FMM}}\label{DM-Mix}

In this subsection, we specialize the DM framework to finite mixture models (FMMs) described in Subsection~2.1. 
For each choice of divergence generator $G(\cdot)$, the corresponding population-level DM objective can be expressed in closed form. Recall that,
\begin{align} 
Q_G(\bm\theta'\!\mid\bm\theta)
=\sum_{k=1}^K\int \pi_k' h(y;\bm\phi_k')\,G\!\Big(\tau_k(y)\Big)\,dy \quad \text{and} \quad
\tau_k(y):=\frac{g(y)\,w_k(y;\bm\theta)}{\pi_k' h(y;\bm\phi_k')}.
\end{align}
We list below the objectives for several important divergences; detailed update formulas for $\bm{\theta}'$ will be given in Section~\ref{Chap5:S:DM-Alg}.

For EM algorithm, $\tilde{Q}(\cdot|\cdot)$ for finite mixture model is given by 
\begin{align*}
Q_{\text{EM}}(\bm{\theta}'|\bm{\theta})\coloneqq \tilde{Q}(\bm{\theta}'|\bm{\theta}) =  \sum_{k=1}^K \int_{\mathcal{Y}} g(y) w_k(y;\bm{\theta}) \log \left( \pi_k' h(y; \bm{\phi}_k') \right) dy.
\end{align*}
The HMIX objective function $\mathcal{A}(\cdot|\cdot)$ is given by
\begin{align*}
Q{\text{HD}}(\bm{\theta}'|\bm{\theta}) \coloneqq \mathcal{A}(\bm{\theta}'|\bm{\theta}) = \sum_{k=1}^K \int_{\mathcal{Y}} \left[ g(y) w_k(y; \bm{\theta}) \pi_k' h(y;\bm{\phi}_k')\right]^{\frac{1}{2}} dy.
\end{align*}
The NED objective function for finite mixture model. It is given by
\begin{align*}
Q_{\text{NED}}(\bm{\theta}'|\bm{\theta})
= \sum_{k=1}^K \int_{\mathcal{Y}} \exp\left(-\frac{g(y)w_k(y;\bm{\theta})}{\pi_k'h(y; \bm{\phi}_k')} \right) \pi_k' h(y;\bm{\phi}_k')dy.
\end{align*}
Next consider the vNED objective function for finite mixture model. It is given by
\begin{align*}
Q_{\text{vNED}}(\bm{\theta}'|\bm{\theta})
= \sum_{k=1}^K \int_{\mathcal{Y}} \exp\left(-\frac{\pi_k'h(y; \bm{\phi}_k')}{g(y)w_k(y;\bm{\theta})} \right) g(y)w_k(y;\bm{\theta})dy.
\end{align*}
Next consider the Blended weight Hellinger distance (BWHD) objective function for finite mixture model, where for $0 <\tau <1$, $Q_{\text{\tiny{BWHD}}}(\cdot|\cdot)$ is given by
\begin{align*}
Q_{\text{\tiny{BWHD}}}(\bm{\theta}'|\bm{\theta}) = \sum_{k=1}^K \int_{\mathcal{Y}} \left( \frac{g(y)w_k(y;\bm{\theta}) - \pi_k' h(y;\bm{\phi}_k') }{\tau(g(y)w_k(y;\bm{\theta}))^{\frac{1}{2}} + (1-\tau) (\pi_k' h(y;\bm{\phi}_k'))^{\frac{1}{2}}} \right)^2 dy.
\end{align*}
Next consider the CR family, where $\alpha \in \mathbb{R}-{-1, 0}$. $Q_{\text{\tiny{CR}}}(\cdot|\cdot)$ for finite mixture model is given by
\begin{align*}
Q_{\text{\tiny{CR}}}(\bm{\theta}'|\bm{\theta}) = \frac{1}{\alpha(1+\alpha)} \sum_{k=1}^K \int_{\mathcal{Y}} \left[\left( \frac{g(y)w_k(y;\bm{\theta}) }{\pi_k' h(y;\bm{\phi}_k')} \right)^{\alpha+1}-1  \right] \pi_k' h(y;\bm{\phi}_k') dy.
\end{align*}
Next consider the PD family (an equivalent alternative parametrization scaled to satisfy $A'(0)=1$. $Q_{\text{\tiny{PD}}}(\cdot|\cdot)$ for finite mixture model is given by
\begin{align*}
Q_{\text{\tiny{PD}}}(\bm{\theta}'|\bm{\theta}) &= \sum_{k=1}^K \int_{\mathcal{Y}}  \Bigg\{\frac{1}{\alpha(1+\alpha)}\left[ \left(\frac{g(y)w_k(y;\bm{\theta}) }{\pi_k' h(y;\bm{\phi}_k')} \right)^{1+\alpha} - \left(\frac{g(y)w_k(y;\bm{\theta}) }{\pi_k' h(y;\bm{\phi}_k')} \right)  \right]  \\
&\quad +\frac{1}{1+\alpha} \left[1- \frac{g(y)w_k(y;\bm{\theta}) }{\pi_k' h(y;\bm{\phi}_k')} \right] \Bigg\} \pi_k' h(y;\bm{\phi}_k') dy.
\end{align*}
We will provide specific algorithms for updating $\bm{\theta}'$ in finite mixture models in Section \ref{Chap5:S:DM-Alg}.


\newpage 

\section*{G: Discrete Kernels and K-means Clustering Algorithm}\label{DSK-sec}

\noindent \textbf{Discrete Kernels}

In general, let $Y_1, \cdots, Y_n$ be i.i.d. random variables with an unknown probability mass function (p.m.f.) $f$ on $\mathbb{Z}$. 
A discrete kernel estimator of $f$ can be defined as
\begin{align}\label{discrete-kernel}
g_n(y) = \frac{1}{n} \sum_{i=1}^{n} \mathcal{K}_{y,c}(Y_i), \quad y \in \mathbb{Z},
\end{align}
where $\mathcal{K}_{y,c}(\cdot)$ is a discrete kernel p.m.f.\ centered at $y$, and $c_n$ is a sequence of smoothing bandwidths. 
More rigorously, following the notation from \cite{Kok11}, the discrete associated kernel is defined as follows.

\begin{defn}
Let $\mathbb{T}$ be the discrete support of the p.m.f.\ $f$ to be estimated, $y$ a fixed target in $\mathbb{T}$, and $c>0$ a bandwidth.  
A p.m.f.\ $\mathcal{K}_{y,c}(\cdot)$ on support $\mathbb{S}_y$ (not depending on $c$) is said to be an \emph{associated kernel} if it satisfies:
\begin{align*}
y \in \mathbb{T}, \quad 
\lim_{c \to 0} \mathbf{E}[X_{y,c}] = y, 
\quad \text{and} \quad  
\lim_{c \to 0} \mathbf{Var}[X_{y,c}] = 0,
\end{align*}
where $X_{y,c}$ is a random variable with p.m.f.\ $\mathcal{K}_{y,c}(\cdot)$.
\end{defn}

\begin{enumerate}
\item \textbf{The empirical kernel:}
\begin{align}\label{emp-ker}
\mathcal{K}_{y,c}(x) = I_{x=y}, \quad x \in \mathbb{T}, \; c \geq 0,
\end{align}
where $I_A$ is the indicator function.  

\item \textbf{Discrete triangular kernel:} (\cite{Kok07})  
For support $\mathbb{T}$ (bounded or unbounded), bandwidth $c>0$, and integer $a>0$, define
\begin{align*}
\mathcal{K}_{y,c}(x) = \frac{(a+1)^c - |x-y|^c}{P(a,c)}, \quad x \in \{y-a, \dots, y+a\},
\end{align*}
with normalizing constant $P(a,c) = (2a+1)(a+1)^c - 2\sum_{k=0}^a k^c$.

\item \textbf{Poisson kernel:}  
For $y \in \mathbb{N}, c>0$,
\begin{align*}
\mathcal{K}_{y,c}(x) = \frac{(y+c)^x e^{-(y+c)}}{x!}, \quad x \in \mathbb{N}.
\end{align*}

\item \textbf{Binomial kernel:}  
For $y \in \mathbb{N}, c \in (0,1]$,
\begin{align*}
\mathcal{K}_{y,c}(x) = \binom{y+1}{x} \left(\frac{y+c}{y+1}\right)^x 
\left(\frac{1-c}{y+1}\right)^{y+1-x}, \quad x \in \{0,1,\dots,y+1\}.
\end{align*}

\item \textbf{Negative binomial kernel:}  
For $y \in \mathbb{N}, c>0$,
\begin{align*}
\mathcal{K}_{y,c}(x) = \binom{y+x}{x} 
\left(\frac{y+c}{2y+1+c}\right)^x 
\left(\frac{y+1}{2y+1+c}\right)^{y+1}, \quad x \in \mathbb{N}.
\end{align*}
\end{enumerate}

In order to measure the performance of different discrete kernel estimators, we use the practical criterion integrated squared error (ISE) given by
\begin{align*}
    \text{ISE} = \sum_{y \in \mathbb{T}} \left(g_n(y) - f(y; \bm{\theta}) \right)^2,
\end{align*}

\noindent \textbf{K-means clustering algorithm}

Given samples $y_1, \cdots, y_n$ and fix the number of clusters $K$, place initial centroids $c_1^{(0)}, \cdots, c_k^{(0)}$ at random locations.
\begin{algorithm}[H]
    \footnotesize
	\caption{The K-means Clustering Algorithm}
	\label{K-means}	 
	1. Set m = 0, 
	\begin{algorithmic}
		\REPEAT
		\STATE 2. For each point $y_i$: \\
		~~~(i). Find nearest centroid, i.e. $k^* = \underset{k}{argmin}~ L(y_i,
		            c_k^{(m)})$, where $L(\cdot)$ represents some distance metric. \\
		~~~(ii). Assign $y_i$ to cluster $k^*$.            		            		
		\STATE 3. Find new centroids: new centroid $c_{k}^{(m+1)} =$ mean of all points $y_i$ assigned to cluster $k^*$. 	
		\STATE 4. Set $m=m+1$.	
		\UNTIL{None of the cluster assignments change.}
	\end{algorithmic}
\end{algorithm}

\begin{rem}
	In step 2 (i), the examples of distance metrics are Euclidean distance or 
	$L_1$ distance.
	Besides, the choice of distance metric depends on the model.	
\end{rem}

\begin{rem}
	This algorithm works well without any outliers or with few outliers. If there exists a considerable percentage of outliers in the sample, then the K-means clustering algorithm is more likely to treat the outlier group as a new cluster, and the initial points will be more likely to be away from the real values.
\end{rem}

\newpage 

\section*{H: Additional Algorithms}\label{Add_algorithms}

If we choose the Hellinger distance divergence, then the algorithm (referred to as the DM-HMIX algorithm or simply HMIX algorithm) is provided in Algorithm \ref{HD:1}. 
\begin{algorithm}[htbp]
    \scriptsize
	\caption{The DM-HMIX Algorithm}
	\label{HD:1}
	Set m = 0.
	\begin{algorithmic}
		\REPEAT
		\STATE 1. Compute
		\begin{align*}
		HD_n(\bm{\theta}^{(m)}) = \sum_{k=1}^K \int_\mathcal{Y} \left(g_n(y)w_k(y;\bm{\theta}^{(m)}) \pi_k^{(m)} h(y; \bm{\phi}_k^{(m)}) \right)^{\frac{1}{2}} dy, \quad \text{where} \quad w_k(y;\bm{\theta}) = \frac{\pi_k h(y;\bm{\phi}_k)}{\sum_{l=1}^K \pi_l h(y;\bm{\phi}_l)}.
		\end{align*}
		\STATE 2. Update $\bm{\phi}_k^{(m+1)}$:
		\begin{align*}
		\bm{\phi}_k^{(m+1)} ~ & = ~ \underset{\bm{\phi}_k' \in \bm{\Theta}}{\text{argmax}} ~  \int_\mathcal{Y} \left(g_n(y)w_k(y;\bm{\theta}^{(m)}) \pi_k^{(m)} h(y; \bm{\phi}_k') \right)^{\frac{1}{2}} dy.
		\end{align*}
		\STATE 3. Update $\pi_k^{(m+1)}$:
		\begin{align*}
		\pi_k^{(m+1)} = \frac{HD_{n,k}^2(\bm{\phi}^{(m+1)}|\bm{\theta}^{(m)})}{\sum_{l=1}^K HD_{n,l}^2(\bm{\phi}^{(m+1)}|\bm{\theta}^{(m)})}, \quad \text{(HMIX type update)}
		\end{align*}
        ~~~~or
        \begin{align*}
		\pi_k^{(m+1)} = \frac{\sqrt{\pi_k^{(m)}} HD_{n,k}(\bm{\phi}^{(m+1)}|\bm{\theta}^{(m)})}{\sum_{l=1}^K\sqrt{\pi_l^{(m)}} HD_{n,l}(\bm{\phi}^{(m+1)}|\bm{\theta}^{(m)})}, \quad \text{(DM-Mix type update (HDMIX))} 
		\end{align*}
		~~~~where 
        \begin{align*}
       HD_{n,k}(\bm{\phi}^{(m+1)}|\bm{\theta}^{(m)}) = \int_\mathcal{Y} \left(g_n(y) w_k(y;\bm{\theta}^{(m)}) h(y; \bm{\phi}_k^{(m+1)})\right)^{\frac{1}{2}}dy.
        \end{align*}       		        
		\STATE 4. Update $w_k(y; \bm{\theta}^{(m+1)})$, 
        compute $HD(\bm{\theta}^{(m+1)})$ and the difference 
        \begin{align*}
        \epsilon_{m+1} = |HD_n(\bm{\theta}^{(m+1)}) - HD_n(\bm{\theta}^{(m)}) |.
        \end{align*}		
		\STATE 5. Set m = m+1.
		\UNTIL{$\epsilon_{m}< \text{threshold}$.}
	\end{algorithmic}
\end{algorithm}

Algorithm \ref{NED:1}, we use vNED divergence. Specifically, the update for $\pi_k'$ therefore can be written as
\begin{small}
    \begin{align*}
    \pi_k' = \frac{\text{vNED}_k(\bm{\theta}'|\bm{\theta})}{\sum_{l=1}^K \text{vNED}_l(\bm{\theta}'|\bm{\theta}) }, \quad \text{where} \quad  \text{vNED}_k(\bm{\theta}'|\bm{\theta}) = \int_\mathcal{Y} \exp\left( -\frac{ \pi_k h(y; \bm{\phi}_k') }{g(y)w_k(y;\bm{\theta})} \right)\pi_k h(y; \bm{\phi}_k') dy.
\end{align*}
\end{small}

\begin{algorithm}[htbp]
    \scriptsize
	\caption{The DM-vNEDMIX Algorithm}
	\label{NED:1}
	Set m = 0.
	\begin{algorithmic}
		\REPEAT
		\STATE 1. Compute
		\begin{align*}
		\text{vNED}_n(\bm{\theta}^{(m)}) = \sum_{k=1}^K \int_\mathcal{Y} \exp \left(-\frac{\pi_k^{(m)} h(y; \bm{\phi}_k^{(m)})}{g_n(y)w_k(y;\bm{\theta}^{(m)}) }\right)g_n(y)w_k(y;\bm{\theta}^{(m)})  dy, \quad \text{where} \quad w_k(y;\bm{\theta}) = \frac{\pi_k h(y;\bm{\phi}_k)}{\sum_{l=1}^K \pi_l h(y;\bm{\phi}_l)}.
		\end{align*}
		\STATE 2. Update $\bm{\phi}_k^{(m+1)}$:
		\begin{align*}
		\bm{\phi}_k^{(m+1)} ~ & = ~ \underset{\bm{\phi}_k' \in \bm{\Theta}}{\text{argmin}} ~  \int_\mathcal{Y} \exp \left(-\frac{\pi_k^{(m)} h(y; \bm{\phi}_k')}{g_n(y)w_k(y;\bm{\theta}^{(m)}) }  \right)g_n(y)w_k(y;\bm{\theta}^{(m)})  dy.
		\end{align*}		
		\STATE 3. Update $\pi_k^{(m+1)}$:
		\begin{align*}
		\pi_k^{(m+1)} =  \frac{\text{vNED}_{n,k}(\bm{\phi}_k^{(m+1)}| \bm{\theta}^{(m)})}{\sum_{l=1}^K \text{vNED}_{n,l}(\bm{\phi}_l^{(m+1)}| \bm{\theta}^{(m)}) }, \quad \text{where}
		\end{align*}
        \begin{align*}
       \text{vNED}_{n,k}(\bm{\phi}_k^{(m+1)}| \bm{\theta}^{(m)}) = \int_\mathcal{Y} \exp\left( -\frac{ \pi_k^{(m)} h(y; \bm{\phi}_k^{(m+1)}) }{g_n(y)w_k(y;\bm{\theta}^{(m)})} \right)\pi_k^{(m)} h(y; \bm{\phi}_k^{(m+1)}) dy.
        \end{align*}       		
		\STATE 4. Update $w_k(y; \bm{\theta}^{(m+1)})$, 
        compute $\text{vNED}_n(\bm{\theta}^{(m+1)})$ and the difference 
        \begin{align*}
        \epsilon_{m+1} = | \text{vNED}_n(\bm{\theta}^{(m+1)}) - \text{vNED}_n(\bm{\theta}^{(m)}) |.
        \end{align*}		
		\STATE 5. Set m = m+1.
		\UNTIL{$\epsilon_{m}< \text{threshold}$.}
	\end{algorithmic}
\end{algorithm}

For Algorithm \ref{NELMIX}, by using recurrence relation in Poisson distribution, we can derive an algorithm similar to the HELMIX algorithm based on vNED. We refer this algorithm as the DM-NELMIX algorithm. 

Let $h\left(y;\lambda\right) = e^{\lambda}\lambda^{y}/ y!$. For given initial values $\bm{\theta}^{(0)} = (\pi_1^{(0)},\cdots, \pi_K^{(0)}, \lambda_1^{(0)}, \cdots, \lambda_K^{(0)} )$.
\begin{algorithm}[htbp]
    \footnotesize
	\caption{The DM-NELMIX Algorithm}
	\label{NELMIX}
	Set m = 0.
	\begin{algorithmic}
		\REPEAT
		\STATE 1. Compute
		\begin{align*}
		\text{vNED}_n^*(\bm{\theta}^{(m)}) = \sum_{k=1}^K \sum_{y \in \mathcal{Y}} \exp \left(-\frac{\pi_k^{(m)} h(y; \bm{\phi}_k^{(m)})}{g_n(y)w_k(y;\bm{\theta}^{(m)}) }\right)g_n(y)w_k(y;\bm{\theta}^{(m)}),
		\end{align*}
		~~~~ where 
        \begin{align*}
        w_k(y;\bm{\theta}^{(m)}) = \frac{\pi_k^{(m)} h(y ; \lambda_k^{(m)})}{\sum_{l=1}^{K}\pi_l^{(m)} h(y ; \lambda_l^{(m)})}. 
        \end{align*}     
		\STATE 2. Find $\lambda_k^{(m+1)}$:
		\begin{align*}
		\lambda_k^{(m+1)} ~ & = ~ \sum_{y \in \mathcal{Y}} y w_{yk}^{(m)}  \Big/\sum_{y \in \mathcal{Y}}  w_{yk}^{(m)},
		\end{align*}
		~~~~where 		
		\begin{align*}
		w_{yk}^{(m)} = \exp\left(- \frac{\pi_k^{(m)} h(y; \lambda_k^{(m)})}{ g_{n}(y)w_k(y;\bm{\theta}^{(m)})  } \right)\pi_k^{(m)} h(y; \lambda_k^{(m)}).
		\end{align*}		
		\STATE 3. Update $\pi_k^{(m+1)}$:
		\begin{align*}
			\pi_k^{(m+1)} = \text{vNED}_{n,k}\left(\pi_k^{(m)}, \lambda_k^{(m+1)}\right) \Big/\sum_{l=1}^{K} \text{vNED}_{n,l} \left(\pi_l^{(m)}, \lambda_l^{(m+1)} \right),
		\end{align*}
		~~~~where 
		\begin{align*}
		\text{vNED}_{n,k}\left(\pi_k^{(m)}, \phi_k^{(m+1)}\right) = \sum_{y \in \mathcal{Y}} \exp\left(- \frac{\pi_k^{(m)} h(y ; \lambda_k^{(m+1)})}{ g_{n}(y)w_k(y; \bm{\theta}^{(m+1)}) } \right)\pi_k^{(m)} h(y;\lambda_k^{(m+1)}).
		\end{align*}				
		\STATE 4. Update $w_k(y; \bm{\theta}^{(m+1)})$, 
        compute $\text{vNED}_n^*(\bm{\theta}^{(m+1)})$ and the difference 
        \begin{align*}
        \epsilon_{m+1} = |\text{vNED}_n^*(\bm{\theta}^{(m+1)}) - \text{vNED}_n^*(\bm{\theta}^{(m)}) |.
        \end{align*}					
		\STATE 5. Set m = m+1.
		\UNTIL{$\epsilon_{m}< \text{threshold}$.}
	\end{algorithmic}
\end{algorithm}

\newpage 

\section*{I: Additional Simulation Results}\label{add-simulations}

\noindent \textbf{Discrete Models}

In this section, we assume $K$ is known. 
For discrete data we estimate the nonparametric distribution $g_n(\cdot)$ using discrete kernels. 
We implemented both cross-validation and moment-based bandwidth selection; the results were similar, while the moment-based method was significantly faster, so we report only the moment-based results.

\noindent \textbf{Poisson Mixture Model}

Suppose $f(\cdot; \bm{\theta}) = 0.4\,\mathrm{Pois}(0.5) + 0.6\,\mathrm{Pois}(10)$.
Figure~\ref{fig:pi1} shows the average estimates of $\pi_1$ across kernels and sample sizes $n\in\{20,50,100,200\}$;
the dashed line indicates the true value $0.4$.



\begin{figure}[H]
    \centering
    \includegraphics[width=0.8\textwidth]{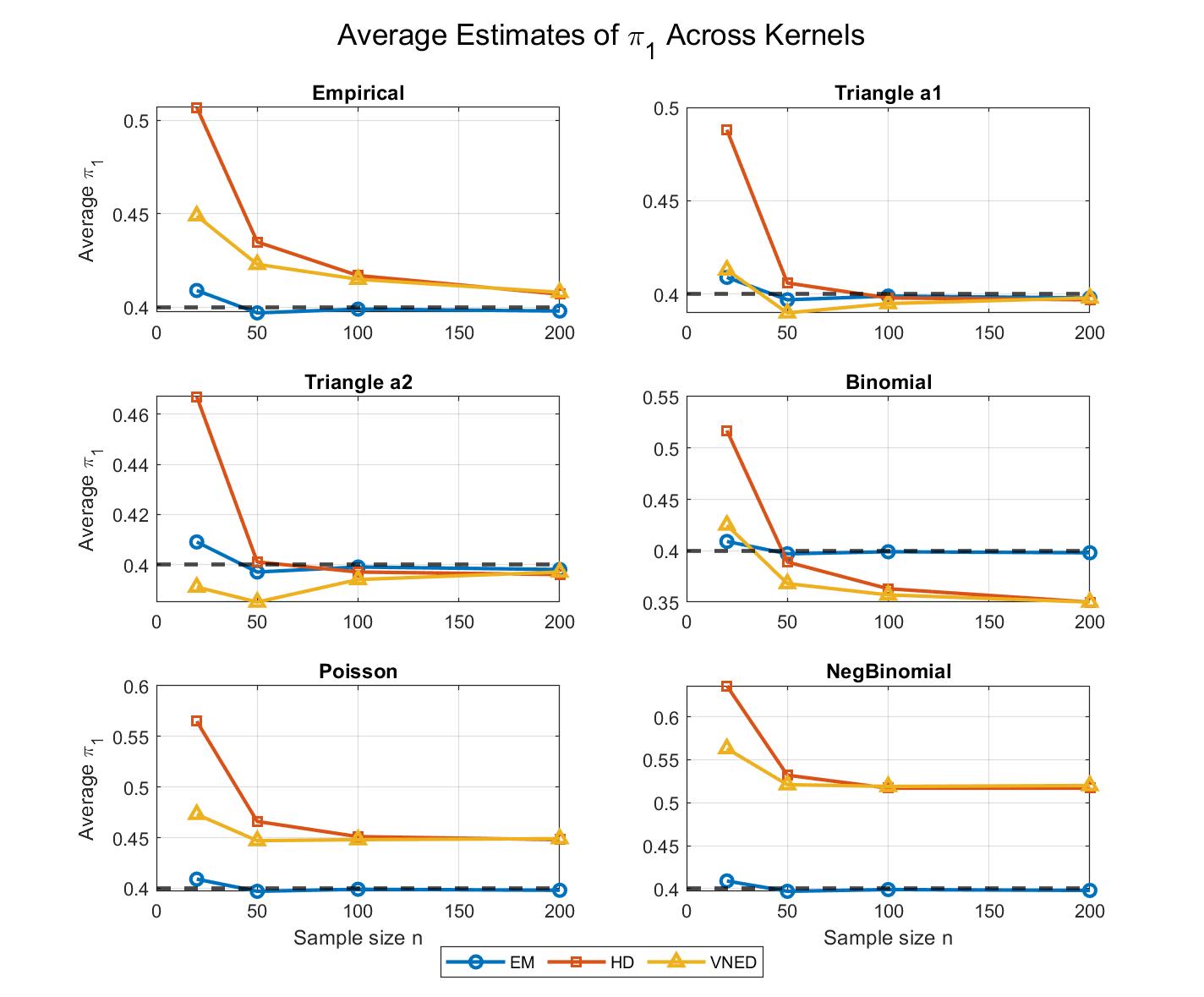}
    \caption{\small Average estimates of $\pi_1$ across kernels for sample sizes $n=20,50,100,200$. 
    Solid lines represent EM, HD, and vNED methods; dashed horizontal line indicates the true value $0.4$.}
    \label{fig:pi1}
\end{figure}

\newpage 
Figure~\ref{Fig-ISE-PG-1} reports the mean ISE for the two-component Poisson mixture as $n$ increases at several bandwidths $c$.
The results indicate: (i) with small $n$, discrete kernels are beneficial; (ii) bandwidth choice is important; 
and (iii) the triangle kernel with $a=1$ performs well overall for this model.


\begin{figure}[H] 
  \begin{subfigure}[b]{0.5\linewidth}
    \centering
    \includegraphics[width=0.75\linewidth]{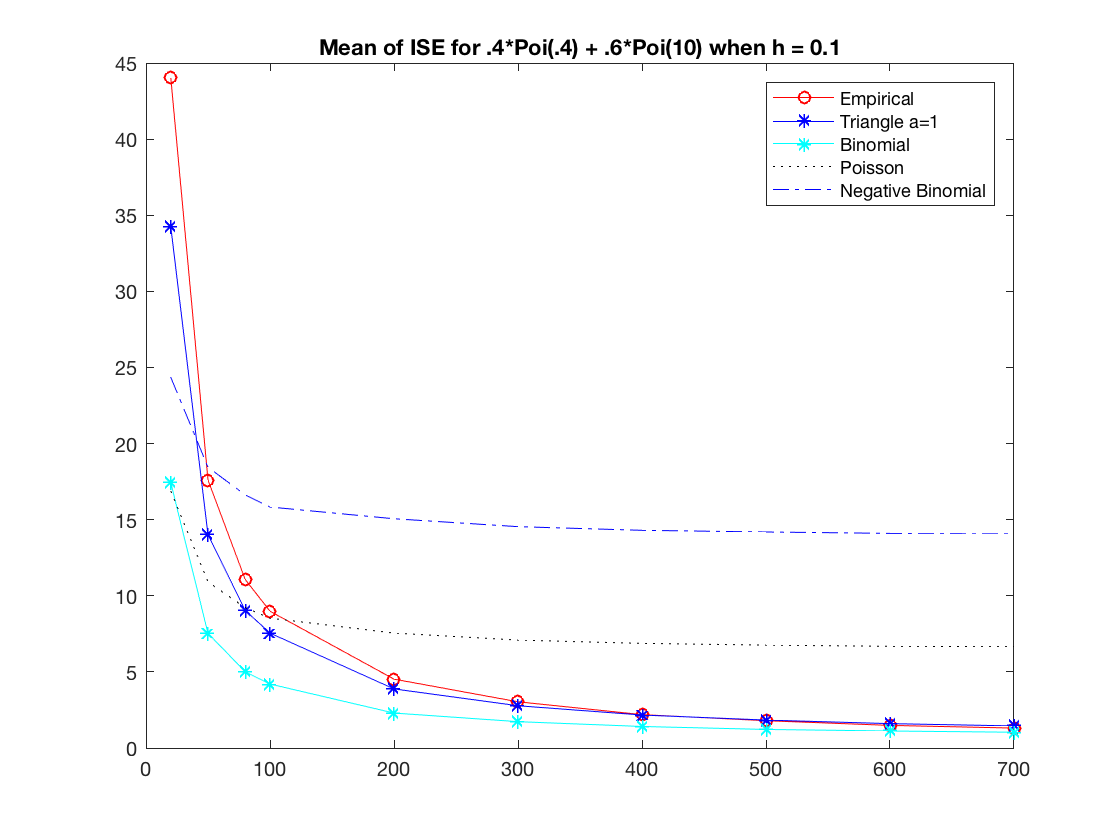} 
    \caption{$c=0.1$} 
    \label{fig7:a1} 
    \vspace{4ex}
  \end{subfigure}
  \begin{subfigure}[b]{0.5\linewidth}
    \centering
    \includegraphics[width=0.75\linewidth]{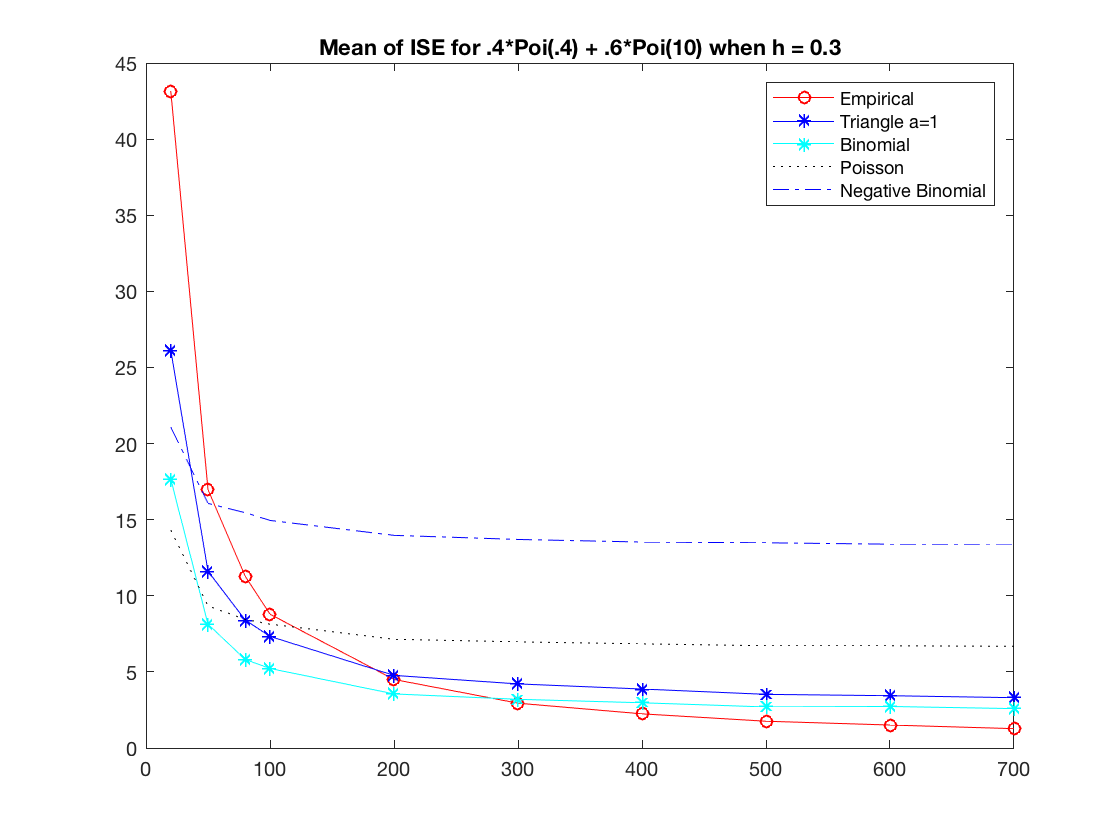} 
    \caption{$c=0.3$} 
    \label{fig7:b1} 
    \vspace{4ex}
  \end{subfigure} 
  \begin{subfigure}[b]{0.5\linewidth}
    \centering
    \includegraphics[width=0.75\linewidth]{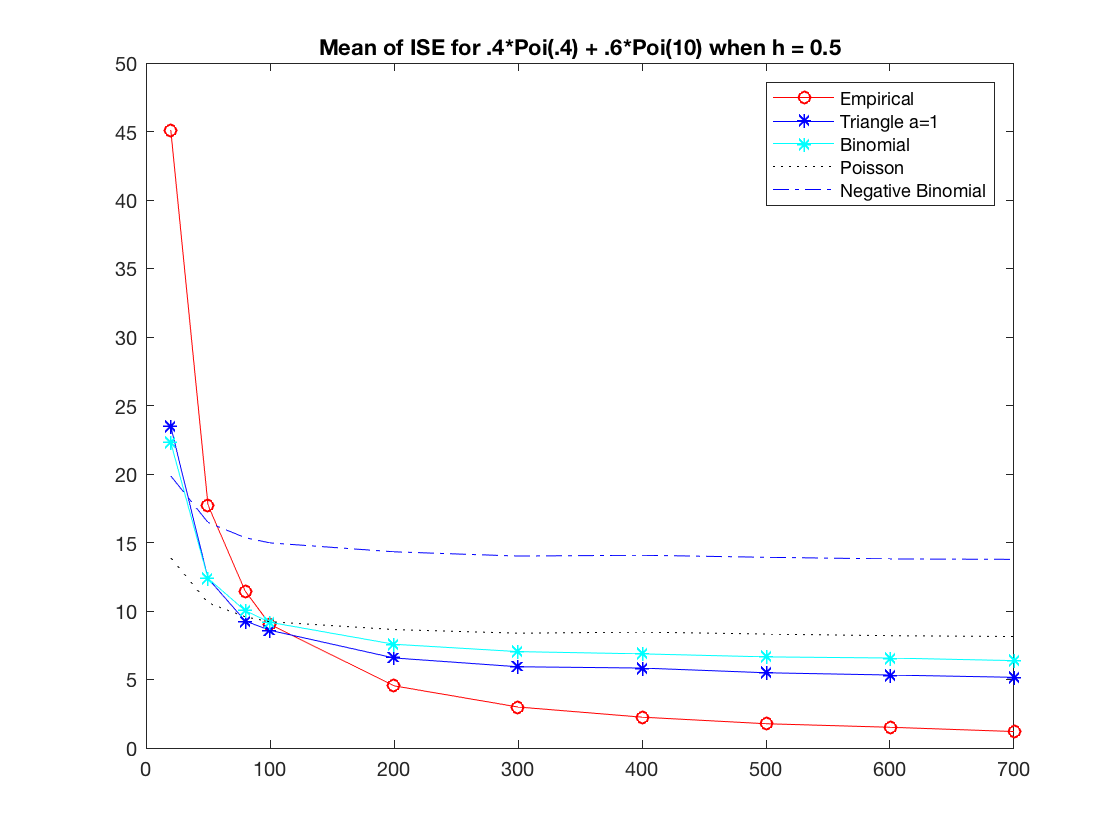} 
    \caption{$c=0.5$} 
    \label{fig7:c1} 
  \end{subfigure}
  \begin{subfigure}[b]{0.5\linewidth}
    \centering
    \includegraphics[width=0.75\linewidth]{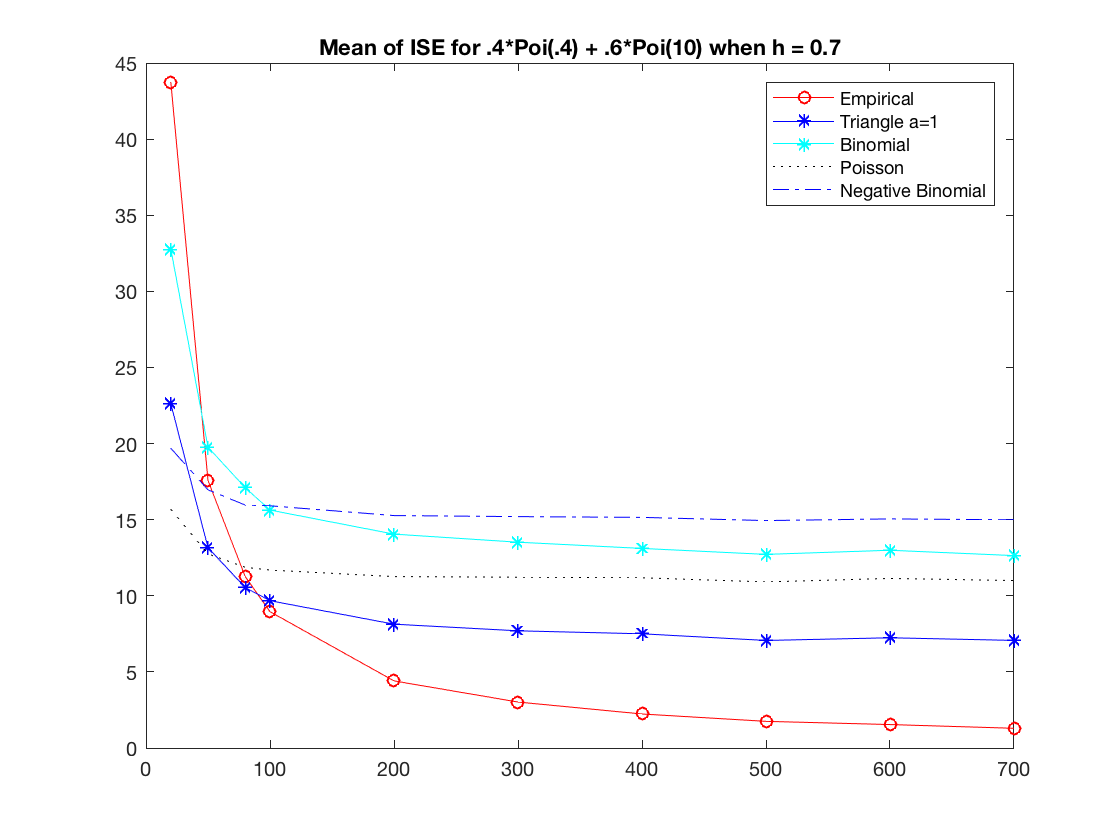} 
    \caption{$c=0.7$} 
    \label{fig7:d1} 
  \end{subfigure} 
  \caption{Mean of ISE for Model $0.4Poi(0.5) + 0.6Poi(10)$ as $n$ changes}
  \label{Fig-ISE-PG-1}
\end{figure}

In this section for discrete models, all simulation results use the empirical p.m.f.\ 
\[
g_n(y) = \frac{n_y}{n},\qquad n_y=\sum_{i=1}^n \mathbf{1}\{Y_i=y\}.
\]



\begin{table}[H]
\centering
\captionsetup{font=small}
\caption{Estimates Using Different Kernels for $0.4 Poi(0.5) + 0.6Poi(10)$}
\resizebox{1.0\textwidth}{!}{
\centering
\begin{tabular}{cc|ccc|ccc|ccc}
\toprule
Kernel               &     & \multicolumn{3}{c|}{Empirical}                          & \multicolumn{3}{c|}{Triangle $a=1$}                     & \multicolumn{3}{c}{Triangle $a=2$}                     \\ \hline
$n$                  &     & $\hat{\pi}_1$   & $\hat{\lambda}_1$ & $\hat{\lambda}_2$ & $\hat{\pi}_1$   & $\hat{\lambda}_1$ & $\hat{\lambda}_2$ & $\hat{\pi}_1$   & $\hat{\lambda}_1$ & $\hat{\lambda}_2$ \\ \hline
\multirow{3}{*}{20}  & EM  & 0.409 (0.113) & 0.544 (0.282)   & 9.904 (1.007)   & 0.409 (0.113) & 0.544 (0.282)   & 9.904 (1.007)   & 0.409 (0.113) & 0.544 (0.282)   & 9.904 (1.007)   \\
                     & HD  & 0.507 (0.148) & 0.410 (0.249)   & 9.605 (1.087)   & 0.488 (0.176) & 0.487 (0.270)   & 9.604 (1.145)   & 0.467 (0.183) & 0.498 (0.300)   & 9.578 (1.144)   \\
                     & vNED & 0.449 (0.135) & 0.456 (0.268)   & 9.556 (1.120)   & 0.413 (0.155) & 0.589 (0.292)   & 9.639 (1.146)   & 0.391 (0.157) & 0.631 (0.349)   & 9.638 (1.135)   \\ \hline
\multirow{3}{*}{50}  & EM  & 0.397 (0.074) & 0.490 (0.173)   & 10.01 (0.548)   & 0.397 (0.074) & 0.490 (0.173)   & 10.01 (0.548)   & 0.397 (0.074) & 0.490 (0.173)   & 10.01 (0.548)   \\
                     & HD  & 0.435 (0.083) & 0.442 (0.171)   & 9.788 (0.598)   & 0.406 (0.089) & 0.570 (0.211)   & 9.772 (0.599)   & 0.401 (0.091) & 0.572 (0.221)   & 9.768 (0.599)   \\
                     & vNED & 0.423 (0.079) & 0.463 (0.175)   & 9.866 (0.598)   & 0.390 (0.083) & 0.621 (0.216)   & 9.873 (0.582)   & 0.385 (0.085) & 0.620 (0.227)   & 9.873 (0.583)   \\ \hline
\multirow{3}{*}{100} & EM  & 0.399 (0.050) & 0.502 (0.113)   & 10.02 (0.415)   & 0.399 (0.050) & 0.502 (0.113)   & 10.02 (0.415)   & 0.399 (0.050) & 0.502 (0.113)   &10.02 (0.415)   \\
                     & HD  & 0.417 (0.053) & 0.478 (0.114)   & 9.849 (0.462)   & 0.398 (0.056) & 0.584 (0.152)   & 9.841 (0.473)   & 0.397 (0.057) & 0.584 (0.152)   & 9.835 (0.474)   \\
                     & vNED & 0.415 (0.052) & 0.482 (0.110)  & 9.917 (0.463)   & 0.395 (0.054) & 0.600 (0.147)   & 9.922 (0.464)   & 0.394 (0.055) & 0.598 (0.145)   & 9.915 (0.464)   \\ \hline
\multirow{3}{*}{200} & EM  & 0.398 (0.032)   & 0.502 (0.073)     & 10.07 (0.296)     & 0.398 (0.032)   & 0.502 (0.073)     & 10.07 (0.296)     & 0.398 (0.032)   & 0.502 (0.073)     & 10.07 (0.296)     \\
                     & HD  & 0.407 (0.033)   & 0.489 (0.074)     & 9.879 (0.305)     & 0.397 (0.035)   & 0.551 (0.085)     & 9.875 (0.307)     & 0.396 (0.035)   & 0.553 (0.087)     & 9.869 (0.307)     \\
                     & vNED & 0.408 (0.033)   & 0.488 (0.070)     & 9.941 (0.299)     & 0.398 (0.035)   & 0.553 (0.082)     & 9.942 (0.300)     & 0.397 (0.035)   & 0.554 (0.084)     & 9.934 (0.300)     \\ \hline
Kernel               &     & \multicolumn{3}{c|}{Binomial}                           & \multicolumn{3}{c|}{Poisson}                            & \multicolumn{3}{c}{Negative Binomial}                  \\ \hline
$n$                  &     & $\hat{\pi}_1$   & $\hat{\lambda}_1$ & $\hat{\lambda}_2$ & $\hat{\pi}_1$   & $\hat{\lambda}_1$ & $\hat{\lambda}_2$ & $\hat{\pi}_1$   & $\hat{\lambda}_1$ & $\hat{\lambda}_2$ \\ \hline
\multirow{3}{*}{20}  & EM  & 0.409 (0.113)   & 0.544 (0.282)     & 9.904 (1.007)     & 0.409 (0.113)   & 0.544 (0.282)     & 9.904 (1.007)     & 0.409 (0.113)   & 0.544 (0.282)     & 9.904 (1.007)     \\
                     & HD  & 0.517 (0.176)   & 0.354 (0.225)     & 9.441 (1.210)     & 0.565 (0.183)   & 0.426 (0.252)     & 9.380 (1.221)     & 0.636 (0.171)   & 0.466 (0.268)     & 9.400 (1.270)     \\
                     & vNED & 0.425 (0.153)   & 0.408 (0.230)     & 9.453 (1.234)     & 0.473 (0.165)   & 0.519 (0.255)     & 9.396 (1.220)     & 0.563 (0.164)   & 0.565 (0.287)     & 9.366 (1.511)     \\ \hline
\multirow{3}{*}{50}  & EM  & 0.397 (0.074)   & 0.490 (0.173)     & 10.01 (0.548)     & 0.397 (0.074)   & 0.490 (0.173)     & 10.01 (0.548)     & 0.397 (0.074)   & 0.490 (0.173)     & 10.01 (0.548)     \\
                     & HD  & 0.389 (0.090)   & 0.438 (0.153)     & 9.503 (0.635)     & 0.466 (0.095)   & 0.604 (0.213)     & 9.582 (0.727)     & 0.532 (0.093)   & 0.681 (0.255)     & 9.586 (0.805)     \\
                     & vNED & 0.368 (0.081)   & 0.463 (0.144)     & 9.554 (0.614)     & 0.447 (0.086)   & 0.651 (0.191)     & 9.649 (0.733)     & 0.521 (0.084)   & 0.732 (0.230)     & 9.618 (0.896)     \\ \hline
\multirow{3}{*}{100} & EM  & 0.399 (0.050)   & 0.502 (0.113)     & 10.02 (0.415)     & 0.399 (0.050)   & 0.502 (0.113)     & 10.02 (0.415)     & 0.399 (0.050)   & 0.502 (0.113)     & 10.02 (0.415)     \\
                     & HD  & 0.363 (0.053)   & 0.492 (0.111)     & 9.510 (0.522)     & 0.451 (0.058)   & 0.765 (0.196)     & 9.752 (0.576)     & 0.517 (0.060)   & 0.903 (0.260)     & 9.812 (0.630)     \\
                     & vNED & 0.357 (0.050)   & 0.502 (0.101)     & 9.562 (0.498)     & 0.448 (0.054)   & 0.777 (0.168)     & 9.818 (0.552)     & 0.519 (0.055)   & 0.912 (0.224)     & 9.832 (0.646)     \\ \hline
\multirow{3}{*}{200} & EM  & 0.398 (0.032)   & 0.502 (0.073)     & 10.07 (0.296)     & 0.398 (0.032)   & 0.502 (0.073)     & 10.07 (0.296)     & 0.398 (0.032)   & 0.502 (0.073)     & 10.07 (0.296)     \\
                     & HD  & 0.350 (0.031)   & 0.513 (0.072)     & 9.491 (0.322)     & 0.448 (0.035)   & 0.878 (0.125)     & 9.911 (0.374)     & 0.517 (0.035)   & 1.089 (0.177)     & 10.09 (0.445)     \\
                     & vNED & 0.350 (0.031)   & 0.515 (0.067)     & 9.547 (0.310)     & 0.449 (0.034)   & 0.861 (0.104)     & 9.942 (0.353)     & 0.520 (0.033)   & 1.055 (0.148)     & 10.05 (0.438)     \\ \bottomrule
\end{tabular}
}
\label{tab:PoisAve}
\end{table}

Figure \ref{fig:Poislambda1} and Figure \ref{fig:Poislambda2} are figure illustrations of Table \ref{tab:PoisAve}.

\begin{figure}[H]
    \centering
    \includegraphics[width=0.8\textwidth]{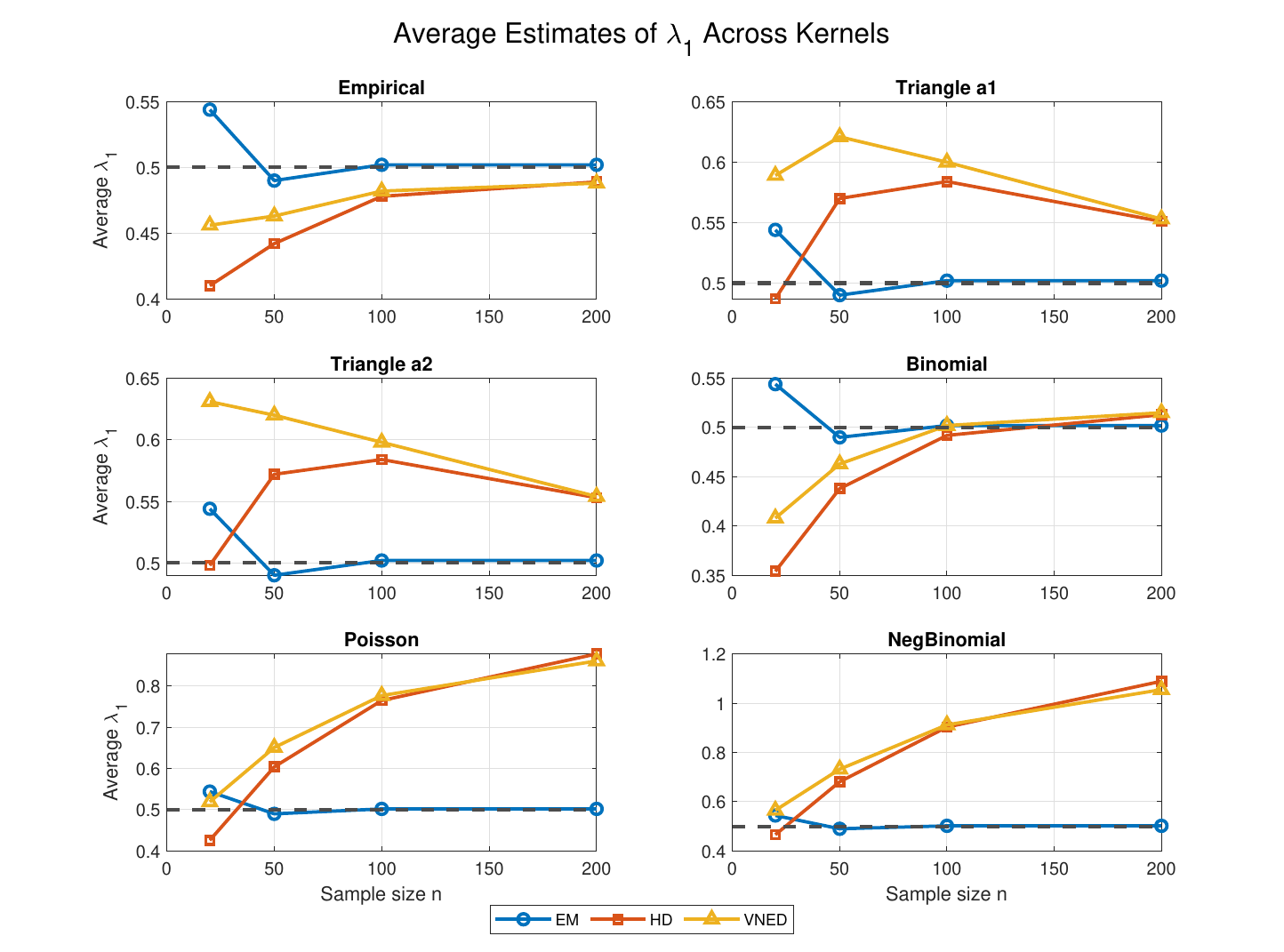}
    \caption{\small Average estimates of $\lambda_1$ across kernels for sample sizes $n=20,50,100,200$. 
    Solid lines represent EM, HD, and vNED methods; dashed horizontal line indicates the true value $0.4$.}
    \label{fig:Poislambda1}
\end{figure}

\begin{figure}[H]
    \centering
    \includegraphics[width=0.8\textwidth]{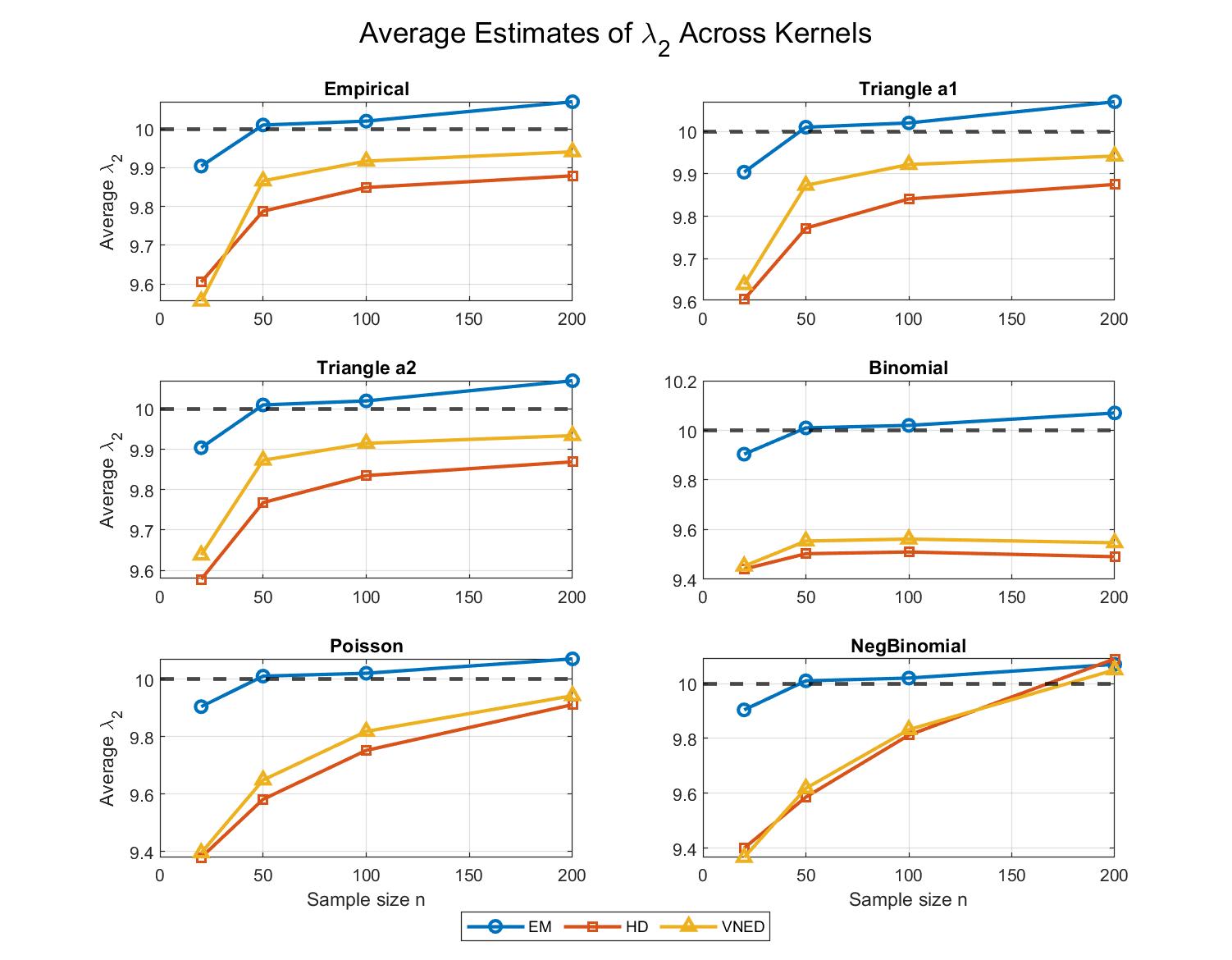}
    \caption{\small Average estimates of $\lambda_2$ across kernels for sample sizes $n=20,50,100,200$. 
    Solid lines represent EM, HD, and vNED methods; dashed horizontal line indicates the true value $0.4$.}
    \label{fig:Poislambda2}
\end{figure}

\noindent \textbf{PG Mixture Model with Discrete Kernel}

Now suppose the underlying model is two component PG mixture given by 
$ f(\cdot; \bm{\theta}) = 0.3PG(10, 1) + 0.7PG(1, 2),$ where $PG(\alpha,\beta)$ denotes a Poisson--Gamma mixture with $\mathrm{Gamma}(\alpha,\beta)$ mixing (shape $\alpha$, rate $\beta$).
For $\epsilon\in\{0,0.05,0.10,0.15,0.20,0.30\}$, we inject outliers by replacing an $\epsilon$ fraction of observations with the fixed value $50$:
$Y_i^{(\epsilon)}=(1-B_i)Y_i + 50\,B_i$ with $B_i\sim\mathrm{Bernoulli}(\epsilon)$ independently.
At each $\epsilon$ we estimate by EM, HMIX, and vNEDMIX. For each parameter, the figure plots the Monte Carlo average (over $5{,}000$ replications) versus $\epsilon$; the dotted line marks the truth.

\begin{table}[H]
\centering
\scriptsize
\captionsetup{font=small}
\caption{Estimates Using Different Kernels for $0.3 PG(10, 1) + 0.7PG(1, 2)$}
\label{ISE-PG-2}
\begin{tabular}{ccccccc}
\toprule
$n=1000$    &        & $\hat{\pi}_1$ & $\hat{\alpha}_1$ & $\hat{\beta}_1$ & $\hat{\alpha}_2$ & $\hat{\beta}_2$ \\ \hline
EM          &         & 0.298 (0.018) & 10.63 (2.827)    & 1.058 (0.269)   & 1.091 (0.378)    & 2.216 (0.938)   \\ \hline
\multirow{2}{*}{Empirical}         & HD  & 0.289 (0.018) & 13.60 (4.089)    & 1.364 (0.391)   & 1.009 (0.315)    & 1.985 (0.760)   \\
                                   & vNED & 0.289 (0.018) & 12.82 (3.876)    & 1.287 (0.374)   & 1.043 (0.346)    & 2.081 (0.845)   \\ \hline
\multirow{2}{*}{Triangle $a=1$}    & HD  & 0.292 (0.018) & 13.54 (4.033)    & 1.358 (0.386)   & 1.079 (0.331)    & 2.072 (0.776)   \\
                                   & vNED & 0.292 (0.018) & 12.68 (3.766)    & 1.273 (0.363)   & 1.117 (0.363)    & 2.175 (0.862)   \\ \hline
\multirow{2}{*}{Triangle $a=2$}    & HD  & 0.292 (0.018) & 13.89 (4.270)    & 1.388 (0.407)   & 1.033 (0.281)    & 1.865 (0.624)   \\
                                   & vNED & 0.293 (0.018) & 12.93 (3.904)    & 1.293 (0.375)   & 1.064 (0.299)    & 1.945 (0.672)   \\ \hline
\multirow{2}{*}{Binomial}          & HD  & 0.272 (0.017) & 14.21 (4.622)    & 1.420 (0.442)   & 1.226 (0.178)    & 1.826 (0.368)   \\
                                   & vNED & 0.275 (0.017) & 12.61 (3.787)    & 1.262 (0.363)   & 1.263 (0.182)    & 1.907 (0.376)   \\ \hline
\multirow{2}{*}{Poisson}           & HD  & 0.209 (0.016) & 15.90 (4.911)    & 1.416 (0.423)   & 0.865 (0.044)    & 0.703 (0.069)   \\
                                   & vNED & 0.216 (0.016) & 12.76 (3.264)    & 1.146 (0.283)   & 0.890 (0.042)    & 0.745 (0.067)   \\ \hline
\multirow{2}{*}{Negative Binomial} & HD  & 0.148 (0.014) & 29.87 (16.50)    & 2.346 (1.290)   & 0.824 (0.029)    & 0.447 (0.036)   \\
                                   & vNED & 0.152 (0.014) & 23.22 (9.331)    & 1.845 (0.745)   & 0.834 (0.025)    & 0.459 (0.031)   \\ \bottomrule
\end{tabular}
\end{table}

Figure \ref{Fig-ISE-PG-2} compare the mean of ISE for two component PG mixture model as $n$ increases when choosing different bandwidth $c$. This indicates that (i) when the sample size is small, discrete kernels are beneficial, (ii) choosing proper discrete bandwidth is also crucial, and (iii) empirical kernel may be the best choice for this model. 

\begin{figure}[H] 
  \begin{subfigure}[b]{0.5\linewidth}
    \centering
    \includegraphics[width=0.75\linewidth]{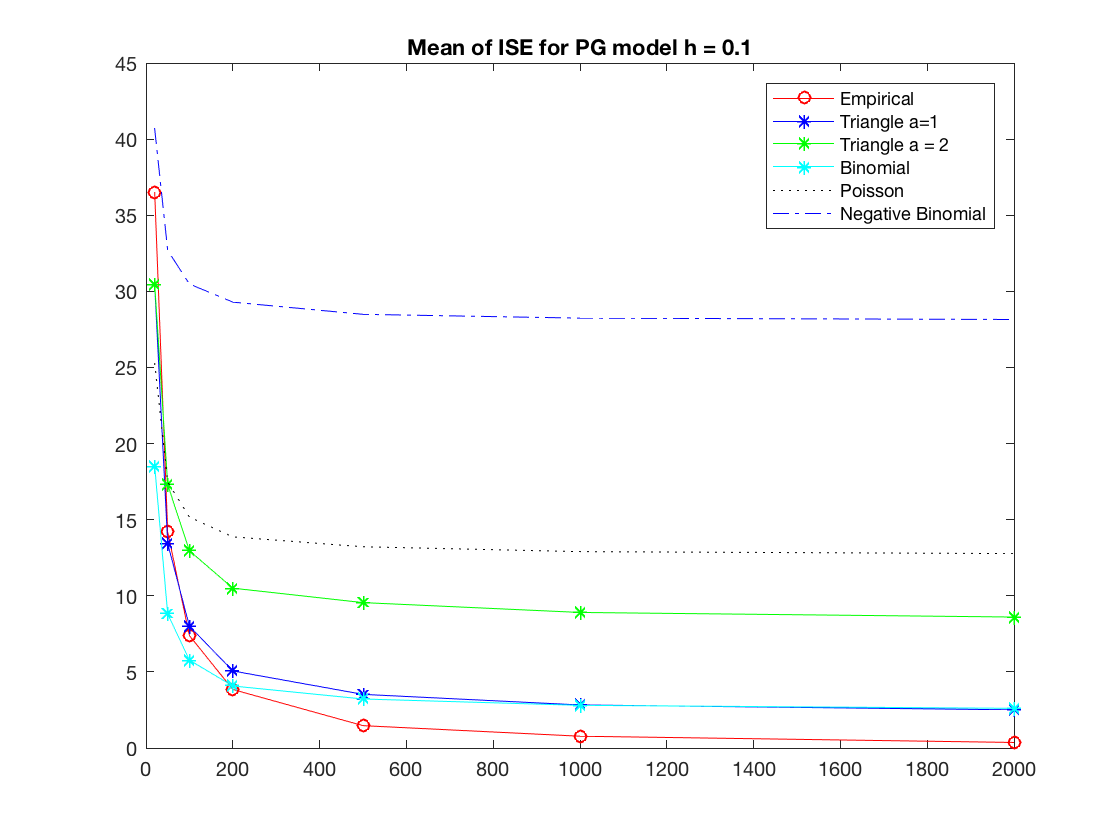} 
    \caption{$c=0.1$} 
    \label{fig7:a2} 
    \vspace{4ex}
  \end{subfigure}
  \begin{subfigure}[b]{0.5\linewidth}
    \centering
    \includegraphics[width=0.75\linewidth]{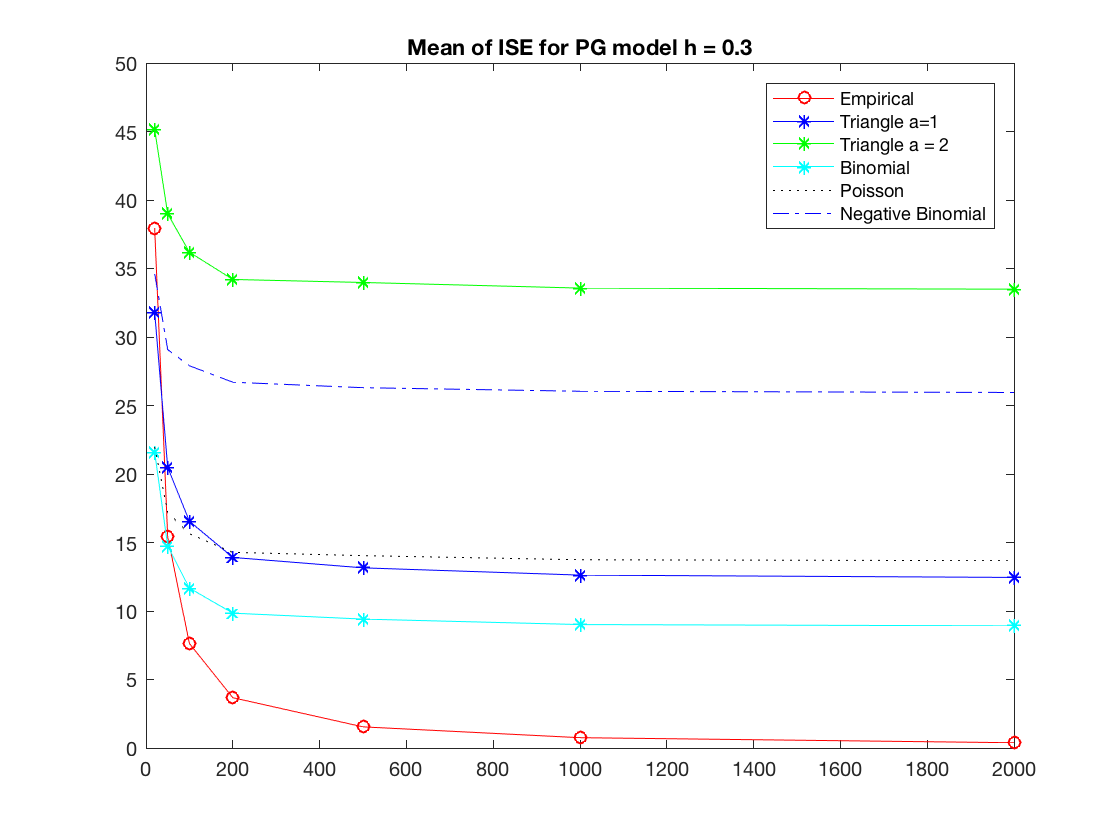} 
    \caption{$c=0.3$} 
    \label{fig7:b2} 
    \vspace{4ex}
  \end{subfigure} 
  \begin{subfigure}[b]{0.5\linewidth}
    \centering
    \includegraphics[width=0.75\linewidth]{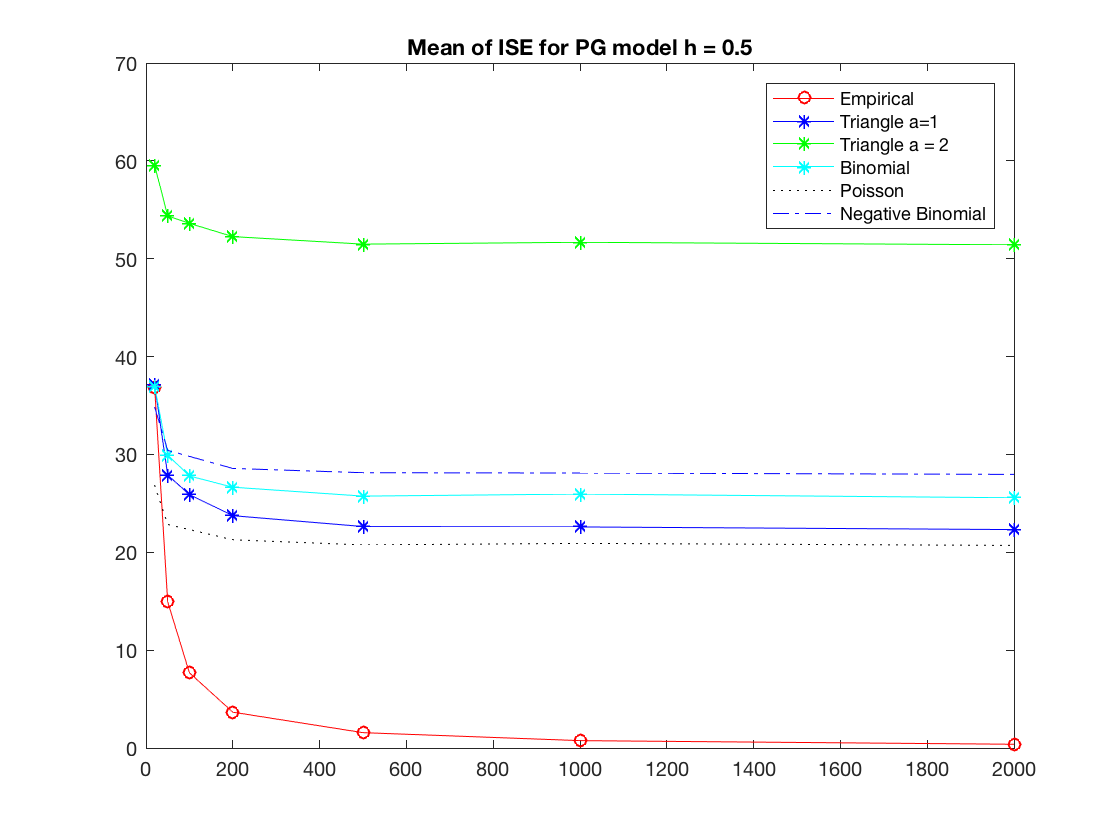} 
    \caption{$c=0.5$} 
    \label{fig7:c2} 
  \end{subfigure}
  \begin{subfigure}[b]{0.5\linewidth}
    \centering
    \includegraphics[width=0.75\linewidth]{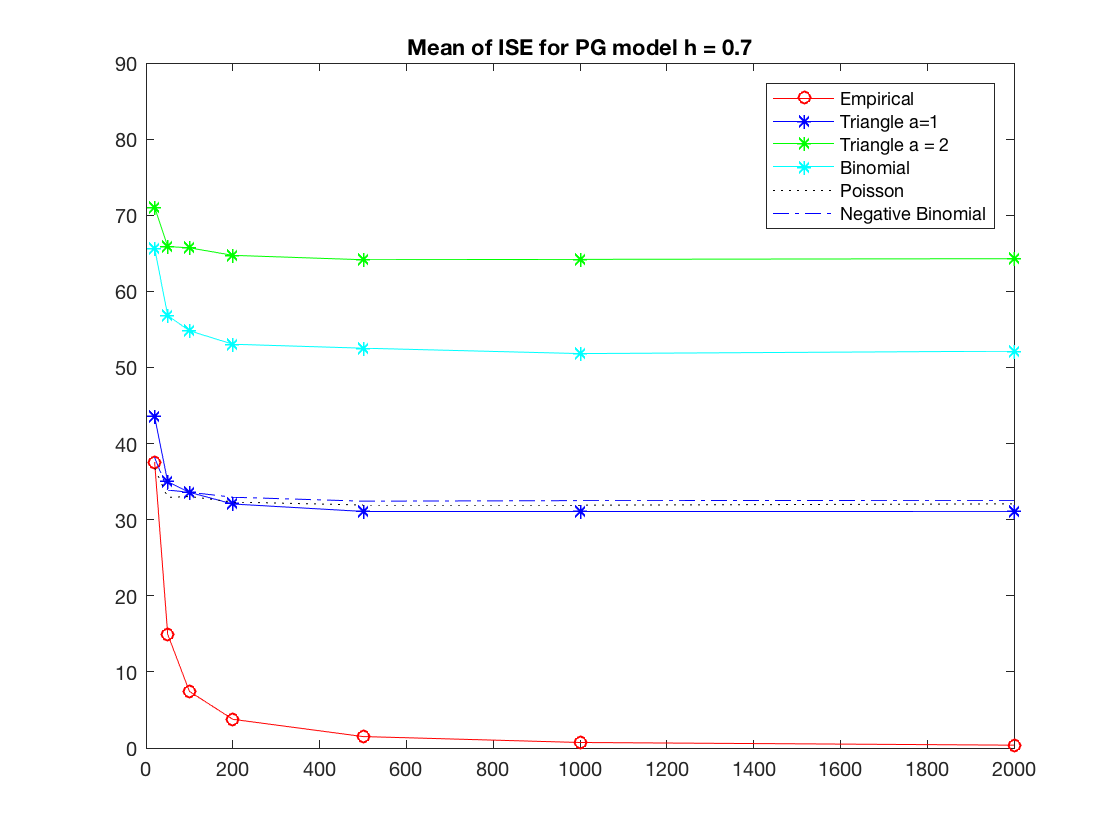} 
    \caption{$c=0.7$} 
    \label{fig7:d2} 
  \end{subfigure} 
  \caption{Mean of ISE for Model $0.3PG(10, 1) + 0.7PG(1,2)$ as $n$ changes}
  \label{Fig-ISE-PG-2}
\end{figure}

\noindent \textbf{PL Mixture to Test Robustness}


We simulate from a two–component PL mixture with
  $\pi_1=0.3$, $(\mu_1,\sigma_1)=(3,0.5)$ and $(\mu_2,\sigma_2)=(1,0.5)$.
  For each contamination level $\epsilon\in\{0,0.05,0.10,0.15,0.20,0.30\}$, an $\epsilon$-fraction
  of observations is replaced by a large outlier value ($50$). Sample size $n=2000$,
  replications $R=5000$.
  Panels (a)--(e) plot the replication averages of the estimated parameters
  $(\pi_1,\mu_1,\sigma_1,\mu_2,\sigma_2)$ versus $\epsilon$ for three methods:
  EM (MLE), HMIX (Hellinger divergence), and vNEDMIX (vNED divergence).
  Horizontal dotted lines mark the true parameter values.
  For PL likelihood evaluations we use fixed Gauss--Hermite quadrature (common nodes/weights across methods).
  Warm starts are used across contamination levels to stabilize fits, and axis limits are set robustly
  so occasional non–convergent runs do not distort the display.

\begin{figure}[H]
    \centering
    \includegraphics[width=0.8\textwidth]{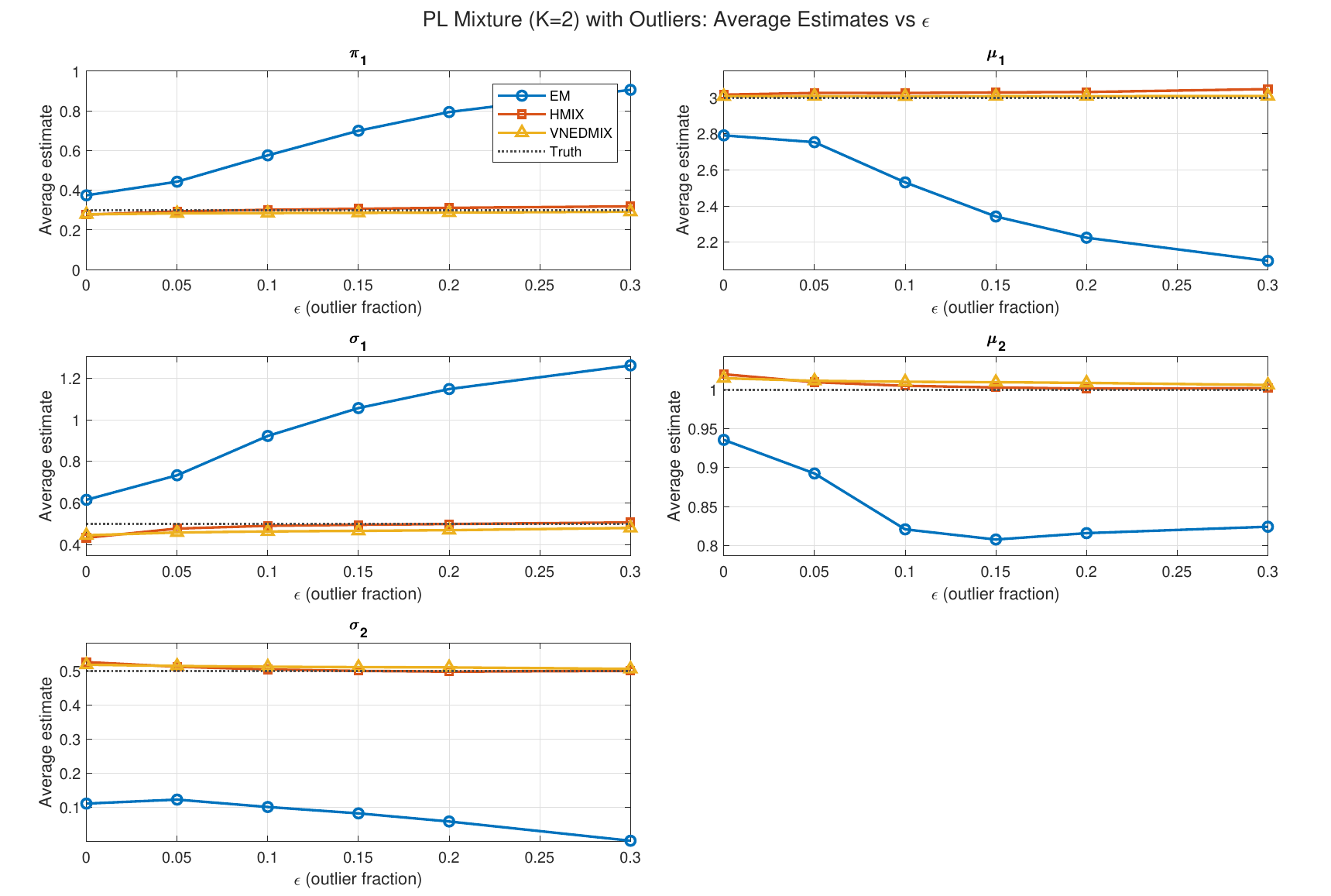}
    \caption{\small Outlier robustness for a Poisson--Lognormal (PL) mixture (known $K=2$).}
    \label{fig:pl_outlier}
\end{figure}

This figure assesses robustness of three estimation strategies for the PL mixture under increasing contamination.
All methods perform well at $\epsilon=0$. As $\epsilon$ grows, EM becomes
progressively more sensitive to the injected large counts, typically inflating location/scale parameters for the
high-mean component and, to a lesser extent, perturbing the mixture weight $\pi_1$. In contrast, HMIX and vNEDMIX downweight the influence
of extreme observations through their discrepancy objectives, and therefore track the dotted truth lines more closely
over a wider range of $\epsilon$. The divergence-based methods thus exhibit markedly improved stability for both
location $(\mu_k)$ and dispersion $(\sigma_k)$ parameters under contamination, while maintaining competitive accuracy
in the clean setting. The use of fixed Gauss--Hermite quadrature provides fast, consistent likelihood approximations
for the PL components across all methods.

\noindent \textbf{Continuous Models}

In this subsection, we consider the continuous model case. Specifically, we consider the two component mixture normal model, i.e.,
\begin{align*}
    f\left(\cdot; \bm{\theta} \right) = \pi_1 N\left(\mu_1, \sigma_1^2\right) + (1-\pi_1) N\left(\mu_2, \sigma_2^2\right) .
\end{align*}
In the following simulation, we let true parameters $\pi_1 = 0.3, \mu_1 = 10; \sigma_1 = 1, \mu_2 =0, \sigma_2 = 1,$ and the sample size is $n= 200$. For the nonparametric density estimate $g_n(\cdot)$, we use the Epanechnikov kernel and the optimal bandwidth is applied. Also, the number of components $K$ is unknown and estimated based on the DIC criteria as described before. 

\begin{table}[H]
\centering
\scriptsize
\caption{Two-Component Normal (Epanechnikov kernel; 100\% chose $K=2$)}
\begin{tabular}{lcccccc}
\toprule
 &  & $\hat{\pi}_1$ & $\hat{\mu}_1$ & $\hat{\sigma}_1$ & $\hat{\mu}_2$ & $\hat{\sigma}_2$ \\
\midrule
\multirow{3}{*}{EM}
 & Ave & 0.299 & 9.998 & 0.986 & 0.000 & 0.995 \\
 & StD & 0.032 & 0.126 & 0.091 & 0.086 & 0.059 \\
 & MSE & 0.001 & 0.016 & 0.009 & 0.007 & 0.004 \\
\midrule
\multirow{3}{*}{HMIX}
 & Ave & 0.299 & 9.997 & 1.236 & 0.000 & 1.250 \\
 & StD & 0.032 & 0.126 & 0.073 & 0.086 & 0.050 \\
 & MSE & 0.001 & 0.016 & 0.061 & 0.007 & 0.065 \\
\midrule
\multirow{3}{*}{vNEDMIX}
 & Ave & 0.299 & 9.997 & 1.243 & 0.000 & 1.253 \\
 & StD & 0.032 & 0.126 & 0.073 & 0.086 & 0.050 \\
 & MSE & 0.001 & 0.016 & 0.064 & 0.007 & 0.067 \\
\bottomrule
\end{tabular}
\end{table}

We observe that for continuous normal mixture models, the HMIX, vNEDMIX are competitive to the EM algorithm.

\begin{table}[H]
\scriptsize
\centering
\caption{Average Runtime (seconds) per simulation for mixture models with $n=20{,}000$ and true $K=2$.}
\label{tab:mix-runtime-rotated}
\begin{tabular}{lcc}
\toprule
 & \textbf{$K$ known} & \textbf{$K$ unknown} \\
\midrule
\textbf{Pois} & 0.2 s & 16 s \\
\textbf{PG}   & 0.7 s & 48 s \\
\textbf{PL}   & 46 s & 66 s \\
\bottomrule
\end{tabular}
\end{table}

\newpage
\section*{J: Additional Image Analysis}\label{supp:add-image-ana}

\textbf{Original Image Recovery}


We observe from Table \ref{chap5:model-selection-supp} that the BIC, HIC, and vNEDIC all prefer $K=5$. However, when we recover the images the difference is negligible (figures not shown here). Thus, we choose  $K=3$. Figure \ref{phan-image} and Figure \ref{Lena-image} show that the EM, HMIX, and vNED algorithms are comaparable in reconstructing the images.

\begin{figure}[H] 
  \begin{subfigure}[b]{0.5\linewidth}
    \centering
    \includegraphics[width=0.75\linewidth]{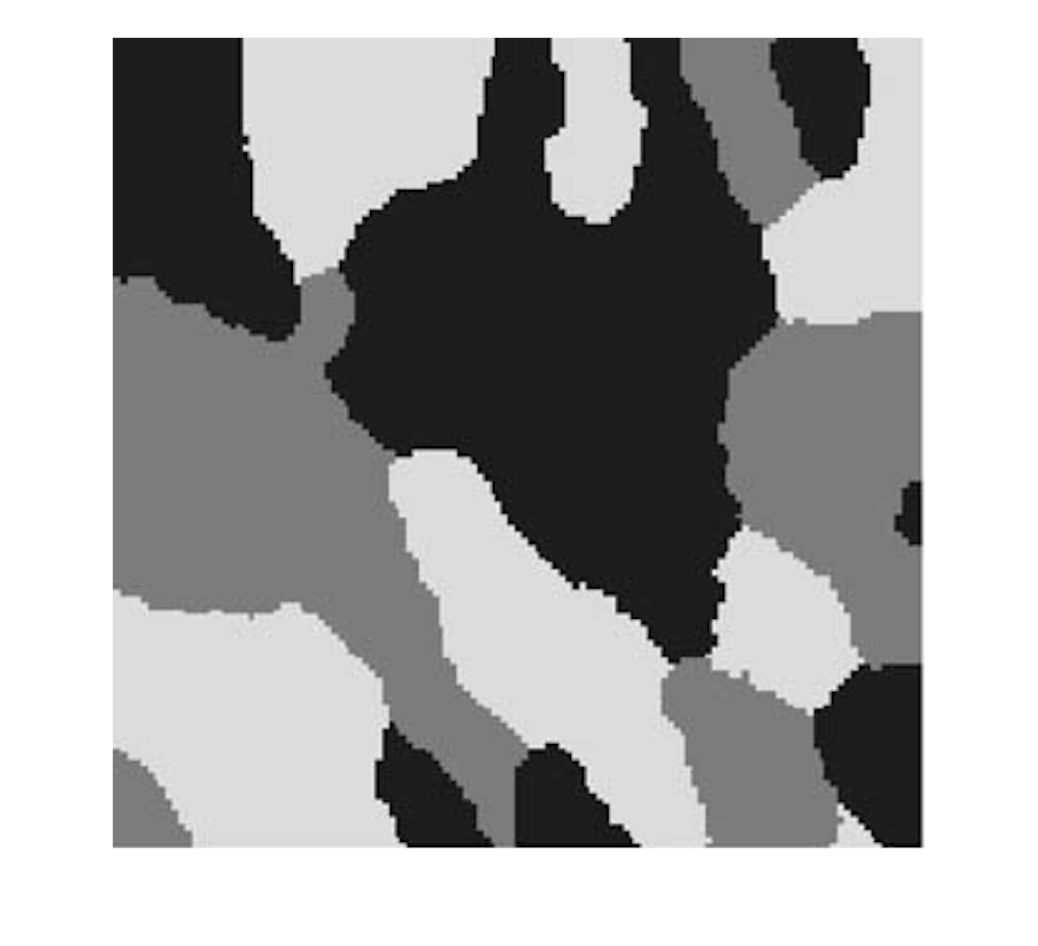} 
    \caption{Original three-class phantom image} 
    \label{ori:image_a} 
    \vspace{4ex}
  \end{subfigure}
  \begin{subfigure}[b]{0.5\linewidth}
    \centering
    \includegraphics[width=0.75\linewidth]{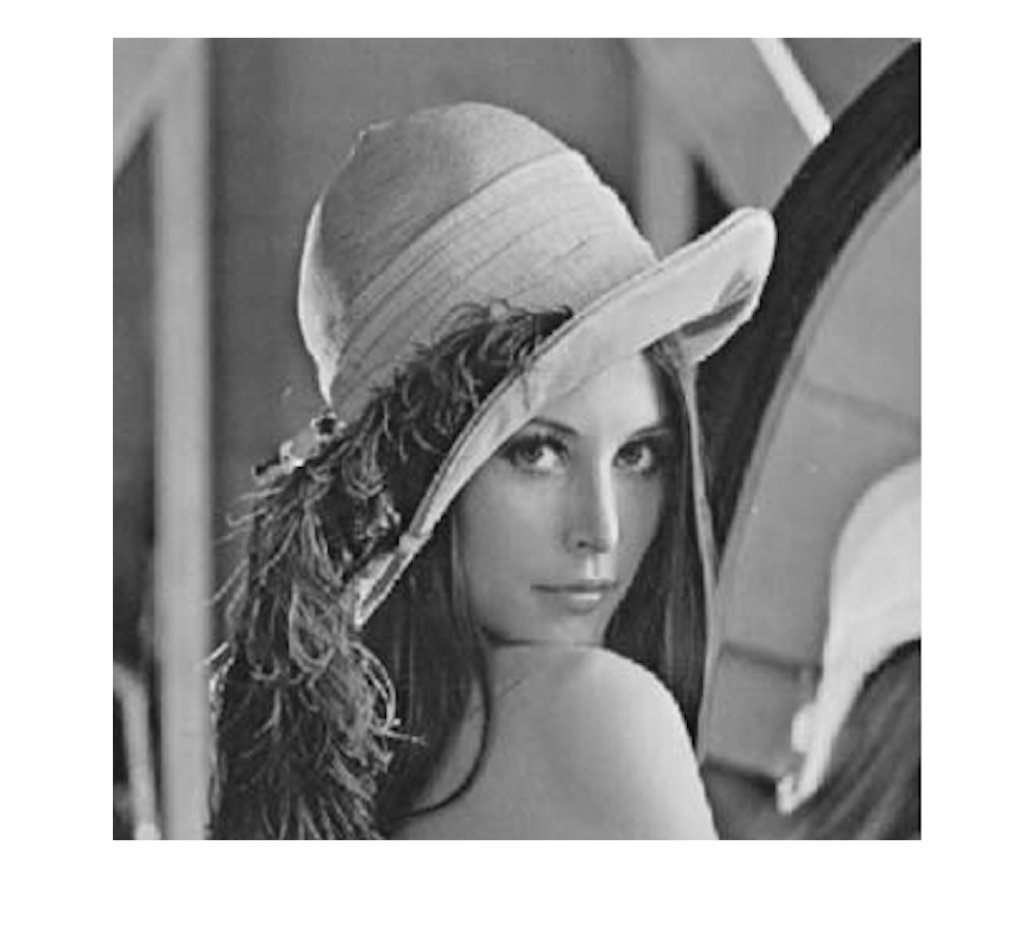} 
    \caption{Original lena image with 256 gray scale values} 
    \label{ori:image_b} 
    \vspace{4ex}
  \end{subfigure} 
\end{figure}

\begin{figure}[H] 
  \begin{subfigure}[b]{0.33\linewidth}
    \centering
    \includegraphics[width=0.75\linewidth]{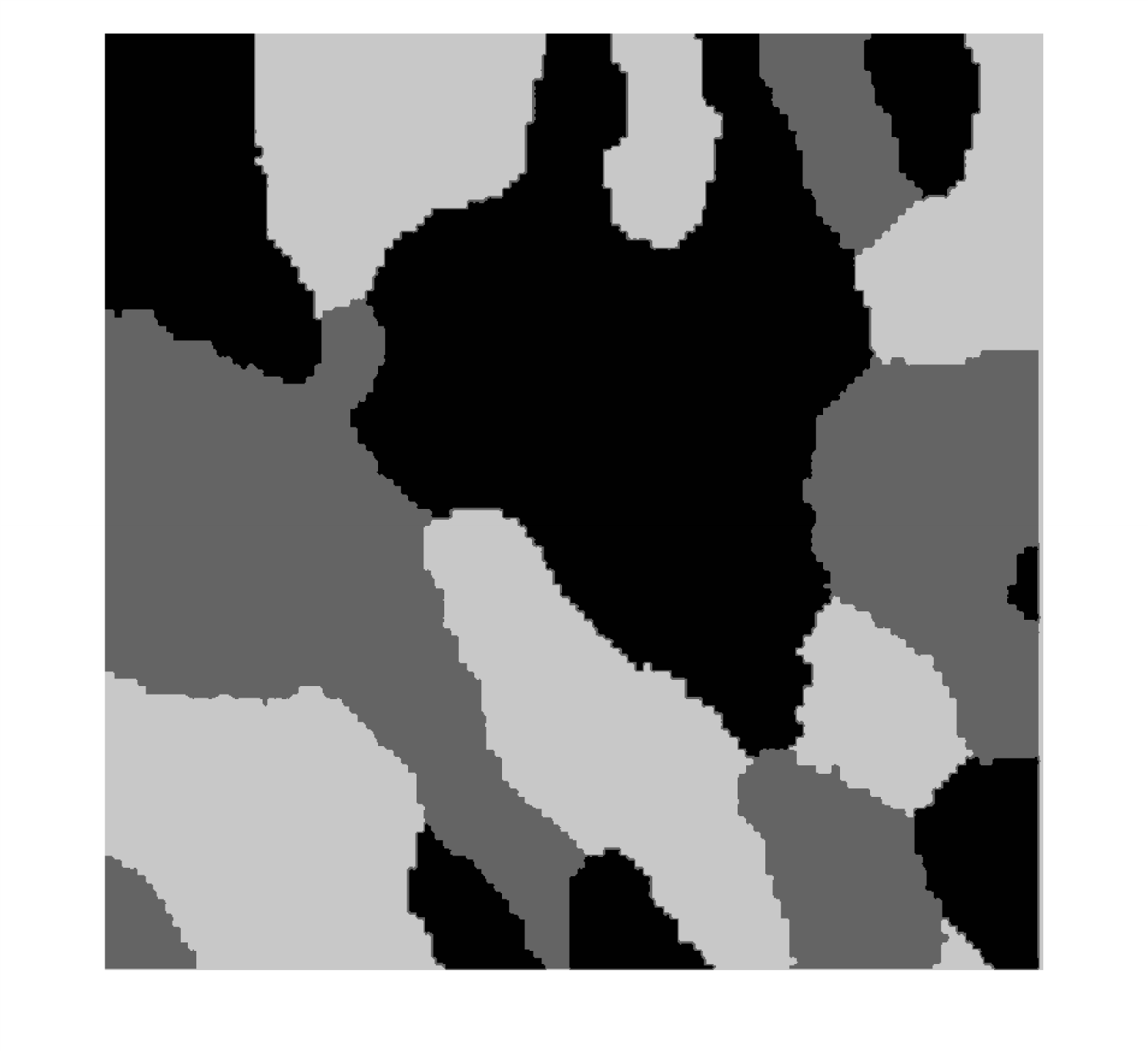} 
    \caption{EM algorithm} 
    \label{em_a:poiss} 
    \vspace{4ex}
  \end{subfigure}
  \begin{subfigure}[b]{0.33\linewidth}
    \centering
    \includegraphics[width=0.75\linewidth]{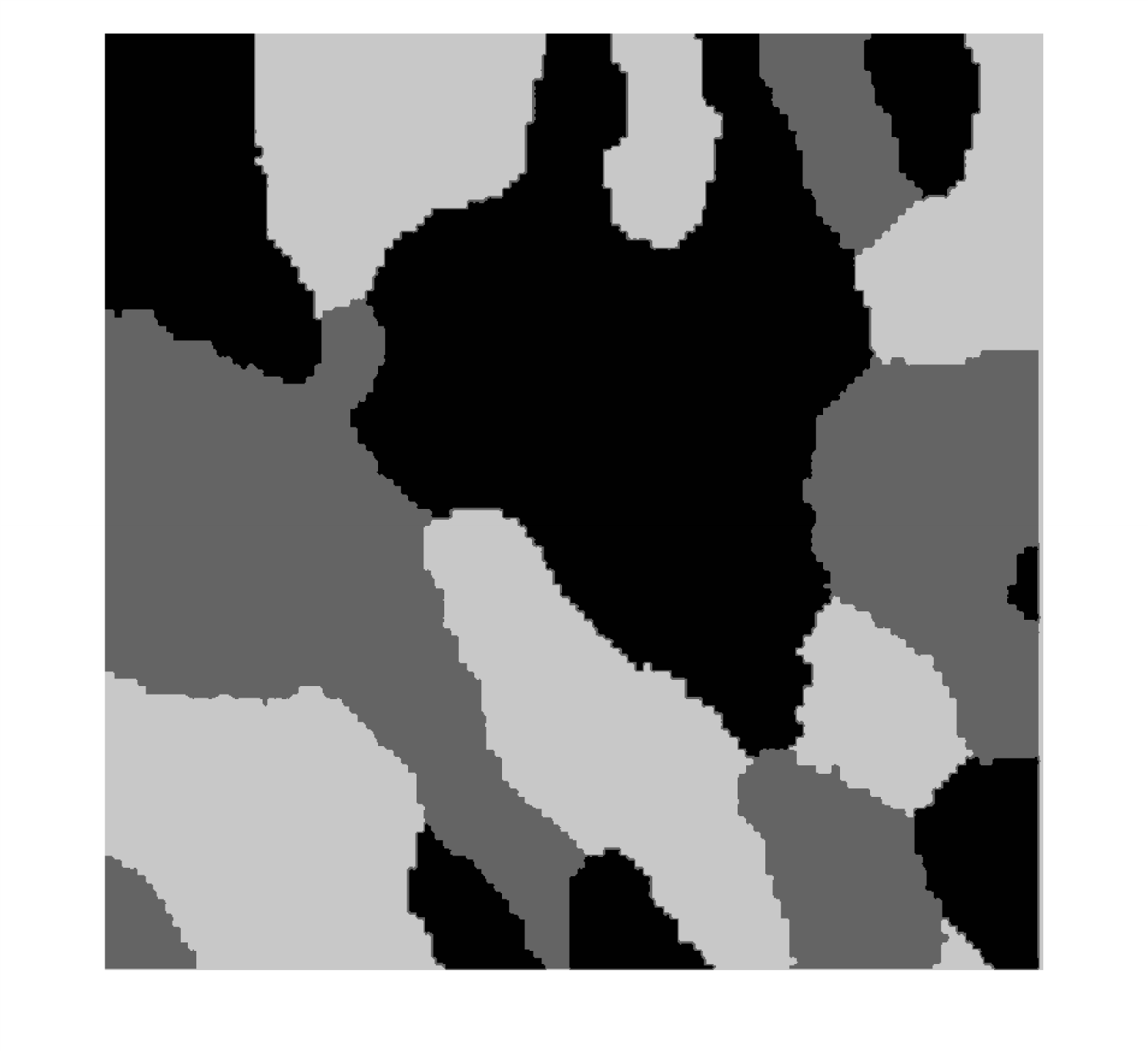} 
    \caption{HMIX algorithm} 
    \label{hmix_a:poiss} 
    \vspace{4ex}
  \end{subfigure} 
  \begin{subfigure}[b]{0.33\linewidth}
    \centering
    \includegraphics[width=0.75\linewidth]{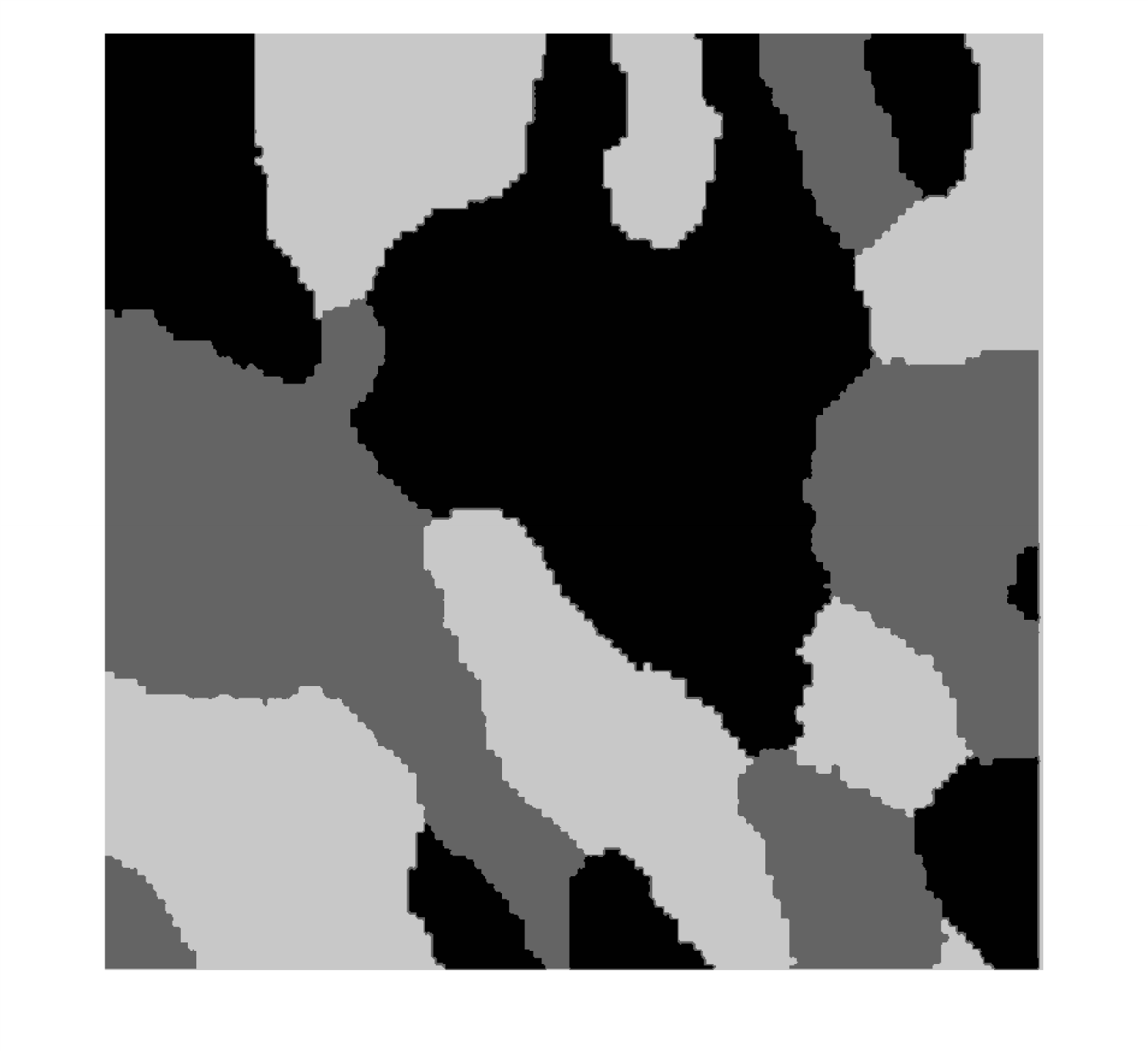} 
    \caption{vNEDMIX algorithm} 
    \label{vNEDmix_a:poiss} 
    \vspace{4ex}
  \end{subfigure}
  \caption{Three-class phantom image segamentation $K=3$}
  \label{phan-image}
\end{figure}

\begin{figure}[H] 
  \begin{subfigure}[b]{0.33\linewidth}
    \centering
    \includegraphics[width=0.75\linewidth]{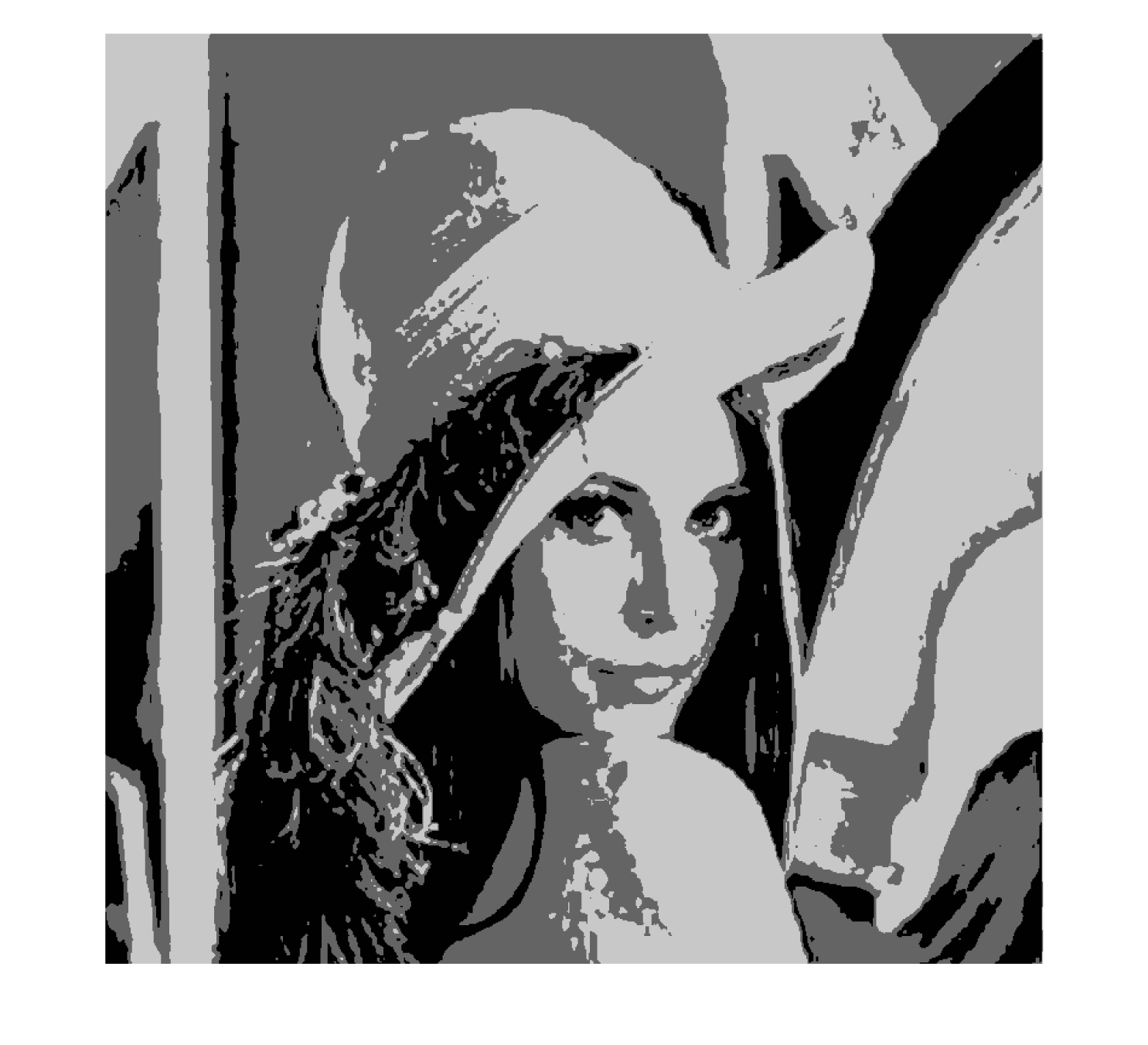} 
    \caption{EM algorithm} 
    \label{em_c:poiss} 
    \vspace{4ex}
  \end{subfigure}
  \begin{subfigure}[b]{0.33\linewidth}
    \centering
    \includegraphics[width=0.75\linewidth]{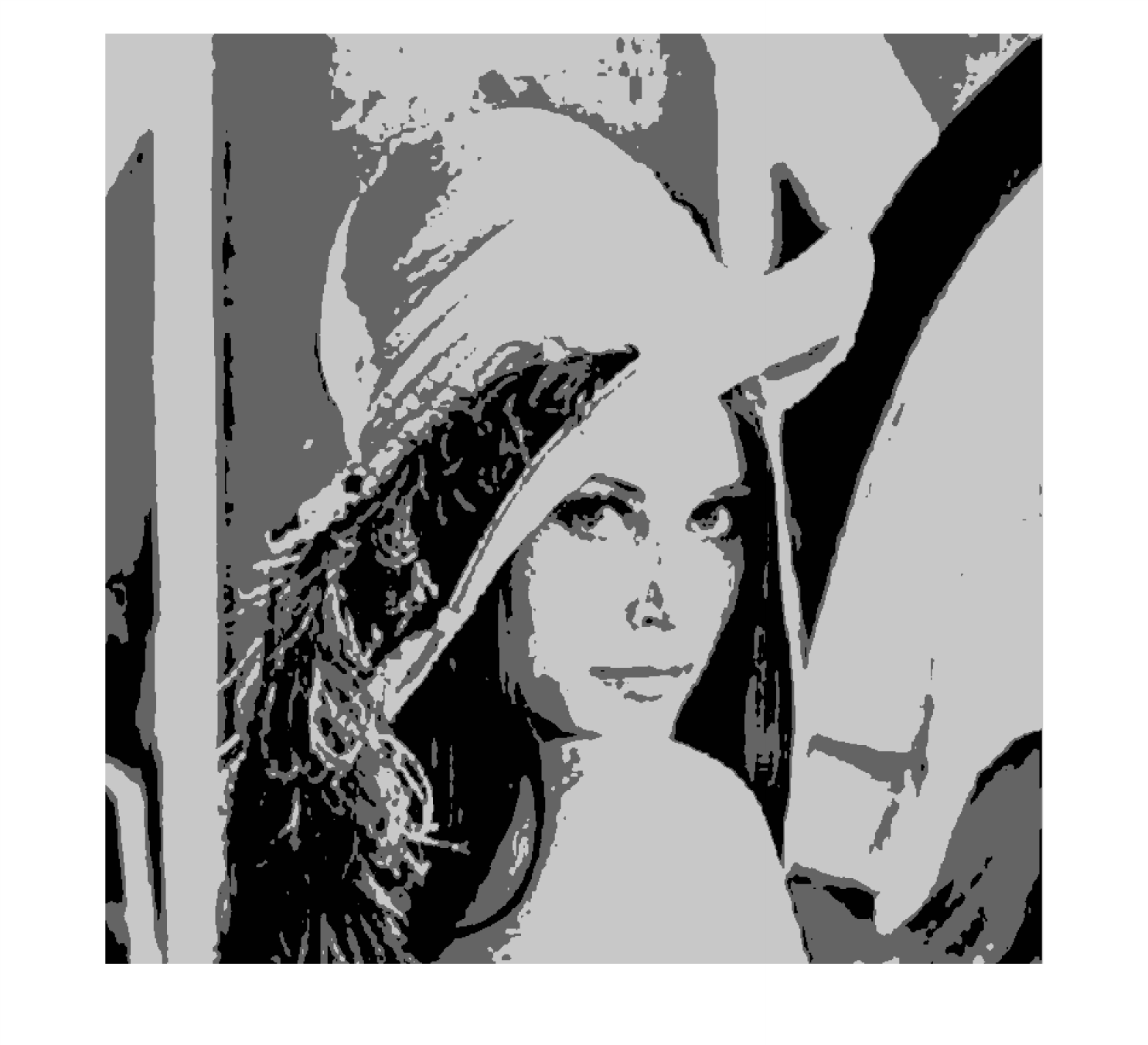} 
    \caption{HMIX algorithm} 
    \label{hmix_c:poiss} 
    \vspace{4ex}
  \end{subfigure} 
  \begin{subfigure}[b]{0.33\linewidth}
    \centering
    \includegraphics[width=0.75\linewidth]{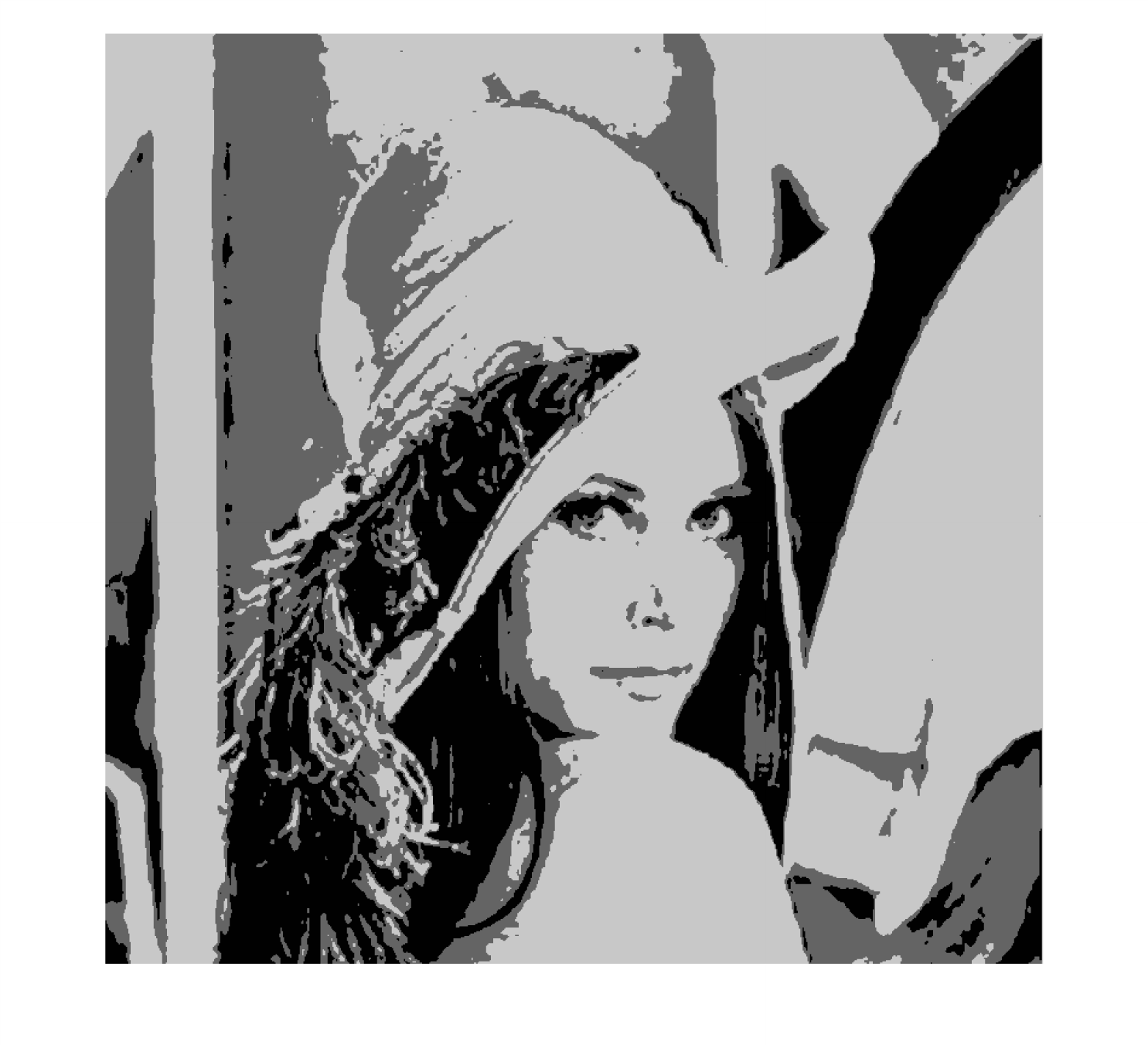} 
    \caption{vNEDMIX algorithm} 
    \label{vNEDmix_c:poiss} 
    \vspace{4ex}
  \end{subfigure}
  \caption{Lena Image Segamentation $K=3$}
  \label{Lena-image}
\end{figure}

The following tables are additional analysis for Phontom image and Lena image.

\begin{table}[H]
\centering
\scriptsize
\caption{Comparison of Different Model Selection Criteria (Phantom and Lena)}
\label{chap5:model-selection-supp}
\resizebox{\textwidth}{!}{%
\begin{tabular}{@{}l l c c c c c@{}}
\toprule
Image & Criterion & \# of components = 1 & 2 & 3 & 4 & 5 \\
\midrule
\multirow{3}{*}{Phantom}
& BIC    & $2.327\times 10^7$ & $1.021\times 10^7$ & $\mathbf{5.584}\times \mathbf{10}^{\mathbf{6}}$ & $5.584\times 10^6$ & $5.584\times 10^{6}$ \\
& HIC    & 1.353              & 0.998              & \textbf{0.804}                                   & 0.804               & 0.804               \\
& vNEDIC & 0.897              & 0.817              & \textbf{0.768}                                   & 0.768               & 0.768               \\
\midrule
\multirow{3}{*}{Lena}
& BIC    & $1.599\times 10^7$ & $8.497\times 10^6$ & $6.965\times 10^6$ & $6.399\times 10^6$ & $\mathbf{6.290}\times \mathbf{10}^{\mathbf{6}}$ \\
& HIC    & 0.689              & 0.396              & 0.192              & 0.062              & \textbf{0.022} \\
& vNEDIC & 0.648              & 0.546              & 0.479              & 0.407              & \textbf{0.382} \\
\bottomrule
\end{tabular}%
}
\end{table}

\begin{table}[H]
\centering
\scriptsize
\caption{Parameter Point Estimates for Phantom Image}
\resizebox{\textwidth}{!}{%
\begin{tabular}{@{}l l *{9}{c}@{}}
\toprule
    &        & $\hat{\pi}_1$ & $\hat{\lambda}_1$ & $\hat{\pi}_2$ & $\hat{\lambda}_2$ & $\hat{\pi}_3$ & $\hat{\lambda}_3$ & $\hat{\pi}_4$ & $\hat{\lambda}_4$ & $\hat{\lambda}_5$ \\
\midrule
\multirow{3}{*}{$K=2$}
    & EM        & 0.3665 & 28.99 & 0.6335 & 173.9 &        &       &       &       &       \\
    & HMIX     & 0.5835 & 28.79 & 0.4165 & 218.1 &        &       &       &       &       \\
    & vNEDMIX   & 0.6030 & 28.96 & 0.3970 & 218.7 &        &       &       &       &       \\
\midrule
\multirow{3}{*}{$K=3$}
    & EM        & 0.3636 & 28.65 & 0.3050 & 123.5 & 0.3314 & 219.3 &       &       &       \\
    & HMIX     & 0.4097 & 28.79 & 0.2987 & 124.0 & 0.2916 & 218.2 &       &       &       \\
    & vNEDMIX   & 0.3851 & 28.90 & 0.2857 & 124.2 & 0.3291 & 218.7 &       &       &       \\
\midrule
\multirow{3}{*}{$K=4$}
    & EM        & 0.1966 & 28.65 & 0.1670 & 28.65 & 0.3050 & 123.5 & 0.3314 & 219.3 &       \\
    & HMIX     & 0.2215 & 28.79 & 0.1882 & 28.79 & 0.2987 & 124.0 & 0.3609 & 218.2 &       \\
    & vNEDMIX   & 0.1991 & 28.90 & 0.1860 & 28.90 & 0.2857 & 124.2 & 0.3291 & 218.7 &       \\
\midrule
\multirow{3}{*}{$K=5$}
    & EM        & 0.1818 & 28.65 & 0.1818 & 28.65 & 0.3050 & 123.5 & 0.1657 & 219.3 & 219.3 \\
    & HMIX     & 0.2049 & 28.79 & 0.2049 & 28.79 & 0.2987 & 124.0 & 0.1458 & 218.2 & 218.2 \\
    & vNEDMIX   & 0.1926 & 28.90 & 0.1926 & 28.90 & 0.2857 & 124.2 & 0.1646 & 218.7 & 218.7 \\
\bottomrule
\end{tabular}%
}
\end{table}

\begin{table}[H]
\centering
\scriptsize
\caption{Parameter Point Estimates for Lena Image}
\resizebox{\textwidth}{!}{%
\begin{tabular}{@{}l l *{9}{c}@{}}
\toprule
    &        & $\hat{\pi}_1$ & $\hat{\lambda}_1$ & $\hat{\pi}_2$ & $\hat{\lambda}_2$ & $\hat{\pi}_3$ & $\hat{\lambda}_3$ & $\hat{\pi}_4$ & $\hat{\lambda}_4$ & $\hat{\lambda}_5$ \\
\midrule
\multirow{3}{*}{$K=2$}
    & EM        & 0.3788 & 67.34 & 0.6212 & 151.1 &        &       &       &       &       \\
    & HMIX     & 0.3747 & 91.62 & 0.6253 & 146.1 &        &       &       &       &       \\
    & vNEDMIX   & 0.3119 & 48.28 & 0.6881 & 141.1 &        &       &       &       &       \\
\midrule
\multirow{3}{*}{$K=3$}
    & EM        & 0.2284 & 52.03 & 0.3731 & 108.8 & 0.3985 & 167.9 &       &       &       \\
    & HMIX     & 0.2254 & 48.77 & 0.2977 & 96.55 & 0.4769 & 148.1 &       &       &       \\
    & vNEDMIX   & 0.2470 & 48.01 & 0.5180 & 137.2 & 0.2350 & 189.9 &       &       &       \\
\midrule
\multirow{3}{*}{$K=4$}
    & EM        & 0.1907 & 48.16 & 0.2331 & 90.85 & 0.3935 & 138.0 & 0.1827 & 189.8 &       \\
    & HMIX     & 0.1932 & 48.29 & 0.2370 & 92.36 & 0.3939 & 139.3 & 0.1760 & 191.0 &       \\
    & vNEDMIX   & 0.1945 & 47.74 & 0.2401 & 94.13 & 0.3950 & 141.3 & 0.1705 & 192.8 &       \\
\midrule
\multirow{3}{*}{$K=5$}
    & EM        & 0.1794 & 47.08 & 0.1832 & 84.43 & 0.2162 & 118.7 & 0.2832 & 150.8 & 196.3 \\
    & HMIX     & 0.1792 & 47.48 & 0.1802 & 85.45 & 0.2187 & 119.3 & 0.2839 & 150.9 & 196.3 \\
    & vNEDMIX   & 0.1782 & 47.46 & 0.1807 & 87.02 & 0.2213 & 120.3 & 0.2823 & 151.2 & 196.5 \\
\bottomrule
\end{tabular}%
}
\end{table}






\begin{figure}[H]
  \centering
  \captionsetup{font=footnotesize}
  \newcommand{\panelw}{0.48\linewidth} 

  \begin{subfigure}{\panelw}
    \centering
    \includegraphics[width=\linewidth,keepaspectratio]{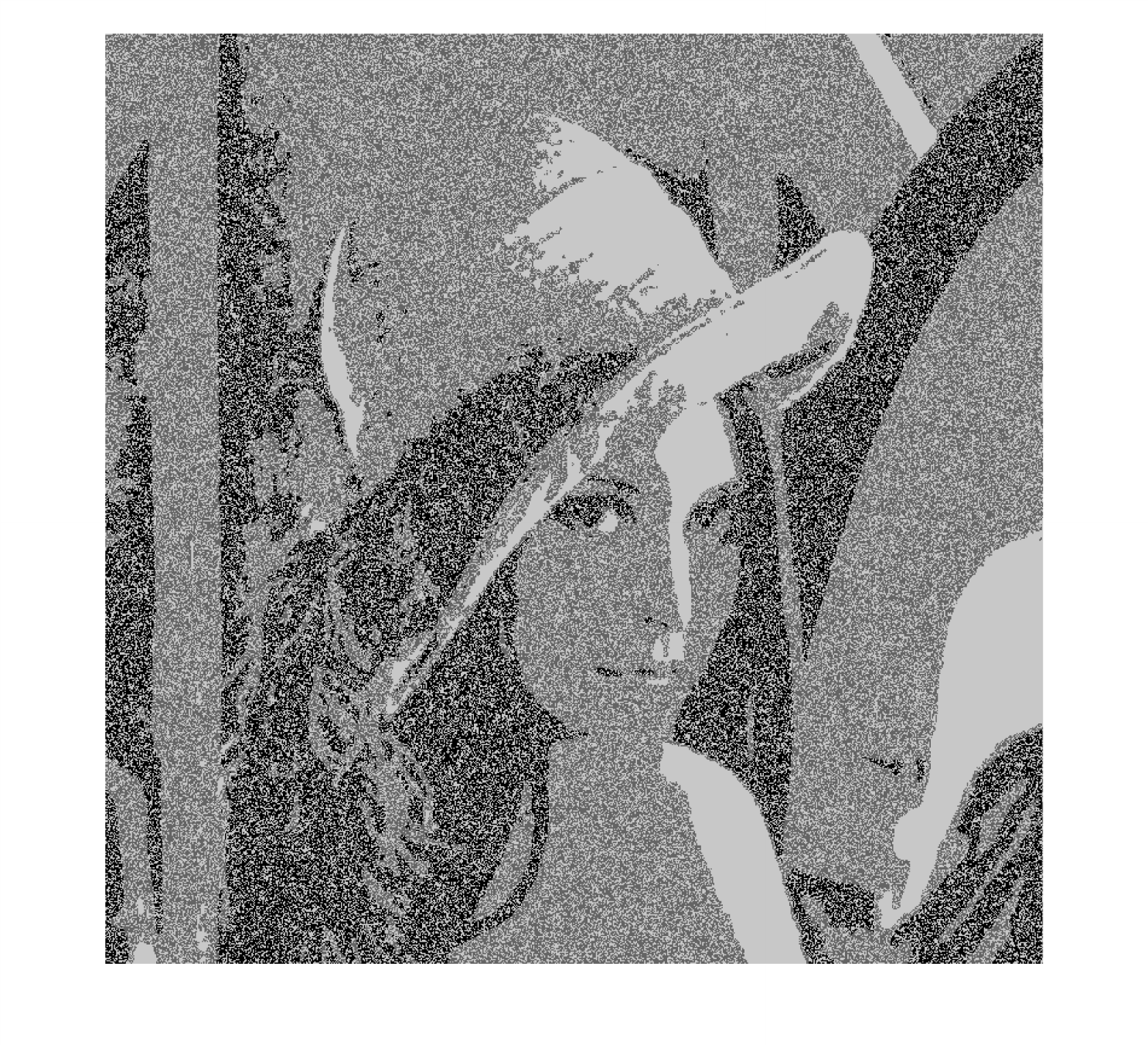}
    \caption{Recovered image from EM algorithm}
    \label{fig:em_recon_lena}
  \end{subfigure}\hfill
  \begin{subfigure}{\panelw}
    \centering
    \includegraphics[width=\linewidth,keepaspectratio]{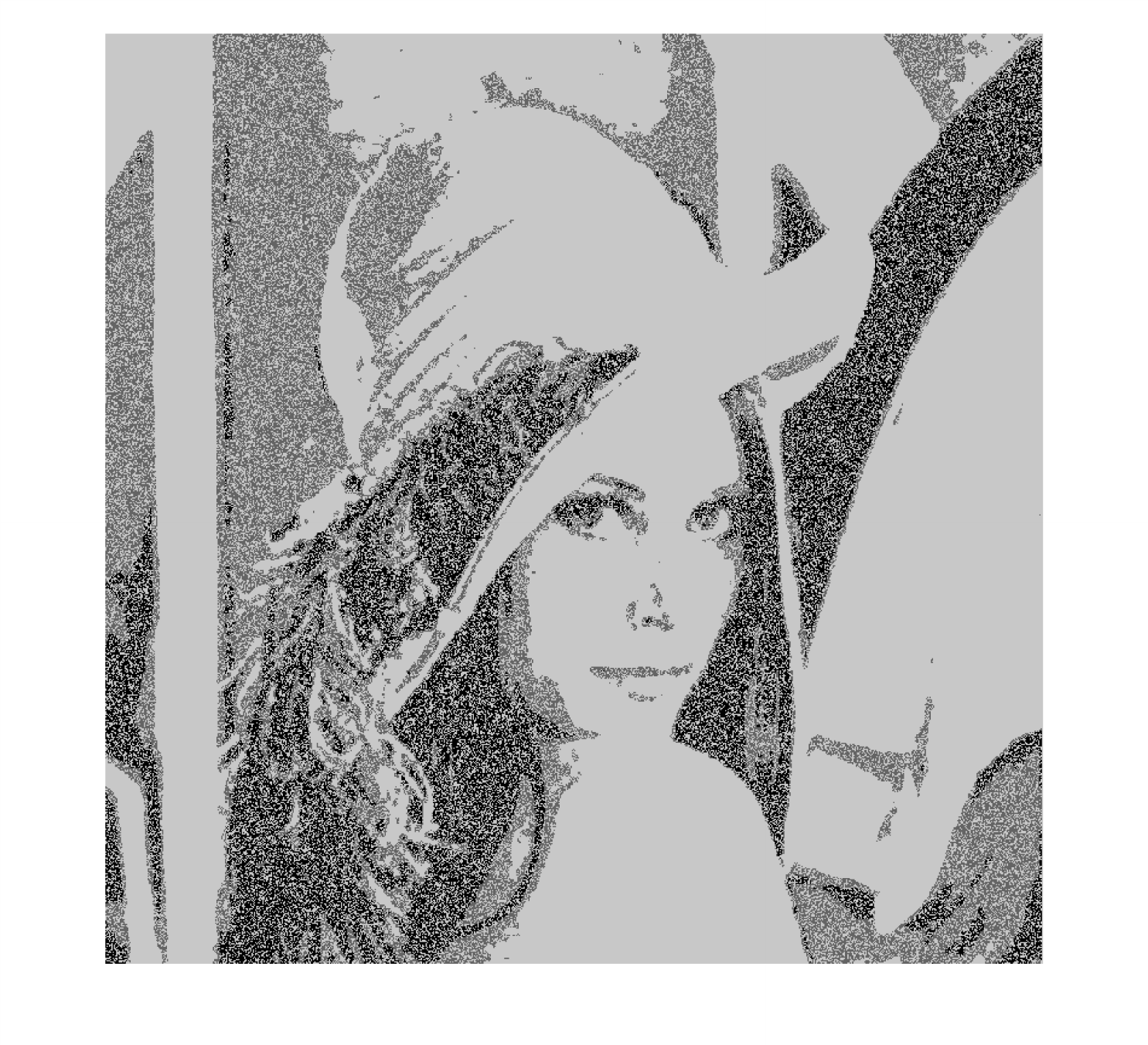}
    \caption{Recovered image from vNEDMIX algorithm}
    \label{fig:vNEDmix_recon_lena}
  \end{subfigure}

  \caption{Lena image reconstruction after adding $30\%$ outliers: EM vs.\ vNEDMIX.}
  \label{fig:lena-image-recons}
\end{figure}

\newpage 

\section*{K: Assumptions}\label{App-assumptions}

This section lays out all the necessary assumptions for the proofs of the main lemmas and
theorems in Sections 3 and 4.

\noindent\textbf{Connection with Z--estimation.}
The standing assumptions in this appendix are tailored so that the minimum--divergence
estimator and its DM--based approximations can be analyzed along the lines of the
classical Z--estimation scheme. In particular, for fixed $K$ we can write the MDE as the
solution of a system of estimating equations
\[
\Psi_n({\bm{\theta}})
  := \nabla_{\bm{\theta}} D_G(g_n,f_{\bm{\theta}})
  \;=\; 0,
\qquad
\Psi({\bm{\theta}};g)
  := \nabla_{\bm{\theta}} D_G(g,f_{\bm{\theta}}),
\]
with population target ${\bm{\theta}}^\star\in\arg\min_{\bm{\theta}} D_G(g,f_{\bm{\theta}})$ satisfying
$\Psi({\bm{\theta}}^\star;g)=0$. Assumptions \((\mathrm{C0})\)–\((\mathrm{C3})\), \((\mathrm{F1})\), \((\mathrm{K1})\)–\((\mathrm{K2})\) and
\((\mathrm{M1})\)–\((\mathrm{M8})\) guarantee: (i) identification and a local quadratic expansion of $D_G$
around ${\bm{\theta}}^\star$; (ii) a uniform law of large numbers for $\Psi_n({\bm{\theta}})$ on a neighborhood of
${\bm{\theta}}^\star$; and (iii) Fr\'echet differentiability and nonsingularity of the Jacobian
$H({\bm{\theta}}^\star):=\nabla_{\bm{\theta}}\Psi({\bm{\theta}}^\star;g)$. These are exactly the ingredients that
replace the abstract stochastic--equicontinuity conditions in the general Z--estimation
theory, and they allow us to derive consistency and $\sqrt{n}$--normality for the solution of
$\Psi_n({\bm{\theta}})=0$ in the usual way.

At the same time, the DM algorithm introduces two additional layers: the plug--in density
$g_n$ and the finite number of iterations $m_n$. The operator--level bounds in
Section~3(in particular the contraction and noisy--contraction results)
are used to control the optimization error
$\|{\bm{\theta}}_{m_n,n}-\hat{\bm{\theta}}_n\|$, where $\hat{\bm{\theta}}_n$ solves $\Psi_n({\bm{\theta}})=0$, while the
assumptions in this appendix ensure that $\hat{\bm{\theta}}_n$ itself behaves as a standard
Z--estimator. Appendix~N makes these connections explicit by working out the linearization
of $\Psi_n$ around ${\bm{\theta}}^\star$ and the associated Godambe information, and by showing
how the DM--specific error terms can be absorbed into the usual Z--estimation expansions
under our choice of $m_n$.

\noindent \textbf{Assumptions for Model Identifiability}

\noindent (\textbf{A0}) Assume that $G(0) = 0, G'(0) = 0$ and $G{''} (0) = 1$. Note that the above conditions imply $G'(\delta) <0$ when $\delta <0$ and $G'(\delta)>0$ when $\delta > 0.$

\noindent We now state two identifiability conditions below which enable us to deduce results concerning the mixture model with random effects.
\\
\noindent (\textbf{A1}) $ \lambda_1 \neq \lambda_2$ implies $s\left(y| \lambda_1\right) \neq s\left(y| \lambda_2\right)$ for all $y$.

\noindent (\textbf{A2}) $\bm{\phi}_1 \neq \bm{\phi}_2$ implies $r\left(\lambda, \bm{\phi}_1 \right) \neq r\left(\lambda, \bm{\phi}_2 \right)$ on a set of positive Lebesgue measure.

\noindent 
Note that (\textbf{D2}) is satisfied if the condition (\textbf{A1})-(\textbf{A2}) hold. 


\noindent \textbf{Assumptions for $Q_n(\cdot|\cdot)$, $\bm{\Theta}_{0,n}$, and $D_n(\cdot)$.}








\noindent \textbf{Parameter Space, Stationary Points, and Divergence Function Assumptions}

\noindent (\textbf{B1}) $Q(\bm{\theta}'|\bm{\theta})$ and $Q_n(\bm{\theta}'|\bm{\theta})$ are both continuous in $\bm{\theta}'$ and $\bm{\theta}$. Also, $Q(\bm{\theta}'|\bm{\theta})$ is continuously differentiable w.r.t. $\bm{\theta}'$. 

\noindent (\textbf{B2})\label{Level-comp} $\bm{\Theta}_{\bm{\theta}_0} = \{\bm{\theta}' \in \bm{\Theta}: D(\bm{\theta}') \leq  D(\bm{\theta}_{0})  \}$ and $\bm{\Theta}_{\bm{\theta}_{0,n}} = \{\bm{\theta}_n' \in \bm{\Theta}: D_n(\bm{\theta}_n') \leq  D_n(\bm{\theta}_{0,n})  \}$ are compact.

\noindent (\textbf{B3}) $D(\bm{\theta}') < \infty$ and $D_n(\bm{\theta}') < \infty$ for $\bm{\theta}'\in \bm{\Theta}$.

\noindent (\textbf{B4}) All stationary points in $S_{D}$ and $S_{n,D}$ are isolated; that is, for any $\bm{\theta} \in S_{D}$ and $\bm{\theta} \in S_{n,D}$, there exists a neighborhood of $\bm{\theta}$ which does not contain any other points of $S_{D}$ and $S_{n,D}$.

\noindent \textbf{Assumptions for Consistency of the DM Algorithm Sequence}


\noindent (\textbf{C0}) The density $s(\cdot|\lambda)$ is bounded, $r(\cdot; \bm{\phi}_k)$ is continuous in $\bm{\phi}_k$ for each $1\leq k \leq K.$ Also, assume that there exists a function $R(\cdot)$ such that $r(\lambda;\bm{\phi}_k) \leq R(\lambda)$ for all $1\leq k \leq K$ and $\int_{\Real}R(\lambda) d\lambda < \infty.$

\noindent (\textbf{C1}) The parameter space $\bm{\Theta}$ is compact.

\noindent (\textbf{C2}) $G(-1) + G'(\infty) < \infty,$ where 
$G'(\infty) = \lim_{u \to \infty} G(u)/u.$

\noindent (\textbf{C3}) $g_n(\cdot)$ is strongly consistent to $g(\cdot)$.

\noindent (\textbf{C4}) $c_n \to 0, n c_n \to \infty$ as $n \to \infty.$

\noindent \textbf{DM Algorithm Objective Function Assumptions}


{\textbf{(D1)} Closed graph / outer semicontinuity of the update map.}
The DM update correspondences $M_n:\bm \Theta \rightrightarrows \bm \Theta$ and $M: \bm \Theta \rightrightarrows \bm \Theta$
have closed graphs on $\Theta$; that is, if $\bm{\theta}^{(j)}\to\bar{\bm{\theta}}$ and $\bm{\eta}^{(j)}\in M_n(\bm{\theta}^{(j)})$ with
$\bm{\eta}^{(j)}\to\bar{\bm{\eta}}$, then $\bar{\bm{\eta}}\in M_n(\bar{\bm{\theta}})$ (and analogously for $M$). Equivalently,
$M_n$ and $M$ are outer semicontinuous on $\Theta$.
\emph{Single-valued specialization.} When, for each $\bm{\theta}$ in a compact neighborhood, the minimizer of
$Q_n(\cdot\mid\bm{\theta})$ (resp.\ $Q(\cdot\mid\bm{\theta})$) is unique, $M_n$ (resp.\ $M$) is single-valued and
\textbf{(D1)} reduces to ordinary continuity of the function $\bm{\theta}\mapsto M_n(\bm{\theta})$
(resp.\ $\bm{\theta}\mapsto M(\bm{\theta})$) by standard argmin-continuity/Berge’s maximum theorem.

\noindent (\textbf{D2}) For given $\bm{\theta} \in \bm{\Theta}$, both $Q(\cdot|\bm{\theta})$ and $Q_n(\cdot|\bm{\theta})$ have a unique global minimum.

Note that, under $\textbf{B}$, if (\textbf{D2}) holds, then (\textbf{D1}) is always satisfied. However, the converse may not hold. We can obtain the continuity of $M(\cdot)$  by replacing conditions (\textbf{B}') in the proposition with conditions (\textbf{B}).

\noindent \textbf{Kernel Assumptions for Central Limit Theorem}

\noindent (\textbf{K1}) The kernel $\mathcal{K}(\cdot)$ is symmetric at 0, has second moment, and twice continuously differentiable on a compact support $Supp(\mathcal{K})$.

\noindent (\textbf{K2}) The bandwidth $c_n$ satisfies $c_n \to 0$, $\sqrt{n}c_n \to \infty,$ and $\sqrt{n} c_n^4 \to 0$ as $n \to \infty$.

\noindent \textbf{Regularity Conditions for Central Limit Theorem}

Below we state the main assumptions that are required in the proof of the CLT. These assumptions describe the smoothness properties of the postulated density.

\noindent (\textbf{M1}) $A(\delta)$, $A'(\delta)$, $A'(\delta)(\delta+1)$ and $A{''}(\delta)(\delta+1)$ are bounded on $[-1, \infty)$.

\noindent (\textbf{M2}) 
The following conditions ensure $L^1$-continuity of both the Hessian of the density and the score–quadratic component under parameter convergence.
Specifically, as $n \to \infty$, $\bm{\phi}_n \to \bm{\theta}^*$,
\begin{align*}
 &\int_{\Real} |\nabla^2 f(y; \bm{\phi}_n) - \nabla^2 f(y; \bm{\theta}^*)|dy = o_p(1) \quad \text{and}& \\
 &\int_{\Real} |u(y; \bm{\phi}_n)u'(y; \bm{\phi}_n)f(y;\bm{\phi}_n) - u(y; \bm{\theta}^*)u'(y;\bm{\theta}^*)f(y;\bm{\theta}^*) |dy = o_p(1).&
\end{align*}
\noindent (\textbf{M3}) The matrix $I(\bm{\theta}^*)$ given by 
\begin{align*}
    I(\bm{\theta}^*) = \int_{\Real} u(y; \bm{\theta}^*)u'(y;\bm{\theta}^*) f(y;\bm{\theta}^*)dy
\end{align*}
 is finite.
 
\noindent (\textbf{M4}) 
An integrability constraint ensuring that the squared second derivative, scaled by the influence ratio, remains finite.
\begin{align*}
\int_{\Real} \frac{|u(y;\bm{\theta}^*)|}{f(y;\bm{\theta}^*)} f{''}^2(y;\bm{\theta}^*) dy < \infty.
\end{align*}
\noindent In addition, let $\{\alpha_n: n\geq 1 \}$ be a sequence diverging to infinity, and 
\begin{align*}
    &\int_\Real \int_{|y|\leq \alpha_n} \frac{|u(y;\bm{\theta}^*)|}{f(y;\bm{\theta}^*)} \tilde{H}_n(t, y) dydt < \infty,
\end{align*}    
\text{where} $\tilde{H}_n(t, y) =  \sup\{ H_n(t, y), t \in Supp(\mathcal{K}), |y|\leq \alpha_n \}$, $H_n(t, y) = \mathcal{K}(t) t^2 f{''}(\tilde{y}_n(t);\theta)$, and $\tilde{y}_n(t)$ is a point between $(\min(y-c_nt, y), \max(y-c_nt, y))$.

\noindent (\textbf{M5}) 
This condition controls the tail probability that the shifted observation $Y_1 - c_n t$ falls outside the truncation region.
\begin{align*}
    h_n \equiv n \sup_{t \in Supp(\mathcal{K})} \mathbb{P}\left( |Y_1 - c_nt| > \alpha_n \right) = O_p(1).
\end{align*}
\noindent (\textbf{M6}) 
The following condition ensures that the truncated mass of the score component becomes negligible under the normalization $1/(\sqrt{n}\,c_n).$
\begin{align*}
   \lim_{n\to \infty} \frac{1}{\sqrt{n} c_n} \int_{|y|\leq \alpha_n}u(y;\bm{\theta}^*)dy = 0.
\end{align*}
\noindent (\textbf{M7}) 
This condition requires that the ratio of the shifted density to the original density remains uniformly bounded over the truncation region.
\begin{align*}
    \sup_{|y| \leq \alpha_n} \sup_{t \in Supp(\mathcal{K})} \frac{f(y-c_nt;\bm{\theta}^*)}{f(y;\bm{\theta}^*)} = O_p(1).
\end{align*}
\noindent (\textbf{M8}) 
The following condition ensures that the squared $L^2$ distance between the shifted and unshifted score functions vanishes uniformly over t.
\begin{align*}
    \lim_{n\to \infty} \sup_{t \in Supp(\mathcal{K})}\int_{\Real}\left( u(y+c_nt;\bm{\theta}^*)- u(y;\bm{\theta}^*)\right)\left(u(y+c_nt; \bm{\theta}^*) - u(y;\bm{\theta}^*)\right)' f(y;\bm{\theta}^*) dy = 0.
\end{align*}

\noindent \textbf{Parametric Model Assumptions}

\noindent (\textbf{F1}) $f(\cdot;\bm{\theta})$ is twice continuously differentiable.

\noindent \textbf{Regularity Conditions for Breakdown Point Analysis}

\noindent (\textbf{O1}) $\int_\Real |g(y) - \eta_n(y) |dy \to 2$ as $n \to \infty$. 

\noindent (\textbf{O2}) $\int_\Real |f(y;\bm{\theta}) - \eta_n(y) |dy \to 2$ as $n \to \infty$ uniformly for $|\bm{\theta}| \leq c$ for every fixed $c$.

\noindent (\textbf{O3}) $\int_\Real |g(y) - f(y;\bm{\theta}_{\epsilon,n}^*) |dy \to 2 $ as $n \to \infty$ if $|\bm{\theta}_{\epsilon,n}^*| \to \infty$ as $n\to \infty$, where $\bm{\theta}_{\epsilon, n}^* = T(g_{\epsilon, n})$.

\newpage 

\newpage

\section*{L: Proofs}\label{all_pfs}

\subsection{Proofs of Results in Section 2 of the Main}

\begin{proof}[Detailed Proof of Lemma \ref{Inequality00}]
Let $U(t) = tG(-1+a/t)$, where $a \geq 0, t >0$. We first show that $U(\cdot)$ is a convex function on $(0, \infty)$. Since $G(\cdot)$ is thrice differentiable and convex, we take the second derivative with respect to $t$. Let $\delta = \left(-1 + a/t\right),$ 
so the second derivative of $U(\cdot)$ w.r.t. $t$ is 
\begin{align*}
\begin{split}
\frac{\partial^2 U}{\partial t^2} & =  \frac{\partial G(\delta)}{\partial \delta}\left(-\frac{a}{t^2}\right)  - \left(\frac{\partial^2 G(\delta)}{\partial \delta^2} \left(-\frac{a}{t^2}\right) \frac{a}{t} + \frac{\partial G(\delta)}{\partial \delta} \left( -\frac{a}{t^2} \right)    \right) \\
&=\frac{\partial^2 G(\delta)}{\partial \delta^2} \frac{a^2}{t^3} \geq 0
\end{split}
\end{align*}
as $a \geq 0$, $t >0$ and convexity of $G(\cdot)$. Hence $U(\cdot)$ is convex. Next we will show that for almost all $y$, $a \geq 0$,  $\bm{\theta}' \in \bm{\Theta}$, and every $\tilde{q}(\cdot|y)$,
\begin{align}\label{lem:ine:1}
f(y;\bm{\theta}')G\left( -1 +\frac{a}{f(y;\bm{\theta}')} \right) \leq \int_{\mathcal{Z}} f(y;\bm{\theta}') w(Z|Y;\bm{\theta}') G\left( -1 + \frac{a \tilde{q}(z|y)}{f(y;\bm{\theta}') w(Z|Y;\bm{\theta}')} \right) dz.
\end{align}
First note that the left hand side (LHS) of (\ref{lem:ine:1}) is $U(f(y;\bm{\theta}'))$; for the right hand side (RHS) of (\ref{lem:ine:1}), by Jensen's inequality, we have 
\begin{align*}
U(f(y;\bm{\theta}')) &= U\left(\bm{E}_{\tilde{q}} \left[ \frac{f(y;\bm{\theta}') w(Z|y;\bm{\theta}')}{\tilde{q}(Z|y)}  \right] \right) \\
&\leq \bm{E}_{\tilde{q}} \left[ U\left(  \frac{f(y;\bm{\theta}') w(Z|y;\bm{\theta}')}{\tilde{q}(Z|y)}     \right)   \right] \\
&= \int_{\mathcal{Z}} f(y;\bm{\theta}') w(Z|y;\bm{\theta}') G\left( -1 + \frac{a \tilde{q}(z|y)}{f(y;\bm{\theta}') w(Z|y;\bm{\theta}')} \right) dz.
\end{align*}
Hence (\ref{lem:ine:1}) holds. Now let $a = g(y)$, and taking another integral w.r.t. $\mathcal{Y}$, we have $\mathscr{D}(q, \bm{\theta}') \geq D(\bm{\theta}').$ In addition, the equality holds if and only if $\tilde{q}(z|y) = w(z|y;\bm{\theta}')$ for almost every $y$. The proof is complete.
\end{proof}

\setcounter{equation}{0}\renewcommand{\theequation}{A.\arabic{equation}}


\begin{proof}[Proof of Lemma~\textup{(finite stationary points at a fixed value)}]
Fix $D_n^\star\in\Real$ and set
\[
S\;:=\;\{\bm{\theta}\in\bm{\Theta}:\ \nabla D_n(\bm{\theta})=0,\; D_n(\bm{\theta})=D_n^\star\}.
\]
By \textup{(B1)} $D_n$ and $\nabla D_n$ are continuous; hence $S=\nabla D_n^{-1}(\{0\})\cap D_n^{-1}(\{D_n^\star\})$
is closed. Let $c>D_n^\star$; by \textup{(B2)} the sublevel set $\{\bm{\theta}:\ D_n(\bm{\theta})\le c\}$ is compact, so $S$ is compact.
By \textup{(B4)} every stationary point is isolated: for each $\theta\in S$ there exists $r_\theta>0$ such that
$B(\bm{\theta},r_{\bm{\theta}})\cap\{\nabla D_n=0\}=\{\bm{\theta}\}$.
Then $\{B(\bm{\theta},r_{\bm{\theta}}/2):\bm{\theta}\in S\}$ is an open cover of the compact set $S$, so it
admits a finite subcover, forcing $S$ to be finite (each ball contains exactly one point of $S$ by isolation).
\end{proof}

\subsection*{Proof of Proposition~3}

\begin{proof}[Proof of Proposition~3 (limit--set structure: convergence or finite cycle)]
Let $\{\bm{\theta}_{m,n}\}_{m\ge0}$ be a sample--level DM sequence with $\bm{\theta}_{m+1,n}\in M_n(\bm{\theta}_{m,n})$.
By Proposition~1 (majorization descent), $D_n(\bm{\theta}_{m+1,n})\le D_n(\bm{\theta}_{m,n})$ for all $m$, and
$D_n(\bm{\theta}_{m,n})\downarrow D_n^\star$ for some $D_n^\star\in\Real$. 
Denote the $\omega$--limit set $\Omega:=\{\,\bar{\bm{\theta}}:\ \exists\,m_j\uparrow\infty,\ \bm{\theta}_{m_j,n}\to\bm{\bar\theta},\}$.
By \textup{(B2)}, all iterates lie in a compact sublevel set, so $\Omega\neq\varnothing$.
By continuity of $D_n$ and the convergence of $D_n(\bm{\theta}_{m,n})$, every $\bar{\bm{\theta}}\in\Omega$ satisfies $D_n(\bar{\bm{\theta}})=D_n^\star$.
By Proposition~1, every $\bar{\bm{\theta}}\in\Omega$ is stationary for $D_n$; together with the previous lemma,
$\Omega$ is a \emph{finite} set.

Next, we show $\Omega$ is $M_n$--invariant and that $M_n$ acts as a permutation on $\Omega$.
Assume \textup{(D1)} (continuity of $M_n$). If $\bm{\theta}_{m_j,n}\to\bar{\bm{\theta}}$, then
$\bm{\theta}_{m_j+1,n}=M_n(\bm{\theta}_{m_j,n})\to M_n(\bar{\bm{\theta}})$; hence $M_n(\bar{\bm{\theta}})\in\Omega$, so $M_n(\Omega)\subseteq\Omega$.
Conversely, given any $z\in\Omega$, pick $\{m_j\}$ with $\bm{\theta}_{m_j+1,n}\to z$; compactness yields a subsequence
$\bm{\theta}_{m_{j_k},n}\to y\in\Omega$ with $M_n(y)=z$ by continuity, so $M_n:\Omega\to\Omega$ is surjective,
hence bijective on the finite set $\Omega$. Therefore $M_n|\_\Omega$ is a \emph{permutation}, so $\Omega$
is a finite union of cycles. Because $\{\bm{\theta}_{m,n}\}$ has $\Omega$ as its \emph{entire} limit set, only one such
cycle can occur: otherwise, invariance would force the orbit to eventually remain in a neighborhood of a single
cycle and it could not accumulate on the others. Hence $\Omega=\{\bm{\theta}^\ast_{1,n},\dots,\bm{\theta}^\ast_{t,n}\}$ with
\[
M_n(\bm{\theta}^\ast_{i,n})= \bm \theta^\ast_{i+1,n}\quad(i=1,\dots,t-1),\qquad M_n(\bm{\theta}^\ast_{t,n})=\bm{\theta}^\ast_{1,n},
\]
and each $\bm{\theta}^\ast_{i,n}$ is stationary with $D_n(\bm{\theta}^\ast_{i,n})=D_n^\star$.

It remains to establish (iii): convergence of the parallel subsequences $\{\bm{\theta}_{tm+i,n}\}_{m\ge0}$ to $\bm{\theta}^\ast_{i,n}$.
By continuity of $M_n$ and isolation \textup{(B4)}, fix disjoint open balls $U_i:=B(\bm{\theta}^\ast_{i,n},r)$ so small that
$M_n\big(\overline{U_i}\big)\subset U_{i+1}$ for all $i$ (indices modulo $t$); this is possible since
$M_n(\bm{\theta}^\ast_{i,n})=\bm{\theta}^\ast_{i+1,n}$ and $M_n$ is continuous. Because $\mathrm{dist}(\bm{\theta}_{m,n},\Omega)\to0$ and the $U_i$
cover a neighborhood of $\Omega$, there exists $m_0$ after which every iterate belongs to $\bigcup_i U_i$ and, moreover,
membership cycles deterministically: $\bm{\theta}_{m,n}\in U_i \Rightarrow \bm{\theta}_{m+1,n}\in U_{i+1}$.
In particular, for each $i$ there exists $m_i\ge m_0$ with $\bm{\theta}_{m_i,n}\in U_i$, and then by $t$-step invariance
$M_n^t(\overline{U_i})\subset U_i$ we have $\bm{\theta}_{m_i+kt,n}\in U_i$ for all $k\ge0$. Any limit point of
$\{\bm{\theta}_{m_i+kt,n}\}_{k\ge0}$ lies in $\Omega\cap U_i=\{\bm{\theta}^\ast_{i,n}\}$, so the subsequence converges:
$\bm{\theta}_{tm+i,n}\to\bm{\theta}^\ast_{i,n}$ as claimed.
\end{proof}

\subsection*{Proof of Proposition~4}

\begin{proof}[Proof of Proposition~4 (no cycles under uniqueness; strict descent)]
Assume \textup{(B)} and \textup{(D2)}. Uniqueness in \textup{(D2)} implies that for each $\bm{\theta}$, the map
$\bm{\theta}'\mapsto Q_n(\bm{\theta}'\mid\bm{\theta})$ has a \emph{unique} minimizer, so the update is single--valued and continuous
(standard argmin continuity), i.e., \textup{(D1)} holds.

\emph{Strict descent away from stationarity.} If $\bm{\theta}_{m+1,n}=M_n(\bm{\theta}_{m,n})\neq \bm{\theta}_{m,n}$, then by uniqueness
\[
Q_n(\bm{\theta}_{m+1,n}\mid\bm{\theta}_{m,n})\;<\;Q_n(\bm{\theta}_{m,n}\mid\bm{\theta}_{m,n})=D_n(\bm{\theta}_{m,n}),
\]
and by majorization $D_n(\bm{\theta}_{m+1,n})\le Q_n(\bm{\theta}_{m+1,n}\mid\bm{\theta}_{m,n})$, hence
$D_n(\bm{\theta}_{m+1,n})<D_n(\bm{\theta}_{m,n})$.

\emph{Excluding cycles.} If $\Omega$ contained a cycle of length $t\ge2$, then along that cycle all points are stationary
and have objective value $D_n^\star$, so once the orbit enters a small invariant neighborhood of the cycle, the updates
cannot be strictly decreasing (the values along one pass around the cycle remain equal), contradicting the strict
descent above unless the update lands exactly at a fixed point. Therefore $\Omega$ is a singleton $\{\bm{\theta}^\ast_n\}$,
and the sequence converges: $\bm{\theta}_{m,n}\to \bm{\theta}^\ast_n$.

\emph{Fixed point identity.} By continuity of $M_n$ and the update relation $\bm{\theta}_{m+1,n}=M_n(\bm{\theta}_{m,n})$,
passing to the limit yields $\bm{\theta}^\ast_n=M_n(\bm{\theta}^\ast_n)$. The final monotonicity statement follows directly from the
strict descent shown above unless $\bm{\theta}_{m,n}$ has already reached $\bm{\theta}^\ast_n$.
\end{proof}

\subsection*{Proof of Theorem \ref{THM:Population:contraction}}

\begin{proof}[Proof of Theorem \ref{THM:Population:contraction}]
Since both $M(\bm{\theta})$ and $\bm{\theta}^*$ are in $\bm{\Theta}$, we may apply condition 
\begin{align}\label{First-order-1}
\langle \nabla Q(\bm{\theta}^*|\bm{\theta}^*), \bm{\theta}' - \bm{\theta}^*  \rangle \geq 0 \quad \text{for all} \quad \bm{\theta}' \in \bm{\Theta}.
\end{align}
with $\bm{\theta}'=M(\bm{\theta})$:
\begin{align*}
\langle \nabla Q(\bm{\theta}^*|\bm{\theta}^*), M(\bm{\theta}) - \bm{\theta}^*  \rangle \geq 0 \quad \text{for all} \quad \bm{\theta}' \in \bm{\Theta},
\end{align*}
and apply condition 
\begin{align}\label{First-order-2}
\langle  \nabla Q(M(\bm{\theta})|\bm{\theta}), \bm{\theta}' - M(\bm{\theta}) \rangle \geq 0 \quad \text{for all} \quad \bm{\theta}' \in \bm{\Theta}.
\end{align}
with $\bm{\theta}'=\bm{\theta}^*$:
\begin{align*}
\langle \nabla Q(M(\bm{\theta})|\bm{\theta}) , \bm{\theta}^* - M(\bm{\theta}) \rangle \geq 0 \quad \text{for all} \quad \bm{\theta}' \in \bm{\Theta}.
\end{align*}
Adding the above two inequalities and then perform some algebra yields the condition
\begin{small}
    \begin{align}\label{Thm1:Ineq1}
\langle \nabla Q(M(\bm{\theta})|\bm{\theta}^*) - \nabla Q(\bm{\theta}^*|\bm{\theta}^*) , M(\bm{\theta}) - \bm{\theta}^* \rangle \leq
\langle \nabla Q(M(\bm{\theta})|\bm{\theta}^*) - \nabla Q(M(\bm{\theta})|\bm{\theta}) , M(\bm{\theta}) - \bm{\theta}^* \rangle.
\end{align}
\end{small}
Now the $\lambda$-strong convexity condition implies that the left-hand side is lower bounded as (by letting $\bm{\theta}_1 = M(\bm{\theta}), \bm{\theta}_2 = \bm{\theta}^*$ and $\bm{\theta}_1 = \bm{\theta}^, \bm{\theta}_2 = M(\bm{\theta})$ respectively)
\begin{align}\label{Thm:1-Lower-bound}
\langle \nabla Q(M(\bm{\theta})|\bm{\theta}^*) - \nabla Q(\bm{\theta}^*|\bm{\theta}^*) , M(\bm{\theta}) - \bm{\theta}^* \rangle \geq  \lambda ||\bm{\theta}^* - M(\bm{\theta})||_2^2 .
\end{align}
On the other hand, the FOS($\gamma$) condition together with the Cauchy-Schwarz inequality implies that the right-hand side upper bounded as 
\begin{align}\label{Thm:1-Upper-bound}
\langle \nabla Q(M(\bm{\theta})|\bm{\theta}^*) - \nabla Q(M(\bm{\theta})|\bm{\theta}), M(\bm{\theta}) - \bm{\theta}^* \rangle \leq \gamma ||\bm{\theta}^* - M(\bm{\theta})||_2 ||\bm{\theta} - \bm{\theta}^* ||_2,
\end{align}
Combining inequalities (\ref{Thm:1-Lower-bound}) and (\ref{Thm:1-Upper-bound}) with original bound (\ref{Thm1:Ineq1}) yields
\begin{align}
\lambda \vert \vert \bm{\theta}^* - M(\bm{\theta}) \vert \vert_2 \leq \gamma \vert \vert \bm{\theta}^* - M(\bm{\theta}) \vert \vert_2 \vert \vert \bm{\theta}- \bm{\theta}^* \vert \vert_2,
\end{align}
and canceling terms completes the proof.
\end{proof}

\subsection*{Proof of Theorem \ref{thm:noisy}}

\begin{proof}[Proof of Theorem \ref{thm:noisy}]
From (\ref{pop-sam-lower}), for any fixed $\bm{\theta} \in \mathbb{B}_2(r';\bm{\theta}^*)$,
\begin{align*}
 ||  M_n(\bm{\theta}) - M(\bm{\theta}) ||_2 \leq \epsilon_M(n, \rho)
\end{align*}
with probability at least $1-\rho$. 
It suffices to show that 
\begin{align}\label{THM-2:pf}
|| \bm{\theta}_{m+1, n} - \bm{\theta}^* ||_2 \leq \kappa ||\bm{\theta}_{m, n} -\bm{\theta}^*  ||_2 + \epsilon_M(n, \rho).
\end{align}
Indeed, when this bound holds, we may iterate it to show that
\begin{align*}
||\bm{\theta}_{m, n} -\bm{\theta}^*  ||_2 &\leq \nonumber \kappa ||\bm{\theta}_{m-1, n} -\bm{\theta}^*  ||_2 + \epsilon_M(n, \rho)   \\
&\leq \nonumber \kappa\left\{ \kappa ||\bm{\theta}_{m-2, n} -\bm{\theta}^* ||_2 +  \epsilon_M(n, \rho) \right\} + \epsilon_M(n, \rho)\\
&\leq \nonumber \kappa^m ||\bm{\theta}_{0, n} -\bm{\theta}^* ||_2 + \left\{\sum_{k=0}^{m-1} \kappa^k   \right\} \epsilon_M(n, \rho) \\
&\leq  \kappa^m ||\bm{\theta}_{0, n} -\bm{\theta}^* ||_2 + \frac{1}{1-\kappa} \epsilon_M(n, \rho).
\end{align*}
It remains to prove (\ref{THM-2:pf}), and we do so by induction. Beginning with $m=0$, we have 
\begin{align*}
||\bm{\theta}_{1, n} - \bm{\theta}^*  ||_2 = ||M_n(\bm{\theta}_{0,n}) - \bm{\theta}^*  ||_2 &\leq  ||M(\bm{\theta}_{0,n}) - \bm{\theta}^*  ||_2 + ||M_n(\bm{\theta}_{0,n}) - M(\bm{\theta}_{0,n}) ||_2 \\
&\leq \kappa ||\bm{\theta}_{0, n} - \bm{\theta}^*   ||_2 + \epsilon_M(n, \rho) \\
&\leq  \kappa r' + (1-\kappa)r' = r'.
\end{align*}
In the induction from $m\mapsto m+1$, suppose that $||\bm{\theta}_{m,n} - \bm{\theta}^*||_2 \leq r$, and the bound (\ref{THM-2:pf}) at iteration $m$. Now at iteration $m+1$, 
\begin{align*}
||\bm{\theta}_{m+1,n} - \bm{\theta}^* ||_2 = ||M_n(\bm{\theta}_{m,n}) - \bm{\theta}^*  ||_2 &\leq  ||M(\bm{\theta}_{m,n}) - \bm{\theta}^*  ||_2 + ||M_n(\bm{\theta}_{m,n}) - M(\bm{\theta}_{m,n})||_2 \\
&\leq \kappa || \bm{\theta}_{m,n} -\bm{\theta}^* ||_2 + \epsilon_M(n, \rho) \\
&\leq \kappa r' + (1-\kappa)r'  = r'.
\end{align*}
Thus $||\bm{\theta}_{m+1, n} - \bm{\theta}^* ||_2 \leq r'$, and this completes the proof.
\end{proof}
\subsection*{Proof of Theorem 3}

\noindent \textbf{(i) Consistency.} Under \textbf{(C0)–(C3)} the DM surrogate decomposition
$Q_n({\bm{\theta}}'\!\mid{\bm{\theta}})=D_n({\bm{\theta}}')+H_G({\bm{\theta}}'\!\mid{\bm{\theta}})$ with $H_G\ge 0$ yields
$D_n({\bm{\theta}}_{m+1,n})\le D_n({\bm{\theta}}_{m,n})$ (monotone descent).
By sublevel compactness the sequence $\{{\bm{\theta}}_{m,n}\}$ is tight for each $n$; any limit point is stationary.
Uniform convergence and identification (C0)–(C3) imply the stationary set concentrates at ${\bm{\theta}}^\star$
as $n\to\infty$; hence $\lim_{n}\lim_{m}{\bm{\theta}}_{m,n}={\bm{\theta}}^\star$ a.s.
If $m_n\to\infty$ then ${\bm{\theta}}_{m_n,n}\stackrel{p}{\to}{\bm{\theta}}^\star$.

\noindent \textbf{(ii) $\sqrt n$–normality for truncated iterates.}
Under Theorem 1
there exists $\kappa\in(0,1)$ such that $M$ is $\kappa$–contractive on $B_2(r;{\bm{\theta}}^\star)$.
Let $\widehat{{\bm{\theta}}}_n:=\arg\min_{\bm{\theta}} D_n({\bm{\theta}})$ and note that
$\widehat{{\bm{\theta}}}_n\in M_n(\widehat{{\bm{\theta}}}_n)$ (unique by the fixed-order regularity).
Under \textbf{(F1)}, \textbf{(C1)}–\textbf{(C2)}, \textbf{(K1)}–\textbf{(K2)}, \textbf{(M1)}–\textbf{(M8)} the
fixed-order MDE satisfies
\[
\sqrt n\,(\widehat{{\bm{\theta}}}_n-{\bm{\theta}}^\star)\ \Rightarrow\ \mathcal N\!\big(0,\,I({\bm{\theta}}^\star)^{-1}\big).
\]
Moreover, the sample update map $M_n$ inherits the contraction locally with factor
$\kappa_n\le \kappa+o_p(1)$ (same FOS/strong-convexity argument with $g$ replaced by $g_n$ and uniform LLN).
Hence for any initialization ${\bm{\theta}}_{0,n}\in B_2(r;{\bm{\theta}}^\star)$,
\[
\|{\bm{\theta}}_{t,n}-\widehat{{\bm{\theta}}}_n\|_2\ \le\ \kappa_n^{\,t}\,\|{\bm{\theta}}_{0,n}-\widehat{{\bm{\theta}}}_n\|_2
\qquad\text{(w.p.\ $\to1$).}
\]
Choose
\(
m_n\ \ge\ \big\lceil (\tfrac12\log n + c_0)/|\log\kappa| \big\rceil
\)
so that $\sqrt n\,\kappa^{\,m_n}\to0$ and therefore $\sqrt n\,\kappa_n^{\,m_n}\to0$.
Decompose
\[
\sqrt n\,({\bm{\theta}}_{m_n,n}-{\bm{\theta}}^\star)
= \underbrace{\sqrt n\,({\bm{\theta}}_{m_n,n}-\widehat{{\bm{\theta}}}_n)}_{\to 0\ \text{in prob.}}
+ \underbrace{\sqrt n\,(\widehat{{\bm{\theta}}}_n-{\bm{\theta}}^\star)}_{\Rightarrow\ \mathcal N(0,I({\bm{\theta}}^\star)^{-1})}.
\]
The first term is $o_p(1)$ by the geometric bound above, the second term is the fixed-order CLT.
Slutsky’s lemma gives
\(
\sqrt n\,({\bm{\theta}}_{m_n,n}-{\bm{\theta}}^\star)\Rightarrow \mathcal N(0,I({\bm{\theta}}^\star)^{-1}).
\)
Finally, the displayed choice implies \(m_n = \lceil (\tfrac12 \log n)/|\log\kappa| \rceil + O(1)\).
\qed

\subsection*{Proof of Corollary 3 (Finite-step Godambe CLT)}
Let $\widehat{{\bm{\theta}}}_n$ solve $\Psi_n({\bm{\theta}}):=\nabla_{\bm{\theta}} D_G(g_n,f_{\bm{\theta}})=0$ and let
${\bm{\theta}}_{t+1,n}\in M_n({\bm{\theta}}_{t,n})$ with ${\bm{\theta}}_{0,n}\in B_2(r;{\bm{\theta}}^\star)$.
Under \textup{(G1)–(G4)} at ${\bm{\theta}}^\dagger:=\arg\min_{\bm{\theta}} D_G(g,f_{\bm{\theta}})$,
the Z–estimation CLT gives
\(
\sqrt n\,(\widehat{{\bm{\theta}}}_n-{\bm{\theta}}^\dagger)\Rightarrow \mathcal N(0,H^{-1} V H^{-1})
\)
with
\(
H:=\nabla_{\bm{\theta}}\Psi({\bm{\theta}}^\dagger;g),\
V:=Var_g\!\big[A'(\tfrac{g}{f_{{\bm{\theta}}^\dagger}}-1)s_{{\bm{\theta}}^\dagger}(Y)\big].
\)
By Theorem 1 and the same uniformity argument as in Theorem 1 (ii),
the sample update map contracts locally with factor $\kappa_n\le \kappa+o_p(1)$, hence
\(
\|{\bm{\theta}}_{m_n,n}-\widehat{{\bm{\theta}}}_n\|_2 \le \kappa_n^{\,m_n}\|{\bm{\theta}}_{0,n}-\widehat{{\bm{\theta}}}_n\|_2.
\)
With the threshold $\sqrt n\,\kappa^{m_n}\to0$, we get
\(
\sqrt n\,\|{\bm{\theta}}_{m_n,n}-\widehat{{\bm{\theta}}}_n\|_2 \to 0
\)
in probability. Decompose
\[
\sqrt n({\bm{\theta}}_{m_n,n}-{\bm{\theta}}^\dagger)
= \sqrt n({\bm{\theta}}_{m_n,n}-\widehat{{\bm{\theta}}}_n) + \sqrt n(\widehat{{\bm{\theta}}}_n-{\bm{\theta}}^\dagger),
\]
apply Slutsky’s lemma, and conclude
\(
\sqrt n({\bm{\theta}}_{m_n,n}-{\bm{\theta}}^\dagger)\Rightarrow \mathcal N(0,H^{-1} V H^{-1}).
\)
At the model ($g=f_{{\bm{\theta}}^\star}$) and with $A'(0)=1$, $H=V=I({\bm{\theta}}^\star)$, so the covariance is $I({\bm{\theta}}^\star)^{-1}$.
\qed

\subsection*{Proof of Corollary 4 (Finite-step (Godambe-Wilks)}
Under the assumptions of Corollary 4 and $\sqrt n\,\kappa^{m_n}\to0$,
a quadratic expansion of $D_G(g_n,f_{\bm{\theta}})$ at $\widehat{{\bm{\theta}}}_n$ yields
\[
2n\{D_G(g_n,f_{{\bm{\theta}}^\dagger})-D_G(g_n,f_{{\bm{\theta}}_{m_n,n}})\}
= n({\bm{\theta}}_{m_n,n}-{\bm{\theta}}^\dagger)^\top H\,({\bm{\theta}}_{m_n,n}-{\bm{\theta}}^\dagger)\ +\ o_p(1),
\]
since \(n\|{\bm{\theta}}_{m_n,n}-\widehat{{\bm{\theta}}}_n\|_2^2=(\sqrt n\,\kappa^{m_n})^2=o_p(1)\).
Let \(J:=H^{-1/2} V H^{-1/2}\) with eigenvalues \(\{\lambda_j\}_{j=1}^{p(K_0)}\).
By Corollary 3,
\(n({\bm{\theta}}_{m_n,n}-{\bm{\theta}}^\dagger)^\top H({\bm{\theta}}_{m_n,n}-{\bm{\theta}}^\dagger)\Rightarrow
\sum_{j=1}^{p(K_0)}\lambda_j\,\chi^2_{1,j}\).
If $g=f_{{\bm{\theta}}^\star}$ and $A'(0)=1$, then $H=V=I({\bm{\theta}}^\star)$ and the limit is $\chi^2_{p(K_0)}$.
\qed

\subsection*{Proof of Corollary 5 (Population contraction under contamination)}
\label{supp:cor-pop-contr-eps}
Define $\Psi_\epsilon(\bm\theta;g):=\nabla_{\bm\theta} D_G(g,f_{\bm\theta})
=\int f_{\bm\theta}(y)\,s_{\bm\theta}(y)\,B\!\big(g(y)/f_{\bm\theta}(y)\big)\,dy$,
$s_{\bm\theta}=\nabla_{\bm\theta}\log f_{\bm\theta}$, $B(u)=G(u)-uG'(u)$.
By \textbf{(F1)}, \textbf{(C1)}–\textbf{(C2)}, \textbf{(K1)}–\textbf{(K2)}, \textbf{(M1)}–\textbf{(M8)},
the maps $\bm\theta\mapsto \nabla_{\bm\theta}\Psi_\epsilon(\bm\theta;g_\epsilon)$ and the Gâteaux derivative 
$h\mapsto \partial_g\Psi_\epsilon(\bm\theta;g_\epsilon)[h]$ are continuous in $(\bm\theta,g_\epsilon)$ on 
$B_2(r;\bm\theta^\star)\times\{g_\epsilon:\epsilon\in[0,\epsilon_0]\}$ (see Appendix N, Gâteaux lemma).
At $\epsilon=0$, Theorem~1 provides FOS$(\gamma)$ and local curvature $\lambda>0$
on $B_2(r_0;\bm\theta^\star)$ with $\kappa=\gamma/\lambda<1$. By continuity in $\epsilon$,
there exist $r\in(0,r_0]$ and $\bar\kappa\in(0,1)$ such that, for all $\epsilon\in[0,\epsilon_0]$,
FOS$(\gamma_\epsilon)$ and local $\lambda_\epsilon$–strong convexity of $D_G(g_\epsilon,\cdot)$ hold on 
$B_2(r;\bm\theta^\dagger_\epsilon)$ and $\gamma_\epsilon/\lambda_\epsilon\le\bar\kappa<1$.
Let $\bm\theta_+\in M_\epsilon(\bm\theta)$ be any minimizer of $Q_\epsilon(\cdot\mid\bm\theta)$.
By strong convexity of $q_\epsilon(\cdot):=Q_\epsilon(\cdot\mid\bm\theta^\dagger_\epsilon)$ and its optimality at $\bm\theta^\dagger_\epsilon$,
\[
\lambda_\epsilon\,\|\bm\theta_+-\bm\theta^\dagger_\epsilon\|_2^2
\ \le\ \big\langle \nabla_{\bm\theta'} q_\epsilon(\bm\theta_+)-\nabla_{\bm\theta'} q_\epsilon(\bm\theta^\dagger_\epsilon),\,\bm\theta_+-\bm\theta^\dagger_\epsilon\big\rangle.
\]
Add and subtract $\nabla_{\bm\theta'}Q_\epsilon(\bm\theta_+\mid\bm\theta)$ and use the optimality condition
$\langle \nabla_{\bm\theta'} Q_\epsilon(\bm\theta_+\mid\bm\theta),\,\bm\theta^\dagger_\epsilon-\bm\theta_+\rangle\ge 0$ to get
\[
\lambda_\epsilon\,\|\bm\theta_+-\bm\theta^\dagger_\epsilon\|
\ \le\ \big\|\nabla_{\bm\theta'}Q_\epsilon(\bm\theta_+\mid\bm\theta)-\nabla_{\bm\theta'}Q_\epsilon(\bm\theta_+\mid\bm\theta^\dagger_\epsilon)\big\|
\ \le\ \gamma_\epsilon\,\|\bm\theta-\bm\theta^\dagger_\epsilon\|,
\]
where the last step is FOS$(\gamma_\epsilon)$. Hence
$\|\bm\theta_+-\bm\theta^\dagger_\epsilon\|\le(\gamma_\epsilon/\lambda_\epsilon)\|\bm\theta-\bm\theta^\dagger_\epsilon\|
\le \bar\kappa\,\|\bm\theta-\bm\theta^\dagger_\epsilon\|$, proving the contraction.
\qed

\subsection*{Proof of Corollary~6 (Noisy contraction and opt-to-stat under contamination)}
\label{supp:proof-noisy-contr-eps}
Let $\bm\theta_{t+1,\epsilon,n}\in M_{\epsilon,n}(\bm\theta_{t,\epsilon,n})$ with
$\bm\theta_{t,\epsilon,n}\in B_2(r;\bm\theta^\dagger_\epsilon)$.
Choose $\bm\eta_t\in M_\epsilon(\bm\theta_{t,\epsilon,n})$ so that
$\|\bm\theta_{t+1,\epsilon,n}-\bm\eta_t\|_2 \le \mathcal M_{\mathrm{unif}}(n,r;\epsilon)$.
By Corollary~5 (population contraction under contamination),
\[
\|\bm\eta_t-\bm\theta^\dagger_\epsilon\|_2\;\le\;\bar\kappa\,\|\bm\theta_{t,\epsilon,n}-\bm\theta^\dagger_\epsilon\|_2.
\]
Therefore
\[
\|\bm\theta_{t+1,\epsilon,n}-\bm\theta^\dagger_\epsilon\|_2
\ \le\ \|\bm\theta_{t+1,\epsilon,n}-\bm\eta_t\|_2+\|\bm\eta_t-\bm\theta^\dagger_\epsilon\|_2
\ \le\ \bar\kappa\,\|\bm\theta_{t,\epsilon,n}-\bm\theta^\dagger_\epsilon\|_2+\mathcal M_{\mathrm{unif}}(n,r;\epsilon),
\]
establishing the contaminated noisy recursion.

\medskip
\noindent\textit{Bounding $\mathcal M_{\mathrm{unif}}(n,r;\epsilon)$.}
Fix $\bm\theta\in B_2(r;\bm\theta^\dagger_\epsilon)$.
By strong convexity of $Q_\epsilon(\cdot\mid\bm\theta)$ on the ball with modulus $\lambda_\epsilon$
and the optimality conditions for $M_\epsilon(\bm\theta)$ and $M_{\epsilon,n}(\bm\theta)$,
the minimizer map is locally inverse–Lipschitz:
\begin{align}\label{eq:inv-Lip}
\operatorname{dist}\!\big(M_{\epsilon,n}(\bm\theta),M_\epsilon(\bm\theta)\big)
\ \le\ \frac{1}{\lambda_\epsilon}\,
\sup_{\bm\eta\in B_2(r;\bm\theta^\dagger_\epsilon)}
\big\|\nabla_{\bm\theta'}Q_{\epsilon,n}(\bm\eta\mid\bm\theta)-\nabla_{\bm\theta'}Q_{\epsilon}(\bm\eta\mid\bm\theta)\big\|.
\end{align}
Using the score representation and the mean–value inequality for the RAF derivative,
for any $\bm\eta\in B_2(r;\bm\theta^\dagger_\epsilon)$,
\begin{align*}
\big\|\nabla_{\bm\theta'}Q_{\epsilon,n}(\bm\eta\mid\bm\theta)-\nabla_{\bm\theta'}Q_{\epsilon}(\bm\eta\mid\bm\theta)\big\|
 &\le  A'_{\max}\,
\sup_{\bm\zeta\in B_2(r;\bm\theta^\dagger_\epsilon)}\|s_{\bm\zeta}\|_{\mathcal H^\ast}\|g_{\epsilon,n}-g_\epsilon\|_{\mathcal H}\\
&\le C_{fos}\,\Env(K)\,A'_{\max}\,\|g_{\epsilon,n}-g_\epsilon\|_{\mathcal H},
\end{align*}
uniformly on the ball (the last inequality is the same envelope bound used in the main after Theorem~2, adapted to $g_\epsilon$).
Combining with \eqref{eq:inv-Lip} and taking the sup over $\bm\theta$ gives the contaminated analogue of (3.7):
\begin{equation}\label{eq:Munif-eps}
\mathcal M_{\mathrm{unif}}(n,r;\epsilon)
\ \lesssim\ \frac{C_{fos}\,A'_{\max}\,\Env(K)}{\lambda_\epsilon}\;\|g_{\epsilon,n}-g_\epsilon\|_{\mathcal H}.
\end{equation}
Under the assumption $\|g_{\epsilon,n}-g_\epsilon\|_{\mathcal H}=o_p(n^{-1/2})$ and the uniform boundedness of
$\lambda_\epsilon^{-1}$ for $\epsilon\in[0,\epsilon_0]$, we obtain
$\mathcal M_{\mathrm{unif}}(n,r;\epsilon)=o_p(n^{-1/2})$.

\medskip
\noindent\textit{Conclusion (opt-to-stat).}
Iterating the noisy recursion and using Lemma~\ref{lem:M-noisyrec} yields
\[
\|\bm\theta^{(m_n)}_{\epsilon,n}-\widehat{\bm\theta}_{\epsilon,n}\|_2
\ \le\ \bar\kappa^{\,m_n}\,\|\bm\theta^{(0)}_{\epsilon,n}-\widehat{\bm\theta}_{\epsilon,n}\|_2
+ \frac{1-\bar\kappa^{\,m_n}}{1-\bar\kappa}\,\mathcal M_{\mathrm{unif}}(n,r;\epsilon).
\]
Multiplying by $\sqrt{n}$ and using $\sqrt n\,\bar\kappa^{\,m_n}\to0$ together with
$\sqrt n\,\mathcal M_{\mathrm{unif}}(n,r;\epsilon)=o_p(1)$ gives
$\sqrt n\,\|\bm\theta^{(m_n)}_{\epsilon,n}-\widehat{\bm\theta}_{\epsilon,n}\|_2\to^p 0$.
\qed

\subsection*{S.R.1 \; Robust contaminated noisy contraction for bounded–RAF class}
\label{supp:robust-bRAF}
\begin{cor}[Robust class (NED/vNED)]
Assume $A'_{\max}:=\sup_{\delta\ge-1}|A'(\delta)|<\infty$ and $A'$ is nonincreasing on $[0,\infty)$
(e.g.,  NED, vNED). Then there exist $r>0$ and $\bar\kappa\in(0,1)$ such that, for every
$\varepsilon\in[0,\varepsilon_0]$ and any selection
$\bm\theta_{t+1,\varepsilon,n}\in M_{\varepsilon,n}(\bm\theta_{t,\varepsilon,n})$ with
$\bm\theta_{t,\varepsilon,n}\in B_2(r;\bm\theta^\dagger_\varepsilon)$,
\[
\|\bm\theta_{t+1,\varepsilon,n}-\bm\theta^\dagger_\varepsilon\|_2
\ \le\ \bar\kappa\,\|\bm\theta_{t,\varepsilon,n}-\bm\theta^\dagger_\varepsilon\|_2
\ +\ \mathcal M_{\mathrm{unif}}(n,r;\varepsilon),
\]
and, if $\|g_{\varepsilon,n}-g_\varepsilon\|_{\mathcal H}=o_p(n^{-1/2})$, then
$\mathcal M_{\mathrm{unif}}(n,r;\varepsilon)=o_p(n^{-1/2})$; consequently any
$m_n\ge \lceil(\tfrac12\log n+c_0)/|\log\bar\kappa|\rceil$ satisfies
$\sqrt n\,\|\bm\theta^{(m_n)}_{\varepsilon,n}-\widehat{\bm\theta}_{\varepsilon,n}\|\to^p 0$.
\end{cor}

\begin{proof}
(1) \emph{Inverse–Lipschitz for the argmin map.} Strong convexity of $Q_\varepsilon(\cdot\mid\bm\theta)$ with modulus
$\lambda_\varepsilon$ on $B_2(r;\bm\theta^\dagger_\varepsilon)$ yields, for every $\bm\theta$ in the ball,
\begin{equation}\label{eq:inv-lip-eps}
\operatorname{dist}\!\big(M_{\varepsilon,n}(\bm\theta),\,M_\varepsilon(\bm\theta)\big)
\ \le\ \frac{1}{\lambda_\varepsilon}\,
\sup_{\bm\eta\in B_2(r;\bm\theta^\dagger_\varepsilon)}
\big\|\nabla_{\bm\theta'}Q_{\varepsilon,n}(\bm\eta\mid\bm\theta)-\nabla_{\bm\theta'}Q_{\varepsilon}(\bm\eta\mid\bm\theta)\big\|.
\end{equation}
(2) \emph{Gradient perturbation with bounded RAF derivative.} Using the score representation and the mean-value form,
\[
\sup_{\bm\eta\in B_2(r;\bm\theta^\dagger_\varepsilon)}
\big\|\nabla_{\bm\theta'}Q_{\varepsilon,n}(\bm\eta\mid\bm\theta)-\nabla_{\bm\theta'}Q_{\varepsilon}(\bm\eta\mid\bm\theta)\big\|
\ \le\ C_{\mathrm{fos}}\;\!A'_{\max}\,\mathrm{Env}(K)\,\|g_{\varepsilon,n}-g_\varepsilon\|_{\mathcal H},
\]
uniformly on the ball (bounded, nonincreasing $A'$ allows the envelope to be absorbed).
Combining with \eqref{eq:inv-lip-eps} and taking sup over $\bm\theta$ yields
\begin{equation}\label{eq:munif-eps}
\mathcal M_{\mathrm{unif}}(n,r;\varepsilon)\ \lesssim\
\frac{C_{\mathrm{fos}}\,A'_{\max}\,\mathrm{Env}(K)}{\lambda_\varepsilon}\,
\|g_{\varepsilon,n}-g_\varepsilon\|_{\mathcal H}.
\end{equation}
Under the plug-in rate and a uniform curvature lower bound
$\inf_{\varepsilon\le\varepsilon_0}\lambda_\varepsilon>0$, we get $\mathcal M_{\mathrm{unif}}=o_p(n^{-1/2})$.
Noisy recursion + Lemma (noisy linear recurrence) imply the opt-to-stat bound.
\end{proof}

\subsection*{S.R.2 \; KL requires a local floor: failure mode without it}
\label{supp:kl-failure}
\begin{prop}[KL needs a local density/score floor]
Let $G(u)=u\log u-u+1$ (KL). Assume only the standing smoothness on $B_2(r;\bm\theta^\dagger_\varepsilon)$
but no local lower density/score bound. Then there exist contamination sequences $\{\eta_n\}$ with
$\|g_{\varepsilon,n}-g_\varepsilon\|_{\mathcal H}=O_p(n^{-1/2})$ for which
$\mathcal M_{\mathrm{unif}}(n,r;\varepsilon)\not=o_p(n^{-1/2})$; hence the opt-to-stat step may fail.
If, in addition, a local density/score floor holds (e.g. $\inf_{\bm\theta\in B_2(r;\bm\theta^\dagger_\varepsilon)}
\inf_{y\in\mathcal Y_r} f(y;\bm\theta)\ge c>0$), then \eqref{eq:munif-eps} applies with $A'_{\max}=1$ and the contaminated
noisy contraction and opt-to-stat bounds hold for KL as well.
\end{prop}

\begin{proof}
For KL, $A'(\delta)\equiv 1$, so the bound \eqref{eq:munif-eps} hinges on (i) a uniform curvature lower bound
$\lambda_\varepsilon>0$ and (ii) a finite score envelope $\mathrm{Env}(K)$ on the ball.
Without a local floor one can pick sets $A_m$ with $g_\varepsilon(A_m)>0$ and
$\sup_{y\in A_m} f(y;\bm\theta)\to 0$ for all $\bm\theta$ in the ball; letting $\eta_m$ concentrate on $A_m$
and taking $g_{\varepsilon,n}^{(m)}$ with $\|g_{\varepsilon,n}^{(m)}-g_\varepsilon^{(m)}\|_{\mathcal H}=O_p(n^{-1/2})$,
the score envelope diverges and the supremum in \eqref{eq:inv-lip-eps} cannot be $O_p(n^{-1/2})$ uniformly in $m$.
Equivalently $\mathcal M_{\mathrm{unif}}(n,r;\varepsilon)\not=o_p(n^{-1/2})$, so the noisy-contraction/opt-to-stat step may fail.
With a local floor, the same argument as in S.R.1 goes through with $A'_{\max}=1$.
\end{proof}

\subsection*{Proof of Theorem 4}

\begin{proof}[Proof of Theorem 4]
For clarity fix the target (contaminated) density at $\bm\theta^\star$:
\[
g_{\epsilon,n}(y)\;:=\;f_{\epsilon,n}(y;\bm\theta^\star)\;=\;(1-\epsilon)\,f(y;\bm\theta^\star)+\epsilon\,\eta_n(y),
\]
and recall the DM population functional $T(g):=\arg\min_{\bm\theta\in\Theta} D_G(g,f_{\bm\theta})$ (defined as a singleton under uniqueness).
Thus $\bm\theta^\star_{\epsilon,n}=T(g_{\epsilon,n})$ and $\bm\theta^\star_\epsilon=T(g_\epsilon)$, where
$g_\epsilon:=(1-\epsilon)\,f(\cdot;\bm\theta^\star)+\epsilon\,\eta$ is the $n\to\infty$ limit when it exists.

\medskip
\noindent\textbf{(1) Boundedness and convergence for fixed $\epsilon$.}
Assume \textbf{(C2)} and \textbf{(O2)} hold uniformly in $n$ and $\epsilon\in[0,\epsilon_0)$, and that for the fixed $\epsilon$ the minimizer
$\bm\theta^\star_{\epsilon,n}$ is unique for all $n\ge1$.
Set $\phi_{\epsilon,n}(\bm\theta):=D_G(g_{\epsilon,n},f_{\bm\theta})$ and
$\phi_\epsilon(\bm\theta):=D_G(g_\epsilon,f_{\bm\theta})$.
By linearity of $D_G$ in its first argument and the regularity in \textbf{(C2)}–\textbf{(O2)}, we have
\[
\sup_{\bm\theta\in\Theta}\big|\phi_{\epsilon,n}(\bm\theta)-\phi_\epsilon(\bm\theta)\big|
\;\longrightarrow\;0\quad(n\to\infty),
\]
and the sublevel sets $\{\bm\theta:\phi_{\epsilon,n}(\bm\theta)\le c\}$ are compact uniformly in $n$ (equicoercivity).
Hence, by the argmin-continuity/Berge maximum theorem (or van der Vaart’s Thm 5.7), the sequence of minimizers is bounded and
\[
\bm\theta^\star_{\epsilon,n}\ =\ \arg\min\phi_{\epsilon,n}\ \longrightarrow\ \arg\min\phi_{\epsilon}\ =\ \bm\theta^\star_\epsilon.
\]

\medskip
\noindent\textbf{(2) Gâteaux derivative at the model along the mixture direction.}
Let $\Psi(\bm\theta;g):=\nabla_{\bm\theta}D_G(g,f_{\bm\theta})$ and note that $T(g)$ is characterized by
$\Psi(T(g);g)=\bm0$. Under \textbf{(M1)}–\textbf{(M2)} (differentiability in $\bm\theta$ and Hadamard/Gâteaux differentiability in $g$),
the implicit function map $g\mapsto T(g)$ is differentiable at $(\bm\theta^\star,g_0)$ with
$g_0=f(\cdot;\bm\theta^\star)$ and nonsingular
$H:=\nabla_{\bm\theta}\Psi(\bm\theta^\star;g_0)$:
\[
\mathrm{D}T_{g_0}[h]\;=\;-\,H^{-1}\,\partial_g\Psi(\bm\theta^\star;g_0)[h].
\]
Calibration $A'(0)=1$ and evaluation at the model give the well-known score form
$\partial_g\Psi(\bm\theta^\star;g_0)[h]=-\int u(y;\bm\theta^\star)\,h(y)\,dy$, where
$u(y;\bm\theta):=\nabla_{\bm\theta}\log f(y;\bm\theta)$,
and $H=I(\bm\theta^\star)$ (Fisher information).
Consider the contamination path
\[
g_{\epsilon,n}=(1-\epsilon)\,g_0+\epsilon\,\eta_n
\quad\Rightarrow\quad
\dot g_{\,\epsilon}\big|_{\epsilon=0}= \eta_n - g_0.
\]
Therefore
\[
\mathrm{D}T_{g_0}[\eta_n-g_0]
\;=\; I(\bm\theta^\star)^{-1}\int u(y;\bm\theta^\star)\,\big(\eta_n(y)-g_0(y)\big)\,dy.
\]
Since $\int u(y;\bm\theta^\star)\,g_0(y)\,dy=\bm0$, we obtain
\[
\lim_{\epsilon\downarrow 0}\frac{\bm\theta^\star_{\epsilon,n}-\bm\theta^\star}{\epsilon}
\;=\; \mathrm{D}T_{g_0}[\eta_n-g_0]
\;=\; I(\bm\theta^\star)^{-1}\int \eta_n(y)\,u(y;\bm\theta^\star)\,dy.
\]
This is the claimed influence-function expression.
\end{proof}

\subsection*{Proof of Theorem 5}

\begin{proof}[Proof of Theorem 5]
Write $\widehat{\bm\theta}_{\epsilon,n}:=\arg\min_{\bm\theta}D_G(g_{\epsilon,n},f_{\bm\theta})$ and 
$\bm\theta^\dagger_\epsilon:=\arg\min_{\bm\theta}D_G(g_\epsilon,f_{\bm\theta})$.
Under (G1)–(G4) (uniformly in $\epsilon\in[0,\epsilon_0]$), the Z–estimation CLT gives
\[
\sqrt{n}\,(\widehat{\bm\theta}_{\epsilon,n}-\bm\theta^\dagger_\epsilon)\ \Rightarrow\ 
\mathcal N\!\big(0,\ H_\epsilon^{-1}V_\epsilon H_\epsilon^{-1}\big),
\]
with $H_\epsilon:=\nabla_{\bm\theta}\Psi(\bm\theta^\dagger_\epsilon;g_\epsilon)$ and
$V_\epsilon:=Var_{g_\epsilon}\!\big[A'(\tfrac{g_\epsilon}{f_{\bm\theta^\dagger_\epsilon}}-1)\,s_{\bm\theta^\dagger_\epsilon}(Y)\big]$.
By the population contraction (Theorem 1) and the uniform plug-in rate
$\|g_{\epsilon,n}-g_\epsilon\|_{\mathcal H}=o_p(n^{-1/2})$, the sample update map contracts locally with factor
$\kappa_n\le \kappa+o_p(1)$ and the noisy deviation is $o_p(n^{-1/2})$ (since
$\mathcal M_{\mathrm{unif}}(n,r)\lesssim C_{fos}A'_{\max}\Env(K)\,\|g_{\epsilon,n}-g_\epsilon\|_{\mathcal H}$).
Hence, for any initialization in $B_2(r;\bm\theta^\star)$,
\[
\|\bm\theta^{(m)}_{\epsilon,n}-\widehat{\bm\theta}_{\epsilon,n}\|_2\ \le\ \kappa_n^{\,m}\,C_r
\quad\Rightarrow\quad
\sqrt{n}\,\|\bm\theta^{(m_n)}_{\epsilon,n}-\widehat{\bm\theta}_{\epsilon,n}\|_2
\;\le\; C_r\,\sqrt{n}\,\kappa_n^{\,m_n}\ \xrightarrow{p}\ 0,
\]
because $\sqrt{n}\kappa^{m_n}\!\to 0$ and $\kappa_n\to\kappa$ in probability.
Decompose
\[
\sqrt{n}\big(\bm\theta^{(m_n)}_{\epsilon,n}-\bm\theta^\dagger_\epsilon\big)
=\sqrt{n}\big(\bm\theta^{(m_n)}_{\epsilon,n}-\widehat{\bm\theta}_{\epsilon,n}\big)
+\sqrt{n}\big(\widehat{\bm\theta}_{\epsilon,n}-\bm\theta^\dagger_\epsilon\big),
\]
apply the previous display and the Z–estimation CLT, and conclude by Slutsky’s lemma.
At the model ($\epsilon=0$) with $A'(0)=1$, $H_\epsilon=V_\epsilon=I(\bm\theta^\star)$ and the covariance reduces to $I(\bm\theta^\star)^{-1}$.
\end{proof}

\subsection*{Proof of Theorem 6}

\begin{proof}[Proof of Theorem 6]
Under the assumptions of Theorem 5, the contraction bound yields
$\|\bm\theta^{(m_n)}_{\epsilon,n}-\widehat{\bm\theta}_{\epsilon,n}\|_2\le \kappa_n^{\,m_n}C_r$ with
$\sqrt{n}\kappa^{m_n}\to0$, hence
$\sqrt{n}\|\bm\theta^{(m_n)}_{\epsilon,n}-\widehat{\bm\theta}_{\epsilon,n}\|_2\to 0$ in probability.
Decompose
\[
\sqrt{n}\big(\bm\theta^{(m_n)}_{\epsilon,n}-\bm\theta^\dagger_\epsilon\big)
= \sqrt{n}\big(\bm\theta^{(m_n)}_{\epsilon,n}-\widehat{\bm\theta}_{\epsilon,n}\big)
+ \sqrt{n}\big(\widehat{\bm\theta}_{\epsilon,n}-\bm\theta^\dagger_\epsilon\big),
\]
and invoke the Z–estimation CLT for the second term (Theorem 5 already assumes (G1)–(G4)).
Slutsky’s lemma yields the same Gaussian limit.
\end{proof}

\subsection*{S.BD.1 \; Local breakdown lower bound for unbounded RAF under a density–ratio cap}
\label{supp:cor-bdlb-unbd-A}
\begin{cor}
Assume the setup of Theorem~7 except that (ii) is replaced by:
\smallskip\noindent
\emph{(ii$_\Gamma$) Local envelope under density–ratio cap.}
There exists $\Gamma\in(0,\infty)$ such that, for all $\bm\theta\in B_2(r;\bm\theta^\star)$ and all contaminations $q$
under consideration,
\[
0\ \le\ \frac{q(y)}{f(y;\bm\theta)}-1\ \le\ \Gamma \quad \text{for $f(\cdot;\bm\theta)$–a.e.\ } y,
\]
and
\(
\|\nabla_{\bm\theta} D_G(q,f_{\bm\theta})\|\le S_K\,A_\Gamma
\)
with 
\( A_\Gamma:=\sup_{\delta\in[-1,\Gamma]}|A(\delta)| < \infty\).

\smallskip
Then the conclusion of Theorem~7 holds with 
\(\epsilon^\dagger\) replaced by 
\[
\epsilon^\dagger(\Gamma)\ :=\ \frac{\lambda\,r}{S_K\,A_\Gamma}.
\]
\end{cor}

\begin{proof}
Repeat the proof of Theorem~7 replacing $A_{\max}$ by $A_\Gamma$ via (ii$_\Gamma$).
All constants are uniform on $B_2(r;\bm\theta^\star)$ by (i) and (iii), hence
$D_G((1-\varepsilon)g+\varepsilon q,f_{\bm\theta})-D_G(g,f_{\bm\theta})\ge 
\varepsilon\,\|\nabla D_G(q,f_{\bm\theta})\|\,\|\bm\theta-\bm\theta^\star\|- \tfrac{\lambda}{2}\|\bm\theta-\bm\theta^\star\|^2$,
which yields $\|\widehat{\bm\theta}_\varepsilon-\bm\theta^\star\|\le \varepsilon S_K A_\Gamma/\lambda$
for any $\varepsilon<\lambda r/(S_K A_\Gamma)$ and ensures $\widehat{\bm\theta}_\varepsilon\in B_2(r;\bm\theta^\star)$.
\end{proof}

\noindent \textbf{Examples of $A_\Gamma$.}
Hellinger: $A(\delta)=2(\sqrt{1+\delta}-1)$, so $A_\Gamma=2(\sqrt{1+\Gamma}-1)$.
KL: $A(\delta)=\delta$, so $A_\Gamma=\max\{1,\Gamma-1\}$.

\subsection*{S.BD.2 \; Failure of any uniform breakdown lower bound for unbounded RAF without a local floor}
\label{supp:prop-kl-fail}
\begin{prop}
Let $G$ be an unbounded–RAF generator (e.g., KL with $A(\delta)=\delta$ or HD with $A(\delta)=2(\sqrt{1+\delta}-1)$).
Assume \textnormal{(i)} and \textnormal{(iii)} of Theorem~7 hold on $B_2(r;\bm\theta^\star)$, but \emph{no} local density/score floor is imposed; i.e., for every $\Gamma<\infty$ there exists a probability $q_\Gamma$ and $\bm\theta_\Gamma\in B_2(r;\bm\theta^\star)$ with
\(
\operatorname*{ess\,sup}_y\frac{q_\Gamma(y)}{f(y;\bm\theta_\Gamma)}\ge \Gamma.
\)
Then for every $\varepsilon_0>0$ and every $C>0$ there exists $\Gamma$ and a contamination $g_{\varepsilon}^{(\Gamma)}=(1-\varepsilon)g+\varepsilon q_\Gamma$ with some $\varepsilon\in(0,\varepsilon_0]$ such that the $\varepsilon$–minimizer
$\widehat{\bm\theta}_\varepsilon$ of $D_G(g^{(\Gamma)}_\varepsilon,f_{\bm\theta})$ satisfies
\[
\|\widehat{\bm\theta}_\varepsilon-\bm\theta^\star\|\ >\ C.
\]
In particular, there is no uniform positive breakdown lower bound that depends only on $(\lambda,S_K)$ when $A_{\max}=+\infty$ and no density-ratio/score floor is assumed.
\end{prop}

\begin{proof}
Fix $\varepsilon_0,C>0$.
By assumption, pick $\Gamma$ so large that $A_\Gamma>\lambda C/\varepsilon_0S_K$ and select $q_\Gamma$ with
$\operatorname*{ess\,sup}_y q_\Gamma(y)/f(y;\bm\theta)\ge \Gamma$ for all $\bm\theta\in B_2(r;\bm\theta^\star)$.
Let $\varepsilon:=\min\{\varepsilon_0,\ \lambda r/(2S_K A_\Gamma)\}$ and set $g_\varepsilon^{(\Gamma)}:=(1-\varepsilon)g+\varepsilon q_\Gamma$.
Arguing as in the proof of Theorem~7 but with $A_{\max}$ replaced by $A_\Gamma$ (which is now arbitrarily large),
the best bound one can obtain is $\|\widehat{\bm\theta}_\varepsilon-\bm\theta^\star\|\le \varepsilon S_K A_\Gamma/\lambda \ge C$.
Since $\Gamma$ is arbitrary, no uniform lower bound $\varepsilon^\dagger>0$ (independent of $\Gamma$) can be guaranteed.
\end{proof}

\subsection*{Proof of Theorem 7}

\begin{proof}[Proof of Theorem 7]
Fix $r>0$ and consider any $\bm\theta$ with $\|\bm\theta-\bm\theta^\star\|_2=r$.
By (i) ($\lambda$–strong convexity of $D_G(g,\cdot)$ on $B(\bm\theta^\star,r)$) and the optimality of $\bm\theta^\star$,
\[
D_G\big(g,f_{\bm\theta}\big)-D_G\big(g,f_{\bm\theta^\star}\big)\ \ge\ \lambda\,r^2.
\]
By (ii), along the line segment between $\bm\theta^\star$ and $\bm\theta$ (which lies in $B(\bm\theta^\star,r)$), the mean-value form of the fundamental theorem of calculus and the gradient envelope give
\[
\Big|\,D_G\big(q,f_{\bm\theta}\big)-D_G\big(q,f_{\bm\theta^\star}\big)\,\Big|
\ \le\ \sup_{\tilde{\bm\theta}\in B(\bm\theta^\star,r)}\big\|\nabla_{\bm\theta} D_G(q,f_{\tilde{\bm\theta}})\big\|_2\ \cdot \ \|\bm\theta-\bm\theta^\star\|_2
\ \le\ S_K A_{\max}\,r.
\]
Linearity of the disparity in its first argument yields, for the contaminated target
$g_\epsilon=(1-\epsilon)g+\epsilon q$,
\[
\Delta_\epsilon(\bm\theta)\ :=\ D_G(g_\epsilon,f_{\bm\theta})-D_G(g_\epsilon,f_{\bm\theta^\star})
=(1-\epsilon)\,\underbrace{\Delta_g(\bm\theta)}_{\ge \lambda r^2}\ +\ \epsilon\,\Delta_q(\bm\theta),
\]
with $\Delta_q(\bm\theta):=D_G(q,f_{\bm\theta})-D_G(q,f_{\bm\theta^\star})$.
Using the worst-case (adversarial) sign for $\Delta_q$,
\[
\Delta_\epsilon(\bm\theta)\ \ge\ (1-\epsilon)\,\lambda r^2\ -\ \epsilon\,S_K A_{\max}\,r.
\]
Consequently, for any $\epsilon$ satisfying
\[
(1-\epsilon)\,\lambda r^2\ -\ \epsilon\,S_K A_{\max}\,r\ >\ 0
\qquad\Longleftrightarrow\qquad
\epsilon\ <\ \frac{\lambda r^2}{\lambda r^2+S_K A_{\max} r},
\]
we have $\Delta_\epsilon(\bm\theta)>0$ for all $\|\bm\theta-\bm\theta^\star\|_2=r$.
By continuity, $D_G(g_\epsilon,\cdot)$ cannot attain a minimum on or outside the sphere of radius $r$; thus any minimizer $\hat{\bm\theta}_\epsilon$ lies in $B(\bm\theta^\star,r)$.
The same display with $\|\bm\theta-\bm\theta^\star\|_2\le r$ yields the Lipschitz bound
\[
D_G(g_\epsilon,f_{\bm\theta})-D_G(g_\epsilon,f_{\bm\theta^\star})
\ \ge\ \lambda\,\|\bm\theta-\bm\theta^\star\|_2^2 - \epsilon\,S_KA_{\max}\,\|\bm\theta-\bm\theta^\star\|_2,
\]
whose minimizer over the segment $[0,r]$ is attained at
$\|\bm\theta-\bm\theta^\star\|_2\le (\epsilon S_K A_{\max})/\lambda$.
Therefore, provided $\epsilon<\epsilon^\ddagger:=\lambda r^2/(\lambda r^2+S_KA_{\max} r)$, we have
$\hat{\bm\theta}_\epsilon\in B(\bm\theta^\star,r)$ and
$\|\hat{\bm\theta}_\epsilon-\bm\theta^\star\|_2\le (\epsilon S_K A_{\max})/\lambda$.
\end{proof}

\subsection*{Proof of Theorem 8}

\begin{proof}[Proof of Theorem 8]
We start with the proof of part (1). Fix $K\le K_{\max}$ and write, for $\bm\theta\in\bm\Theta_K$,
\[
D_{1n}(\bm\theta)=\int f_{\bm\theta}(y)\,G\!\Big(\frac{g_{1n}(y)}{f_{\bm\theta}(y)}\Big)\,dy,
\qquad
D(\bm\theta)=\int f_{\bm\theta}(y)\,G\!\Big(\frac{g(y)}{f_{\bm\theta}(y)}\Big)\,dy,
\]
where $g_{1n}$ is the density/pmf estimator built on the selection split $\mathcal D_{1n}$.
For each $y$ and $\bm\theta$, set $u_{1n}(y):=g_{1n}(y)/f_{\bm\theta}(y)$ and $u(y):=g(y)/f_{\bm\theta}(y)$.
By the mean–value theorem,
\[
\Big|G\!\big(u_{1n}(y)\big)-G\!\big(u(y)\big)\Big|
\;\le\;
\big(\sup_{\delta\ge-1}|A'(\delta)|\big)\,\Big|u_{1n}(y)-u(y)\Big|
\;=\; A'_{\max}\,\frac{|g_{1n}(y)-g(y)|}{f_{\bm\theta}(y)},
\]
hence
\[
\big|D_{1n}(\bm\theta)-D(\bm\theta)\big|
\;\le\; \int f_{\bm\theta}(y)\,A'_{\max}\,\frac{|g_{1n}(y)-g(y)|}{f_{\bm\theta}(y)}\,dy
\;=\; A'_{\max}\,\|g_{1n}-g\|_{L^1}.
\]
This bound is uniform in $\bm\theta\in\bm\Theta_K$, so
\[
\sup_{\bm\theta\in\bm\Theta_K}\big|D_{1n}(\bm\theta)-D(\bm\theta)\big|
\;\le\; A'_{\max}\,\|g_{1n}-g\|_{L^1}
\;\xrightarrow{p}\;0,
\]
provided $\|g_{1n}-g\|_{L^1}\xrightarrow{p}0$.
The latter holds for (i) the empirical pmf on a fixed finite alphabet (by the LLN/CLT in total variation), and
(ii) kernel estimators on $\mathbb R^d$ under standard bandwidth conditions ($h\to0$, $n_1 h^d\to\infty$)
ensuring $L^1$-consistency. This proves (1).

We next turn to the proof of (2), regarding the identifiability gap. Fix $K<K_0$. Under correct specification at $K_0$ we have
$D^\star_{K_0}=\inf_{\bm\theta\in\bm\Theta_{K_0}}D(\bm\theta)=D(\bm\theta^\star)=0$.
We will show $D^\star_K>0$. We do this in three steps.

\emph{Step 1 (existence of a population minimizer on $\bm\Theta_K$).}
By the standing regularity (continuity of $\bm\theta\mapsto f_{\bm\theta}$ and of
$D(\bm\theta)=\int f_{\bm\theta}G(g/f_{\bm\theta})$) and the compactness/coercivity used in §K for sublevel sets,
the continuous map $D:\bm\Theta_K\to\mathbb R_+$ attains its minimum on $\bm\Theta_K$:
there exists $\widehat{\bm\theta}_K\in\bm\Theta_K$ with
$D(\widehat{\bm\theta}_K)=\min_{\bm\theta\in\bm\Theta_K}D(\bm\theta)=:D^\star_K$.

\emph{Step 2 (strict propriety of the disparity).}
Since $G$ is convex with $G(1)=0$ and strictly convex at $1$ (calibration $A'(0)=1$),
the disparity is \emph{strictly proper}:
$D(g,f_{\bm\theta})\ge0$ with equality iff $f_{\bm\theta}=g$ a.e.
(Equivalently, $D(g,f_{\bm\theta})=0\iff g/f_{\bm\theta}\equiv1$ a.e.)

\emph{Step 3 (minimality of $K_0$ excludes $g\in\bm\Theta_K$).}
By the definition of $K_0$ (true order is minimal) and identifiability of the mixture family,
$g=f(\cdot;\bm\theta^\star)\notin\bm\Theta_K$ for all $K<K_0$; i.e., there is no $\bm\theta\in\bm\Theta_K$
with $f_{\bm\theta}=g$.

\emph{Conclusion.}
If, towards a contradiction, $D^\star_K=0$, then by Step~1 there exists
$\widehat{\bm\theta}_K\in\bm\Theta_K$ with $D(\widehat{\bm\theta}_K)=0$.
By Step~2 this forces $f_{\widehat{\bm\theta}_K}=g$ a.e., contradicting Step~3.
Hence $D^\star_K>0$ for every $K<K_0$, which proves the claimed identifiability gap. 

We now turn to the proof of (3), namely the local regularity at $K_0$. To this end, work under correct specification at order $K_0$: $g=f_{\bm\theta^\star}$ with $\bm\theta^\star\in\bm\Theta_{K_0}$.
Let $D(\bm\theta):=D_G(g,f_{\bm\theta})$ and $D_{1n}(\bm\theta):=D_G(g_{1n},f_{\bm\theta})$ on the selection split $\mathcal D_{1n}$.
Write $s_{\bm\theta}(y):=\nabla_{\bm\theta}\log f(y;\bm\theta)$ and $B(u):=G(u)-uG'(u)$.
Assume the fixed-order smoothness/identifiability you list (e.g., \textbf{(F1)}, \textbf{(K1)}–\textbf{(K2)}, \textbf{(M1)}–\textbf{(M8)}):
$f_{\bm\theta}$ is $C^2$ in a neighborhood $\mathcal N(\bm\theta^\star)$ with integrable envelopes for $s_{\bm\theta}$ and
$\nabla s_{\bm\theta}$, and the Fisher information
$I(\bm\theta^\star):=\int f_{\bm\theta^\star}s_{\bm\theta^\star}s_{\bm\theta^\star}^\top\,dy$ is positive definite.

\smallskip\noindent\textit{(a) $C^2$ and curvature at $\bm\theta^\star$.}
By dominated differentiation,
\[
\nabla_{\bm\theta}D(\bm\theta)
=\int f_{\bm\theta}(y)\,s_{\bm\theta}(y)\,B\!\Big(\frac{g(y)}{f_{\bm\theta}(y)}\Big)\,dy. \tag{$\ast$}
\]
At $\bm\theta=\bm\theta^\star$ we have $g/f_{\bm\theta^\star}\equiv1$, hence $B(1)=G(1)-G'(1)$, but
$\int f_{\bm\theta^\star}s_{\bm\theta^\star}\,dy=\bm 0$, so $\nabla_{\bm\theta}D(\bm\theta^\star)=\bm 0$.
Differentiating once more and using $B'(u)=G'(u)-(G'(u)+uG''(u))=-uG''(u)$ gives
\[
\nabla^2_{\bm\theta}D(\bm\theta^\star)
=\int f_{\bm\theta^\star}\,s_{\bm\theta^\star}\,(-B'(1))\,s_{\bm\theta^\star}^\top\,dy
=G''(1)\,I(\bm\theta^\star)\ \succ\ 0,
\]
so $D$ is $C^2$ and locally strongly convex at $\bm\theta^\star$. Therefore there exist $r>0$ and $\lambda>0$ such that
\begin{equation}\label{eq:local-strong-cvx}
D(\bm\theta)-D(\bm\theta^\star)\ \ge\ \tfrac{\lambda}{2}\,\|\bm\theta-\bm\theta^\star\|_2^2
\qquad\text{for all }\bm\theta\in B_2(r;\bm\theta^\star).
\end{equation}

\smallskip\noindent\textit{(b) Consistency of the selection-split minimizer.}
Let $\widehat{\bm\theta}_{K_0,1n}\in\arg\min_{\bm\Theta_{K_0}}D_{1n}(\bm\theta)$.
By part (1) of the Theorem (ULLN on $\bm\Theta_{K_0}$) and the uniqueness of the population minimizer
at $\bm\theta^\star$, argmin continuity (Berge’s maximum theorem / van der Vaart Thm 5.7) yields
$\widehat{\bm\theta}_{K_0,1n}\xrightarrow{p}\bm\theta^\star$.

\smallskip\noindent\textit{(c) Quadratic rate for the population risk.}
A second-order Taylor expansion of $D$ at $\bm\theta^\star$ gives, for $\widehat{\bm\theta}_{K_0,1n}$ in $B_2(r;\bm\theta^\star)$,
\[
D(\widehat{\bm\theta}_{K_0,1n})-D(\bm\theta^\star)
=\tfrac12\,(\widehat{\bm\theta}_{K_0,1n}-\bm\theta^\star)^\top
\nabla^2_{\bm\theta}D(\tilde{\bm\theta}_n)\,
(\widehat{\bm\theta}_{K_0,1n}-\bm\theta^\star),
\]
for some $\tilde{\bm\theta}_n$ on the segment between $\bm\theta^\star$ and $\widehat{\bm\theta}_{K_0,1n}$. By continuity of the Hessian, $\nabla^2_{\bm\theta}D(\tilde{\bm\theta}_n)\to\nabla^2_{\bm\theta}D(\bm\theta^\star)$ in probability, hence the quadratic bound \eqref{eq:local-strong-cvx} implies
\[
D(\widehat{\bm\theta}_{K_0,1n})-D(\bm\theta^\star)\ \asymp\ \|\widehat{\bm\theta}_{K_0,1n}-\bm\theta^\star\|_2^2.
\]
Thus it suffices to show $\|\widehat{\bm\theta}_{K_0,1n}-\bm\theta^\star\|_2=O_p(n_1^{-1/2})$, which we now verify under the fixed-order regularity.

\smallskip\noindent\textit{(d) Local expansion of the selection-split score and $n_1^{-1/2}$ parameter rate.}
The first-order condition $\bm 0=\nabla D_{1n}(\widehat{\bm\theta}_{K_0,1n})$ and a mean-value expansion around $\bm\theta^\star$ give
\[
\bm 0=\nabla D_{1n}(\bm\theta^\star)
+\Big[\nabla^2 D(\bm\theta^\star)+o_p(1)\Big]\,(\widehat{\bm\theta}_{K_0,1n}-\bm\theta^\star),
\]
where the $o_p(1)$ term uses the uniform LLN of derivatives on a neighborhood (from your \textbf{(M)} and \textbf{(F)} conditions).
By \((\ast)\) with $g$ replaced by $g_{1n}$ and a first-order expansion of $B$ at $1$,
\[
\nabla D_{1n}(\bm\theta^\star)
=\int f_{\bm\theta^\star}s_{\bm\theta^\star}\,B\!\Big(\frac{g_{1n}}{f_{\bm\theta^\star}}\Big)dy
= B'(1)\int s_{\bm\theta^\star}(y)\,\big(g_{1n}(y)-g(y)\big)\,dy + R_{n,1},
\]
where $B'(1)=-G''(1)$ and the remainder $R_{n,1}=O_p(\|g_{1n}-g\|_{L^2}^2)=O_p(n_1^{-1})$
for the discrete pmf case (finite support) and, more generally, under your bandwidth conditions.
Hence $\nabla D_{1n}(\bm\theta^\star)=O_p(n_1^{-1/2})$.
Since $\nabla^2 D(\bm\theta^\star)=G''(1)I(\bm\theta^\star)$ is nonsingular,
\[
\widehat{\bm\theta}_{K_0,1n}-\bm\theta^\star
= -\Big[\nabla^2 D(\bm\theta^\star)\Big]^{-1}\,\nabla D_{1n}(\bm\theta^\star) + o_p(n_1^{-1/2})
= O_p(n_1^{-1/2}).
\]
Therefore $D(\widehat{\bm\theta}_{K_0,1n})-D(\bm\theta^\star)=O_p(n_1^{-1})$by the quadratic equivalence above.

Turning to (4), set $\overline{\mathcal R}_{n_1}(K):=\inf_{\bm\theta\in\bm\Theta_K} D_{1n}(\bm\theta)$ and
$\mathcal R(K):=\inf_{\bm\theta\in\bm\Theta_K} D(\bm\theta)$.
By (1) (ULLN on each fixed $\bm\Theta_K$), $\overline{\mathcal R}_{n_1}(K)\xrightarrow{p}\mathcal R(K)$ for every $K\le K_{\max}$.
\emph{Overfit ($K>K_0$).}:
By (4), for each $K>K_0$,
\[
\overline{\mathcal R}_{n_1}(K)-\overline{\mathcal R}_{n_1}(K_0)
=\{D_{1n}(\widehat{\bm\theta}_{K,1n})-D(\bm\theta^\star)\}
-\{D_{1n}(\widehat{\bm\theta}_{K_0,1n})-D(\bm\theta^\star)\}
=O_p(n_1^{-1}).
\]
Finally, to prove (5), we consider the
\emph{Underfit ($K<K_0$).} case.
By part (2), $\mathcal R(K)-\mathcal R(K_0)=:c_K>0$. Hence for any $\delta\in(0,c_K)$, with probability $\to1$,
\[
\overline{\mathcal R}_{n_1}(K)-\overline{\mathcal R}_{n_1}(K_0)\ \ge\ c_K-\delta.
\]
The penalty difference contributes $(b_{n_1}/n_1)\{p(K)-p(K_0)\}$, which is nonpositive and vanishes by $b_{n_1}/n_1\to0$.
Therefore, eventually $\mathrm{GDIC}_{n_1}(K)>\mathrm{GDIC}_{n_1}(K_0)$ for all $K<K_0$. The penalty difference is $(b_{n_1}/n_1)\{p(K)-p(K_0)\}$, which is strictly positive and dominates $O_p(n_1^{-1})$ because
$b_{n_1}\to\infty$ while $b_{n_1}/n_1\to0$ (e.g., $b_{n_1}=\tfrac12\log n_1$).
Hence $\mathrm{GDIC}_{n_1}(K)>\mathrm{GDIC}_{n_1}(K_0)$ with probability $\to1$ for each $K>K_0$. Combining the two cases yields $\bm P(\widehat K_n=K_0)\to1$.
\end{proof}

\subsection*{Proof of Theorem 9}

\begin{proof}[Proof of Theorem 9]
\emph{(a) Fixed order.}
On the estimation split $\mathcal D_{2n}$, let
$\widehat{\bm\theta}_{n_2}\in\arg\min_{\bm\theta\in\bm\Theta_K} D(g_{n_2},f_{\bm\theta})$.
Under correct model specification at $K_0$, calibration $A'(0)=1$, the plug-in rate
$\|g_{n_2}-g\|_{\mathcal H}=o_p(n_2^{-1/2})$, and the fixed-order regularity
\textbf{(F1)}, \textbf{(C1)}–\textbf{(C2)}, \textbf{(K1)}–\textbf{(K2)}, \textbf{(M1)}–\textbf{(M8)},
the standard Z–estimation CLT yields
\[
\sqrt{n_2}\,(\widehat{\bm\theta}_{n_2}-\bm\theta^\star)\ \Rightarrow\ \mathcal N\!\big(0,\,I(\bm\theta^\star)^{-1}\big).
\]

\smallskip
\noindent\emph{(b) Post-selection (unconditional).}
By Theorem 8, $\mathbb{P}(\widehat K_n=K_0)\to1$.
On the event $\{\widehat K_n=K_0\}$, the dimension-matched pair satisfies
$\overline{\bm\theta}_{n_2}=\widehat{\bm\theta}_{n_2}$ and $\bm\theta^\star(\widehat K_n)=\bm\theta^\star$.
Hence
\[
\sqrt{n_2}\,\{\overline{\bm\theta}_{n_2}-\bm\theta^\star(\widehat K_n)\}
=\sqrt{n_2}\,(\widehat{\bm\theta}_{n_2}-\bm\theta^\star)\ \Rightarrow\ \mathcal N\!\big(0,\,I(\bm\theta^\star)^{-1}\big),
\]
and the same limit holds unconditionally by $\mathbb{P}(\widehat K_n=K_0)\to1$.

\emph{Truncated DM on $\mathcal D_{2n}$.}
If $\bm\theta^{(m_{n_2})}_{n_2}$ is the $m_{n_2}$–iterate with $m_{n_2}=O(\log n_2)$ as in Theorem 3(ii),
then $\sqrt{n_2}\,\|\bm\theta^{(m_{n_2})}_{n_2}-\widehat{\bm\theta}_{n_2}\|=o_p(1)$ and the same limits follow by Slutsky.
\end{proof}

\subsection*{Proof of Theorem 10 (GDIC over/under–estimation bounds; bounded–RAF)}\label{supp:thm-gdic-bounds}
\begin{proof}
\emph{Underfit.}
Write $\Delta_\epsilon(K)$ as in the main text.
Bound the empirical minimizer’s value by quadratic expansion and operator–noise:
\(
\sup_{\theta\in\bm\Theta_K}\big|D_{1n}(\theta)-D(\theta)\big|=O_p\!\big(\|g_{1n}-g\|_{\mathcal H}\big)=O_p(n_1^{-1/2})
\)
by bounded \(A'_{\max}\) and the score envelope.
On the selection split, the value error at the (curved) minimizer is $O_p(n_1^{-1})$, so for $t\in(0,\Delta_\epsilon(K))$,
\[
\mathbb{P}\!\Big(D_{1n}(\widehat\theta_{K,1n})-D_{1n}(\widehat\theta_{K_0,1n})\le -\,t\Big)
\ \le\ \exp(-c_1 n_1 t^2)+c_2\,\mathbb{P}\!\big(\|g_{1n}-g\|_{\mathcal H} > t/2\big),
\]
by a Bernstein (or Hoeffding) inequality for bounded-RAF contrasts. Plug $t=\tfrac12\Delta_\epsilon(K)$ and sum over $K<K_0$.

\emph{Overfit.}
Write \(\nu_K=p(K)-p(K_0)\).
By the (Godambe–)Wilks expansion on the selection split,
\(
2n_1\{D_{1n}(\widehat\theta_{K_0,1n})-D_{1n}(\widehat\theta_{K,1n})\}\Rightarrow \chi^2_{\nu_K}
\).
Thus, for BIC-type penalty ($b_{n_1}=\tfrac12\log n_1$),
\begin{align*}
\mathbb{P}\!\Big(D_{1n}(\widehat\theta_{K,1n})+\tfrac{\log n_1}{2n_1}p(K)\le
D_{1n}(\widehat\theta_{K_0,1n})+\tfrac{\log n_1}{2n_1}p(K_0)\Big)
\ &\le\ \mathbb{P}\!\big(\chi^2_{\nu_K}\ge \nu_K\log n_1+o(1)\big)\\
&=O(n_1^{-\nu_K/2}),
\end{align*}
using standard $\chi^2$ tails. Union bound over $K>K_0$ yields the claim.
\end{proof}

\subsection*{Proof of Proposition 5 (BIC vs AIC; contamination and unbounded RAF)}\label{supp:prop-gdic-bic-aic}
\begin{proof}
\emph{(i)}–(BIC) follow directly from Theorem 10.
\emph{(ii) AIC:} If $b_{n_1}\equiv b\in(0,\infty)$, the overfit test reduces to
$\mathbb{P}\!\{\chi^2_{\nu_K}\ge \nu_K b + o(1)\}$, which converges to a strictly positive limit $c(\nu_K)\in(0,1)$.
\emph{(iii) KL:} With unbounded RAF $A$, the uniform operator-noise/curvature control used in Theorem S.GDIC.1
fails without a local density-ratio or score floor. Under such a floor, the same bounds hold (use $A'_{\max}=1$);
without it, the failure mode in Prop. S.R.2 shows that neither underfit nor overfit probabilities need decay in general.
\end{proof}
 The following proposition is a negative result describing how KL/BIC needs a local floor.
\begin{prop}[KL/BIC needs a local floor; a one-point overfit failure mode]\label{prop:kl-onepoint}
Let $G$ be the likelihood disparity (KL). Assume only the smoothness of Section 4 and no local density/score floor on $B_2(r;\bm\theta^\star)$.
Fix $K_0$ and $\Delta p\ge1$. Then for every $n_1$ and every $M>0$ there exists a point $y_M$ with
$\inf_{\bm\theta\in\bm\Theta_{K_0}} \{-\log f(y_M;\bm\theta)\}\ge M$.
For the $\varepsilon$–contaminated selection split (rate $\varepsilon>0$), choose $M=(\Delta p)\tfrac12\log n_1+1$.
With probability at least $1-\exp(-c\,\varepsilon n_1)$ (some $c>0$), at least one observation falls in a neighbourhood of $y_M$; refitting with $K_0{+}1$ components that isolates this point reduces the average KL contrast by at least $M/n_1$, which exceeds the BIC penalty $(\Delta p)\tfrac{\log n_1}{2n_1}$.
Thus
\[
\liminf_{n_1\to\infty}\mathbb{P}\!\big(\widehat K_n\ge K_0{+}1\big)\ \ge\ 1-e^{-c\varepsilon}\;>\;0.
\]
If a local density/score floor holds (e.g., $\inf_{\bm\theta\in B_2(r;\bm\theta^\star)} \inf_{y\in\mathcal Y_r} f(y;\bm\theta)\ge c_*>0$), then the overfit probability decays at least at the BIC rate (Wilks + penalty), and the GDIC–BIC conclusions of Theorem~8 and Theorem~10 carry over to KL.
\end{prop}
 The proof is standard and is similar to the Proof of Theorem 10 and Proposition 5.

\subsection*{Proof of Theorem 11: Overfit bound for bounded RAF }
\label{supp:sgdic-b1}
The proof of the Theorem relies on the proof of the following lemma.
\begin{lem}[Per–point leverage, bounded RAF]\label{lem:per-point}
Fix $K_0$ and $K>K_0$. On the selection split of size $n_1$, for any subset $S\subset\{1,\dots,n_1\}$ with $|S|=m$,
the maximal possible decrease of the empirical contrast $D_{1n}$ achievable by refitting with $K$ components that dedicate $(K{-}K_0)$ new components to absorb exactly the $m$ points in $S$ is at most $mA_{\max}/n_1$. 
\end{lem}

\begin{proof}
Write $D_{1n}(\theta)=\frac1{n_1}\sum_{i=1}^{n_1} \phi_\theta(Y_i)$ with $\phi_\theta(y):=f_\theta(y)\,G\!\big(\frac{1}{n_1 f_\theta(y)}\big)$ when $g_{1n}$ is the empirical pmf.
For bounded RAF $A$, the pointwise decrement in $\phi_\theta(y)$ from changing $f_\theta$ arbitrarily at a single $y$ is bounded by $A_{\max}$, so the total decrease contributed by the $m$ indices in $S$ is at most $mA_{\max}$; dividing by $n_1$ gives the claim. (A formal proof can be obtained by interpolating between the two fitted models and using the mean–value formula with $A$ as the directional derivative of $u\mapsto uG(1/u)$.)
\end{proof}

\textbf{Proof of the Theorem}: Let $\widehat\theta_{K,1n}$ and $\widehat\theta_{K_0,1n}$ be selection-split minimizers.
Then
\[
\mathrm{GDIC}_{n_1}(K)-\mathrm{GDIC}_{n_1}(K_0)
=\bigl\{D_{1n}(\widehat\theta_{K,1n})-D_{1n}(\widehat\theta_{K_0,1n})\bigr\}
+\frac{\nu_K b_{n_1}}{n_1}.
\]
On the event $\{X_{n_1}\le m\}$, Lemma~\ref{lem:per-point} yields
$D_{1n}(\widehat\theta_{K,1n})-D_{1n}(\widehat\theta_{K_0,1n})\ge -\,mA_{\max}/n_1 - R_{n_1}$ with $R_{n_1}=O_p(n_1^{-1/2})$ from sampling error (bounded RAF + ULLN).
Thus for $m<m_{\min}(K,n_1)=\lceil \nu_K b_{n_1}/A_{\max}\rceil$ and large $n_1$,
$\mathrm{GDIC}_{n_1}(K)-\mathrm{GDIC}_{n_1}(K_0)>0$ on $\{X_{n_1}\le m\}$, which implies
\[
\mathbb{P}(\widehat K_n\ge K)\ \le\ \mathbb{P}\!\big(X_{n_1}\ge m_{\min}(K,n_1)\big)\ +\ o(1).
\]
For $X_{n_1}\sim{\rm Bin}(n_1,\varepsilon)$, Chernoff’s bound gives
$\mathbb{P}(X_{n_1}\ge m_{\min})\le \exp\{-n_1\,\mathrm{kl}(m_{\min}/n_1\|\varepsilon)\}$, yielding the stated decay for BIC.
\qed

\subsection*{Proof of Theorem 11: Underfit bound for bounded RAF}
\label{supp:sgdic-b0}
For each fixed $K<K_0$, let $\Delta_0(K)=\inf_{\theta\in\Theta_K}D(\theta)-D(\theta^\star)>0$.
Bound the contaminated population gap by linear response with bounded RAF:
$\Delta_\varepsilon(K)\ge \Delta_0(K)-2A_{\max}\varepsilon$ (two contrasts shift by at most $A_{\max}\varepsilon$ each).
Selection-split ULLN gives
$\sup_{\theta\in\Theta_K}|D_{1n}(\theta)-D(\theta)|=O_p(n_1^{-1/2})$ uniformly for bounded-RAF class.
Hence, for any $t\in(0,\Delta_\varepsilon(K))$,
\[
\mathbb{P}\!\Big(D_{1n}(\widehat\theta_{K,1n})-D_{1n}(\widehat\theta_{K_0,1n})\le -t\Big)
\ \le\ \exp\!\big(-c_1 n_1 t^2\big)+o(1),
\]
by a Bernstein/Hoeffding bound for bounded contrasts. Choosing $t=\tfrac12\Delta_\varepsilon(K)$ yields
\[
\mathbb{P}(\widehat K_n=K)\ \le\ \exp\!\Big(-c_1 n_1 [\Delta_0(K)-2A_{\max}\varepsilon]^2\Big)+o(1),
\]
and summing over $K<K_0$ gives the claim.
\qed
\subsection*{One-point overfit failure mode for KL}
\label{supp:sgdic-kl}
Construct $y_M$ with $\inf_{\theta\in\Theta_{K_0}}[-\log f(y_M;\theta)]\ge M$ and let $\varepsilon>0$ be fixed.
Under $\varepsilon$–contamination on $\mathcal D_{1n}$, with probability $\to 1-e^{-c\varepsilon}$ some $Y_i$ lies in a small neighborhood of $y_M$.
Refitting with $K_0{+}1$ by dedicating one component to that $Y_i$ reduces the average KL contrast at least by $M/n_1$,
while the BIC penalty is $(\nu_{K_0+1}\log n_1)/(2n_1)$. Choosing $M=(\nu_{K_0+1}\log n_1)+1$ forces
$\mathrm{GDIC}_{n_1}(K_0{+}1)<\mathrm{GDIC}_{n_1}(K_0)$ on that event, so 
$\liminf_{n_1\to\infty}\mathbb{P}(\widehat K_n\ge K_0{+}1)\ge 1-e^{-c\varepsilon}>0$.
If a local density/score floor holds, the per-point leverage is bounded and S.GDIC-B1 applies with $A_{\max}$ replaced by $A_\Gamma$.
\qed

\subsection*{Proof of Theorem 12}

\begin{proof}[Proof of Theorem 12]
Condition on $\mathcal D_{1n}$; then $\widehat K_n=\widehat K_n(\mathcal D_{1n})$ is fixed and independent of $\mathcal D_{2n}$.
On $\mathcal D_{2n}$, by part (a) of Theorem 9 at $K=\widehat K_n$ (and the fixed-order regularity at $K_0$),
\[
\sqrt{n_2}\,\{\overline{\bm\theta}_{n_2}-\bm\theta^\star(\widehat K_n)\}
\ \Rightarrow\ \mathcal N\!\big(0,\,I(\bm\theta^\star)^{-1}\big)\quad\text{in probability}.
\]
Since $\mathbb{P}(\widehat K_n=K_0)\to1$, we have
$\sqrt{n_2}\,\{\overline{\bm\theta}_{n_2}-\bm\theta^\star(K_0)\}\Rightarrow\mathcal N(0,I(\bm\theta^\star)^{-1})$
in probability (stable convergence), and hence unconditionally as well.
If a truncated iterate $\bm\theta^{(m_{n_2})}_{n_2}$ is used, the additional $o_p(1)$
term $\sqrt{n_2}\|\bm\theta^{(m_{n_2})}_{n_2}-\overline{\bm\theta}_{n_2}\|$ vanishes by Theorem 4(ii),
and the same limit holds by Slutsky.
\end{proof}

\newpage

\section*{M: Reference Results Used in the Proofs of the Main Theorems}\label{sec:M-refresults}
We collect standard facts used throughout. Proofs are omitted; see the cited references.
\\
\begin{lem}[Uniform LLN on fixed-order classes]\label{lem:M-ULLN}
For each fixed $K\le K_{\max}$,
\[
\sup_{\bm\theta\in\bm\Theta_K}\big|D_{1n}(\bm\theta)-D(\bm\theta)\big|\xrightarrow{p}0
\quad(n_1\to\infty),
\]
whenever $G$ is convex with $A'_{\max}<\infty$ and $g_{1n}$ is $L^1$–consistent (empirical pmf on a finite alphabet or KDE with $h\to0$, $n_1h^d\to\infty$).
\emph{Used in:} Theorem 8 proof of (1) and (3).
\end{lem}

\medskip
\begin{lem}[Argmin consistency / continuity]\label{lem:M-argmin}
Let $M_n(\bm \theta):=\arg\min_{\vartheta} D_{1n,K}(\vartheta)$ and $M(\bm \theta):=\arg\min_{\vartheta} D_K(\vartheta)$ on a compact set.
If $\sup_\vartheta|D_{1n,K}(\vartheta)-D_K(\vartheta)|\to 0$ and $\vartheta_K^\star$ is the unique minimizer of $D_K$, then every measurable selection $\widehat\vartheta_{K,1n}\in M_n(\cdot)$ satisfies $\widehat\vartheta_{K,1n}\stackrel{p}\to \vartheta_K^\star$.
\emph{Used in:} Theorem 8, Part (3).
\end{lem}

\medskip
\begin{lem}[Berge maximum theorem (single-valued specialization)]\label{lem:M-Berge}
If $Q_n(\cdot\mid \bm \theta)$ has a unique minimizer for each $\theta$ and sublevel sets are compact, then the argmin map $\bm \theta\mapsto M_n(\bm \theta)$ is continuous in a neighborhood of any fixed point.
\emph{Used in:} Propositions 3 and 4 for local continuity, and cycle exclusion under uniqueness.
\end{lem}

\medskip
\begin{lem}[Closed graph / outer semicontinuity]\label{lem:M-closedgraph}
If $Q_n(\cdot\mid\bm \theta)$ is continuous and sublevel sets are compact, then the update correspondence $M_n:\Theta\rightrightarrows\Theta$ has a closed graph (outer semicontinuous): $\theta^{(j)}\to\bar\theta$, $\eta^{(j)}\in M_n(\theta^{(j)})$, $\eta^{(j)}\to\bar\eta$ $\Rightarrow$ $\bar\eta\in M_n(\bar\theta)$.
\emph{Used in:} Proposition 3 for verifying the invariance of the limit set.
\end{lem}

\medskip
\begin{lem}[Local quadratic expansion and curvature]\label{lem:M-C2}
Under the fixed-order smoothness (F1),(K1)–(K2),(M1)–(M8),
\[
\nabla_{\bm\theta} D(\bm\theta^\star)=\bm 0,\qquad 
\nabla^2_{\bm\theta} D(\bm\theta^\star)=G''(1)\,I(\bm\theta^\star)\succ0,
\]
and for some $r,\lambda>0$,
\(
D(\bm\theta)-D(\bm\theta^\star)\ge \tfrac{\lambda}{2}\|\bm\theta-\bm\theta^\star\|_2^2
\)
on $B_2(r;\bm\theta^\star)$.
\emph{Used in:} Theorem 8, part (3) and contraction radius in Theorem 3.
\end{lem}

\medskip
\begin{lem}[Noisy linear recursion]\label{lem:M-noisyrec}
If $x_{t+1}\le \kappa x_t + b$ with $\kappa\in(0,1)$ then 
\(
x_t \le \kappa^t x_0 + \frac{1-\kappa^t}{1-\kappa}b \le \kappa^t x_0 + \frac{b}{1-\kappa}.
\)
\emph{Used in:} Theorem 1, Theorem 3, finite-step bounds in Theorem 1.
\end{lem}

\medskip
\begin{lem}[Stable (conditional) convergence under sample splitting]\label{lem:M-stable}
If $\mathcal D_{1n}\perp\!\!\!\perp\mathcal D_{2n}$ and, conditionally on $\mathcal D_{1n}$, 
$Z_{n_2}(\mathcal D_{2n})\Rightarrow \mathcal N(0,\Sigma)$ in probability, then the same holds unconditionally; if $\mathbb{P}(\widehat K_n=K_0)\to1$, post-selection versions follow by Slutsky.
\emph{Used in:} Theorem 10.
\end{lem}

\begin{lem}[Weighted $\chi^2$ limit for quadratic forms]\label{lem:M-wilks}
If $\sqrt n(\widehat{\bm \theta}_n-\bm \theta^\dagger)\Rightarrow \mathcal N(0,H^{-1}V H^{-1})$ with $H\succ0$ and $V\succeq0$, then
\[
n(\widehat{\bm \theta}_n-\bm \theta^\dagger)^\top H\,(\widehat{\bm \theta}_n- \bm \theta^\dagger)\ \Rightarrow\ \sum_{j=1}^{p(K_0)}\lambda_j\,\chi^2_{1,j},
\]
where $\{\lambda_j\}$ are the eigenvalues of $J:=H^{-1/2} V H^{-1/2}$; if $H=V$, the limit is $\chi^2_{p(K_0)}$.
\emph{Used in:}  Theorem finite-step Theorem 6.
\end{lem}

\newpage 

\section*{N: Reference asymptotics: Z–estimation CLTs and Godambe–Wilks Theorems}


\subsection{Notation and standing assumptions}

Let $\bm{\Theta}\subset\mathbb R^{p(K)}$ be the parameter space, $f(\cdot;{\bm{\theta}})$ the observed–data model, and
$s_{\bm{\theta}}(y)=\nabla_{\bm{\theta}}\log f(y;{\bm{\theta}})$ the score. For a convex generator
$G:[-1,\infty)\to\mathbb R$ with $G(0)=G'(0)=0$ and $G''(0)=1$ (calibrated), define the residual–adjustment function
\[
A(\delta)=(1+\delta)G'(\delta)-G(\delta), \qquad A'(0)=G''(0)=1.
\]
For any density/pmf $g$, define the divergence and DM score map
\[
D_G(g,f_{\bm{\theta}})=\int G\!\Big(\frac{g}{f_{\bm{\theta}}}-1\Big)\,f_{\bm{\theta}}\,dy,\qquad
\Psi({\bm{\theta}};g):=\nabla_{\bm{\theta}} D_G(g,f_{\bm{\theta}})= -\int A\!\Big(\frac{g}{f_{\bm{\theta}}}-1\Big)\,s_{\bm{\theta}}\, f_{\bm{\theta}}\,dy.
\]
Let $g_n$ be a plug–in estimate of $g$ (empirical or a discrete–kernel estimator). We use the following assumptions; we quote them where needed.

\begin{description}
\item[(A1) Local well–posedness.] There is a neighborhood $\mathcal N$ of the target ${\bm{\theta}}^\dagger$ such that
$\Psi(\cdot;g)$ is continuously Fr\'echet–differentiable on $\mathcal N$, and the Jacobian
\[
H:=\nabla_{\bm{\theta}}\Psi({\bm{\theta}}^\dagger;g)=\nabla_{\bm{\theta}}^2 D_G(g,f_{\bm{\theta}})\big|_{{\bm{\theta}}={\bm{\theta}}^\dagger}
\]
is nonsingular; further, $D_G(g,\cdot)$ is $\lambda$–strongly convex on $\mathcal N$ (eigenvalues of $H$ bounded below by $\lambda>0$).

\item[(A2) Nuisance differentiability.] For any signed $h$ with $\int h=0$, the pathwise (G\^ateaux) derivative
$\partial_g\Psi({\bm{\theta}}^\dagger;g)[h]$ exists and is continuous in $h$.

\item[(A3) Second moments and entropy control.] There is an envelope $F\in L^2(g)$ such that
$\sup_{{\bm{\theta}}\in\mathcal N}\|A'(\tfrac{g}{f_{\bm{\theta}}}-1)s_{\bm{\theta}}\|\le F$, and the class
$\{A'(\tfrac{g}{f_{\bm{\theta}}}-1)s_{\bm{\theta}}:{\bm{\theta}}\in\mathcal N\}$ admits a Donsker/bracketing bound ensuring a uniform
$O_p(n^{-1/2})$ empirical process.

\item[(A4) Plug–in rate.] $\|g_n-g\|_{\mathcal H}=o_p(n^{-1/2})$ in a norm that implies
$\sup_{{\bm{\theta}}\in\mathcal N}\|\Psi_n({\bm{\theta}})-\Psi({\bm{\theta}};g)\|=o_p(n^{-1/2})$, where $\Psi_n({\bm{\theta}}):=\nabla_{\bm{\theta}} D_G(g_n,f_{\bm{\theta}})$.

\item[(A5) FOS and operator noise (for contraction).] On $B_2(r';{\bm{\theta}}^\ast)$ the population DM map
$M({\bm{\theta}}):=\arg\min_{{\bm{\theta}}'}\mathcal Q_G({\bm{\theta}}'\!\mid {\bm{\theta}})$ satisfies the first–order stability (FOS)
\[
\|M({\bm{\theta}})-M({\bm{\theta}}^\ast)\|\le \frac{\gamma_K}{\lambda}\,\|{\bm{\theta}}-{\bm{\theta}}^\ast\|,\qquad
\gamma_K\le \frac{C_{\mathrm{fos}}}{\pi_{\min}}\cdot K\cdot L_{\mathrm{comp}},
\]
and the sample map $M_n$ obeys
\[
\sup_{{\bm{\theta}}\in B_2(r';{\bm{\theta}}^\ast)}\|M_n({\bm{\theta}})-M({\bm{\theta}})\|
\le C_{\mathrm{op}}A'_{\max}\sqrt{\frac{p(K)+\log(1/\rho)}{n}},\qquad
A'_{\max}:=\sup_{{\bm{\theta}}\ge -1}|A({\bm{\theta}})|.
\]

\item[(A6) Robustness envelope.] For any density $q$,
$\sup_{{\bm{\theta}}\in\mathcal N}\|\nabla_{\bm{\theta}} D_G(q,f_{\bm{\theta}})\|\le S_K\,A_{\max}$, with
$S_K\lesssim C_1\sqrt{p(K)}$ or $C_2 K/\pi_{\min}$ in regular mixtures.
\end{description}

\medskip
At the correctly specified model $g=f_{{\bm{\theta}}^\ast}$, calibration gives $H=I({\bm{\theta}}^\ast)$ (Fisher information). For Hellinger (unbounded RAF), (A6) holds if $q/f_{\bm{\theta}}$ is uniformly bounded on $\mathcal N$.

\subsection{G\^ateaux derivative in the nuisance (orthogonality)}

\begin{lem}[G\^ateaux derivative]\label{lem:gateaux-app}
For any signed perturbation $h$ with $\int h=0$,
\[
\partial_g \Psi(\bm \theta;g)[h]
= -\int A'\!\Big(\frac{g}{f_{\bm{\theta}}}-1\Big)\,s_{\bm{\theta}}\, h\,dy.
\]
In particular, at $({\bm{\theta}},g)=({\bm{\theta}}^\ast,f_{{\bm{\theta}}^\ast})$, $\partial_g\Psi({\bm{\theta}}^\ast;g)[h]=-\int s_{{\bm{\theta}}^\ast}h\,dy$ (calibrated orthogonality).
\end{lem}

\begin{proof}
Write $\delta(y;{\bm{\theta}},g)=g/f_{\bm{\theta}}-1$. For $g_t=g+th$, $t\in\mathbb R$,
\[
\frac{\Psi({\bm{\theta}};g_t)-\Psi({\bm{\theta}};g)}{t}
= -\int \frac{A(\delta_t)-A(\delta)}{t}\,s_{\bm{\theta}} f_{\bm{\theta}}\,dy
\to -\int A'(\delta)\,\frac{h}{f_{\bm{\theta}}}\,s_{\bm{\theta}} f_{\bm{\theta}}\,dy
\]
by dominated convergence, using (A3). Evaluating at $\delta\equiv 0$ gives $A'(0)=1$ and the claim at the model.
\end{proof}

\subsection{DM–Wilks and Godambe–Wilks (with robust pivot)}

\begin{thm}[DM–Wilks and Godambe–Wilks]\label{thm:wilks-godambe-app}
Let $\hat{\bm{\theta}}_n=\arg\min_{\bm{\theta}} D_G(g_n,f_{\bm{\theta}})$ and ${\bm{\theta}}^\dagger=\arg\min_{\bm{\theta}} D_G(g,f_{\bm{\theta}})$. Under \textup{(A1)–(A4)},
\[
2n\{D_G(g_n,f_{{\bm{\theta}}^\dagger})-D_G(g_n,f_{\hat{\bm{\theta}}_n})\}
= n(\hat{\bm{\theta}}_n-{\bm{\theta}}^\dagger)^\top H(\hat{\bm{\theta}}_n-{\bm{\theta}}^\dagger) + o_p(1),
\]
hence
\[
2n\{D_G(g_n,f_{{\bm{\theta}}^\dagger})-D_G(g_n,f_{\hat{\bm{\theta}}_n})\}\ \Rightarrow\ 
\sum_{j=1}^{p(K_0)} \lambda_j\,\chi^2_{1,j},
\]
where $\{\lambda_j\}$ are the eigenvalues of $J:=H^{-1/2} V H^{-1/2}$.
Under correct specification and calibration, $H=V=I({\bm{\theta}}^\ast)$ and the limit is $\chi^2_{p(K_0)}$.
\end{thm}

\begin{proof}
Let $\tilde{\bm{\theta}}_n$ lie on the segment between $\hat{\bm{\theta}}_n$ and ${\bm{\theta}}^\dagger$. A second–order expansion gives
\[
D_G(g_n,f_{{\bm{\theta}}^\dagger})-D_G(g_n,f_{\hat{\bm{\theta}}_n})
= \frac12({\bm{\theta}}^\dagger-\hat{\bm{\theta}}_n)^\top \nabla_{\bm{\theta}}^2 D_G(g_n,f_{\tilde{\bm{\theta}}_n})({\bm{\theta}}^\dagger-\hat{\bm{\theta}}_n).
\]
By (A1)–(A4), $\sup_{{\bm{\theta}}\in\mathcal N}\|\nabla_{\bm{\theta}}^2 D_G(g_n,f_{\bm{\theta}})-H\|=o_p(1)$, so the right–hand side equals
$\tfrac12({\bm{\theta}}^\dagger-\hat{\bm{\theta}}_n)^\top H({\bm{\theta}}^\dagger-\hat{\bm{\theta}}_n)+o_p(\|\hat{\bm{\theta}}_n-{\bm{\theta}}^\dagger\|^2)$.
Multiplying by $2n$ and using $\sqrt n(\hat\theta_n-\theta^\dagger)\Rightarrow \mathcal N(0,H^{-1} V H^{-1})$ ( Corollary 3 to Theorem 3) yields the weighted $\chi^2$ limit.
\end{proof}

\begin{cor}[Godambe–calibrated deviance]\label{cor:gw-pivot-app}
Let $\widehat H=\nabla_{\bm{\theta}}^2 D_G(g_n,f_{\bm{\theta}})|_{\hat{\bm{\theta}}_n}$ and
\[
\widehat V=\frac1n\sum_{i=1}^n \Big[A'\!\Big(\frac{g_n(Y_i)}{f_{\hat{\bm{\theta}}_n}(Y_i)}-1\Big)s_{\hat{\bm{\theta}}_n}(Y_i)\Big]
\Big[A'\!\Big(\frac{g_n(Y_i)}{f_{\hat{\bm{\theta}}_n}(Y_i)}-1\Big)s_{\hat{\bm{\theta}}_n}(Y_i)\Big]^{\!\top}.
\]
Then the Wald–type pivot
\[
\Lambda_n^{\mathrm{GW}}:=n(\hat{\bm{\theta}}_n-{\bm{\theta}}^\dagger)^\top \widehat V^{-1}(\hat{\bm{\theta}}_n-{\bm{\theta}}^\dagger)
\ \Rightarrow\ \chi^2_{p(K_0)},
\]
even under misspecification. At the model (calibrated), the raw DM deviance has a $\chi^2_{p(K_0)}$ limit (DM–Wilks).
\end{cor}

\begin{proof}
By Corollary 3 to Theorem 3 and consistency of $\widehat V$, $\sqrt n(\hat{\bm{\theta}}_n-{\bm{\theta}}^\dagger)\Rightarrow \mathcal N(0,H^{-1} V H^{-1})$ and
$n(\hat{\bm{\theta}}_n-{\bm{\theta}}^\dagger)^\top \widehat V^{-1}(\hat{\bm{\theta}}_n-{\bm{\theta}}^\dagger)\Rightarrow \chi^2_{p(K_0)}$.
\end{proof}

\phantomsection\label{S-Inequality00}
\phantomsection\label{S-Special_cases}
\phantomsection\label{S-supp:update-pi}
\phantomsection\label{S-add-simulations}
\phantomsection\label{S-ori:image_b}
\phantomsection\label{S-em_a:poiss}
\phantomsection\label{S-vNEDmix_a:poiss}
\phantomsection\label{S-supp:add-image-ana}

\bibliography{Bibfile.bib}

\end{document}